\PassOptionsToPackage{dvipsnames}{xcolor}
\documentclass{article}

\usepackage{amsmath,amsthm,amsfonts}
\usepackage{colonequals}
\usepackage{todonotes}
\usepackage{mathtools}
\usepackage{latexsym}
\usepackage{tikz,tikz-cd}
 \usepackage{quiver}
\usepackage{comment}
\usepackage{tikzit}				
	\usetikzlibrary{decorations.pathreplacing,fit,arrows.meta}	
\usepackage{enumitem}
	\setlist[enumerate]{label=(\roman*)}  
	\setlist[enumerate,2]{label=(\alph*)}  
\usepackage[T1,T2A]{fontenc}			
\usepackage[noadjust]{cite}
\usepackage{tikz-cd}
\usepackage{amssymb,xparse}		
\usepackage{scalerel}
\usepackage{geometry}
\usepackage{bm}				
\usepackage{centernot}
\usepackage{xcolor}
\usepackage[outline]{contour}
\usepackage{authblk}

\colorlet{dgray}{black!45}
\colorlet{cgray}{dgray!64}
\colorlet{mgray}{dgray!37}
\colorlet{lgray}{mgray!40}

\tikzset{
	cbox/.style={
		rectangle, 
		thin, 
		draw=mgray, 
		fill=lgray, 
		inner sep=0mm,
	},
	setlabel/.style={
		anchor=south west,
		font=\scriptsize,
		color=cgray, 
		inner sep=2pt, 
	},
	catlabel/.style={
		anchor=south west,
		font=\small,
		color=cgray, 
		inner sep=3pt, 
	}
}

\definecolor{fillcolor}{HTML}{FFFFFF}
\definecolor{light}{HTML}{FFFFFF}
\definecolor{dark}{HTML}{000000}

\usepackage[draft=false]{hyperref}
\hypersetup{colorlinks,
	linkcolor=black,
	citecolor=dgray,
	urlcolor=dgray}

\tikzstyle{none}=[text=dark]
\tikzstyle{bn}=[text=dark, fill=dark, draw=dark, shape=circle, inner sep=1.5pt]
\tikzstyle{morphism}=[text=dark, fill=light, draw=dark, shape=rectangle, tikzit fill=white]
\tikzstyle{state}=[text=dark, fill=light, draw=dark, regular polygon, regular polygon sides=3, minimum width=0.8cm, shape border rotate=180, inner sep=0pt, tikzit fill=white]
\tikzstyle{horiz state}=[text=dark, fill=light, draw=dark, regular polygon, regular polygon sides=3, minimum height=0.9cm, shape border rotate=90, inner sep=0pt, tikzit fill=white]
\tikzstyle{medium box}=[text=dark, fill=light, draw=dark, shape=rectangle, minimum width=0.7cm, minimum height=0.7cm, tikzit fill=white]
\tikzstyle{medium state}=[text=dark, fill=light, draw=dark, regular polygon, regular polygon sides=3, minimum width=1.3cm, inner sep=0pt, shape border rotate=180, tikzit fill=white]
\tikzstyle{large morphism}=[text=dark, fill=light, draw=dark, shape=rectangle, minimum width=1.7cm, minimum height=1cm, tikzit fill=white]
\tikzstyle{large state}=[text=dark, fill=light, draw=dark, regular polygon, regular polygon sides=3, minimum width=2.2cm, shape border rotate=180, inner sep=0pt, tikzit fill=white]
\tikzstyle{wide state}=[text=dark, fill=light, draw=dark, shape=isosceles triangle, minimum width=0.8cm, shape border rotate=270, inner ysep=5pt, inner xsep=1.4pt, minimum height=0.5cm, isosceles triangle apex angle=80, tikzit fill=white]
\tikzstyle{effect}=[text=dark, fill=light, draw=dark, regular polygon, regular polygon sides=3, minimum width=0.5cm, shape border rotate=0, inner sep=0pt, tikzit fill=white]
\tikzstyle{horiz effect}=[text=dark, fill=light, draw=dark, regular polygon, regular polygon sides=3, minimum width=0.5cm, shape border rotate=270, inner sep=0pt, tikzit fill=white]
\tikzstyle{wn}=[text=dark, fill=light, draw=dark, shape=circle, inner sep=1.5pt, tikzit fill={rgb,255: red,250; green,250; blue,250}]
\tikzstyle{ds morphism}=[text=dark, dash pattern=on 2pt off 2pt, fill=light, draw=dark, shape=rectangle, tikzit fill=white, tikzit draw={rgb,255: red,191; green,191; blue,191}]
\tikzstyle{ds state}=[text=dark, dash pattern=on 2pt off 2pt, fill=light, draw=dark, regular polygon, regular polygon sides=3, minimum width=0.5cm, shape border rotate=180, inner sep=0pt, tikzit fill=white, tikzit draw={rgb,255: red,191; green,191; blue,191}]
\tikzstyle{ds horiz state}=[text=dark, dash pattern=on 2pt off 2pt, fill=light, draw=dark, regular polygon, regular polygon sides=3, minimum width=0.5cm, shape border rotate=90, inner sep=0pt, tikzit fill=white, tikzit draw={rgb,255: red,191; green,191; blue,191}]
\tikzstyle{ds medium box}=[text=dark, dash pattern=on 2pt off 2pt, fill=light, draw=dark, shape=rectangle, minimum width=0.7cm, minimum height=0.7cm, tikzit fill=white]
\tikzstyle{ds horiz wide state}=[text=dark, dash pattern=on 2pt off 2pt, fill=light, draw=dark, shape=isosceles triangle, minimum width=0.8cm, shape border rotate=270, inner ysep=5pt, inner xsep=1.4pt, minimum height=0.5cm, isosceles triangle apex angle=80, tikzit fill=white]
\tikzstyle{ds effect}=[text=dark, dash pattern=on 2pt off 2pt, fill=light, draw=dark, regular polygon, regular polygon sides=3, minimum width=0.5cm, shape border rotate=0, inner sep=0pt, tikzit fill=white, tikzit draw={rgb,255: red,191; green,191; blue,191}]
\tikzstyle{ds horiz effect}=[text=dark, dash pattern=on 2pt off 2pt, fill=light, draw=dark, regular polygon, regular polygon sides=3, minimum width=0.5cm, shape border rotate=270, inner sep=0pt, tikzit fill=white, tikzit draw={rgb,255: red,191; green,191; blue,191}]

\tikzstyle{wire}=[draw=dark]
\tikzstyle{ds}=[-, dash pattern=on 2pt off 2pt, draw=dark, tikzit draw={rgb,255: red,191; green,191; blue,191}]
\tikzstyle{arrow}=[->, draw=dark]
\tikzstyle{mapsto}=[{|->}, draw=dark]
\tikzstyle{curly brace}=[-, draw=dark, decorate, decoration={brace,amplitude=5pt}]
\tikzstyle{dashed box}=[-, dashed, draw=dark]

\usepackage[noabbrev,capitalize]{cleveref}	

\allowdisplaybreaks

\newtheorem{theorem}{Theorem}[section]
\newtheorem{proposition}[theorem]{Proposition}
\newtheorem{lemma}[theorem]{Lemma}
\newtheorem{corollary}[theorem]{Corollary}
\newtheorem{definition}[theorem]{Definition}

\theoremstyle{definition}

\newtheorem{example}[theorem]{Example}

\newtheorem{remark}[theorem]{Remark}

\numberwithin{equation}{section}

\let\originalleft\left
\let\originalright\right
\renewcommand{\left}{\mathopen{}\mathclose\bgroup\originalleft}
\renewcommand{\right}{\aftergroup\egroup\originalright}

\newcommand{\R}{\mathbb{R}}

\newcommand{\cat}[1]{{\mathbf{#1}}} 
\newcommand{\op}{\mathrm{op}}
\newcommand{\id}{\mathrm{id}} 		

\DeclareMathOperator*{\colim}{colim}

\tikzset{pullback/.style={minimum size=1.2ex,path picture={	
			\draw[opacity=1,black,-,#1] (-0.5ex,-0.5ex) -- (0.5ex,-0.5ex) -- (0.5ex,0.5ex);%
}}}
\newcommand{\Yon}{\mathrm{Yon}} 
\newcommand{\El}{\cat{El}} 
\newcommand{\funtimes}{\boldsymbol{\times}} 
\newcommand{\funto}{\boldsymbol{\rightarrow}} 
\newcommand{\ffunto}{\boldsymbol{\hookrightarrow}} 
\newcommand{\profunto}{\boldsymbol{\nrightarrow}} 
\tikzset{functor/.style={line width=0.6pt,-{>[scale=1]}}} 
\tikzset{profunctor/.style={line width=0.6pt,-{>[scale=1]},"\boldsymbol{/}"{description}}} 
\tikzset{virtual/.style={ds,->}} 
\newcommand{\funcat}[2]{\boldsymbol{[}#1,#2\boldsymbol{]}}
\newcommand{\sfuncat}[1]{\funcat{#1}{\cat{Set}}} 
\newcommand{\End}[1]{\displaystyle\int_{#1}} 
\newcommand{\Coend}[1]{\displaystyle\int^{#1}} 

\newcommand{\newterm}[1]{\textbf{#1}}

\title{\bf How to Represent Non-Representable Functors}

\author{Paolo Perrone}
\affil{University of Oxford}
\date{2024}

\begin{document}

\maketitle

\begin{abstract}
	The arrows of a category are elements of particular sets, the hom-sets. These sets are functorial, and their functoriality specifies how to compose the arrows with other arrows of the same category. In particular, it allows to form diagrams, making many abstract concepts graphically visible.
	
	Presheaves and set-valued functors, in general, are not representable, and so their elements are not arrows in the usual sense. They can however still be seen as ``arrow-like structures'', which can be post-composed but not pre-composed (for the case of set functors), or pre-composed but not post-composed (for the case of presheaves). 
	Therefore, we can still represent their structure graphically. 
	
	In this exposition we show how to draw and interpret these generalized diagrams, and how to use them to prove theorems. 
	We will then study in detail, and represent graphically, a few concepts of category theory which are often considered hard to visualize: representability, weighted limits and colimits, and Cauchy completion (for unenriched categories). 
	
	We also sketch how to interpret the more general case of profunctors, and the Day convolution of presheaves.
\end{abstract}

{
\tableofcontents
}

\section{Introduction}

One of the most useful features of category theory and related fields is \emph{diagrams}, and the technique of \emph{diagram chasing}. 
Diagrams are a way to make abstract relations graphically evident, and to employ our visual or spatial skills to aid our reasoning.

The main idea behind diagrams is to \emph{compose things of a certain class to create more complex things of the same class}. 
Many structures in mathematics are instances of this pattern, for example groups, monoids, and categories. Diagrams for these structures, and for their refinements, such as monoidal categories, arise naturally from this idea.
\[
\begin{tikzcd}
	X \ar{r}{f} & Y \ar{r}{g} & Z
\end{tikzcd}
\]

However, these are not the only type of structure that we use in abstract mathematics. Of at least equal importance we have the more general idea of \emph{things of type $A$ which act on things of type $B$ to give us new things of type $B$}. Group representations, monoid actions and modules are instances of this pattern. 
One may then wonder, \emph{can we draw these as well?}

The good news is that we can.
However, diagrams for this type of structures are less known and less established. Because of that, while they are useful, they tend to be less accessible to newcomers. 
This exposition aims at bridging this gap, with two canonical examples of such structures: set-valued functors (which we will simply call \emph{set functors}) and presheaves.

The main idea is that \emph{we can represent graphically any functorial action on sets}, and not just the ones given by categorical composition.

For example, given a functor $F:\cat{C}\funto\cat{Set}$, we have sets $FX$ which depend functorially on an object $X$.
This is similar to what happens with the hom-functor, with entries $\cat{C}(A,X)$, but this time we have a single argument, $X$.
In a certain sense, the functorial action of $F$ is similar to the one of $\cat{C}(A,-):\cat{C}\funto\cat{Set}$, in its second variable.
This suggests to interpret the elements $f\in FX$ as of being \emph{like arrows, but which can only be post-composed, not pre-composed}. A possible way to draw them is then as arrows to the object $X$, coming from a ``virtual object'' which is not part of the category:
\[
\begin{tikzcd}
	\bullet \ar[virtual]{r}{f} & X
\end{tikzcd}
\]
The functorial action on a morphism $g:X\to Y$ can then be represented as follows:
\[
\begin{tikzcd}
	\bullet \ar[virtual]{r}{f} & X
\end{tikzcd}
\qquad{\color{dgray}\longmapsto}\qquad
\begin{tikzcd}
	\bullet \ar[virtual]{r}{f} & X \ar{r}{g} & Y
\end{tikzcd}
\]
Presheaves can be represented similarly, and dually. 

While this idea is simple, and while it is a straightforward generalization of usual categorical composition, it seems to be less standard both in the literature and in introductions. 
A reason could be that for many basic concepts of category theory, the usual diagrams are enough.
There are however concepts, such as weighted limits and Cauchy completion, which \emph{do} benefit from these more general diagrams. Indeed, these structures are rather hard to ``draw'' otherwise, and probably because of that, they are often considered much harder to understand than, for example, ordinary limits. 
Still, they are quite important, especially as stepping stones towards their generalization to the enriched context.

In this exposition we will address precisely these two concepts, weighted limits and Cauchy completion. We will first establish our more general diagrammatic formalism (\Cref{diagrams}) and start by representing universal properties. We will then define and study weighted limits and colimits (\Cref{wlim}), and turn to Cauchy completions in their many forms (\Cref{cauchy}). Finally, at the end, we will give some intuition on how to use this technique in more advanced situations, such as with profunctors, and with the Day convolution (\Cref{further}).

\paragraph{Where these ideas have appeared before.}
The idea of set functors and presheaves as arrows from and to ``extra objects'' has always been present in the category theory folklore at least since Grothendieck, and has had a strong influence on higher category theory, see for example \cite{htt} and the nLab page on \emph{motivation for sheaves, cohomology and higher stacks}.\footnote{\url{https://ncatlab.org/nlab/show/motivation+for+sheaves\%2C+cohomology+and+higher+stacks}}

The first recorded use of \emph{diagram chasing} with these ideas seems to be \cite{pareigis}, see Section 2.2 therein (where what here we call  ``virtual arrows'' are depicted as normal arrows). 
The first explicit, systematic use of these techniques in applications seems to be Leinster's work on self-similarity \cite{leinster-similarity1,leinster-similarity2,leinster-similarity3,leinster-similarity4} (where   ``virtual arrows'' are written with the symbol $\nrightarrow$). 
The most recent one seems to be Di Liberti and Rogers' work \cite{topoiwithenoughpoints} (where ``virtual arrows'' are depicted as crossing a dotted line).

A deep analysis of these ideas and their pervasive, yet elusive presence in the category theory literature can be found in Ellerman's writings \cite{ellerman1,ellerman2,ellerman3} (where ``virtual arrows'' are called ``chimera morphisms'' or ``heteromorphisms''---see \Cref{profunctors}---and are sometimes written with the symbol $\Rightarrow$). 

Finally, I learned from personal communication that in a few category theory courses taught around the world, for example in Cambridge and in Tallinn, this point of view is sometimes used, at least informally, to describe concepts such as representability and adjunctions.

\paragraph{Prerequisites.}
The material presented here should be accessible to anyone with a background in basic category theory, as contained for example in \cite{basiccats}, \cite{riehl2016category}, or \cite{startingcats}. 
Knowledge of some enriched category theory may give an additional perspective, but it is not required. 

\paragraph{Acknowledgements.}
I learned many of the ideas presented here through fruitful discussions with Bartosz Milewski, David Jaz Myers, Mario Román, Walter Tholen, and with the ItaCa collective, in particular Ivan Di Liberti and Fosco Loregian.

I also want to thank Nathanael Arkor, John Baez, David Corfield, Tom Leinster, Morgan Rogers, and Mike Shulman, for pointing out literature sources where these ideas are used.

This work was developed to provide infrastructure for my work on probability, funded at the time of writing by Sam Staton's ERC grant ``BLaSt -- A Better Language for Statistics''. 

\section{Diagrams for set functors and presheaves}\label{diagrams}

Throughout this paper, let $\cat{C}$ be a locally small category.

Consider a set functor on $\cat{C}$, i.e.\ a functor $F:\cat{C}\funto\cat{Set}$. On objects, it maps an object $A$ of $\cat{C}$ to a set, which we denote $FA$. On morphisms, it maps a morphism $g:A\to B$ to a function $FA\to FB$ which we denote as follows.
\[
\begin{tikzcd}[row sep=0]
	FA \ar{r}{g_*} & FB \\
	\color{dgray} f \ar[mapsto,dgray ]{r} & \color{dgray} g_*f
\end{tikzcd}
\]
Particular set functors are the \emph{representable} ones. Those are functors $\cat{C}\funto\cat{Set}$ in the form (up to natural isomorphism) 
\[
\begin{tikzcd}[row sep=0]
	\cat{C}(R,A) \ar{r}{g\circ-} & \cat{C}(R,B) \\
	\color{dgray} \left(R\xrightarrow{f} A\right) \ar[mapsto,dgray]{r} & \color{dgray} \left(R\xrightarrow{f} A\xrightarrow{g} B\right)
\end{tikzcd}
\]
for some object $R$ of $\cat{C}$.

Similarly, and dually, consider a presheaf $P:\cat{C}^\op\funto\cat{Set}$. On objects, it again maps $A$ to a set $PA$. On morphisms, it maps $g:A\to B$ to a function in the opposite direction, $g^*:PB\to PA$, which we denote as follows.
\[
\begin{tikzcd}[row sep=0]
	PB \ar{r}{g^*} & PA \\
	\color{dgray} p \ar[mapsto,dgray]{r} & \color{dgray} g^*f
\end{tikzcd}
\]
Particular presheaves are the \emph{representable} ones. Those are presheaves $\cat{C}^\op\funto\cat{Set}$ in the form (up to natural isomorphism)
\[
\begin{tikzcd}[row sep=0]
	\cat{C}(B,R) \ar{r}{-\circ g} & \cat{C}(A,R) \\
	\color{dgray} \left(B\xrightarrow{p} R\right) \ar[mapsto,dgray]{r} & \color{dgray} \left(A\xrightarrow{g} B\xrightarrow{p} B\right)
\end{tikzcd}
\]
for some object $R$ of $\cat{C}$.

In other words: 
\begin{itemize}
	\item In general, set functors and presheaves map objects to generic sets, and morphisms to generic functions (in a functorial way).
	\item \emph{Representable} set functors and presheaves map objects to \emph{sets of arrows of $\cat{C}$}, and morphisms to the operations of pre- and postcomposition of these arrows.
\end{itemize}
The main idea that we want to present here is to \emph{interpret all set functors and presheaves in terms of ``arrows'', but of a more general kind}.

\subsection{Virtual objects and virtual arrows}\label{virtual}

Let $R$ be an object of $\cat{C}$, and consider the representable functor $\cat{C}(R,-):\cat{C}\funto\cat{Set}$. 
We can visualize its action as a \emph{cone-like diagram over $\cat{C}$} in the following way. 
\begin{itemize}
	\item On objects, it assigns to an object $A$ to the set of all arrows $R\to A$:
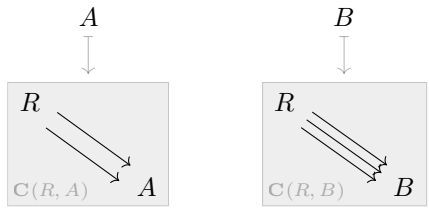
	\[
	\begin{tikzcd}[row sep=small, column sep=tiny,
		blend group=multiply,
		/tikz/execute at end picture={
			\node [cbox, fit=(A) (R1)] (FA) {};
			\node [setlabel] at (FA.south west) {$\cat{C}(R,A)$};
			\node [cbox, fit=(B) (R2)] (FB) {};
			\node [setlabel] at (FB.south west) {$\cat{C}(R,B)$};
		}]
		& A \ar[mapsto, shorten >=3.5mm, color=cgray]{dd} 
		&&&&&&&& B \ar[mapsto, shorten >=3.5mm , color=cgray]{dd} \\ \\ 
		|[alias=R1]| R \ar[shift left=1.5]{ddrr} 
		\ar[shift right=1.5]{ddrr} 
		&\; &&&&&&&
		|[alias=R2]| R \ar[shift left=1.5]{ddrr} 
		\ar{ddrr}
		\ar[shift right=1.5]{ddrr}  &\; \\ \\
		&& |[alias=A]| A &&&&&&&& |[alias=B]| B
	\end{tikzcd}
	\]
	(Here we draw sets of two or three arrows, but these sets could also be empty, or infinite, etc.)
	\item On morphisms, it assigns to an arrow $g:A\to B$ the postcomposition function:
	
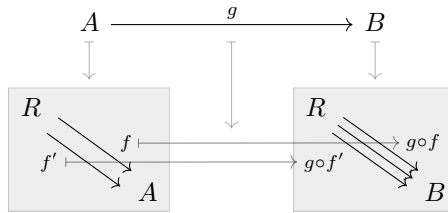
\[
	\begin{tikzcd}[row sep=small, column sep=tiny,
		blend group=multiply,
		/tikz/execute at end picture={
			\node [cbox, fit=(A) (R1)] (FA) {};
			\node [cbox, fit=(B) (R2)] (FB) {};
		}]
		& A \ar[mapsto, shorten >=3.5mm, color=cgray]{dd} \ar{rrrrrrrr}{g} 
		&&&& \; \ar[mapsto, shorten >=-3mm, color=cgray]{dd} 
		&&&& B \ar[mapsto, shorten >=3.5mm , color=cgray]{dd} \\ \\ 
		|[alias=R1]| R \ar[shift left=1.5, "f"{near end, name=F}]{ddrr} 
		\ar[shift right=1.5, "f'"'{near start, name=FP}]{ddrr} 
		&\; &&&& \; &&& 
		|[alias=R2]| R \ar[shift left=1.5, "g\circ f"{near end, name=GF}]{ddrr} 
		\ar{ddrr}
		\ar[shift right=1.5, "g\circ f'"'{near start, name=GFP}]{ddrr}  &\; \\ \\
		&& |[alias=A]| A &&&&&&&& |[alias=B]| B
		\ar[mapsto, color=dgray, from=F, to=GF]
		\ar[mapsto, color=dgray, from=FP, to=GFP]
	\end{tikzcd}
	\]
	(Not every arrow $R\to B$ may be in the form $g\circ f$ for some $f:R\to A$. In the diagram above, and in the examples that follow, we assume that the unlabeled middle arrow $R\to B$ is not of this form.)
\end{itemize} 
The postcomposition function, equivalently, is a way of telling us which triangles in the diagram below commute:
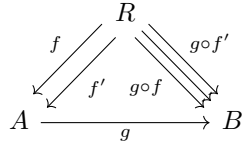
\[
 \begin{tikzcd}[row sep=large]
 	& R \ar[shift right=1.5]{dl}[swap]{f} 
 	\ar[shift left=1.5]{dl}{f'} 
 	\ar[shift right=1.5]{dr}[swap]{g\circ f} 
 	\ar{dr}
 	\ar[shift left=1.5]{dr}{g\circ f'} \\
 	A \ar{rr}[swap]{g} && B
 \end{tikzcd}
\]
For example, the triangle formed by $f$, $g$ and $g\circ f$ commutes. (The unlabeled arrow does not make any triangle with $g$ commute.)
 
In general there may be several arrows $A\to B$. So there are many possible triangles as follows, and the functor $\cat{C}(R,-):\cat{C}\funto\cat{Set}$ tells us which ones are commutative.
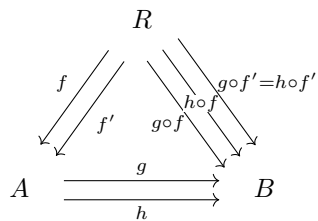
\[
\begin{tikzcd}[row sep=5em, column sep=3em]
	& R \ar[shift right=1.5, shorten=3mm]{dl}[swap]{f} 
	\ar[shift left=1.5, shorten=3mm]{dl}{f'} 
	\ar[shift right=3, shorten=2mm]{dr}[swap,inner sep=0mm]{g\circ f} 
	\ar[shift right=0, shorten=2mm, "h\circ f"{description, inner sep=0}]{dr}
	\ar[shift left=3, shorten=2mm]{dr}[inner sep=0mm]{g\circ f'=h\circ f'} \\
	A \ar[shift left=1.5, shorten=3mm]{rr}{g}
	\ar[shift right=1.5, shorten=3mm]{rr}[swap]{h} && B
\end{tikzcd}
\]
(For example, in the diagram above, there are four commutative triangles.)

So we can view the representable functor $\cat{C}(R,-):\cat{C}\funto\cat{Set}$ as forming a large conical diagram where the base of the cone is the entire category $\cat{C}$, the vertex of the cone is the object $R$, and sets of arrows connect the vertex to the objects at the base. 
Note that:
\begin{itemize}
	\item In a traditional cone, one usually wants a single arrow from the tip to each point of the base. Our conical diagram is more general, there is an entire \emph{set} instead (possibly infinite, possibly empty);
	\item In a traditional cone, one usually wants all triangles to commute. Here instead the rules for which triangles commute are specified by the functor (or equivalently, by composition of arrows in $\cat{C}$); 
	\item The object $R$ appears both at the vertex and at the base, and the identity is one of the connecting arrows.
\end{itemize}

Similarly, and dually, consider a representable presheaf $\cat{C}(-,R):\cat{C}^\op\funto\cat{Set}$. 
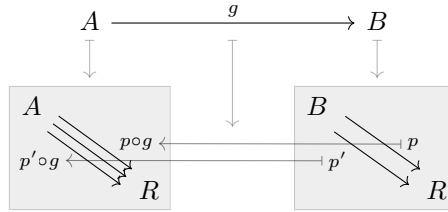
\[
\begin{tikzcd}[row sep=small, column sep=tiny,
	blend group=multiply,
	/tikz/execute at end picture={
		\node [cbox, fit=(A) (R1)] (FA) {};
		\node [cbox, fit=(B) (R2)] (FB) {};
	}]
	& A \ar[mapsto, shorten >=3.5mm, color=cgray]{dd} \ar{rrrrrrrr}{g} 
	&&&& \; \ar[mapsto, shorten >=-3mm, color=cgray]{dd} 
	&&&& B \ar[mapsto, shorten >=3.5mm , color=cgray]{dd} \\ \\ 
	|[alias=R1]| A \ar[shift left=1.5, "p\circ g"{near end, name=F}]{ddrr} 
	\ar{ddrr}
	\ar[shift right=1.5, "p'\circ g"'{near start, name=FP}]{ddrr} 
	&\; &&&& \; &&& 
	|[alias=R2]| B \ar[shift left=1.5, "p"{near end, name=GF}]{ddrr} 
	\ar[shift right=1.5, "p'"'{near start, name=GFP}]{ddrr}  &\; \\ \\
	&& |[alias=A]| R &&&&&&&& |[alias=B]| R
	\ar[mapsto, color=dgray, from=GF, to=F]
	\ar[mapsto, color=dgray, from=GFP, to=FP]
\end{tikzcd}
\]
We can view it as forming an upside-down conical shape (``co-conical'') as follows.
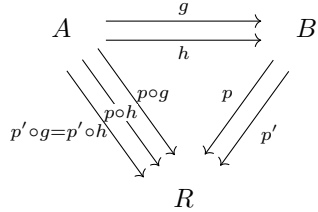
\[
\begin{tikzcd}[row sep=5em, column sep=3em]
	A \ar[shift left=1.5, shorten=3mm]{rr}{g}
	\ar[shift right=1.5, shorten=3mm]{rr}[swap]{h} 
	\ar[shift left=3, shorten=2mm]{dr}[inner sep=0mm]{p\circ g}
	\ar[shift left=0, shorten=2mm, "p\circ h"{description, inner sep=0}]{dr}
	\ar[shift right=3, shorten=2mm]{dr}[swap,inner sep=-0.4mm]{p'\circ g=p'\circ h}
	&& B \ar[shift right=1.5, shorten=3mm]{dl}[swap]{p} 
	\ar[shift left=1.5, shorten=3mm]{dl}{p'} \\
	& R 
\end{tikzcd}
\]
Once again, the action on morphisms (precomposition) specifies which triangles commute.

Let's now introduce the main idea of this work: we can interpret non-representable set functors as forming cones like the ones above, but where \emph{the tip is not an object of our category}. 
We will call it a \emph{virtual object}.\footnote{This is not related to \emph{virtual double categories} as far as I know. Since I'm not using any double categories here, this terminology should not cause any confusion.}
More precisely, let $F:\cat{C}\funto\cat{Set}$ be any set functor. For every object $A$, $FA$ is a set. We draw the elements of the set $FA$ as dashed arrows as follows,
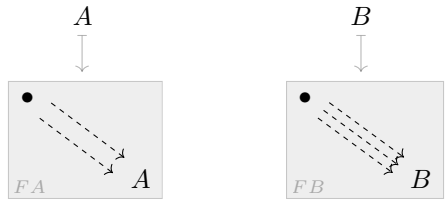
\[
\begin{tikzcd}[row sep=small, column sep=tiny,
	blend group=multiply,
	/tikz/execute at end picture={
		\node [cbox, fit=(A) (R1)] (FA) {};
		\node [setlabel] at (FA.south west) {$FA$};
		\node [cbox, fit=(B) (R2)] (FB) {};
		\node [setlabel] at (FB.south west) {$FB$};
	}]
	& A \ar[mapsto, shorten >=2.8mm, color=cgray]{dd}
	&&&&&&&& B \ar[mapsto, shorten >=2.8mm , color=cgray]{dd} \\ \\ 
	|[alias=R1]| \bullet \ar[virtual, shift left=1.5]{ddrr} 
	\ar[virtual, shift right=1.5]{ddrr} 
	&\; &&&& \; &&& 
	|[alias=R2]| \bullet \ar[virtual, shift left=1.5]{ddrr} 
	\ar[virtual]{ddrr}
	\ar[virtual, shift right=1.5]{ddrr}  &\; \\ \\
	&& |[alias=A]| A &&&&&&&& |[alias=B]| B
\end{tikzcd}
\]
and call them \emph{virtual arrows} from the \emph{virtual object} $\bullet$ to $A$.
(When one is considering many functors at once, one may want to distinguish the virtual point of each functor by labeling it for example as $\bullet_F$. See also \Cref{catpsh}.)

On morphisms, $F$ assigns to an arrow $g:A\to B$ a function as follows:
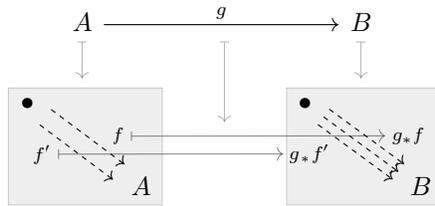
\[
\begin{tikzcd}[row sep=small, column sep=tiny,
	blend group=multiply,
	/tikz/execute at end picture={
		\node [cbox, fit=(A) (R1)] (FA) {};
		\node [cbox, fit=(B) (R2)] (FB) {};
	}]
	& A \ar[mapsto, shorten >=2.8mm, color=cgray]{dd} \ar{rrrrrrrr}{g} 
	&&&& \; \ar[mapsto, shorten >=-3mm, color=cgray]{dd} 
	&&&& B \ar[mapsto, shorten >=2.8mm , color=cgray]{dd} \\ \\ 
	|[alias=R1]| \bullet \ar[virtual, shift left=1.5, "f"{near end, name=F}]{ddrr} 
	\ar[virtual, shift right=1.5, "f'"'{near start, name=FP}]{ddrr} 
	&\; &&&& \; &&& 
	|[alias=R2]| \bullet \ar[virtual, shift left=1.5, "g_* f"{near end, name=GF}]{ddrr} 
	\ar[virtual]{ddrr}
	\ar[virtual, shift right=1.5, "g_* f'"'{near start, name=GFP}]{ddrr}  &\; \\ \\
	&& |[alias=A]| A &&&&&&&& |[alias=B]| B
	\ar[mapsto, color=dgray, from=F, to=GF]
	\ar[mapsto, color=dgray, from=FP, to=GFP]
\end{tikzcd}
\]
We can view it as a \emph{rule to post-compose a virtual arrow into $A$ with $g$, and obtain a virtual arrow into $B$}:
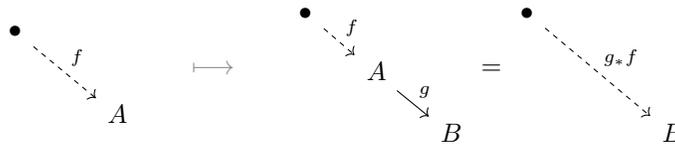
\[
\begin{tikzcd}
	\bullet \ar[virtual]{dr}{f} \\ & A
\end{tikzcd}
\qquad{\color{dgray}\longmapsto}\qquad
\begin{tikzcd}[sep=small]
	\bullet \ar[virtual]{dr}{f} \\ & A \ar{dr}{g} \\ && B
\end{tikzcd}
=
\begin{tikzcd}[sep=small]
	\bullet \ar[virtual]{ddrr}{g_*f} \\ & \phantom{A} \\ && B
\end{tikzcd}
\]

In other words, the virtual arrows are things that we can postcompose (with arrows of $\cat{C}$), but not precompose.
Just as for representable functors, the action of $F$ on morphism tells us which triangles in the following diagram commute:
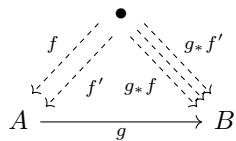
\[
\begin{tikzcd}[row sep=large]
	& \bullet \ar[virtual, shift right=1.5]{dl}[swap]{f} 
	\ar[virtual, shift left=1.5]{dl}{f'} 
	\ar[virtual, shift right=1.5]{dr}[swap]{g_* f} 
	\ar[virtual]{dr}
	\ar[virtual, shift left=1.5]{dr}{g_* f'} \\
	A \ar{rr}[swap]{g} && B
\end{tikzcd}
\]
In summary, can view a set functor as forming a large conical diagram, which we call \emph{virtual cone}, where the base of the cone is the entire category $\cat{C}$, the vertex of the cone is the virtual object $\bullet$, and sets of virtual arrows connect the vertex to the objects at the base:
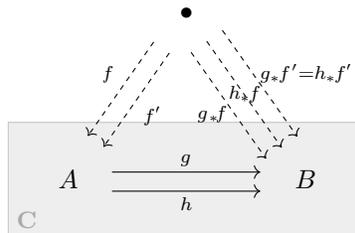
\[
\begin{tikzcd}[row sep=5em, column sep=3em,
	blend group=multiply,
	/tikz/execute at end picture={
		\node [cbox, fit=(A) (B), inner sep=5mm] (C) {};
		\node [catlabel] at (C.south west) {$\cat{C}$};
	}]
	& \bullet \ar[virtual, shift right=1.5, shorten=3mm]{dl}[swap]{f} 
	\ar[virtual, shift left=1.5, shorten=3mm]{dl}{f'} 
	\ar[virtual, shift right=3, shorten=2mm]{dr}[swap,inner sep=0mm]{g_*\!f} 
	\ar[virtual, shift right=0, shorten=2mm, "h_*\!f"{description, inner sep=0}]{dr}
	\ar[virtual, shift left=3, shorten=2mm]{dr}[inner sep=0mm]{g_*\!f'=h_*\!f'} \\
	|[alias=A]| A \ar[shift left=1.5, shorten=3mm]{rr}{g} 
	\ar[shift right=1.5, shorten=3mm]{rr}[swap]{h} && |[alias=B]| B 
\end{tikzcd}
\]
The functor action on morphisms tells us which triangles in the cone commute. 
Note moreover that by functoriality of $F$,
\begin{itemize}
	\item For all $A$ and $f\in FA$, $(\id_A)_*f=f$, meaning that identities indeed behaves like (left) identities on virtual arrows;
	\item For all $g:A\to B$, $g:B\to C$ and $f\in FA$, $(g'\circ g)_*f=g'_*(g_*f)$, meaning that the composition between real and virtual arrows is associative. 
\end{itemize}
In other words, virtual arrows behave ``as much as possible like actual arrows of a category''.
Possible interpretations are:
\begin{itemize}
	\item The virtual object ``behaves like an object of $\cat{C}$, but we don't have access to it, we can only see its image projected onto the other objects''.
	\item The virtual object is ``an object we would like to add to $\cat{C}$ \emph{from the left}, or \emph{from above}, or \emph{from the in-direction},  without breaking the category structure''.
	\item The virtual object, together with its arrows, is an ``entry point to $\cat{C}$'', or ``a \emph{consistent} way to arrive into some object of $\cat{C}$ from elsewhere''.  
\end{itemize}
\emph{Consistent} means, for example, that we can enter $A$ using the virtual arrow $f$ and then move to $B$ using the real arrow $g$, or we can equivalently enter $B$ directly, using the virtual arrow $g_* f$, and these two ways are equivalent.

In order to make rigorous arguments about virtual arrows, it is helpful, in what follows, to promote $\bullet$ to an actual object of a category, an extension of $\cat{C}$. To do so, we have to add its identity map (the functor $F$ does not specify any arrows \emph{into} $\bullet$). So let's define the following category.
\begin{definition}\label{CplusF}
	Let $\cat{C}$ be a category, and let $F:\cat{C}\funto\cat{Set}$ be a functor. 
	The category $\cat{C}^{+F}$ has 
	\begin{itemize}
		\item As objects, the ones of $\cat{C}$ plus an extra object $E$;
		\item As morphisms between the objects coming from $\cat{C}$, the ones of $\cat{C}$;
		\item As morphisms $E\to A$, where $A$ is an object of $\cat{C}$, the elements of $FA$ (``virtual arrows'');
		\item The identity of $E$ as the unique morphism with codomain $E$.
	\end{itemize}
	The composition of morphisms of $\cat{C}$ is as in $\cat{C}$, and between morphisms $E\to A$ and $A\to B$ is as specified by functoriality of $F$.
\end{definition}
In $\cat{C}^{+F}$, all arrows are real arrows, of course. But it is still helpful to keep track of which arrows come from $\cat{C}$ and which from $F$. Therefore we will draw morphisms out of $E$ still as dashed, and refer to them still as ``virtual arrows''. 

Let's now turn to presheaves. We can represent them similarly, and dually, as \emph{virtual co-cones}, where the virtual arrows now point \emph{to} a virtual object:
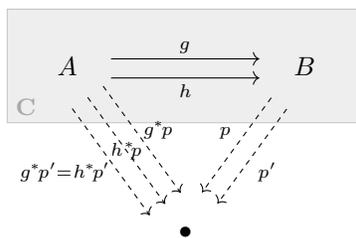
\[
\begin{tikzcd}[row sep=5em, column sep=3em,
	blend group=multiply,
	/tikz/execute at end picture={
		\node [cbox, fit=(A) (B), inner sep=5mm] (C) {};
		\node [catlabel] at (C.south west) {$\cat{C}$};
	}]
	|[alias=A]| A \ar[shift left=1.5, shorten=3mm]{rr}{g}
	\ar[shift right=1.5, shorten=3mm]{rr}[swap]{h} 
	\ar[virtual, shift left=3, shorten=2mm]{dr}[inner sep=0.1mm]{g^*\!p}
	\ar[virtual, shift left=0, shorten=2mm, "h^*\!p"{description, inner sep=0}]{dr}
	\ar[virtual, shift right=3, shorten=2mm]{dr}[swap,inner sep=0mm]{g^*\!p'=h^*\!p'}
	&& |[alias=B]| B \ar[virtual, shift right=1.5, shorten=3mm]{dl}[swap]{p} 
	\ar[virtual, shift left=1.5, shorten=3mm]{dl}{p'} \\
	& \bullet
\end{tikzcd}
\]
This time, the virtual arrows are things that we can precompose (with arrows of $\cat{C}$), but not postcompose.
Possible interpretations are: 
\begin{itemize}
	\item The virtual object ``behaves like an object of $\cat{C}$, but we don't have access to it, we can only probe it by mapping the other objects into it''.
	\item The virtual object is ``an object we would like to add to $\cat{C}$ \emph{from the right}, or \emph{from below}, or \emph{from the out-direction},  without breaking the category structure''.
	\item The virtual object, together with its arrows, is an ``exit point of $\cat{C}$'', or ``a \emph{consistent} way to leave the objects of $\cat{C}$ and go elsewhere''.  
\end{itemize}
Again, \emph{consistent} means for example that we could leave $A$ equivalently by using the virtual arrow $g^*p$, or by moving from $A$ to $B$ using the (real) arrow $g$ and then leaving $B$ using $p$. 

Note that the idea that ``set functors are things that go \emph{into} our objects'' and ``presheaves are things that go \emph{out} of our objects'' is implicitly used very often in category theory:

\begin{example}
	Let $F:\cat{Top}\funto\cat{Set}$ the functor mapping a topological space to the set of its path-connected components. (Readers less familiar with topology may want to think of graphs, or even categories, instead.)
	This functor is not representable. 
	It is however quite similar to other representable functors $\cat{Top}\funto\cat{Set}$ which ``probe our spaces using particular shapes'', such as points, open curves, or loops. (These functors are represented by the one-point space, by $\R$, and by the circle $S^1$ respectively.)
	Indeed, $F$ can be seen as entering a space $A$, and instead of ``looking'' for points or curves, it ``looks'' for entire path components of $A$, no matter their shape. No actual space exists that has such a general, shifting shape, and so we can view elements of $FA$ only as \emph{virtual} arrows into $A$. (While no such object exists, using virtual arrows we can preserve the ``going into $A$'' flavor of this functor.)
\end{example}

\begin{example}\label{tuple_out}
	Suppose that the category $\cat{C}$ does not have all products. 
	Given objects $X$ and $Y$, the presheaf $P=\cat{C}(-,X)\funtimes\cat{C}(-,Y)$ may fail to be representable. 
	Now given an object $A$, the virtual arrows out of $A$, i.e.\ the elements of $PA = \cat{C}(A,X)\funtimes\cat{C}(A,Y)$ are \emph{pairs} of (real) arrows out of $A$, namely $A\to X$ and $A\to Y$:
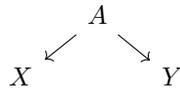
	\[
	\begin{tikzcd}[sep=small]
		& A \ar{dr} \ar{dl} \\
		X && Y
	\end{tikzcd}
	\]
	Of course, \emph{pairs} of arrows are not (single) arrows, and if the product $X\times Y$ does not exist, these pairs don't even \emph{correspond} to single arrows in a natural way. But still, each one of these pairs is a \emph{thing that goes out of $A$}. The elements of $PA$, for each $A$, (and for each presheaf $P$), still have a sense of ``leaving $A$''.
	In this sense it is useful to see elements of non-representable presheaves as \emph{things that go out}, as if they were (single) arrows, but potentially more general.
	(This example is particularly nice, in general not every presheaf has tuples of real arrows as elements. But the interpretation of ``things that go out'' can still be used in general.)
\end{example}

\begin{example}\label{equalizer}
	Suppose that the category $\cat{C}$ does not have all equalizers. 
	Given parallel arrows $g,h:X\to Y$, the presheaf 
	$P=\cat{C}(-,X)\funtimes\cat{C}(-,Y)$ 
	\[
	A \quad{\color{dgray}\longmapsto}\quad PA = \{f:A\to X : g\circ f = h\circ f\} \subseteq \cat{C}(A,X)
	\]
	may fail to be representable. 
	In this case, one could see the virtual arrows $f\in PA$ as \emph{some} real arrows $A\to X$, since $PA$ is a subset of $\cat{C}(A,X)$. The ``virtual'' part is however that these are not \emph{all} arrows $A\to X$. That is, there is no object, $X$ or any other, such that the arrows $A\to X$ are exactly the elements of $PA$. 
\end{example}

\begin{definition}\label{CplusP}
	Let $\cat{C}$ be a category, and let $P:\cat{C}^\op\funto\cat{Set}$ be a presheaf. 
	The category $\cat{C}_{+P}$ has 
	\begin{itemize}
		\item As objects, the ones of $\cat{C}$ plus an extra object $E$;
		\item As morphisms between the objects coming from $\cat{C}$, the ones of $\cat{C}$;
		\item As morphisms $A\to E$, where $A$ is an object of $\cat{C}$, the elements of $PA$ (``virtual arrows'');
		\item The identity of $E$ as the unique morphism with domain $E$.
	\end{itemize}
	The composition of morphisms of $\cat{C}$ is as in $\cat{C}$, and between morphisms $A\to B$ and $B\to E$ is as specified by functoriality of $P$.
\end{definition}

This idea of virtual cones and co-cones is quite simple, but as we will see, it can help our intuition with several advanced category theory concepts.

\subsection{Representable functors and universal properties}\label{representable}

The first categorical concept that we analyze using our point of view on set functors and presheaves is representability.
We use the definition of representable functor given for example in \cite[Chapter~2]{riehl2016category}, which we recall:
\begin{definition}
	A set functor $F:\cat{C}\funto\cat{Set}$ is called \newterm{representable} if there exists a natural isomorphism $\phi:F\xrightarrow{\cong}\cat{C}(R,-)$ for some object $R$ of $\cat{C}$. 
	
	Similarly, a presheaf $P:\cat{C}^\op\funto\cat{Set}$ is called \newterm{representable} if there exists a natural isomorphism $\phi:P\xrightarrow{\cong}\cat{C}(-,R)$ for some object $R$ of $\cat{C}$. 
	
	In either case we call $R$ the \newterm{representing object}, and $\phi$ its \newterm{universal property}.
\end{definition}

We can redefine representability using our virtual arrows: a functor is representable if and only if \emph{its virtual object and arrows correspond to some real objects and arrows, inside our category}.
More precisely:
\begin{theorem}\label{repr_virtual}
	A functor $F:\cat{C}\funto\cat{Set}$ is represented by an object $R$ if and only if the inclusion functor $I:\cat{C}\ffunto\cat{C}^{+F}$ admits a retraction functor $\Pi:\cat{C}^{+F}\funto\cat{C}$ mapping $\cat{C}\subseteq\cat{C}^{+F}$ to itself and $E$ to $R$, and satisfying any of the following equivalent conditions:
	\begin{enumerate}
		\item $\Pi$ (and not just $I$) is fully faithful, meaning that for every $A$ of $\cat{C}$, $\Pi$ maps each ``virtual arrow'' $E\to A$ bijectively to a ``real arrow'' $R\to A$;
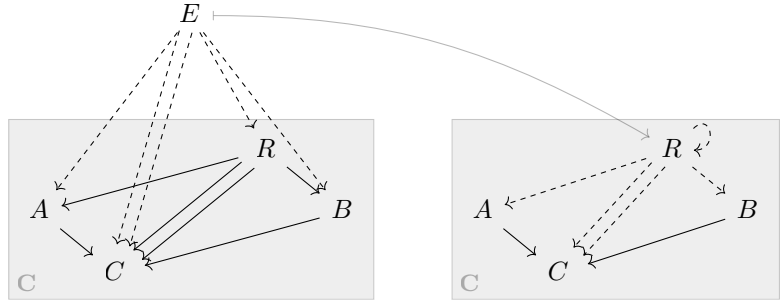
		\[
		\begin{tikzcd}[sep=small,
			blend group=multiply,
			/tikz/execute at end picture={
				\node [cbox, fit=(A1) (B1) (C1) (R1), inner sep=1.2mm] (CL) {};
				\node [catlabel] at (CL.south west) {$\cat{C}$};
				\node [cbox, fit=(A2) (B2) (C2) (R2), inner sep=1.2mm] (CR) {};
				\node [catlabel] at (CR.south west) {$\cat{C}$};
			}]
			& |[alias=E, xshift=10mm]| E 
			 &&&&&& |[xshift=5mm]|\phantom{E} \\ \\ \\ \\ 
			&& |[alias=R1]| R \ar[shorten <=0.8mm]{dll} \ar[shift left, shorten <=-0.8mm, shorten >=0.1mm]{ddl} \ar[shift right]{ddl} \ar{dr}
		 	 &&&&&& |[alias=R2]| R \ar[virtual, shorten <=0.5mm]{dll} \ar[virtual, shift left, shorten <=-0.8mm, shorten >=0.1mm]{ddl} \ar[virtual, shift right]{ddl} \ar[virtual]{dr} \ar[virtual,loop, in=0,out=45,looseness=3] \\
			|[alias=A1]| A \ar{dr} &&& |[alias=B1]| B \ar[shorten >=1mm]{dll} 
			 &&& |[alias=A2]| A \ar{dr} &&& |[alias=B2]| B \ar[shorten >=1mm]{dll} \\
			& |[alias=C1]| C 
			 &&&&&& |[alias=C2]| C
			\ar[virtual, from=E, to=A1]
			\ar[virtual, from=E, to=B1]
			\ar[virtual, from=E, to=C1, shift left, shorten <=-0.2mm, shorten >=1mm]
			\ar[virtual, from=E, to=C1, shift right, shorten >=0.8mm]
			\ar[virtual, from=E, to=R1]
			\ar[mapsto,color=cgray, out=0, in=150, from=E, to=R2]
		\end{tikzcd}
		\]
		(Sometimes, as in the diagram above, we will write the resulting ``real arrows'' still as dashed. See \Cref{dashed}.)
		\item $\Pi$ is left-adjoint to the inclusion $I$ (i.e.\ it is a reflector).
	\end{enumerate}
\end{theorem}

(The categories $\cat{C}^{+F}$ and $\cat{C}_{+P}$ are defined in \Cref{CplusF,CplusP}.)

\begin{corollary}
	Dually, a presheaf $P:\cat{C}^\op\funto\cat{Set}$ is representable if and only if the inclusion functor $\cat{C}\ffunto\cat{C}_{+P}$ admits a retraction which is either fully faithful, or equivalently, right-adjoint (a co-reflector). 
\end{corollary}

As we will see in the proof, we don't just have an equivalence of \emph{properties} (there exists a bijection, etc.), but even of \emph{structures} (the bijections themselves correspond).

\begin{proof}
	First let's prove the equivalence with condition (i).	
	Note that, since any retraction $\Pi$ must restrict to the identity on $\cat{C}\subseteq\cat{C}^{+F}$, we only need to specify its action on $E$ and on the arrows with domain $E$ (i.e.\ the ``virtual object and arrows'').
			
	First suppose that $F$ is represented by $R$. Then each object $A$ of $\cat{C}$ there an isomorphism $\phi_A:FA\xrightarrow{\cong}\cat{C}(R,A)$ (natural in $A$). 
	Recall that the hom-set $\cat{C}^{+F}(E,A)$ is defined to be exactly the set $FA$, so equivalently we have an isomorphism $\phi_A:\cat{C}^{+F}(E,A)\to \cat{C}(R,A)$.
	Define now $\Pi(E)\coloneqq R$ and $\Pi(\id_E)\coloneqq\id_R$, and for every $f:E\to A$, $\Pi(f)\coloneqq \phi_A(f)$. (On the subcategory $\cat{C}$, we take $\Pi$ to be the identity.)
	This is functorial: for every $f:E\to A$ and $g:A\to B$, we have that 
	\[
	\Pi(g\circ f) \;=\; \phi_B(g_* f) \;=\; g\circ\phi_A(f) \;=\; \Pi(g)\circ\Pi(f) ,
	\]
	where in the second equality we used naturality of $\phi$, i.e.~commutativity of the following diagram:
	\[
	\begin{tikzcd}
		FA \ar{d}{g_*} \ar{r}{\phi_A} & \cat{C}(R,A) \ar{d}{g\circ-} \\
		FB \ar{r}{\phi_B} & \cat{C}(R,B)
	\end{tikzcd}
	\]
	By construction $\Pi$ is a retract of the inclusion $\cat{C}\ffunto\cat{C}_{+P}$.
	Moreover, since $\phi_A$ is a bijection, we have that the mapping $\Pi$ gives a bijection between ``virtual arrows'' $E\to A$ and ``real arrows'' $R\to A$.
	
	Conversely, suppose that $\Pi:\cat{C}^{+F}\funto\cat{C}$ is fully faithful. Again, on $\cat{C}$ it must necessarily be the identity. The action of $\Pi$ on arrows gives an assignment
	\[
	\begin{tikzcd}[row sep=0]
		\cat{C}^{+F}(E,A) = FA \ar{r}{\phi_A} & \cat{C}(R,A) \\
		\color{dgray} f \ar[mapsto, dgray]{r} & \color{dgray} \Pi(f)
	\end{tikzcd}
	\]
	which is natural in $A$ by functoriality: for every $g:A\to B$, 
	\[
	\phi(g\circ f) \;=\; \Pi(g\circ f) \;=\; \Pi(g)\circ\Pi(f) \;=\; g\circ\Pi(f) .
	\]
	Moreover, this mapping $\phi_A$ is a bijection by hypothesis. Therefore we have a natural isomorphism $\phi:F\to \cat{C}(R,-)$. 
	
	For the equivalence of (i) and (ii), notice that any retraction $\Pi:\cat{C}^{+F}\funto\cat{C}$ gives a bijection
	\begin{equation}\label{caseE}
	\begin{tikzcd}[row sep=0]
		\cat{C}(\Pi(E),A) = \cat{C}(R,A) \ar[leftrightarrow]{r}{\cong} & \cat{C}^{+F} (E,A) = \cat{C}^{+} (E,I(A)) \\
		\color{dgray} \Pi(f) & \color{dgray} f \ar[mapsto,dgray]{l}
	\end{tikzcd}
	\end{equation}
	natural in $A$, as above, if and only if it is left-adjoint to the inclusion $I$. The latter means that we have bijections
	\[
	\begin{tikzcd}[row sep=0]
		\cat{C}(\Pi(X),A) \ar[leftrightarrow]{r}{\cong} & \cat{C}^{+F} (X,I(A)) \\
		\color{dgray} \epsilon_A\circ\Pi(f) = \Pi(f) & \color{dgray} f \ar[mapsto,dgray]{l}
	\end{tikzcd}
	\]
	for all objects $A$ of $\cat{C}$ and $X$ of $\cat{C}^{+}$, natural in $X$ and $A$. 
	(Here we can take the counit $\epsilon$ of the adjunction to be the identity, since $\Pi$ is a retraction.)
	The case $X=E$ is exactly \eqref{caseE}, and for $X\in\cat{C}$, $\Pi$ acts as the identity,
	\[
	\begin{tikzcd}[row sep=0]
		\cat{C}(\Pi(X),A) = \cat{C}(X,A) \ar[leftrightarrow]{r}{\id} & \cat{C}(X,A) = \cat{C}^{+F} (X,I(A)) \\
		\color{dgray} \Pi(f) = f & \color{dgray} f \ar[mapsto,dgray]{l}
	\end{tikzcd}
	\]
	naturally in $A$.
	For naturality in $X$, consider a morphism $h:X\to Y$ of $\cat{C}^{+F}$, and the naturality diagram as follows.
	\[
	\begin{tikzcd}
		\cat{C}(\Pi(Y),A) \ar{d}{-\circ \Pi(h)} \ar[leftrightarrow]{r}{\cong} & \cat{C}^{+F} (Y,I(A)) \ar{d}{-\circ h} \\
		\cat{C}(\Pi(X),A) \ar[leftrightarrow]{r}{\cong} & \cat{C}^{+F} (X,I(A))
	\end{tikzcd}
	\]
	We have a few cases:
	\begin{itemize}
		\item If $X,Y\in\cat{C}$, the square commutes since $\Pi(g)=g$ and the horizontal arrows are identities;
		\item If $X=E$, and $Y\in\cat{C}$, the diagram reduces to the following,
		\[
		\begin{tikzcd}
			\cat{C}(Y,A) \ar{d}{-\circ \Pi(h)} \ar[leftrightarrow]{r}{\cong} & \cat{C}^{+F} (Y,A) \ar{d}{-\circ h} \\
			\cat{C}(R,A) \ar[leftrightarrow]{r}{\cong} & \cat{C}^{+F} (E,A)
		\end{tikzcd}
		\]
		which commutes by functoriality of $\Pi$;
		\item For the identity $E\to E$, the condition is trivial;
		\item There are no morphisms $A\to E$ for $A\in\cat{C}$. \qedhere
	\end{itemize}
\end{proof}

To summarize, representable set functors and presheaves are those for which ``virtual arrows are not-so-virtual''.
In general, the virtual arrows of functors (or presheaves) can serve as ``blueprint'' for how actual arrows to (or from) a representing object, if it exists, look like. 

\begin{remark}\label{dashed}
Usually, in category theory, those ``not-so-virtual arrows'' are represented as dashed or dotted arrows.
For example, suppose that the product $X\times Y$ exists. Then by definition, for every object $A$ and every pair of arrows $A\to X$ and $A\to Y$ there exists a unique arrow $A\to X\times Y$ making the following diagram commute:
\[
\begin{tikzcd}
	& A \ar{dr} \ar{dl} \ar[virtual]{d} \\
	X & \ar{l} X\times Y \ar{r} & Y
\end{tikzcd}
\]
Let's now see why the dashed arrow above is exactly the ``not-so-virtual arrow'' in the sense of \Cref{repr_virtual}.
The product $X\times Y$ is a representing object of the presheaf $P=\cat{C}(-,X)\funtimes\cat{C}(-,Y)$.
As we saw in \Cref{tuple_out}, the elements of $PA$, i.e.~the virtual arrows, are exactly pairs of real arrows $A\to X$ and $A\to Y$.
The universal property of the product says exactly that to each such pair there exists a (real, single) arrow $A\to X\times Y$.  
In the terminology of \Cref{repr_virtual}, this arrow $A\to X\times Y$ (or $A\to R$) is the one obtained by the functor $\Pi:\cat{C}^{+F}\funto\cat{C}$.

Now, in our formalism, these ``not-so-virtual arrows'' are real arrows of $\cat{C}$, not virtual, and so, technically, we should depict them as solid arrows, not dashed. 
However, they are exactly the real representatives of virtual arrows (in the sense of \Cref{repr_virtual}), and so, writing them as dashed, while technically an abuse of notation, is not too far a stretch. 
Somewhat conversely, one can see our dashed, virtual arrows as a generalization of \emph{those dashed arrows which appear in universal property diagrams}. (Generalization to the case where no object actually satisfies the universal property.)
\end{remark}

\subsection{Natural transformations and the Yoneda embedding}\label{catpsh}

Virtual arrows also allow us to visualize natural transformations between set functors or presheaves. Let's see how, starting with presheaves this time. 

Consider two presheaves $P,Q:\cat{C}^\op\funto\cat{Set}$. 
A natural transformation $\alpha:P\Rightarrow Q$ amounts to functions $\alpha_A:PA\to QA$ for all objects $A$, and natural in $A$.
We can view a natural transformation $\alpha$ as a \emph{virtual arrow between virtual objects}.
First of all, we draw now two virtual objects $\bullet_P$ and $\bullet_Q$:
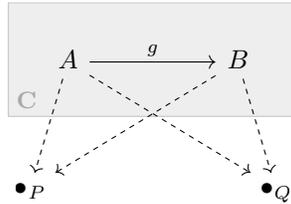
\[
\begin{tikzcd}[sep=huge,
	blend group=multiply,
	/tikz/execute at end picture={
		\node [cbox, fit=(A) (B), inner sep=5mm] (CL) {};
		\node [catlabel] at (CL.south west) {$\cat{C}$};
	}]
	|[alias=A]| A \ar{r}{g} & |[alias=B]| B \\
	|[xshift=-5mm,alias=F]| \bullet_P & |[xshift=5mm,alias=G]| \bullet_Q
	\ar[virtual, from=A, to=F]
	\ar[virtual, from=B, to=F]
	\ar[virtual, from=A, to=G]
	\ar[virtual, from=B, to=G]
\end{tikzcd}
\]
Now we can view $\alpha$ as an arrow between $\bullet_P$ and $\bullet_Q$: 
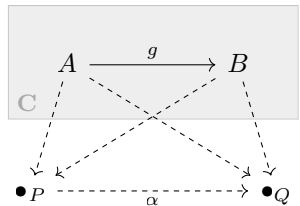
\[
\begin{tikzcd}[sep=huge,
	blend group=multiply,
	/tikz/execute at end picture={
		\node [cbox, fit=(A) (B), inner sep=5mm] (CL) {};
		\node [catlabel] at (CL.south west) {$\cat{C}$};
	}]
	|[alias=A]| A \ar{r}{g} & |[alias=B]| B \\
	|[xshift=-5mm,alias=F]| \bullet_P & |[xshift=5mm,alias=G]| \bullet_Q
	\ar[virtual, from=A, to=F]
	\ar[virtual, from=B, to=F]
	\ar[virtual, from=A, to=G]
	\ar[virtual, from=B, to=G]
	\ar[virtual, from=F, to=G, "\alpha"']
\end{tikzcd}
\]
Now, the assignment $\alpha_A:PA\to QA$ should be a mapping from virtual arrows $A\dashrightarrow\bullet_P$ to virtual arrows $A\dashrightarrow\bullet_Q$.
It can be seen as a ``virtual'' postcomposition with $\alpha$. That is, the map $\alpha_A$ tells us that the following triangle commutes:
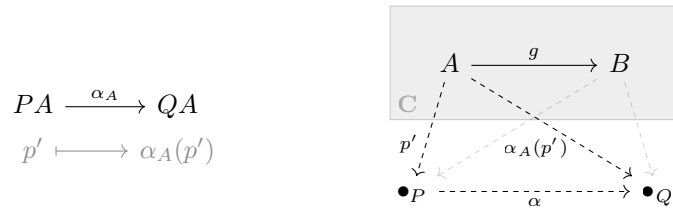
\[
\begin{tikzcd}[row sep=0]
	PA \ar{r}{\alpha_A} & QA \\
	\color{dgray} p' \ar[mapsto,dgray]{r} & \color{dgray} \alpha_A(p')
\end{tikzcd}
\qquad\qquad\qquad
\begin{tikzcd}[sep=huge,
	blend group=multiply,
	/tikz/execute at end picture={
		\node [cbox, fit=(A) (B), inner sep=5mm] (CL) {};
		\node [catlabel] at (CL.south west) {$\cat{C}$};
	}]
	|[alias=A]| A \ar{r}{g} & |[alias=B]| B \\
	|[xshift=-5mm,alias=F]| \bullet_P & |[xshift=5mm,alias=G]| \bullet_Q
	\ar[virtual, from=A, to=F, "p'"{swap,pos=0.8,inner sep=0.5mm}]
	\ar[virtual, from=B, to=F, mgray]
	\ar[virtual, from=A, to=G, "\alpha_A(p')"{swap,pos=0.6,inner sep=0}]
	\ar[virtual, from=B, to=G, mgray]
	\ar[virtual, from=F, to=G, "\alpha"']
\end{tikzcd}
\]
Moreover, naturality in $A$ says that this ``virtual composition'' is associative,
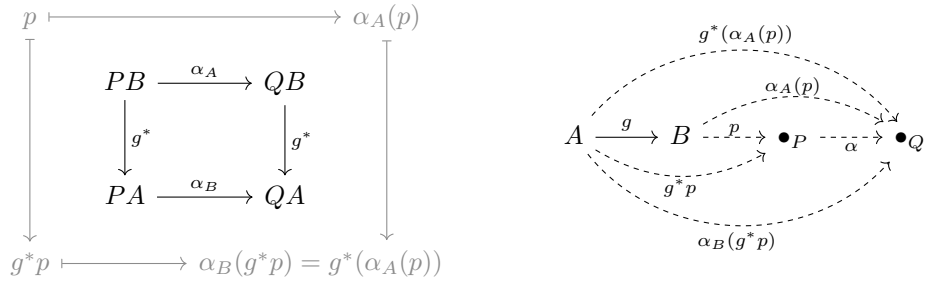
\[
\begin{tikzcd}[sep=small]
	\color{dgray} p \ar[mapsto,dgray]{ddddd} \ar[mapsto,dgray]{rrrrr} &&&&& |[xshift=5mm, overlay]| \color{dgray} \alpha_A(p) \ar[mapsto,dgray]{ddddd} \\
	& PB \ar{rrr}{\alpha_A} \ar{ddd}{g^*} &&& QB \ar{ddd}{g^*} \\ \\ \\ 
	& PA \ar{rrr}{\alpha_B} &&& QA \\
	\color{dgray} g^*p \ar[mapsto,dgray]{rrrrr} &&&&& |[xshift=5mm, overlay]| \color{dgray} \alpha_B(g^*p) = g^*(\alpha_A(p)) \qquad\qquad\quad
\end{tikzcd}
\qquad\qquad\qquad
\begin{tikzcd}
	A \ar{r}{g} \ar[virtual,bend right]{rr}[swap,inner sep=0.3mm]{g^*p} \ar[virtual,bend left=50]{rrr}{g^*(\alpha_A(p))} \ar[virtual,bend right=50]{rrr}[swap]{\alpha_B(g^*p)} & B \ar[virtual]{r}[inner sep=0.1mm]{p} \ar[virtual,bend left]{rr}[inner sep=0.1mm]{\alpha_A(p)} & \bullet_P \ar[virtual]{r}[swap]{\alpha} & \bullet_Q 
\end{tikzcd}
\]
i.e.\ that the following tetrahedron commutes:
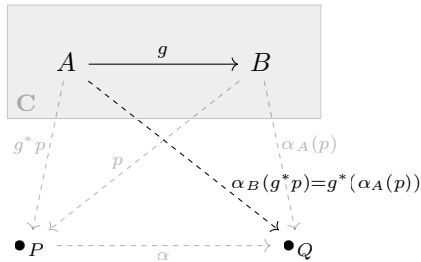
\[
\begin{tikzcd}[sep=2cm,
	blend group=multiply,
	/tikz/execute at end picture={
		\node [cbox, fit=(A) (B), inner sep=5mm] (CL) {};
		\node [catlabel] at (CL.south west) {$\cat{C}$};
	}]
	|[alias=A]| A \ar{r}{g} & |[alias=B]| B \\
	|[xshift=-5mm,alias=F]| \bullet_P & |[xshift=5mm,alias=G]| \bullet_Q
	\ar[virtual, from=A, to=F, "g^*p"{swap,inner sep=0.2mm},cgray]
	\ar[virtual, from=B, to=F, "p"{swap,pos=0.6, inner sep=0.2mm},cgray]
	\ar[virtual, from=A, to=G, "\alpha_B(g^*p) = g^*(\alpha_A(p))"{near end, inner sep=0mm}]
	\ar[virtual, from=B, to=G, "\alpha_A(p)"{inner sep=0.2mm},cgray]
	\ar[virtual, from=F, to=G, "\alpha"',cgray]
\end{tikzcd}
\]
(One could even form a category with two extra objects analogous to $\cat{C}_{+P}$.)

Let's now turn to set functors. Given $F,G:\cat{C}\funto\cat{Set}$, we can similarly view a natural transformation 
as an arrow between virtual points, but crucially, \emph{in the opposite direction}.
Let's see how. First of all, a natural transformation $\alpha:F\Rightarrow G$ amounts to functions $\alpha_A:FA\to GA$ for all objects $A$.
This would have to map virtual arrows $\bullet_F\dashrightarrow A$ to virtual arrows $\bullet_G\dashrightarrow A$. We see immediately that composition with a virtual arrow $\bullet_F\dashrightarrow\bullet_G$ would not work:
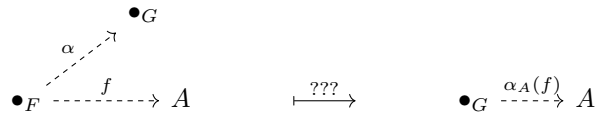
\[
\begin{tikzcd}
	& |[xshift=-5mm]| \bullet_G \\
	\bullet_F \ar[virtual]{ur}{\alpha} \ar[virtual]{r}{f} & A &\, \ar[mapsto]{r}{???} &\, & \bullet_G \ar[virtual]{r}{\alpha_A(f)} & A 
\end{tikzcd}
\]
Instead, we visualize a natural transformation $\alpha:F\Rightarrow G$ as precomposition with a virtual arrow $\bullet_G\dashrightarrow\bullet_F$:
\[
\begin{tikzcd}
	\bullet_G \ar[virtual]{r}{\alpha^\op} & \bullet_F \ar[virtual]{r}{f} & A &\, \ar[mapsto]{r} &\, & \bullet_G \ar[virtual]{r}{\alpha_A(f)} & A 
\end{tikzcd}
\]
In other words, it's a way to say that the following triangle commutes:
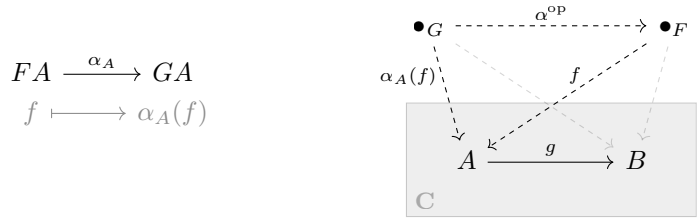
\[
\begin{tikzcd}[row sep=0]
	FA \ar{r}{\alpha_A} & GA \\
	\color{dgray} f \ar[mapsto,dgray]{r} & \color{dgray} \alpha_A(f)
\end{tikzcd}
\qquad\qquad\qquad
\begin{tikzcd}[sep=huge,
	blend group=multiply,
	/tikz/execute at end picture={
		\node [cbox, fit=(A) (B), inner sep=5mm] (CL) {};
		\node [catlabel] at (CL.south west) {$\cat{C}$};
	}]
	|[xshift=-5mm,alias=G]| \bullet_G & |[xshift=5mm,alias=F]| \bullet_F \\
	|[alias=A]| A \ar{r}{g} & |[alias=B]| B 
	\ar[virtual, from=F, to=A, "f"{swap,pos=0.4,inner sep=0.2mm}]
	\ar[virtual, from=F, to=B, mgray]
	\ar[virtual, from=G, to=A, "\alpha_A(f)"{swap,pos=0.2,inner sep=0.2mm}]
	\ar[virtual, from=G, to=B, mgray]
	\ar[virtual, from=G, to=F, "\alpha^\op"]
\end{tikzcd}
\]
Just as for the case of presheaves, naturality in $A$ says that this composition is associative, i.e.\ that the tetrahedron above commutes.

Let's now look at these natural transformations from a larger perspective.
The second-most famous result of category theory is the \emph{Yoneda embedding theorem}, which says that the embedding of a category $\cat{C}$ into its category of presheaves
\[
\begin{tikzcd}[row sep=0, column sep=large]
	\cat{C} \ar[functor]{r}{\Yon} & \sfuncat{\cat{C}^\op} \\
	\color{dgray} A \ar[dgray]{dd}[swap]{g} \ar[mapsto,dgray,shorten=2mm]{r} & \color{dgray} \cat{C}(-,A) \ar[dgray]{dd}{g_*=\,g\circ-} \\
	\phantom{A} \ar[mapsto,dgray,shorten=2mm]{r} & \phantom{\cat{C}(-,A)} \\
	\color{dgray} B \ar[mapsto,dgray,shorten=2mm]{r} & \color{dgray} \cat{C}(-,B)
\end{tikzcd}
\]
is full and faithful.
Replacing $\cat{C}$ by its opposite, we get the embedding on the left below, and taking opposites, we get the one on the right:
\[
\begin{tikzcd}[row sep=0, column sep=large]
	\cat{C}^\op \ar[functor]{r} & \sfuncat{\cat{C}} \\
	\color{dgray} A \ar[leftarrow,dgray]{dd}[swap]{g^\op} \ar[mapsto,dgray,shorten=2mm]{r} & \color{dgray} \cat{C}(A,-) \ar[leftarrow,dgray]{dd}{-\circ g} \\
	\phantom{A} \ar[mapsto,dgray,shorten=2mm]{r} & \phantom{\cat{C}(-,A)} \\
	\color{dgray} B \ar[mapsto,dgray,shorten=2mm]{r} & \color{dgray} \cat{C}(B,-)
\end{tikzcd}
\qquad\qquad\qquad
\begin{tikzcd}[row sep=0, column sep=large]
	\cat{C} \ar[functor]{r} & \sfuncat{\cat{C}}^\op \\
	\color{dgray} A \ar[dgray]{dd}[swap]{g} \ar[mapsto,dgray,shorten=2mm]{r} & \color{dgray} \cat{C}(A,-) \ar[dgray]{dd}{g_*=\,(-\circ g)^\op} \\
	\phantom{A} \ar[mapsto,dgray,shorten=2mm]{r} & \phantom{\cat{C}(-,A)} \\
	\color{dgray} B \ar[mapsto,dgray,shorten=2mm]{r} & \color{dgray} \cat{C}(B,-)
\end{tikzcd}
\]
Both are again fully faithful. Whenever this does not lead to confusion, we will denote the embedding $\cat{C}\funto\sfuncat{\cat{C}}^\op$ again by $\Yon$, and refer to it again as the Yoneda embedding. 

(Keep in mind that, even if $\cat{C}$ is locally small, the categories $\sfuncat{\cat{C}^\op}$ and $\sfuncat{\cat{C}}^\op$ may fail to be locally small.) 

We can view these embeddings as ways to add ``all virtual objects at once'' to $\cat{C}$, in a certain sense, but let's proceed with a little care. 
First of all, instead of adding new objects to $\cat{C}$, we are rather embedding $\cat{C}$ into a larger category. So we switch our notation as follows. Let's do this for presheaves first:
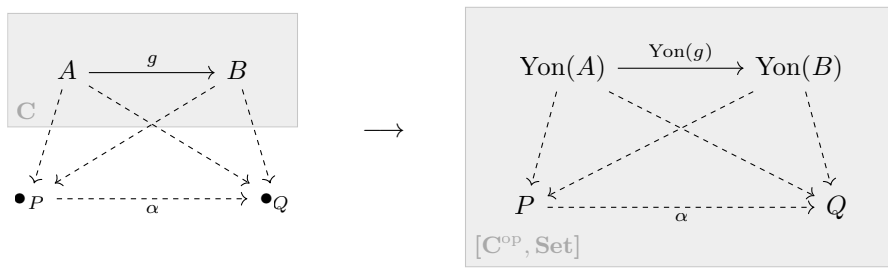
\[
\begin{tikzcd}[sep=huge,
	blend group=multiply,
	/tikz/execute at end picture={
		\node [cbox, fit=(A) (B), inner sep=5mm] (CL) {};
		\node [catlabel] at (CL.south west) {$\cat{C}$};
	}]
	|[alias=A]| A \ar{r}{g} & |[alias=B]| B \\
	|[xshift=-5mm,alias=F]| \bullet_P & |[xshift=5mm,alias=G]| \bullet_Q
	\ar[virtual, from=A, to=F]
	\ar[virtual, from=B, to=F]
	\ar[virtual, from=A, to=G]
	\ar[virtual, from=B, to=G]
	\ar[virtual, from=F, to=G, "\alpha"']
\end{tikzcd}
\qquad\longrightarrow\qquad 
\begin{tikzcd}[sep=huge,
	blend group=multiply,
	/tikz/execute at end picture={
		\node [cbox, fit=(A) (B) (F) (G), inner sep=5mm] (CL) {};
		\node [catlabel] at (CL.south west) {$\sfuncat{\cat{C}^\op}$};
	}]
	|[alias=A]| \Yon(A) \ar{r}{\Yon(g)} & |[alias=B]| \Yon(B) \\
	|[xshift=-5mm,alias=F]| P & |[xshift=5mm,alias=G]| Q
	\ar[virtual, from=A, to=F]
	\ar[virtual, from=B, to=F]
	\ar[virtual, from=A, to=G]
	\ar[virtual, from=B, to=G]
	\ar[virtual, from=F, to=G, "\alpha"']
\end{tikzcd}
\]
All our virtual arrows have now been promoted as actual arrows of our category. (We will sometimes still write them as dashed.)
Now the \emph{first}most famous result of category theory, the Yoneda lemma, says that we have a natural bijection as follows.
\[
\begin{tikzcd}
	PA \ar{r}{\cong} & \sfuncat{\cat{C}^\op}\big(\Yon(A),P\big)
\end{tikzcd}
\]
In other words to each virtual arrow $A\dashrightarrow\bullet_P$ there corresponds a unique real arrow $\Yon(A)\to P$ in $\sfuncat{\cat{C}^\op}$, as the diagram above suggests. 
Since the Yoneda embedding is fully faithful, moreover, we know that to each arrow $g:A\to B$ there corresponds a unique arrow $\Yon(A)\to\Yon(B)$, the solid one in the diagram. 
However, \emph{in $\sfuncat{\cat{C}^\op}$ there may be arrows $P\to\Yon(A)$ that we did not have before.}
So the picture may look more like the following one:
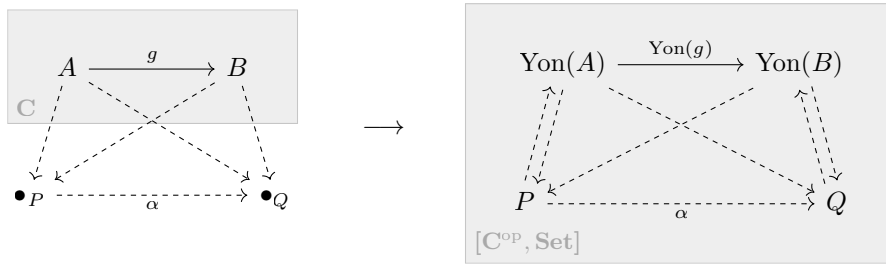
\[
\begin{tikzcd}[sep=huge,
	blend group=multiply,
	/tikz/execute at end picture={
		\node [cbox, fit=(A) (B), inner sep=5mm] (CL) {};
		\node [catlabel] at (CL.south west) {$\cat{C}$};
	}]
	|[alias=A]| A \ar{r}{g} & |[alias=B]| B \\
	|[xshift=-5mm,alias=F]| \bullet_P & |[xshift=5mm,alias=G]| \bullet_Q
	\ar[virtual, from=A, to=F]
	\ar[virtual, from=B, to=F]
	\ar[virtual, from=A, to=G]
	\ar[virtual, from=B, to=G]
	\ar[virtual, from=F, to=G, "\alpha"']
\end{tikzcd}
\qquad\longrightarrow\qquad 
\begin{tikzcd}[sep=huge,
	blend group=multiply,
	/tikz/execute at end picture={
		\node [cbox, fit=(A) (B) (F) (G), inner sep=5mm] (CL) {};
		\node [catlabel] at (CL.south west) {$\sfuncat{\cat{C}^\op}$};
	}]
	|[alias=A]| \Yon(A) \ar{r}{\Yon(g)} & |[alias=B]| \Yon(B) \\
	|[xshift=-5mm,alias=F]| P & |[xshift=5mm,alias=G]| Q
	\ar[virtual, from=A, to=F,shift left]
	\ar[virtual, from=B, to=F]
	\ar[virtual, from=A, to=G]
	\ar[virtual, from=B, to=G,shift left]
	\ar[virtual, from=F, to=A,shift left]
	\ar[virtual, from=G, to=B,shift left]
	\ar[virtual, from=F, to=G, "\alpha"']
\end{tikzcd}
\]

For functors, the situation is similar, but we have to keep in mind that the natural transformations are now reversed. Equivalently, we are embedding $\cat{C}$ into $\sfuncat{\cat{C}}^\op$, not $\sfuncat{\cat{C}}$. The Yoneda lemma now reads as follows: 
\begin{equation}\label{yon_F}
\begin{tikzcd}
	FA \ar{r}{\cong} & \sfuncat{\cat{C}}\big(\Yon(A),F\big) = \sfuncat{\cat{C}}^\op\big(F,\Yon(A)\big) .
\end{tikzcd}
\end{equation}
In other words, to each virtual arrow $\bullet_F\dashrightarrow A$ there corresponds a unique real arrow $F\to\Yon(A)$. However, as for presheaves, there may be extra arrows $\Yon(A)\to F$ that we did not have before.
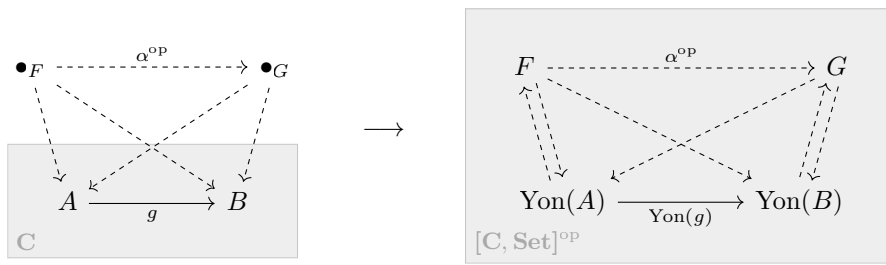
\[
\begin{tikzcd}[sep=huge,
	blend group=multiply,
	/tikz/execute at end picture={
		\node [cbox, fit=(A) (B), inner sep=5mm] (CL) {};
		\node [catlabel] at (CL.south west) {$\cat{C}$};
	}]
|[xshift=-5mm,alias=F]| \bullet_F & |[xshift=5mm,alias=G]| \bullet_G \\
	|[alias=A]| A \ar{r}[swap]{g} & |[alias=B]| B 
	\ar[virtual, from=F, to=A]
	\ar[virtual, from=F, to=B]
	\ar[virtual, from=G, to=A]
	\ar[virtual, from=G, to=B]
	\ar[virtual, from=F, to=G, "\alpha^\op"]
\end{tikzcd}
\qquad\longrightarrow\qquad 
\begin{tikzcd}[sep=huge,
	blend group=multiply,
	/tikz/execute at end picture={
		\node [cbox, fit=(A) (B) (F) (G), inner sep=5mm] (CL) {};
		\node [catlabel] at (CL.south west) {$\sfuncat{\cat{C}}^\op$};
	}]
	|[xshift=-5mm,alias=F]| F & |[xshift=5mm,alias=G]| G \\
	|[alias=A]| \Yon(A) \ar{r}[swap]{\Yon(g)} & |[alias=B]| \Yon(B) 
	\ar[virtual, from=F, to=A,shift left]
	\ar[virtual, from=F, to=B]
	\ar[virtual, from=G, to=A]
	\ar[virtual, from=G, to=B,shift left]
	\ar[virtual, from=A, to=F,shift left]
	\ar[virtual, from=B, to=G,shift left]
	\ar[virtual, from=F, to=G, "\alpha^\op"]
\end{tikzcd}
\]

\section{Weighted limits and colimits}\label{wlim}

Our graphical representation of set functors and presheaves can be used to give an intuitive interpretation of weighted limits and colimits. 

In what follows, we will call a \emph{diagram in $\cat{C}$} a functor $D:\cat{J}\funto\cat{C}$ from some small category $\cat{J}$, which we can see as ``indexing'' the diagram. 

Weighted limits and colimits are defined as follows.

\begin{definition}\label{def_wlim}
	Let $D:\cat{J}\funto\cat{C}$ be a diagram, and let $W:\cat{J}\funto\cat{Set}$ be a set functor. 
	A \newterm{weighted limit} of $D$ with \newterm{weighting} $W$, or \newterm{$W$-weighted limit} of $D$, which we denote by 
	\[
	\lim_{J\in\cat{J}} \big\langle WJ, DJ \big\rangle
	\]
	is an object representing the following presheaf.
	\begin{equation}\label{plim}
	\begin{tikzcd}[row sep=0]
		 \cat{C}^\op \ar[functor]{r} & \cat{Set} \\
		 \color{dgray} A \ar[mapsto,dgray]{r} & \color{dgray} \sfuncat{\cat{J}}\big( W-, \cat{C}(A, D-) \big) 
	\end{tikzcd}
	\end{equation}
	(Note that this is a presheaf on $\cat{C}$, while $W$ is a set functor on $\cat{J}$.)
		
	Dually, let $D:\cat{J}\funto\cat{C}$ be a diagram, and let $W:\cat{J}^\op\funto\cat{Set}$ be a presheaf. 
	A \newterm{weighted colimit} of $D$ with \newterm{weighting} $W$, or \newterm{$W$-weighted colimit} of $D$, which we denote by 
	\[
	\colim_{J\in\cat{J}} \big\langle WJ, DJ \big\rangle 
	\]
	is an object representing the following set functor.
	\begin{equation}\label{pcolim}
	\begin{tikzcd}[row sep=0]
		\cat{C} \ar[functor]{r} & \cat{Set} \\
		\color{dgray} A \ar[mapsto,dgray]{r} & \color{dgray} \sfuncat{\cat{J}^\op}\big( W-, \cat{C}(D-, A) \big) 
	\end{tikzcd}
	\end{equation}
	(Note that this is a set functor on $\cat{C}$, while $W$ is a presheaf on $\cat{J}$.)
\end{definition}

Weighted limits and colimits are alternative called \newterm{indexed limits and colimits}. 

At first, this definition may not seem particularly suggestive.
Let's now interpret this graphically, by means of our virtual arrows.

\begin{remark}
	Sometimes one may be interested more generally in situations where the indexing category $\cat{J}$ is not small. In that case, a necessary condition for the (large-set-valued) presheaf \eqref{plim} to be representable, is that for each $A$ the set $\sfuncat{\cat{J}}\big( W-, \cat{C}(A, D-) \big)$ is small.
	Indeed, that set needs to have the same cardinality as the hom-set $\cat{C}(R,A)$, where $R$ is a representing object.
	The same is true for the functor $\eqref{pcolim}$.
\end{remark}

\subsection{Weighted cones and cocones}\label{wcones}

Consider a diagram $D:\cat{J}\funto\cat{C}$. A (usual) limit of $D$, if it exists, is first of all a \emph{cone} over $D$:
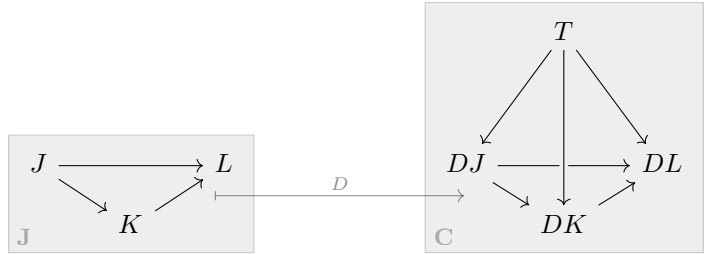
\[
\begin{tikzcd}[sep=small,
	blend group=multiply,
	/tikz/execute at end picture={
		\node [cbox, fit=(J) (L) (K), inner sep=1.2mm] (CL) {};
		\node [catlabel] at (CL.south west) {$\cat{J}$};
		\node [cbox, fit=(DJ) (DL) (DK) (T), inner sep=1.2mm] (CR) {};
		\node [catlabel] at (CR.south west) {$\cat{C}$};
	}]
	&&&&&&&&& |[alias=T]| T  \\ \\ \\ \\ 
	|[alias=J, xshift=-5mm]| J \ar{dr} \ar{rr} && |[alias=L, xshift=5mm]| L 
	 &&&&&& |[alias=DJ]| DJ \ar[-, shorten >=-1mm]{r} \ar{dr} & |[alias=BR2]| \ar[shorten <=-1mm]{r} & |[alias=DL]| DL \\
	& |[alias=K]| K \ar{ur}
	 &&&&&&&& |[alias=DK]| DK \ar{ur}
	\ar[from=T, to=DJ]
	\ar[from=T, to=DK]
	\ar[from=T, to=DL]
	\ar[mapsto,color=dgray, from=L, to=DJ, "D", shift right=4mm, shorten=-4mm]
\end{tikzcd}
\]
That is, an object $T$ of $\cat{C}$ (the ``tip'' of the cone) together with an arrow $c_J:T\to DJ$ for each object $J$ of $\cat{J}$, such that for every arrow $g:J\to K$ of $\cat{D}$, $Dg\circ c_J=c_K$. In other words, each side triangle of the cone (involving $T$) commutes:
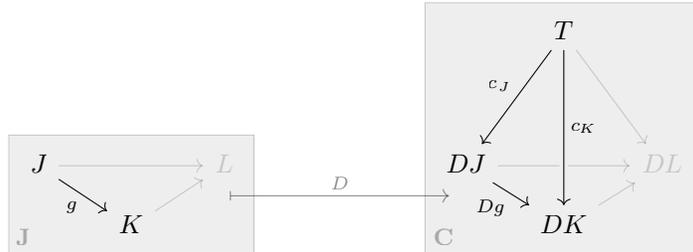
\[
\begin{tikzcd}[sep=small,
	blend group=multiply,
	/tikz/execute at end picture={
		\node [cbox, fit=(J) (L) (K), inner sep=1.2mm] (CL) {};
		\node [catlabel] at (CL.south west) {$\cat{J}$};
		\node [cbox, fit=(DJ) (DL) (DK) (T), inner sep=1.2mm] (CR) {};
		\node [catlabel] at (CR.south west) {$\cat{C}$};
	}]
	&&&&&&&&& |[alias=T]| T  \\ \\ \\ \\ 
	|[alias=J, xshift=-5mm]| J \ar{dr}[swap]{g} \ar[mgray]{rr} && |[alias=L, xshift=5mm]| \color{mgray} L 
	&&&&&& |[alias=DJ]| DJ \ar[-, shorten >=-1mm, mgray]{r} \ar{dr}[swap]{Dg} & |[alias=BR2]| \ar[shorten <=-1mm, mgray]{r} & |[alias=DL]| \color{mgray} DL \\
	& |[alias=K]| K \ar[mgray]{ur}
	&&&&&&&& |[alias=DK]| DK \ar[mgray]{ur}
	\ar[from=T, to=DJ, "c_J"']
	\ar[from=T, to=DK, "c_K"]
	\ar[from=T, to=DL, color=mgray]
	\ar[mapsto,color=dgray, from=L, to=DJ, "D", shift right=4mm, shorten=-2mm]
\end{tikzcd}
\]
A cone $(T,c)$ is now a \emph{limit} if it is a \emph{universal} or \emph{terminal} cone, meaning that for every cone $(T',c')$ over $D$, there exists a unique arrow between the tips $u:T'\to T$ such that for each $J$ of $\cat{J}$, $c_J\circ u=c'_J$, i.e.\ the following triangle commutes:
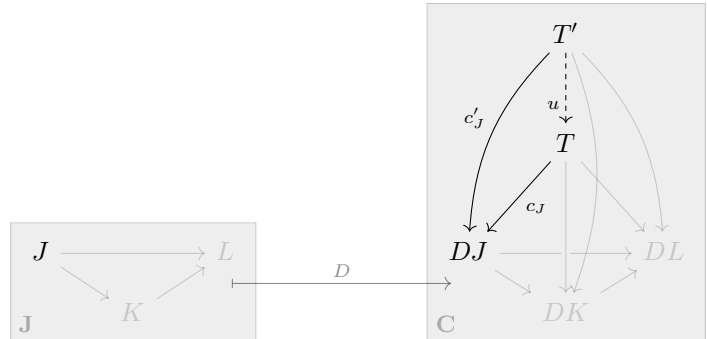
\[
\begin{tikzcd}[sep=small,
	blend group=multiply,
	/tikz/execute at end picture={
		\node [cbox, fit=(J) (L) (K), inner sep=1.2mm] (CL) {};
		\node [catlabel] at (CL.south west) {$\cat{J}$};
		\node [cbox, fit=(DJ) (DL) (DK) (TP), inner sep=1.2mm] (CR) {};
		\node [catlabel] at (CR.south west) {$\cat{C}$};
	}]
	&&&&&&&&& |[alias=TP]| T' \ar[virtual]{ddd}[near end,swap]{u} \\ \\ \\
	&&&&&&&&& |[alias=T]| T  \\ \\ \\
	|[alias=J, xshift=-5mm]| J \ar[mgray]{dr}\ar[mgray]{rr} && |[alias=L, xshift=5mm]| \color{mgray} L 
	&&&&&& |[alias=DJ]| DJ \ar[-, shorten >=-1mm, mgray]{r} \ar[mgray]{dr} & |[alias=BR2]| \ar[shorten <=-1mm, mgray]{r} & |[alias=DL]| \color{mgray} DL \\
	& |[alias=K]| \color{mgray} K \ar[mgray]{ur}
	&&&&&&&& |[alias=DK]| \color{mgray} DK \ar[mgray]{ur}
	\ar[from=T, to=DJ, "c_J"]
	\ar[from=T, to=DK, color=mgray]
	\ar[from=T, to=DL, color=mgray]
	\ar[from=TP, to=DJ, "c'_J"', bend right=20]
	\ar[from=TP, to=DK, color=mgray, bend left=20]
	\ar[from=TP, to=DL, color=mgray, bend left=20]
	\ar[mapsto,color=dgray, from=L, to=DJ, "D", shift right=4mm, shorten=-2mm]
\end{tikzcd}
\]
Similarly, a colimit is a universal (initial) co-cone.

We can interpret a cone over a diagram as a way to \emph{extend} our diagram by adding an object, the tip $T$, and arrows from $T$ to the nodes of the diagram, forming commutative triangles. A limit is then a universal way of doing so.
The interpretation of weighted limits, and more generally of what we call \emph{weighted cones}, is very similar, except that \emph{there may be an entire set of arrows} (or better said, an \emph{indexed tuple}) between $T$ and the nodes of the diagram:
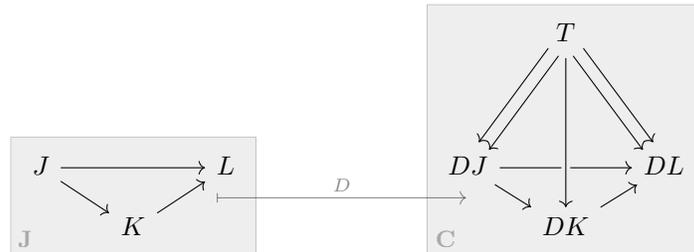
\[
\begin{tikzcd}[sep=small,
	blend group=multiply,
	/tikz/execute at end picture={
		\node [cbox, fit=(J) (L) (K), inner sep=1.2mm] (CL) {};
		\node [catlabel] at (CL.south west) {$\cat{J}$};
		\node [cbox, fit=(DJ) (DL) (DK) (T), inner sep=1.2mm] (CR) {};
		\node [catlabel] at (CR.south west) {$\cat{C}$};
	}]
	&&&&&&&&& |[alias=T]| T  \\ \\ \\ \\ 
	|[alias=J, xshift=-5mm]| J \ar{dr} \ar{rr} && |[alias=L, xshift=5mm]| L 
	&&&&&& |[alias=DJ]| DJ \ar{dr} \ar[-, shorten >=-1mm]{r} & |[alias=BR2]| \ar[shorten <=-1mm]{r} & |[alias=DL]| DL \\
	& |[alias=K]| K \ar{ur}
	&&&&&&&& |[alias=DK]| DK \ar{ur}
	\ar[from=T, to=DJ, shift left=1]
	\ar[from=T, to=DJ, shift right=1]
	\ar[from=T, to=DK, shorten <=0.8mm]
	\ar[from=T, to=DL, shift left=1]
	\ar[from=T, to=DL, shift right=1]
	\ar[mapsto,color=dgray, from=L, to=DJ, "D", shift right=4mm, shorten=-4mm]
\end{tikzcd}
\]
We can view this as a more general way of extending a diagram, by adding an extra object $T$ and arrows from $T$ to the diagram: having single arrows as sides of the cone is just a special case. 
We might now wonder: 
\begin{itemize}
	\item How do we specify each tuple of arrows?
	\item How do we know which triangles commute?
\end{itemize}
The answer to both, which might come expected after reading  \Cref{virtual}, is that we can use a functor $W:\cat{J}\funto\cat{Set}$. We call this a \emph{weighting} for the diagram. 
We can see the weighting functor as a ``blueprint'' for how we want these generalized cones to look:
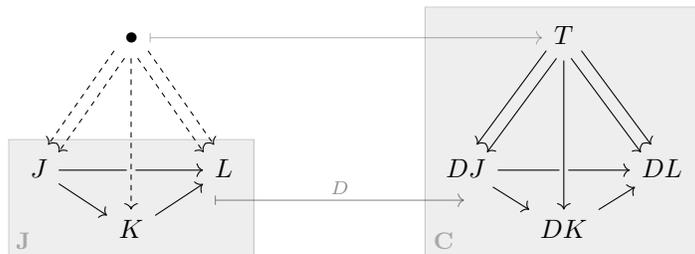
\[
\begin{tikzcd}[sep=small,
	blend group=multiply,
	/tikz/execute at end picture={
		\node [cbox, fit=(J) (L) (K), inner sep=1.2mm] (CL) {};
		\node [catlabel] at (CL.south west) {$\cat{J}$};
		\node [cbox, fit=(DJ) (DL) (DK) (T), inner sep=1.2mm] (CR) {};
		\node [catlabel] at (CR.south west) {$\cat{C}$};
	}]
	& |[alias=E]| \bullet 
	 &&&&&&&& |[alias=T]| T  \\ \\ \\ \\ 
	|[alias=J, xshift=-5mm]| J \ar{dr} \ar[-, shorten >=-1mm]{r} & |[alias=BR1]| \ar[shorten <=-1mm]{r} & |[alias=L, xshift=5mm]| L 
	 &&&&&& |[alias=DJ]| DJ \ar{dr} \ar[-, shorten >=-1mm]{r} & |[alias=BR2]| \ar[shorten <=-1mm]{r} & |[alias=DL]| DL \\
	& |[alias=K]| K \ar{ur}
	 &&&&&&&& |[alias=DK]| DK \ar{ur}
	\ar[virtual, from=E, to=J, shift left=1]
	\ar[virtual, from=E, to=J, shift right=1]
	\ar[virtual, from=E, to=K, shorten <=0.5mm]
	\ar[virtual, from=E, to=L, shift left=1]
	\ar[virtual, from=E, to=L, shift right=1]
	\ar[from=T, to=DJ, shift left=1]
	\ar[from=T, to=DJ, shift right=1]
	\ar[from=T, to=DK, shorten <=0.8mm]
	\ar[from=T, to=DL, shift left=1]
	\ar[from=T, to=DL, shift right=1]
	\ar[mapsto,color=cgray, from=E, to=T]
	\ar[mapsto,color=dgray, from=L, to=DJ, "D", shift right=4mm, shorten=-4mm]
\end{tikzcd}
\]
Before we make this idea precise, in order to have the right intuition, it's worth keeping a few things in mind:
\begin{itemize}
	\item Recall that a set functor comes with a rule for which ``virtual triangles'' commute.
	For example, in the picture above, $W$ may specify that one of the two arrows $\bullet\dashrightarrow L$ (but perhaps not the other one) is the composition of one of the arrows $\bullet\dashrightarrow J$ with $J\to L$. 
	We want the resulting cone in $\cat{C}$ to respect these commutativity relations. More things may commute in $\cat{C}$, but crucially, not fewer.
	\item A way to rephrase the point above is that in a traditional diagram, two arrows may happen to be equal in $\cat{C}$ even if they are written with different symbols in the diagram, that is, if they are indexed by different arrows of $J$ (the functor $J$ is not always faithful). The same is true for these cones: different virtual arrows may be mapped to equal arrows in $\cat{C}$ (but equal virtual arrows cannot be mapped to different arrows). 
	\item Similarly, in a traditional diagram, given any two objects $DJ$ and $DK$, there may be more arrows $DJ\to DK$ of $\cat{C}$ than the ones appearing in the diagram, i.e.\ coming from $\cat{J}$ (the functor $D$ is not always full). The same is true for these cones: there may be arrows $T\to DK$ in $\cat{C}$ which are not considered in the weighted diagram. 
\end{itemize}

With this intuition in mind, let's give the precise definition. 

\begin{definition}\label{def_wcone}
	Let $D:\cat{J}\funto\cat{C}$ be a diagram and $W:\cat{J}\funto\cat{Set}$ be a set functor.
	A \newterm{weighted cone} over $D$ (weighted by $W$), or \newterm{$W$-weighted cone}, is an extension\footnote{Here we mean an extension in the common sense of the word, not a Kan extension. (See however \Cref{ptwise_kan}.)} of $D$ to a functor $D^+:\cat{J}^{+W}\funto\cat{C}$, agreeing with $D$ on $\cat{J}$. We call $D^+\!E$ the \newterm{tip} of the cone.
	
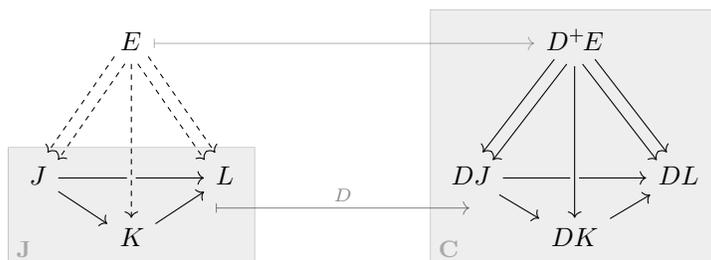
\[
	\begin{tikzcd}[sep=small,
		blend group=multiply,
		/tikz/execute at end picture={
			\node [cbox, fit=(J) (L) (K), inner sep=1.2mm] (CL) {};
			\node [catlabel] at (CL.south west) {$\cat{J}$};
			\node [cbox, fit=(DJ) (DL) (DK) (T), inner sep=1.2mm] (CR) {};
			\node [catlabel] at (CR.south west) {$\cat{C}$};
		}]
		& |[alias=E]| E
		&&&&&&&& |[alias=T]| D^+\!E  \\ \\ \\ \\ 
		|[alias=J, xshift=-5mm]| J \ar{dr} \ar[-, shorten >=-1mm]{r} & |[alias=BR1]| \ar[shorten <=-1mm]{r} & |[alias=L, xshift=5mm]| L 
		&&&&&& |[alias=DJ]| DJ \ar{dr} \ar[-, shorten >=-1mm]{r} & |[alias=BR2]| \ar[shorten <=-1mm]{r} & |[alias=DL]| DL \\
		& |[alias=K]| K \ar{ur}
		&&&&&&&& |[alias=DK]| DK \ar{ur}
		\ar[virtual, from=E, to=J, shift left=1]
		\ar[virtual, from=E, to=J, shift right=1]
		\ar[virtual, from=E, to=K, shorten <=0.5mm]
		\ar[virtual, from=E, to=L, shift left=1]
		\ar[virtual, from=E, to=L, shift right=1]
		\ar[from=T, to=DJ, shift left=1]
		\ar[from=T, to=DJ, shift right=1, shorten <=0.5mm]
		\ar[from=T, to=DK, shorten <=0.8mm]
		\ar[from=T, to=DL, shift left=1, shorten <=0.5mm]
		\ar[from=T, to=DL, shift right=1]
		\ar[mapsto,color=cgray, from=E, to=T]
		\ar[mapsto,color=dgray, from=L, to=DJ, "D", shift right=4mm, shorten=-4mm]
	\end{tikzcd}
	\]
	Dually, let $W:\cat{J}^\op\funto\cat{Set}$ be a presheaf. 
	A \newterm{weighted co-cone} over $D$ (weighted by $W$), or \newterm{$W$-weighted co-cone}, is an extension of $D$ to a functor $D_+:\cat{J}_{+W}\funto\cat{C}$.
	We call $D^+(E)$ the \newterm{tip} of the co-cone.
\end{definition}

(Once again, the categories $\cat{C}^{+F}$ and $\cat{C}_{+P}$ are defined in \Cref{CplusF,CplusP}.)

Equivalently, cones and cocones over $D$ are commutative diagrams of functors as follows.
\begin{equation}\label{wlim_ext}
\begin{tikzcd}[row sep=small]
	\cat{J} \ar[functor,hook]{dd}[swap]{I} \ar[functor]{dr}{D} \\
	& \cat{C} \\
	\cat{J}^{+W} \ar[functor]{ur}[swap]{D^+}
\end{tikzcd}
\qquad\qquad\qquad
\begin{tikzcd}[row sep=small]
	\cat{J} \ar[functor,hook]{dd}[swap]{I} \ar[functor]{dr}{D} \\
	& \cat{C} \\
	\cat{J}_{+W} \ar[functor]{ur}[swap]{D_+}
\end{tikzcd}
\end{equation}
(These diagrams may remind the reader of Kan extensions. That is a correct intuition, see \Cref{ptwise_kan}.) 

The main result of this section is now that \emph{weighted limit and colimits are exactly universal weighted cones}.

\begin{theorem}\label{thm_wlim}
	Let $D:\cat{J}\funto\cat{C}$ be a diagram, and let $W:\cat{D}\funto\cat{Set}$ be a weighting for $D$. 
	A $W$-weighted limit of $D$ is, equivalently, a weighted cone $D^+:\cat{D}^{+W}\funto\cat{D}$ (with tip $T$) which is terminal: given any $W$-weighted cone ${D^+}'$ (with tip $T'$), there is a unique morphism $u:T'\to T$ such that for all $J$ of $\cat{J}$ and $w\in WJ$ we have that $D^+\!w\circ u={D^+}\!w$, i.e.\ the following triangle commutes.
	
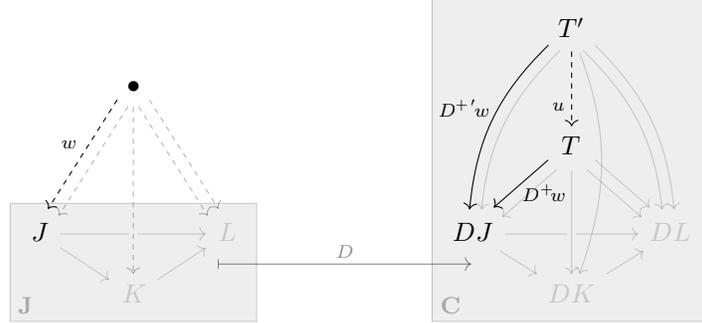
\[
	\begin{tikzcd}[sep=small,
		blend group=multiply,
		/tikz/execute at end picture={
			\node [cbox, fit=(J) (L) (K), inner sep=1.2mm] (CL) {};
			\node [catlabel] at (CL.south west) {$\cat{J}$};
			\node [cbox, fit=(DJ) (DL) (DK) (TP), inner sep=1.2mm] (CR) {};
			\node [catlabel] at (CR.south west) {$\cat{C}$};
		}]
		&&&&&&&&& |[alias=TP]| T' \\ 
		& |[alias=E]| \bullet \\  
		&&&&&&&&& |[alias=T]| T  \\ \\  
		|[alias=J, xshift=-5mm]| J \ar[mgray]{dr} \ar[-, shorten >=-1mm, mgray]{r} & |[alias=BR1]| \ar[shorten <=-1mm, mgray]{r} & |[alias=L, xshift=5mm]| \color{mgray} L 
		&&&&&& |[alias=DJ]| DJ \ar[mgray]{dr} \ar[-, shorten >=-1mm, mgray]{r} & |[alias=BR2]| \ar[shorten <=-1mm, mgray]{r} & |[alias=DL]| \color{mgray} DL \\
		& |[alias=K]| \color{mgray} K \ar[mgray]{ur}
		&&&&&&&& |[alias=DK]| \color{mgray} DK \ar[mgray]{ur}
		\ar[virtual, from=E, to=J, shift left, color=cgray]
		\ar[virtual, from=E, to=J, shift right, "w"']
		\ar[virtual, from=E, to=K, shorten <=0.5mm, color=cgray]
		\ar[virtual, from=E, to=L, shift left=1, color=cgray]
		\ar[virtual, from=E, to=L, shift right=1, color=cgray]
		\ar[from=T, to=DJ, shift left, color=mgray]
		\ar[from=T, to=DJ, shift right, "D^+\!w"'{swap, inner sep=0mm}]
		\ar[from=T, to=DK, shorten <=0.8mm, color=mgray]
		\ar[from=T, to=DL, shift left=1, color=mgray]
		\ar[from=T, to=DL, shift right=1, color=mgray]
		\ar[from=TP, to=DJ, shift left, color=mgray, bend right=20, shorten >=0.7mm]
		\ar[from=TP, to=DJ, shift right, bend right=20, "{D^+}'\!w"{swap, inner sep=0.3mm}]
		\ar[from=TP, to=DK, shorten <=0.8mm, color=mgray, bend left=20]
		\ar[from=TP, to=DL, shift left=1, color=mgray, bend left=20]
		\ar[from=TP, to=DL, shift right=1, color=mgray, bend left=20, shorten >=0.7mm]
		\ar[virtual, from=TP, to=T, "u"', near end, shorten <=0.5mm]
		\ar[mapsto,color=dgray, from=L, to=DJ, "D", shift right=4mm, shorten=-4mm]
	\end{tikzcd}
	\]
	Dually, given a presheaf $W:\cat{J}^\op\funto\cat{Set}$, a $W$-weighted colimit of $D$ with weighting $W$ is equivalently an initial weighted co-cone.
\end{theorem}

Note that the terminality condition must be checked for each of the arrows $T\to DJ$ separately, in case there are more than one. 

Let's now prove this theorem.
The main argument, contained in the next two lemmas, connects weighted cones and cocones with the presheaves appearing in \Cref{def_wlim}.

\begin{lemma}\label{wcone_is_nat}
	Let $D:\cat{J}\funto\cat{C}$ be a diagram with weighting $W:\cat{D}\funto\cat{Set}$. 
	A $W$-weighted cone over $D$ with tip $T$ is equivalently given by a collection of maps 
	\[
	\begin{tikzcd}
		WJ \ar{r}{\phi_J} & \cat{C}(T,DJ)
	\end{tikzcd}
	\]
	natural in $J$.
	
	Dually, given a presheaf $W:\cat{D}^\op\funto\cat{Set}$, a $W$-weighted co-cone with tip $T$ is equivalently given by a collection of maps 
	\[
	\begin{tikzcd}
		WJ \ar{r}{\phi_J} & \cat{C}(DJ,T)
	\end{tikzcd}
	\]
	natural in $J$.
\end{lemma}
The proof is a generalization of the argument in the proof of \Cref{repr_virtual}.
\begin{proof}
	We will prove the case of functors, the case of presheaves is completely analogous and dual. 
	
	First, let $D^+:\cat{J}^{+W}\funto\cat{C}$ be a weighted cone. Its action on morphisms, in particular on the ``virtual arrows'' $E\to J$, induces the following mapping for every $J\in\cat{J}$:
	\[
	\begin{tikzcd}[row sep=0]
		\cat{J}^{+W}(E,J) \ar{r} & \cat{C}(D^+\!E,D^+\!J) \\
		\color{dgray} w \ar[mapsto,dgray]{r} & \color{dgray} D^+\!w .
	\end{tikzcd}
	\]
	Recall now that by definition, the set $\cat{J}^{+W}(E,J)$ is exactly the set of virtual arrows $WJ$.
	Also, note that $D^+\!E=T$, and that since $D^+$ must extend $D$, $D^+\!J=DJ$. Therefore the mapping above can be written equivalently as 
	\[
	\begin{tikzcd}[row sep=0]
		WJ \ar{r} & \cat{C}(T,DJ) \\
		\color{dgray} w \ar[mapsto,dgray]{r} & \color{dgray} D^+\!w ,
	\end{tikzcd}
	\]
	which is in the desired form. Call this mapping $\phi_J$. Naturality in $J$ is now the following commutative diagram for every morphism $g:J\to K$ of $\cat{J}$,
	\begin{equation}\label{nat_is_funct}
	\begin{tikzcd}[sep=small]
		\color{dgray} w \ar[mapsto,dgray]{ddddd} \ar[mapsto,dgray]{rrrr}
		 &&&& \color{dgray} w \ar[mapsto,dgray]{rrrr}
		 &&&& |[xshift=5mm, overlay]| \color{dgray} D^+\!w \ar[mapsto,dgray]{ddddd} \\
		& WJ \ar[equal]{rrr} \ar{ddd}{g_*}
		 &&& \cat{J}^{+W}(E,J) \ar{rrr}{\phi_J} \ar{ddd}{g\circ-}
		 &&& \cat{C}(T,DJ) \ar{ddd}{Dg\circ-} \\ \\ \\
		& WK \ar[equal]{rrr}
		 &&& \cat{J}^{+W}(E,K) \ar{rrr}{\phi_K}
		 &&& \cat{C}(T,DK) \\
		\color{dgray} g_*w \ar[mapsto,dgray]{rrrr} 
		 &&&& \color{dgray} g\circ w \ar[mapsto,dgray]{rrrr} 
		 &&&& |[xshift=5mm, overlay]| \color{dgray} D^+(g\circ w) = D^+\!g\circ D^+\!w \qquad\qquad\quad\;\;
	\end{tikzcd}
	\end{equation}
	which commutes by functoriality of $D^+$ (recalling that $D^+\!g=Dg$ since $g\in\cat{J}$).
	
	Conversely, consider a natural family of maps $\phi_J:WJ\to \cat{C}(T,DJ)$. 
	Define now a weighted cone $D^+:\cat{J}^{+W}\funto\cat{C}$ as follows:
	\begin{itemize}
		\item On the objects and morphisms of $\cat{J}$, it agrees with $D$;
		\item $D^+\!E\coloneqq T$;
		\item For every $J$ and $w:E\to J$ (i.e.~$w\in WJ$), $D^+\!w\coloneqq \phi_J(w)$.
	\end{itemize}
	To show that this is functorial, notice that identities are mapped to identities, and composition of morphisms of $\cat{J}$ is respected by functoriality of $D$. It only remains to show preservation of composition for arrows in the form $E\to J$ and $J\to K$, but this is guaranteed by naturality of $\phi$, since it amounts to a commutative diagram analogous to \eqref{nat_is_funct}.
	
	Finally, as one can readily see, these two procedures are mutually inverse. 
\end{proof}

\begin{lemma}\label{psh_weighted}
	Let $D:\cat{J}\funto\cat{C}$ be a diagram, and let $W:\cat{J}\funto\cat{Set}$ be a functor.
	The presheaf 
	\begin{equation}\label{conew}
	\begin{tikzcd}[row sep=0]
		\cat{C}^\op \ar[functor]{r} & \cat{Set} \\
		\color{dgray} A \ar[mapsto,dgray]{r} & \color{dgray} \sfuncat{\cat{J}}\big( W-, \cat{C}(A, D-) \big) 
	\end{tikzcd}
	\end{equation}
	appearing in \Cref{def_wlim} can be seen as mapping
	\begin{itemize}
		\item Each object $A$ of $\cat{C}$ to the set of $W$-weighted cones over $D$ with tip $A$;
		\item Each arrow $g:A\to B$ in $\cat{C}$ to the function which pre-composes a cone $D^+$ with $g$, to give a cone $g^*(D^+)$ with tip $A$:
	\end{itemize}
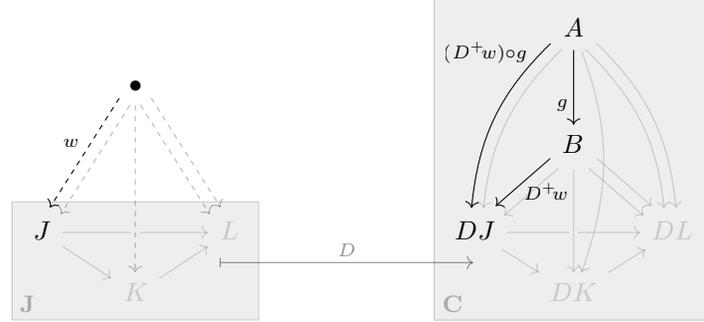
	\[
	\begin{tikzcd}[sep=small,
		blend group=multiply,
		/tikz/execute at end picture={
			\node [cbox, fit=(J) (L) (K), inner sep=1.2mm] (CL) {};
			\node [catlabel] at (CL.south west) {$\cat{J}$};
			\node [cbox, fit=(DJ) (DL) (DK) (TP), inner sep=1.2mm] (CR) {};
			\node [catlabel] at (CR.south west) {$\cat{C}$};
		}]
		&&&&&&&&& |[alias=TP]| A \\ 
		& |[alias=E]| \bullet \\  
		&&&&&&&&& |[alias=T]| B  \\ \\  
		|[alias=J, xshift=-5mm]| J \ar[mgray]{dr} \ar[-, shorten >=-1mm, mgray]{r} & |[alias=BR1]| \ar[shorten <=-1mm, mgray]{r} & |[alias=L, xshift=5mm]| \color{mgray} L 
		&&&&&& |[alias=DJ]| DJ \ar[mgray]{dr} \ar[-, shorten >=-1mm, mgray]{r} & |[alias=BR2]| \ar[shorten <=-1mm, mgray]{r} & |[alias=DL]| \color{mgray} DL \\
		& |[alias=K]| \color{mgray} K \ar[mgray]{ur}
		&&&&&&&& |[alias=DK]| \color{mgray} DK \ar[mgray]{ur}
		\ar[virtual, from=E, to=J, shift left, color=cgray]
		\ar[virtual, from=E, to=J, shift right, "w"']
		\ar[virtual, from=E, to=K, shorten <=0.5mm, color=cgray]
		\ar[virtual, from=E, to=L, shift left=1, color=cgray]
		\ar[virtual, from=E, to=L, shift right=1, color=cgray]
		\ar[from=T, to=DJ, shift left, color=mgray]
		\ar[from=T, to=DJ, shift right, "D^+\!w"'{swap, inner sep=0mm}]
		\ar[from=T, to=DK, shorten <=0.8mm, color=mgray]
		\ar[from=T, to=DL, shift left=1, color=mgray]
		\ar[from=T, to=DL, shift right=1, color=mgray]
		\ar[from=TP, to=DJ, shift left, color=mgray, bend right=20, shorten >=0.7mm]
		\ar[from=TP, to=DJ, shift right, bend right=20, "(D^+\!w)\circ g"{swap, inner sep=0.3mm, pos=0.15}]
		\ar[from=TP, to=DK, shorten <=0.8mm, color=mgray, bend left=20]
		\ar[from=TP, to=DL, shift left=1, color=mgray, bend left=20]
		\ar[from=TP, to=DL, shift right=1, color=mgray, bend left=20, shorten >=0.7mm]
		\ar[from=TP, to=T, "g", near end, swap, shorten <=0.5mm]
		\ar[mapsto,color=dgray, from=L, to=DJ, "D", shift right=4mm, shorten=-4mm]
	\end{tikzcd}
	\]
	Dually, given a presheaf $W:\cat{J}^\op\funto\cat{Set}$, the functor 
	\begin{equation}\label{coconew}
	\begin{tikzcd}[row sep=0]
		\cat{C} \ar[functor]{r} & \cat{Set} \\
		\color{dgray} A \ar[mapsto,dgray]{r} & \color{dgray} \sfuncat{\cat{J}^\op}\big( W-, \cat{C}(D-, A) \big) 
	\end{tikzcd}
	\end{equation}
	appearing in \Cref{def_wlim} can be seen as mapping
	\begin{itemize}
		\item On objects, an object $A\in\cat{C}$ to the set of $W$-weighted co-cones over $D$ with tip $A$;
		\item On morphisms, given $g:A\to B$ in $\cat{C}$, the function which post-composes a cone with tip $A$ with $g$ to give a cone with tip $B$.
	\end{itemize}
\end{lemma}

This motivates the following definition.

\begin{definition}
	Let $D:\cat{J}\funto\cat{C}$ be a diagram, and let $W:\cat{J}\funto\cat{Set}$ be a functor.
	We call the presheaf \eqref{conew} the \newterm{presheaf of $W$-weighted cones}, and denote it by $\mathrm{Cone}^W(D,-):\cat{C}^\op\funto\cat{Set}$.
	
	Dually, given a presheaf $W:\cat{J}^\op\funto\cat{Set}$, we call the functor \eqref{coconew} the \newterm{functor of $W$-weighted co-cones}, and denote it by $\mathrm{Cone}_W(D,-):\cat{C}\funto\cat{Set}$. 
\end{definition}

\begin{proof}[Proof of \Cref{psh_weighted}]
	As always, we will prove the case for cones, the co-cone case is completely analogous and dual.
	
	First of all, on objects, given $A\in\cat{C}$, an element of the set $\sfuncat{\cat{J}}\big( W-, \cat{C}(A, D-) \big)$ is by definition a natural transformation $W\Rightarrow \cat{C}(A, D-)$ between functors $\cat{J}\funto\cat{Set}$. Once again by definition, this is a family of maps $\phi_J:WJ\to \cat{C}(A, DJ)$ natural in $J$. We have seen in \Cref{wcone_is_nat} that this is equivalently a $W$-weighted cone over $D$.
	
	On morphisms, let $g:A\to B$. The mapping 
	\[
	\begin{tikzcd}[row sep=0]
		 \sfuncat{\cat{J}}\big( W-, \cat{C}(B, D-) \big) \ar[Rightarrow]{r}{g^*} & \sfuncat{\cat{J}}\big( W-, \cat{C}(A, D-) \big)
	\end{tikzcd}
	\]
	is induced by (composing):
	\begin{enumerate}
		\item Functoriality of $\cat{C}(-, D-)$ in its first (contravariant) argument, and
		\item Functoriality of $[\cat{J},\cat{Set}](-,-)$ in its second (covariant) argument. 
	\end{enumerate}
	That is:
	\[
	\begin{tikzcd}[row sep=0]
		A \ar{rr}{g} & |[alias=Z0]| & B \\ & \vphantom{\int} \\ & \vphantom{\int} \\
		\cat{C}(A, D-) & |[alias=Z1]| & \cat{C}(B, D-) \ar[Rightarrow]{ll}{g^*} \\
		\color{dgray} \big( A \xrightarrow{g} B \xrightarrow{f} DJ \big)_{J\in\cat{J}} & |[alias=Z2]| & \color{dgray} \ar[mapsto,dgray]{ll} \big( B \xrightarrow{f} DJ \big)_{J\in\cat{J}} \\ & \vphantom{\int} \\ & \vphantom{\int} \\
		\sfuncat{\cat{J}}\big( W-, \cat{C}(B, D-) \big) & |[alias=Z3]| & \sfuncat{\cat{J}}\big( W-, \cat{C}(B, D-) \big) \ar[Rightarrow]{ll}{g^*} \\
		\color{dgray} \Big( w \in WJ \mapsto \big( A \xrightarrow{g} B \xrightarrow{\phi_j(w)} DJ \big) \Big)_{J\in\cat{J}} && \color{dgray} \ar[mapsto,dgray]{ll} \color{dgray} \Big( w \in WJ \mapsto \big( B \xrightarrow{\phi_j(w)} DJ \big) \Big)_{J\in\cat{J}}
		\ar[mapsto,dgray, from=Z0, to=Z1, shorten >=4mm, shorten <=3mm, shift right=2mm, "\mathrm{(i)}"]
		\ar[mapsto,dgray, from=Z2, to=Z3, shorten >=4mm, shorten <=3mm, shift right=2mm, "\mathrm{(ii)}"]
	\end{tikzcd}
	\]
	So, under the correspondence of \Cref{wcone_is_nat}, the functorial action on morphisms corresponds to arrow-wise precomposition of weighted cones.
\end{proof}

With this, we are now ready to prove the theorem.

\begin{proof}[Proof of \Cref{thm_wlim}]
	Once again we focus on the limit case, the colimit case is completely analogous and dual.
	
	Consider the presheaf of $W$-weighted cones:
	\[
	\begin{tikzcd}[row sep=0]
		\cat{C}^\op \ar[functor]{r}{\mathrm{Cone}^W(D,-)} & \cat{Set} \\
		\color{dgray} A \ar[mapsto,dgray]{r} & \color{dgray} \sfuncat{\cat{J}}\big( W-, \cat{C}(A, D-) \big) 
	\end{tikzcd}
	\]
	Applying the Yoneda lemma, an object $T$ representing $\mathrm{Cone}^W(D,-)$ consists first of all of a distinguished (``universal'') element $c\in \sfuncat{\cat{J}}\big( W-, \cat{C}(T, D-) \big)$, the one corresponding to the identity under the bijection characterizing the universal property:
	\[
	\begin{tikzcd}[row sep=0]
		\cat{C}(T,T) \ar[leftrightarrow]{r}{\cong} & \sfuncat{\cat{J}}\big( W-, \cat{C}(T, D-) \big) \\
		\color{dgray} \id_T \ar[mapsto, dgray]{r} & \color{dgray} c
	\end{tikzcd}
	\]
	By \Cref{psh_weighted}, we know that this is equivalently a $W$-weighted cone over $D$ with tip $T$. 
	Moreover, since by naturality of the isomorphism the following diagram must commute for every object $A$ and every $u:A\to T$,
	\[
	\begin{tikzcd}[sep=small]
		\color{dgray} \id_T \ar[mapsto, dgray]{rrrrr} \ar[mapsto,dgray]{ddddd} &&&&& \color{dgray} c \ar[mapsto,dgray]{ddddd} \\
		& \cat{C}(T,T) \ar{ddd}{-\circ u} \ar[leftrightarrow]{rrr}{\cong} &&& \sfuncat{\cat{J}}\big( W-, \cat{C}(T, D-) \big) \ar{ddd}{u^*} \\ \\ \\
		& \cat{C}(A,T) \ar[leftrightarrow]{rrr}{\cong} &&& \sfuncat{\cat{J}}\big( W-, \cat{C}(A, D-) \big) \\
		\color{dgray} u \ar[mapsto, dgray]{rrrrr} &&&&& \color{dgray} u^*c
	\end{tikzcd}
	\]
	and since the bottom arrow is a bijection,
	for every object $A$ and every element $f$ of the set $\sfuncat{\cat{J}}\big( W-, \cat{C}(A, D-) \big)$ (i.e.\ weighted cone over $A$), there exist a unique morphism $u:A\to T$ such that $f=u^*c$. 
	Now again by \Cref{psh_weighted}, $f=u^*c$ means exactly that $f$ is obtained from $c$ by arrow-wise precomposition.
\end{proof}

To summarize, weighted limits and colimits are like ordinary limits and colimits, except that:
\begin{itemize}
	\item The cones have tuples of arrows instead of single arrows as their sides, as specified by the weighting;
	\item Not every side triangle automatically commutes, instead, which triangles commute is again specified by the weighting.
\end{itemize}

\begin{remark}\label{contra_weight}
	Notice that the presheaves of weighted cones and cocones are functorial in $D$. 
	This means that weighted limits and colimits, when enough of them exist, are functorial in the diagram, as it happens for ordinary limits and colimits.
	
	Moreover, the presheaves of weighted cones and cocones (and hence the weighted limits and colimits, if they exist) are also \emph{contravariantly functorial in the weights}. This is a phenomenon that has no analogue in the unweighted case. In what follows we will see some examples.
\end{remark}

\subsection{Basic examples}

Let's see some basic examples of weighted limits used in the literature.
We start with a simple one, which we examine in detail to build some intuition. 

\begin{example}[Power]
	When we form the product of an object $X$ with itself,
	\[
	\begin{tikzcd}
		& A \ar{dl}[swap]{f_1} \ar{dr}{f_2} \ar[virtual]{d}{f} \\
		X & X\times X \ar{l}{\pi_1} \ar{r}[swap]{\pi_2} & X
	\end{tikzcd}
	\]
	we are using the same object, $X$, twice. 
	The universal property reads: we have (universal) maps $\pi_1,\pi_2:X\times X\to X$ such that for every object $A$ and every pair of maps $f_1,f_2:A\to X$, there exists a unique map $f:A\to X\times X$ with $\pi_1\circ f=f_1$ and $\pi_2\circ f=f_2$. 
	
	We can express the same universal property using $X$ only once in the diagram, but replacing the cone $(f_1,f_2)$ by a \emph{weighted} cone:
	\begin{equation}\label{alt_prod}
	\begin{tikzcd}[row sep=huge, column sep=large]
		A \ar[virtual, shorten <=1.5mm]{d}[swap]{f} \ar[shift left]{dr}{f_1} \ar[shift right]{dr}[swap, inner sep=0.5mm]{f_2} \\
		X\times X \ar[shift left, shorten >=1mm]{r}[pos=0.4]{\pi_1} \ar[shift right, shorten >=1mm]{r}[swap, pos=0.4]{\pi_2} & X
	\end{tikzcd}
	\end{equation}
	The price to pay is that this way, not every triangle in the diagram commutes. Indeed, if we want the universal property to read just as before, we want to say that $\pi_1\circ f=f_1$ and $\pi_2\circ f=f_2$, but not, for example, that $\pi_1\circ f=f_2$. 
	(We will use the convention to label the arrows exactly in such a way that matching indices give commutative triangles, whenever possible.)
	
	Let's make this formal. Denote by $\cat{1}$ the one-object, one-arrow category, and denote its single object by $1$. Denote also by $2$ the two-element set, and its elements by $w_1$ and $w_2$. 
	Consider now the following weighted one-object diagram.
	\[
	\begin{tikzcd}[row sep=0]
		\cat{1} \ar[functor]{r}{X} & \cat{C} \\
		\color{dgray} 1 \ar[mapsto,dgray]{r} & \color{dgray} X
	\end{tikzcd}
	\qquad\qquad
	\begin{tikzcd}[row sep=0]
		\cat{1} \ar[functor]{r}{2} & \cat{Set} \\
		\color{dgray} 1 \ar[mapsto,dgray]{r} & |[xshift=9mm]| \color{dgray} 2 = \{w_1, w_2\}
	\end{tikzcd}
	\]
	As a diagram, it has a single object, $X$, and no arrows (except implicitly the identity).
	What's interesting, here, is the weighted \emph{cones}.\footnote{Note that a non-weighted cone over this one-object diagram is just an arrow to $X$, and the ordinary limit is just $X$ (with the identity as limit cone).}
	Indeed, a weighted cone has a tip $A$ and \emph{two} arrows $A\to X$:
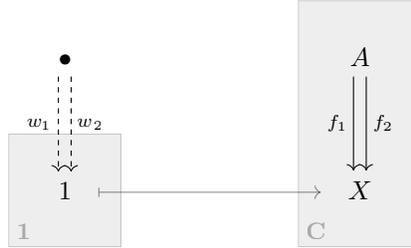
	\[
	\begin{tikzcd}[sep=small,
		blend group=multiply,
		/tikz/execute at end picture={
			\node [cbox, fit=(1), inner sep=5mm] (CL) {};
			\node [catlabel] at (CL.south west) {$\cat{1}$};
			\node [cbox, fit=(X) (A), inner sep=5mm] (CR) {};
			\node [catlabel] at (CR.south west) {$\cat{C}$};
		}]
		|[alias=E]| \bullet
		&&&&&&&& |[alias=A]| A  \\ \\ \\ \\ 
		|[alias=1]| 1 
		&&&&&&&& |[alias=X]| X
		\ar[virtual, from=E, to=1, shift left=1, "w_2"]
		\ar[virtual, from=E, to=1, shift right=1, "w_1"']
		\ar[from=A, to=X, shift left=1, "f_2"]
		\ar[from=A, to=X, shift right=1, "f_1"']
		\ar[mapsto,color=dgray, from=1, to=X, shorten=2mm]
	\end{tikzcd}
	\]
	A weighted limit of this diagram is now a universal weighted cone. It is called a \newterm{power}, and is usually denoted as $X^2$. 
	Denote its universal pair by $\pi_1,\pi_2$.
	The universal property reads: for every weighted cone, i.e.\ for every object $A$ and every pair of maps $f_1,f_2:A\to X$, there is a unique map $f:A\to X$ such that $\pi_1\circ f=f_1$ and $\pi_2\circ f=f_2$:
	
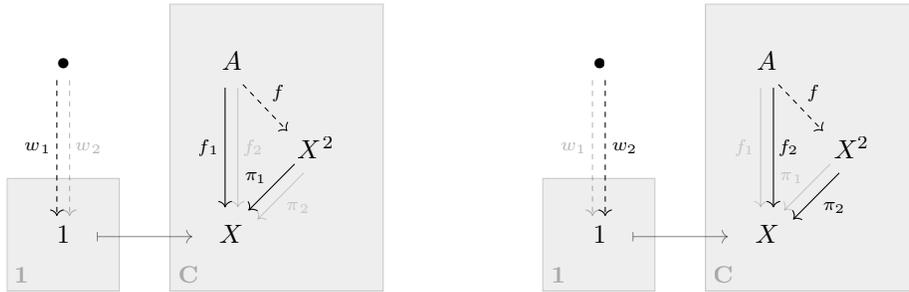
\[
	\begin{tikzcd}[sep=small,
		blend group=multiply,
		/tikz/execute at end picture={
			\node [cbox, fit=(1), inner sep=5mm] (CL) {};
			\node [catlabel] at (CL.south west) {$\cat{1}$};
			\node [cbox, fit=(X) (A) (X2), inner sep=5mm] (CR) {};
			\node [catlabel] at (CR.south west) {$\cat{C}$};
		}]
		|[alias=E]| \bullet
		&&&& |[alias=A]| A  \\ \\  
 		&&&&& |[alias=X2]| X^2 \\ \\ 
		|[alias=1]| 1 
		&&&& |[alias=X]| X
		\ar[virtual, from=E, to=1, shift left=1, "w_2", color=cgray]
		\ar[virtual, from=E, to=1, shift right=1, "w_1"']
		\ar[from=A, to=X, shift left=1, "f_2"{inner sep=0.5mm}, color=mgray, shorten=1mm]
		\ar[from=A, to=X, shift right=1, "f_1"', shorten=1mm]
		\ar[from=X2, to=X, shift left=1, "\pi_2", color=mgray]
		\ar[from=X2, to=X, shift right=1, "\pi_1"{swap,inner sep=0.5mm}]
		\ar[virtual, from=A, to=X2, shift right=1, "f"]
		\ar[mapsto,color=dgray, from=1, to=X, shorten=2mm]
	\end{tikzcd}
	\qquad\qquad\qquad
	\begin{tikzcd}[sep=small,
		blend group=multiply,
		/tikz/execute at end picture={
			\node [cbox, fit=(1), inner sep=5mm] (CL) {};
			\node [catlabel] at (CL.south west) {$\cat{1}$};
			\node [cbox, fit=(X) (A) (X2), inner sep=5mm] (CR) {};
			\node [catlabel] at (CR.south west) {$\cat{C}$};
		}]
		|[alias=E]| \bullet
		&&&& |[alias=A]| A  \\ \\  
		&&&&& |[alias=X2]| X^2 \\ \\ 
		|[alias=1]| 1 
		&&&& |[alias=X]| X
		\ar[virtual, from=E, to=1, shift left=1, "w_2"]
		\ar[virtual, from=E, to=1, shift right=1, "w_1"', color=cgray]
		\ar[from=A, to=X, shift left=1, "f_2"{inner sep=0.5mm}, shorten=1mm]
		\ar[from=A, to=X, shift right=1, "f_1"', color=mgray, shorten=1mm]
		\ar[from=X2, to=X, shift left=1, "\pi_2"]
		\ar[from=X2, to=X, shift right=1, "\pi_1"{swap,inner sep=0.5mm}, color=mgray]
		\ar[virtual, from=A, to=X2, shift right=1, "f"]
		\ar[mapsto,color=dgray, from=1, to=X, shorten=2mm]
	\end{tikzcd}
	\]
	Notice that this is the same as the universal property of $X\times X$, as expressed through the diagram \eqref{alt_prod}.
	Indeed, $X^2\cong X\times X$, as elementary algebra suggests.
	
	Let's now look at this in general.	
	Let $X$ be an object of $\cat{C}$, and let $S$ be a set. The \newterm{$S$-power} or \newterm{$S$-cartesian power} or sometimes $S$-\newterm{cotensor} of $X$, which we denote by $X^S$, is the weighted limit of the following (one-object) diagram and weight.
	\[
	\begin{tikzcd}[row sep=0]
		\cat{1} \ar[functor]{r}{X} & \cat{C} \\
		\color{dgray} 1 \ar[mapsto,dgray]{r} & \color{dgray} X
	\end{tikzcd}
	\qquad\qquad
	\begin{tikzcd}[row sep=0]
		\cat{1} \ar[functor]{r}{S} & \cat{Set} \\
		\color{dgray} 1 \ar[mapsto,dgray]{r} & \color{dgray} S
	\end{tikzcd}
	\]
	Explicitly, we have a \emph{universal set of $S$-many arrows} $X^S\to X$ (the product projections), establishing a bijection between arrows $A\to X^S$ and $S$-tuples of arrows $A\to X$. 
	
	An alternative way to define $X^S$, which readily generalizes to the enriched context, is as follows. First of all, given a set $T$, denote by $T^S$ the usual set cartesian power. (That is, either the $S$-fold cartesian product of $T$, or equivalently the set of functions $S\to T$. And yes, that is the cartesian power in $\cat{C}=\cat{Set}$.)	
	Now for an object $X$ of an arbitrary category $\cat{C}$, the power $X^S$ is equivalently an object representing the following presheaf,
	\[
	\begin{tikzcd}[row sep=0]
		\cat{C}^\op \ar[functor]{r} & \cat{Set} \\
		\color{dgray} A \ar[mapsto,dgray]{r} & \color{dgray} \cat{C}(A,X)^S
	\end{tikzcd}
	\]
	where $\cat{C}(A,X)^S$ is the set cartesian power. Indeed, notice that its elements are exactly $S$-tuples of arrows $A\to X$. More formally, the presheaf of weighted cones for $W=S:\cat{1}\funto\cat{Set}$ is 
	\[
		\sfuncat{\cat{1}}\big(W-, \cat{C}(A,D-) \big) \;=\; \cat{Set}\big(S, \cat{C}(A,X)\big) \;\cong\; \cat{C}(A,X)^S .
	\]
	A category is called \newterm{powered} or \newterm{cotensored} (over $\cat{Set}$) if powers exist for all objects $X$ and all sets $S$.
	
	In a category with products, the power always exists, and is given by the $S$-fold categorical product. Indeed, notice that they represent naturally isomorphic presheaves:
	\[
	A \quad{\color{dgray}\longmapsto}\quad \underbrace{\vphantom{\prod_{s\in S}}\cat{C}(A,X)^S}_{\text{represented by power}} \cong \underbrace{\prod_{s\in S} \cat{C}(A,X)}_{\text{represented by $S$-fold product}} 
	\]
\end{example}

\begin{example}[Copower]\label{copower}
	Let's now consider the dual case. 
	Let $X$ be an object of $\cat{C}$, and let $S$ be a set. The \newterm{$S$-copower} or sometimes $S$-\newterm{tensor} of $X$, which we denote by $S\cdot X$, is, similarly to the power, the weighted colimit of the one-object diagram $X:\cat{1}\funto\cat{C}$ weighted by $S:\cat{1}^\op\funto\cat{Set}$.
	Explicitly, we have a \emph{universal set of $S$-many arrows} $X\to S\cdot X$ (the coproduct inclusions), establishing a bijection between arrows $S\cdot X \to A$ and $S$-tuples of arrows $X\to A$. 
	
	For example, if $S=2=\{w_1,w_2\}$, we have a \emph{universal pair of arrows} $2\cdot X\to X$ (coproduct inclusions), establishing a bijection between (single) arrows $2\cdot X\to A$ and pairs of arrows $X\to A$,
	\[
	\begin{tikzcd}[sep=large]
		X \ar[shift left, shorten <=0.5mm]{r}{\iota_1} \ar[shift right, shorten <=0.5mm]{r}[swap]{\iota_2} \ar[shift left]{dr}[pos=0.6, inner sep=0.2mm]{f_1} \ar[shift right]{dr}[swap, pos=0.6, inner sep=0.2mm]{f_2}
		 & 2\cdot X \ar[virtual]{d}{f} \\
		& A
	\end{tikzcd}
	\qquad
	\qquad
	\begin{tikzcd}[sep=large]
		X \ar[shift left, shorten <=0.5mm]{r}{\iota_1} \ar[shift right, shorten <=0.5mm]{r}[swap]{\iota_2} \ar[shift left]{dr}[pos=0.6, inner sep=0.2mm]{f_1} \ar[shift right]{dr}[swap, pos=0.6, inner sep=0.2mm]{f_2}
		& X + X \ar[virtual]{d}{f} \\
		& A
	\end{tikzcd}
	\] 
	where again we use the convention that we only require $f\circ\iota_1=f_1$ and $f\circ\iota_2=f_2$, and no other triangle to commute in general.
	Alternatively, $S\cdot X$ can be defined as an object representing the following functor,
	\[
	\begin{tikzcd}[row sep=0]
		\cat{C} \ar[functor]{r} & \cat{Set} \\
		\color{dgray} A \ar[mapsto,dgray]{r} & \color{dgray} \cat{C}(X,A)^S
	\end{tikzcd}
	\]
	where $\cat{C}(X,A)^S$ is the set cartesian \emph{power} (not copower). Indeed, its elements are exactly $S$-tuples of arrows $X\to A$. 
	A category is called \newterm{copowered} or \newterm{tensored} (over $\cat{Set}$) if copowers exist for all objects $X$ and all sets $S$.
	
	In a category with coproducts, the power always exists, and is given by the $S$-fold coproduct.
	In particular, for $\cat{C}=\cat{Set}$, the copower is given by the cartesian product:
	\[
	S\cdot X \;\cong\; \coprod_{s\in S} X \;=\; \underbrace{X + \dots + X}_{S\text{ many times}} \;\cong\; S\times X .
	\]
	Note that $S\cdot X\cong S\times X$ only for $\cat{C}=\cat{Set}$: in general, in the expression ``$S\cdot X$'', $X$ is an object of $\cat{C}$, while $S$ is a set. Therefore, the expression $S\times X$ is not meaningful in a generic category.\footnote{In the enriched case, one can replace $\cat{Set}$ by the enriching category.}
\end{example}

\begin{example}[Weighted sum]\label{wsum}
	In basic algebra, in a sum where some entries are repeated, for example $x+x+y$, we can collect terms, and write for example $2\cdot x + y$. 
	Using copowers, we can do something similar. Indeed, we can express a coproduct with repeated entries, equivalently, as a \emph{coproduct of copowers}, or, as it is sometimes called, a \newterm{weighted sum} or \newterm{weighted coproduct}, $2\cdot X + Y$. (We use the same convention as in algebra, where $\cdot$ is applied before $+$.)
	We can depict its universal property as follows,
	\[
	\begin{tikzcd}
		X \ar[shift left]{dr} \ar[shift right]{dr} \ar[shift left, shorten <=0.5mm]{r} \ar[shift right, shorten <=0.5mm]{r} & 2\cdot X + Y \ar[virtual]{d} & Y \ar{dl} \ar{l} \\
		& A
	\end{tikzcd}
	\]
	where as usual, not all possible triangles on the left commute in general, but only the corresponding ones as in \Cref{copower}.
	This is the colimit of a discrete diagram (indexed by a discrete category with two objects $\cat{2}$, as for a coproduct), but where the weight of $X$ is a two-element set:
	\[
	\begin{tikzcd}[row sep=0]
		& \cat{2} \ar[functor]{r}{D} & \cat{C} \\
		|[xshift=-5mm, overlay]| \color{dgray} \mbox{\small(first object of $\cat{2}$)} & \color{dgray} 1 \ar[mapsto,dgray]{r} & \color{dgray} X \\
		|[xshift=-7mm, overlay]| \color{dgray} \mbox{\small(second object of $\cat{2}$)} & \color{dgray} 2 \ar[mapsto,dgray]{r} & \color{dgray} Y 
	\end{tikzcd}
	\qquad\qquad
	\begin{tikzcd}[row sep=0]
		\cat{2}^\op \ar[functor]{r}{W} & \cat{Set} \\
		\color{dgray} 1 \ar[mapsto,dgray]{r} & \color{dgray} 2 & |[xshift=3mm, overlay]| \color{dgray} \mbox{\small(two-element set)\vphantom{j}} \\
		\color{dgray} 2 \ar[mapsto,dgray]{r} & \color{dgray} 1 & |[xshift=3mm, overlay]| \color{dgray} \mbox{\small(one-element set)\vphantom{j}} 
	\end{tikzcd}
	\]
	As basic algebra suggests, the resulting weighted colimit is isomorphic to the repeated coproduct:
	\[
	2\cdot X + Y \;\cong\; X + X + Y .
	\]
	To see why, notice that the functors they represent are naturally isomorphic:
	\[
	A \quad{\color{dgray}\longmapsto}\quad
	\underbrace{\cat{C}(X,A)^2\funtimes\cat{C}(Y,A)}_{\text{represented by } 2\,\cdot X + Y} \;\cong\; \underbrace{\cat{C}(X,A)\funtimes\cat{C}(X,A)\funtimes\cat{C}(Y,A)}_{\text{represented by } X + X + Y}
	\]
	
	More generally, given a set $T$, denote by $\cat{T}$ the corresponding discrete category. Consider a discrete diagram $D:\cat{T}\funto\cat{C}$, and a weighting $W:\cat{T}^\op\funto\cat{Set}$. The resulting weighted colimit looks indeed like a weighted sum in algebra,
	\[
	\coprod_{t\in T} W(t)\cdot D(t) 
	\]
	and is probably what inspired the terminology ``weighted limit''. 
	It represents the following presheaf,
	\[
	A \quad{\color{dgray}\longmapsto}\quad \prod_{t\in T} \cat{C}(D(t),A)^{W(t)} \;\cong\; \prod_{t\in T} \underbrace{\cat{C}(D(t),A)\times\cdots\funtimes\cat{C}(D(t),A)}_{W(t)\text{ many times}} 
	\]
	and so it is isomorphic to a coproduct with repeated entries.
\end{example}

\begin{example}[Weighted product]\label{wprod}
	Similarly and dually to the previous example, we can form \newterm{weighted products} by means of products and powers. This is analogous to how, in algebra, $x \times x \times y = x^2\times y$. 
	As in the example above, given a set $T$, denote by $\cat{T}$ the corresponding discrete category. Consider a discrete diagram $D:\cat{T}\funto\cat{C}$, and a weighting $W:\cat{T}\funto\cat{Set}$. The resulting weighted limit looks like a monomial,
	\[
	\prod_{t\in T} D(t)^{W(t)}
	\]
	It represents the following presheaf,
	\[
	A \quad{\color{dgray}\longmapsto}\quad \prod_{t\in T} \cat{C}(A,D(t))^{W(t)} \;\cong\; \prod_{t\in T} \underbrace{\cat{C}(A,D(t))\times\cdots\funtimes\cat{C}(A,D(t))}_{W(t)\text{ many times}} 
	\]
	and so it is isomorphic to a product with repeated entries.
\end{example}

\begin{example}[Kernel pair]\label{kernel_pair}\!\!\!\footnote{I learned this example from D.\ J.\ Myers.}
	The \newterm{kernel pair} of a morphism $f:X\to Y$ is usually defined as the pullback of $f$ with itself:
	
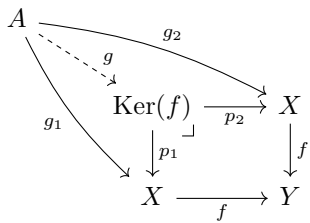
\[
	\begin{tikzcd}
		A \ar[bend right=10]{ddr}[swap]{g_1} \ar[bend left=10]{drr}{g_2} \ar[virtual]{dr}[near end]{g} \\
		& \mathrm{Ker}(f) \ar{r}[swap]{p_2} \ar{d}{p_1}
		 \arrow[dr, phantom," " {pullback}, pos=0]
		 & X \ar{d}{f} \\
		& X \ar{r}[swap]{f} & Y
	\end{tikzcd}
	\]
	Similarly to the case of powers, also here the object $X$ is used twice in the diagram, and so is $f$. Crucially, however, the arrows $p_1,p_2:\mathrm{Ker}(f)\to X$ should be allowed to be different: for example, in categories such as $\cat{Set}$, the kernel pair induces a relation on $X$ denoting \emph{which elements of $X$ can be distinguished by $f$}, and the two arrows denote the two legs of this relation, which is not always contained in the identity relation.\footnote{Note that the usual notation for pullback of arrows, here, is misleading: both arrows would read as $f^*f$.}
	
	We can express the kernel pair as a universal weighted cone as follows,
	\[
	\begin{tikzcd}[column sep=small]
		& \mathrm{Ker}(f) \ar[shift left]{dl}{p_2} \ar[shift right]{dl}[swap]{p_1} \ar{dr}{q} \\
		X \ar{rr}[swap]{f} && Y
	\end{tikzcd}
	\]
	where this time we require both triangles to commute, $f\circ p_1=f\circ p_2=q$. (And so, just like for pullbacks, we do not need to specify the map $q$.) 
	That is, we have the following weighted diagram, where now $\cat{Arr}$ is the ``walking arrow'' category $\{0\to 1\}$:
	\[
	\begin{tikzcd}[row sep=0]
		\cat{Arr} \ar[functor]{r}{D} & \cat{C} \\
		\color{dgray} 0 \ar[dgray]{dd} \ar[mapsto,dgray,shorten=2.5mm]{r} & \color{dgray} X \ar[dgray]{dd}{f} \\
		\phantom{0} \ar[mapsto,dgray,shorten=2.5mm]{r} & \phantom{X} \\
		\color{dgray} 1 \ar[mapsto,dgray,shorten=2.5mm]{r} & \color{dgray} Y
	\end{tikzcd}
	\qquad\qquad
	\begin{tikzcd}[row sep=0]
		\cat{Arr} \ar[functor]{r}{W} & \cat{Set} \\
		\color{dgray} 0 \ar[dgray]{dd} \ar[mapsto,dgray,shorten=3mm]{r} & \color{dgray} 2 \ar[dgray]{dd}{!} & |[xshift=3mm, overlay]| \color{dgray} \mbox{\small(two-element set)} \\
		\phantom{0} \ar[mapsto,dgray,shorten=3mm]{r} & \phantom{2} \\
		\color{dgray} 1 \ar[mapsto,dgray,shorten=3mm]{r} & \color{dgray} 1 & |[xshift=3mm, overlay]| \color{dgray} \mbox{\small(one-element set)}
	\end{tikzcd}
	\]
	The universal property now reads as follows: 
	
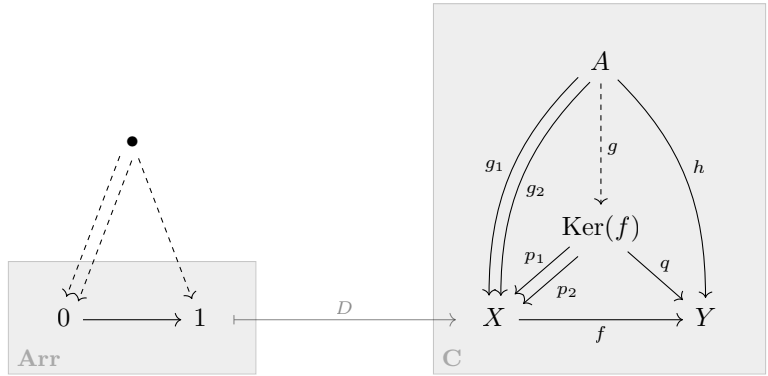
\[
	\begin{tikzcd}[column sep=small,
		blend group=multiply,
		/tikz/execute at end picture={
			\node [cbox, fit=(0) (1), inner sep=5mm] (CL) {};
			\node [catlabel] at (CL.south west) {$\cat{Arr}$};
			\node [cbox, fit=(X) (A) (Y), inner sep=5mm] (CR) {};
			\node [catlabel] at (CR.south west) {$\cat{C}$};
		}]
		 &&&&&&&&&& & |[alias=A]| A \ar[out=-135, in=90, shift left, shorten >=0.6mm]{dddl}[inner sep=0.5mm]{g_2} \ar[out=-135, in=90, shift right]{dddl}[swap, inner sep=0.5mm]{g_1} \ar[virtual, shorten <=0.5mm]{dd}[pos=0.55]{g} \ar[out=-45,in=90]{dddr}{h} \\ 
		& \bullet  \ar[virtual, shift left]{ddl} \ar[virtual, shift right]{ddl} \ar[virtual]{ddr} \\
		 &&&&&&&&&&& \mathrm{Ker}(f) \ar[shift left]{dl}{p_2} \ar[shift right]{dl}[swap, inner sep=0.2mm, pos=0.4]{p_1} \ar{dr}[inner sep=0.5mm]{q} \\
		|[alias=0]| 0 \ar{rr} && |[alias=1]| 1	
		 &&&&&&&& |[alias=X]| X \ar{rr}[swap]{f} && |[alias=Y]| Y
		 \ar[mapsto,dgray,from=1,to=X,shorten=2mm,"D"]
	\end{tikzcd}
	\]
	For every triplet of arrows $g_1,g_2:A\to X$ and $h:A\to Y$ such that $f\circ g_1=f\circ g_2=h$ (again, $h$ is determined by $g_1$ or $g_2$), there exists a unique arrow $g:A\to\mathrm{Ker}(f)$ such that $p_1\circ g=g_1$ and $p_2\circ g=g_2$ (and $q\circ g=h$, but that's implied).
	This is the usual universal property of the kernel pair.
	
	The cokernel pair, dually, can be expressed as a weighted colimit.
\end{example}

\subsection{Core results}

Before we look at more advanced examples, let's see some structural results about weighted limits and colimits.

The first result that we look at is known in the literature as ``Yoneda reduction'', ``co-Yoneda lemma'', or sometimes just ``Yoneda lemma'' (and ``ninja Yoneda lemma'' in \cite{fosco}).
It says that the weighted limit with weights given by a representable set functor, is exactly the image of the representing object.

\begin{proposition}[Yoneda reduction]\label{ninja}
	For any functor $D:\cat{J}\funto\cat{C}$ and every object $J$ of $\cat{J}$,
	\[
	\lim_{K\in\cat{J}} \big\langle \cat{J}(J,K), DK \big\rangle \;\cong\; DJ .
	\]
	(Note that this is a limit weighted by the hom-functor $\cat{J}(J,-):\cat{J}\funto\cat{Set}$.) 
	Similarly, and dually,
	\[
	\colim_{K\in\cat{J}} \big\langle \cat{J}(K,J), DK \big\rangle \;\cong\; DJ .
	\]
\end{proposition}

A possible way to interpret this the following: if an object $J$ is such that the ``virtual arrows'' of a set functor $W$ correspond exactly to the arrows out of $J$, then $DJ$, together with the induced arrows, is a universal cone over the whole diagram. 
This generalizes the usual fact that, if the indexing category $\cat{J}$ has an initial object, its image is the (unweighted) limit of the diagram (or even, that if a set of numbers has a minimum element, then such element is necessarily the greatest lower bound).

\begin{proof}
	As usual, we will consider the limit case. 
	The limit in question has as weighting the hom functor $\cat{J}(J,-)$.
	Therefore, the presheaf of weighted cones looks as follows,
	\[
	A \quad{\color{dgray}\longmapsto}\quad \sfuncat{\cat{J}}( \cat{J}(J,-), \cat{C}(A,D-) ) \;\cong\; \cat{C}(A,DJ) 
	\]
	where the isomorphism is natural and follows from the Yoneda lemma.
	This means exactly that $DJ$ represents the presheaf, and is hence the weighted limit by definition.
\end{proof}

Another direct consequence of the Yoneda lemma is the following, sometimes stated as ``every presheaf is a colimit of representable ones''. In particular, it is a weighted colimit, \emph{weighted by itself}.
We look at the case of small categories, but if one properly takes care of size issues, similar things can be said more generally.

\begin{proposition}\label{wlim_repr}
	Let $\cat{C}$ be a small category, and let $F$ be a set functor on $\cat{C}$. We can express $F$ as the $F$-weighted limit in $\sfuncat{\cat{C}}^\op$ of the Yoneda embedding $\Yon:\cat{C}\funto\sfuncat{\cat{C}}^\op$. 
	
	Dually, let $P$ be a presheaf on $\cat{C}$. Then $P$ is the $P$-weighted colimit in $\sfuncat{\cat{C}^\op}$ of the Yoneda embedding $\Yon:\cat{C}\funto\sfuncat{\cat{C}^\op}$.
\end{proposition}
\begin{proof}
	As usual, let's focus on the functor case.
	By the Yoneda lemma, in the form \eqref{yon_F}, we have an isomorphism
	\begin{align*} \sfuncat{\cat{C}}\big(F-,\sfuncat{\cat{C}}^\op(G,\Yon-)\big)
	&\cong\; \sfuncat{\cat{C}}(F,G) \\
	&\cong\; \sfuncat{\cat{C}}^\op(G,F)
	\end{align*}
	natural in $G$.
	This says exactly that $F$ represents the presheaf of $F$-weighted cones over $\Yon:\cat{C}\funto\sfuncat{\cat{C}}^\op$, and so $F$ is the desired weighted limit.
\end{proof}

Let's now look at how ``weighted limits are limits''.
In all the examples we saw in the previous section, you may have noticed that all the weighted limits and colimits could also be expressed as ordinary, non-weighted limits and colimits. For example, the power could be expressed as an iterated product. 
This is a general phenomenon:

\begin{theorem}\label{not_weighted}
	Every weighted limit in a category $\cat{C}$ can be expressed as an ordinary limit (of a different diagram).
	
	Dually, every weighted colimit can be expressed as an ordinary colimit.
\end{theorem}

One may then ask, why are weighted limits important? There are two main reasons:
\begin{enumerate}
	\item Often, expressing something as a weighted limit gives more compact expressions which reveal interesting properties. We will see examples of this in the next sections;
	\item In enriched category theory, not all weighted limits can be expressed as ordinary limits, and in most cases, it's the theory of \emph{weighted} limits which generalizes to the enriched context. 
\end{enumerate}

Let's now prove the theorem.
We first start with a definition.

\begin{definition}
	Let $W:\cat{J}\funto\cat{Set}$ be a set functor. The \newterm{category of elements} of $W$, which we denote by $\El(W)$, is defined as follows.
	\begin{itemize}
		\item Objects are pairs $(J,w)$ where $J$ is an object of $\cat{J}$, and $w$ is an element of the set $WJ$;
		\item Morphisms $(J,w)\to (K,u)$ are morphisms $g:J\to K$ of $\cat{J}$ such that the function $g_*:WJ\to WK$ maps $w\in WJ$ to $u\in WK$. 
	\end{itemize}
	
	Dually, for a presheaf $W:\cat{J}^\op\funto\cat{Set}$, its \newterm{category of elements} $\El(W)$ has
	\begin{itemize}
		\item As objects, again pairs $(J,w)$ with $J\in\cat{J}$, and $w\in WJ$;
		\item As morphisms $(J,w)\to (K,u)$, morphisms $g:J\to K$ of $\cat{J}$ such that the function $g^*:WK\to WJ$ maps $u\in WK$ to $w\in WJ$ (notice the directions of the arrows). 
	\end{itemize}
	
	In both cases, denote by $\Pi:\El(W)\funto\cat{J}$ the forgetful functor mapping $(J,w)$ to $J$ and $g:(J,w)\to (K,u)$ to $g:J\to K$.
\end{definition}

We can view the category of elements as a category of all possible ``paths'' from the virtual object to the objects of $J$. Indeed, we can view a pair $(J,w)$, with $w:\bullet\dashrightarrow J$, as a copy of $J$ which ``keeps track of how we got there''. Graphically it is helpful to consider the category $\El(W)$ as equipped with the terminal weight functor (singletons), this way we can view $\El(W)$ as a way of \emph{splitting apart parallel virtual arrows}:
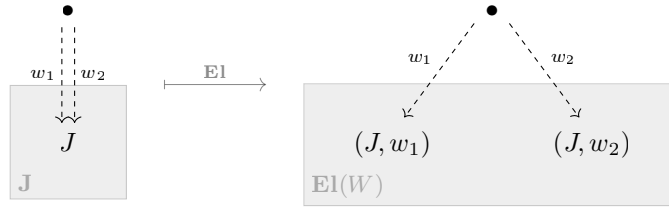
\[
\begin{tikzcd}[column sep=small,
	blend group=multiply,
	/tikz/execute at end picture={
		\node [cbox, fit=(A), inner sep=5mm] (CL) {};
		\node [catlabel] at (CL.south west) {$\cat{J}$};
		\node [cbox, fit=(A1) (A2), inner sep=5mm] (CR) {};
		\node [catlabel] at (CR.south west) {$\El(W)$};
	}]
	& |[alias=E1]| \bullet  
	&&&&&&&&& |[alias=E2]| \bullet \\ \\ 
	& |[alias=A]| J
	&&&&&&&& |[alias=A1]| (J,w_1) && |[alias=A2]| (J,w_2)
	\ar[virtual, from=E1, to=A, shift right, "w_1"']
	\ar[virtual, from=E1, to=A, shift left, "w_2"]
	\ar[virtual, from=E2, to=A1, shift right, "w_1"']
	\ar[virtual, from=E2, to=A2, shift left, "w_2"]
	\ar[mapsto,dgray,from=A,to=A1,shorten=10mm,shift left=8mm,"\El"]
\end{tikzcd}
\]
And so, in particular, it \emph{splits apart weighted cones into their separate commutative triangles}:
\begin{equation}\label{split_open}
	\begin{tikzcd}[column sep=small,
		blend group=multiply,
		/tikz/execute at end picture={
			\node [cbox, fit=(A) (B), inner sep=5mm] (CL) {};
			\node [catlabel] at (CL.south west) {$\cat{J}$};
			\node [cbox, fit=(A1) (A2) (B2), inner sep=7mm] (CR) {};
			\node [catlabel] at (CR.south west) {$\El(W)$};
		}]
		& |[alias=E1]| \bullet  
		&&&&&&&&&& |[xshift=7mm,alias=E2]| \bullet \\ \\
		&&&&&&&&&& |[alias=A1]| (J,w_1) \\ 
		|[alias=A]| J \ar{rr}[swap]{g} && |[alias=B]| K
		&&&&&&&& |[alias=H]| \phantom{(A,w_1)} && |[alias=B2]| (K,w') \\
		&&&&&&&&&& |[alias=A2]| (J,w_2) 
		\ar[from=A1, to=B2, "g^{(1)}"]
		\ar[from=A2, to=B2, "g^{(2)}"']
		\ar[virtual, from=E1, to=A, shift right, "w_1"']
		\ar[virtual, from=E1, to=A, shift left, "w_2"]
		\ar[virtual, from=E1, to=B, "w'\color{dgray}=g_*w_1=g_*w_2"]
		\ar[virtual, from=E2, to=A1, "w_1"']
		\ar[virtual, from=E2, to=A2, "w_2"{pos=0.3}]
		\ar[virtual, from=E2, to=B2, "w'\color{dgray}=g_*w_1=g_*w_2"]
		\ar[mapsto,dgray,from=B,to=H,shorten <=10mm, shorten >=12mm, shift left=8mm,"\El"]
	\end{tikzcd}
\end{equation}
Here is the precise statement.

\begin{lemma}\label{lim_el}
	Let $D:\cat{J}\funto\cat{C}$ be a diagram with weighting $W:\cat{J}\funto\cat{Set}$.
	For every object $A$ of $\cat{C}$, and naturally in $A$, there is a bijective correspondence between 
	\begin{itemize}
		\item $W$-weighted cones over $D$, and
		\item Ordinary cones over the diagram 
		\[
		\begin{tikzcd}
			\El(W) \ar[functor]{r}{\Pi} & \cat{J} \ar[functor]{r}{D} & \cat{C} 
		\end{tikzcd}
		\]
		indexed by the category of elements.
	\end{itemize}
	
	Dually, given a presheaf $W:\cat{J}^\op\funto\cat{Set}$, there is a natural bijection between $W$-weighted co-cones over $D$ and ordinary co-cones over the diagram $D\circ\Pi:\El(W)\funto\cat{J}\funto\cat{C}$.
\end{lemma}
\begin{proof}
	As usual, let's focus on the cone case. 
	
	Let $D^+:\cat{J}^{+W}\funto\cat{C}$ be a $W$-weighted cone over $D$ with tip $A$. 
	Construct now a (traditional) cone $c=\El^*(D^+)$ over $D\circ\Pi:\El(W)\funto\cat{J}\funto\cat{C}$ as follows. First of all, notice that the diagram $D\circ\Pi:\El(W)\funto\cat{J}\funto\cat{C}$ maps an object $(J,w)$ of $\cat{D}$ necessarily to $DJ$. Now set for all $(J,w)\in\El(W)$
	\[
	c_{(J,w)} \;\coloneqq\; D^+\!w : A\to DJ .
	\]
	(It may helpful to look at the diagram in \eqref{split_open}.)
	To see that this is a cone, note that for every morphism $g:(J,w)\to (K,w')$ of $\El(W)$ (such that $g_*w=w'$) we have that, by functoriality of $D^+$, 
	\[
	Dg \circ c_{(J,w)} \;=\; D^+\!g \circ D^+\!w \;=\; D^+\!(g\circ w) \;=\; D^+\!w' \;=\; c_{(K,w')}
	\]
	i.e.\ the following diagram commutes:
	
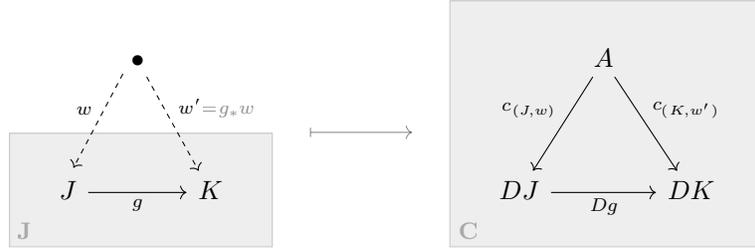
\[
	\begin{tikzcd}[column sep=small,
		blend group=multiply,
		/tikz/execute at end picture={
			\node [cbox, fit=(A) (B), inner sep=5mm] (CL) {};
			\node [catlabel] at (CL.south west) {$\cat{J}$};
			\node [cbox, fit=(A2) (B2) (E2), inner sep=5mm] (CR) {};
			\node [catlabel] at (CR.south west) {$\cat{C}$};
		}]
		& |[alias=E1]| \bullet  
		&&&&&&&&&& |[alias=E2]| A \\ \\
		|[alias=A]| J \ar{rr}[swap]{g} && |[alias=B]| K
		&&&&&&&& |[alias=A2]| DJ  && |[alias=B2]| DK \\
		\ar[from=A2, to=B2, "Dg"']
		\ar[virtual, from=E1, to=A, shift right, "w"']
		\ar[virtual, from=E1, to=B, "w'\color{dgray}=g_*w"]
		\ar[from=E2, to=A2, "c_{(J,w)}"']
		\ar[from=E2, to=B2, "c_{(K,w')}"]
		\ar[mapsto,dgray,from=B,to=A2,shorten=10mm, shift left=8mm]
	\end{tikzcd}
	\]
	This establishes a mapping
	\[
	\begin{tikzcd}
		\mathrm{Cone}^W(D,A) \ar{r}{\El^*} & \mathrm{Cone}^1(D\circ\Pi,A) .
	\end{tikzcd}
	\]
	It remains to show that this map is bijective and natural in $A$.
	
	To show that it is injective, suppose that for cones $D^+$ and ${D^+}'$ we have $c=\El^*(D^+)=\El^*({D^+}')$.
	Then for all $J\in\cat{J}$ and $w\in WJ$, 
	\[
	D^+\!w \;=\; c_{(J,w)} \;=\; {D^+}'\!w ,
	\]
	i.e.\ the weighted cones have exactly the same arrows, so $D^+={D^+}'$.
	
	To show that it is surjective, let $c=(c_{(J,w)}:A\to DJ)_{(J,w)\in\El(W)}$ be a (traditional) cone over $D\circ\Pi$. 
	Then we can write it as $\El^*(D+)$ where for $J\in\cat{J}$ and $w\in WJ$, $D^+\!w=c_{(J,w)}$. To show that this is a weighted cone, i.e.~that it is a functor $D^+:\cat{J}^{+W}\funto\cat{C}$, notice that for every $w\in WJ$ (or $w:E\to J$) and for all $g:J\to K$,
	\[
	D^+\!(g\circ w) \;=\; c_{(K,g_*w)} \;=\; Dg\circ c_{(J,w)} \;=\; D^+\!g\circ D^+\!w ,
	\]
	using the fact that $c$ is a cone.
	
	Finally, to show that it is natural, let $f:A\to B$. Then the following diagram commutes,
	\[
	\begin{tikzcd}
		\mathrm{Cone}^W(D,B) \ar{d}{f^*} \ar{r}{\El^*} & \mathrm{Cone}^1(D\circ\Pi:B) \ar{d}{f^*} \\ 
		\mathrm{Cone}^W(D,A) \ar{r}{\El^*} & \mathrm{Cone}^1(D\circ\Pi,A) 
	\end{tikzcd}
	\]
	since both vertical arrows just precompose the arrows of the cones with $f$.
\end{proof}

\begin{proof}[Proof of \Cref{not_weighted}]
	(As usual, we will focus on the limit case.)
	
	By \Cref{lim_el}, the presheaves $\mathrm{Cone}^W(D,-)$ and $\mathrm{Cone}^1(D\circ\Pi,-):\cat{C}^\op\funto\cat{Set}$ are naturally isomorphic. Therefore one is representable if and only if the other one is, and in that case their representing objects are isomorphic.
	By definition, the representing objects are respectively the $W$-weighted limit of $D$ and the ordinary limit of $D\circ\Pi$, which must then coincide.
\end{proof}

\begin{corollary}
	A category has all (small) weighted limits if and only if it has all (small) ordinary limits, and a functor preserves all (small) weighted limits if and only if it preserves all (small) ordinary limits.
	
	The same can be said about colimits.
\end{corollary}

\begin{corollary}
	Just as ordinary limits and colimits, weighted limits and colimits in functor categories are computed pointwise.
\end{corollary}

Finally, we present a weighted version of the famous result that hom functors preserve limits.

\begin{theorem}\label{hom_wlim}
	Hom-functors preserve weighted limits and colimits in the same way as they preserve ordinary limits and colimits.
	More in detail, given $D:\cat{J}\funto\cat{C}$ and $W:\cat{J}\funto\cat{Set}$, if the weighted limit exists, then we have a natural isomorphism
	\[
	\cat{C} \left( A , \lim_{J\in\cat{J}} \big\langle WJ, DJ \big\rangle \right) \;\cong\;  \lim_{J\in\cat{J}} \big\langle WJ , \cat{C}(A, DJ) \big\rangle 
	\]
	preserving the universal weighted cone.
	
	Dually, given a presheaf $W:\cat{J}^\op\funto\cat{Set}$, if the weighted colimit of $D$ exists, we have a natural isomorphism
	\[
	\cat{C} \left( \colim_{J\in\cat{J}} \big\langle WJ, DJ \big\rangle , A \right) \;\cong\; \lim_{J\in\cat{J}^\op} \big\langle WJ , \cat{C}(DJ,A) \big\rangle 
	\]
	preserving the universal weighted co-cone.
\end{theorem}

(Note that a presheaf $W:\cat{J}^\op\funto\cat{Set}$ can be equivalently written as $W:(\cat{J}^\op)\funto\cat{Set}$, i.e.\ as a set functor on $\cat{J}^\op$, indexing a weighted limit -- and vice versa.)

The proof follows closely the one for the unweighted case.

\begin{proof}
	As usual, let's consider the case of limits. Denote the weighted limit in $\cat{C}$ by $L$, and denote the arrows of the universal weighted cone by $c_{J,w}:L\to DJ$, for $J\in\cat{J}$ and $w\in WJ$. 
	We have to show that $\cat{C}(A,L)$ is a weighted limit in $\cat{Set}$, with the induced cone:
	
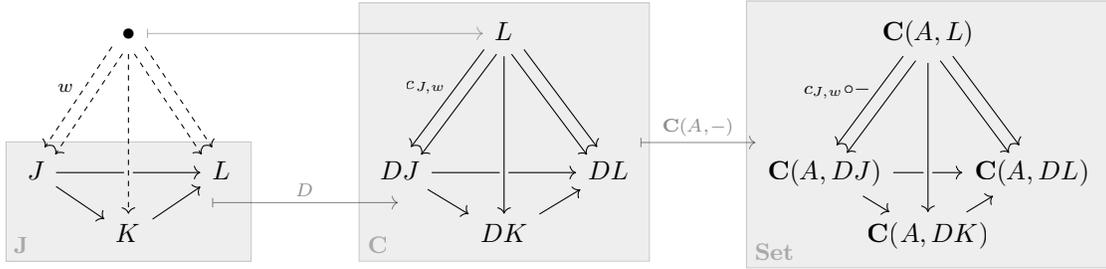
\[
	\begin{tikzcd}[sep=small,
		blend group=multiply,
		/tikz/execute at end picture={
			\node [cbox, fit=(J) (L) (K), inner sep=1.2mm] (CL) {};
			\node [catlabel] at (CL.south west) {$\cat{J}$};
			\node [cbox, fit=(DJ) (DL) (DK) (T), inner sep=1.2mm] (CC) {};
			\node [catlabel] at (CC.south west) {$\cat{C}$};
			\node [cbox, fit=(ADJ) (ADL) (ADK) (AT), inner sep=1.2mm] (CR) {};
			\node [catlabel] at (CR.south west) {$\cat{Set}$};
		}]
		& |[alias=E]| \bullet
		&&&&&& |[alias=T]| L
		&&&&&&&&&& |[alias=AT]| \cat{C}(A,L)  \\ \\ \\ \\ 
		|[alias=J, xshift=-5mm]| J \ar{dr} \ar[-, shorten >=-1mm]{r} & |[alias=BR1]| \ar[shorten <=-1mm]{r} & |[alias=L, xshift=5mm]| L 
		&&&& |[alias=DJ, xshift=-5mm]| DJ \ar{dr} \ar[-, shorten >=-1mm]{r} & |[alias=BR2]| \ar[shorten <=-1mm]{r} & |[alias=DL, xshift=5mm,overlay]| DL
		&&&&&&&& |[alias=ADJ,overlay]| \cat{C}(A,DJ) \ar{dr} \ar[-, shorten >=-1mm]{r} & |[alias=ABR2]| \ar[shorten <=-1mm]{r} & |[alias=ADL, overlay]| \cat{C}(A,DL) \\
		& |[alias=K]| K \ar{ur}
		&&&&&& |[alias=DK]| DK \ar{ur}
		&&&&&&&&&& |[alias=ADK]| \cat{C}(A,DK) \ar{ur}
		\ar[virtual, from=E, to=J, shift left=1]
		\ar[virtual, from=E, to=J, shift right=1,"w"{swap,inner sep=0.6mm}]
		\ar[virtual, from=E, to=K, shorten <=0.5mm]
		\ar[virtual, from=E, to=L, shift left=1]
		\ar[virtual, from=E, to=L, shift right=1]
		\ar[from=T, to=DJ, shift left=1]
		\ar[from=T, to=DJ, shift right=1, shorten <=0.5mm, "c_{J,w}"{swap,inner sep=0.3mm}]
		\ar[from=T, to=DK, shorten <=0.8mm]
		\ar[from=T, to=DL, shift left=1, shorten <=0.5mm]
		\ar[from=T, to=DL, shift right=1]
		\ar[from=AT, to=ADJ, shift left=1]
		\ar[from=AT, to=ADJ, shift right=1, shorten <=0.5mm, "c_{J,w}\circ-"{swap,inner sep=0}]
		\ar[from=AT, to=ADK, shorten <=0.8mm]
		\ar[from=AT, to=ADL, shift left=1, shorten <=0.5mm]
		\ar[from=AT, to=ADL, shift right=1]
		\ar[mapsto,color=cgray, from=E, to=T]
		\ar[mapsto,color=dgray, from=L, to=DJ, "D", shift right=4mm, shorten=-4mm]
		\ar[mapsto,color=dgray, from=DL, to=ADJ, "{\cat{C}(A,-)}", shift left=4mm]
	\end{tikzcd}
	\]
	To this end, let $S$ be a cone over $\cat{C}(A,D-)$, and denote by $d_{J,w}:S\to \cat{C}(A,DJ)$ the arrows of the cone (for $J\in\cat{J}$ and $w\in WJ$).
	We have to show that there is a unique function $u:S\to \cat{C}(A,L)$ such that for all $J\in\cat{J}$ and $w\in WJ$, $c_{J,w}\circ u=d_{J,u}$.
	\begin{equation}\label{liminset}
	\begin{tikzcd}[sep=small]
		& |[alias=S]| S \ar[virtual, shorten <=0.5mm]{dddd}[swap]{u} \\ \\ \\ \\
		& |[alias=AT]| \cat{C}(A,L)  \\ \\ \\ \\ 
		|[alias=ADJ,overlay]| \cat{C}(A,DJ) \ar{dr} \ar[-, shorten >=-1mm]{r} & |[alias=ABR2]| \ar[shorten <=-1mm]{r} & |[alias=ADL, overlay]| \cat{C}(A,DL) \\
		& |[alias=ADK]| \cat{C}(A,DK) \ar{ur}
		\ar[from=AT, to=ADJ, shift left=1,mgray]
		\ar[from=AT, to=ADJ, shift right=1, shorten <=0.5mm, "c_{J,w}\circ-"{inner sep=0}]
		\ar[from=AT, to=ADK, shorten <=0.8mm,mgray]
		\ar[from=AT, to=ADL, shift left=1, shorten <=0.5mm,mgray]
		\ar[from=AT, to=ADL, shift right=1,mgray]
		\ar[from=S, to=ADJ, shift left=1,mgray,bend right, shorten >=0.7mm]
		\ar[from=S, to=ADJ, shift right=1, shorten <=0.5mm, "d_{J,w}"{swap,inner sep=0},bend right]
		\ar[from=S, to=ADK, shorten <=0.8mm,mgray,bend left]
		\ar[from=S, to=ADL, shift left=1, shorten <=0.5mm,mgray,bend left]
		\ar[from=S, to=ADL, shift right=1,mgray,bend left,shorten >=0.7mm]
	\end{tikzcd}
	\end{equation}
	Note now that the function $d_{J,w}:S\to \cat{C}(A,DJ)$ is an $S$-indexed tuple of arrows $A\to DJ$. 
	For each $s\in S$, the element $d_{J,w}(s)$ is an arrow $A\to DJ$, and if we keep $s$ fixed and vary $J\in\cat{J}$ and $w\in WJ$, 
	the tuple $(d_{J,w}(s))_{J\in\cat{J},w\in WJ}$ is exactly a weighted cone in $\cat{C}$. 
	By the universal property of $L$, to this tuple there corresponds a unique morphism $\tilde{u}_s:A\to L$, such that for all $J$ and $w$, 
	\begin{equation}\label{utilde}
	c_{J,w}\circ\tilde{u}_s=d_{J,w}(s) .
	\end{equation}
	We therefore define our map $u:S\to \cat{C}(A,L)$ as follows: 
	\[
	u(s) \coloneqq \tilde{u}_s \in\cat{C}(A,L) .
	\]
	This way, by \eqref{utilde}, the triangles as in \eqref{liminset} commutes. 
	To see that the map $u$ is unique, suppose that another map $u':S\to \cat{C}(A,L)$ makes the triangle \eqref{liminset} commute.
	Then for all $s\in S$, the element $u'(s)\in\cat{C}(A,L)$, which is an arrow $A\to L$, is such that for all $J$ and $w$, $c_{J,w}\circ u'(s)=d_{J,w}(s)$. But we know, by the universal property of $L$, that only one possible map $A\to L$ with this property, namely $\tilde{u}_s$. So necessarily $u'(s)=\tilde{u}_s$, which means that $u'=u$.
\end{proof}

\subsection{Weighted (co)limits as (co)ends}

In the last few years it has become very common to write weighted limits and colimits in the following way.
\begin{equation}\label{wlim_end}
\lim_{J\in\cat{J}}\big\langle WJ, DJ \big\rangle \;\cong\; \End{J\in\cat{J}} {DJ\,}^{WJ} \qquad\qquad \colim_{J\in\cat{J}}\big\langle WJ, DJ \big\rangle \;\cong\; \Coend{J\in\cat{J}} WJ \cdot DJ
\end{equation}
Inside the integral signs we can recognize powers and copowers. The integral signs themselves are not actual integrals in the sense of analysis, they denote particular limits and colimits called \emph{ends} and \emph{coends}. 

Let's now see what these symbols mean. There are at least three levels of understanding them:
\begin{enumerate}
	\item At the most shallow level, one could consider the expressions in \eqref{wlim_end} simply as an alternative way of writing weighted limits and colimits in categories such as $\cat{Set}$. 
	\item One can describe ends and coends in terms of weighted limits and colimits. This will be the topic of the present section. 
	\item Ends and coends have their own very deep theory, interesting for its own sake. The standard reference for that is the book \cite{fosco}, to which we refer the interested readers.
\end{enumerate}

Let's get to (ii).
Let $D:\cat{J}\funto\cat{C}$ be a diagram with weighting $W:\cat{J}\funto\cat{Set}$, and consider its weighted limit. 
For every object $J$ of $\cat{J}$, any $W$-weighted cone, for example the universal one, gives a tuple of arrows from the tip $T$ to $DJ$ indexed by the set $WJ$. 
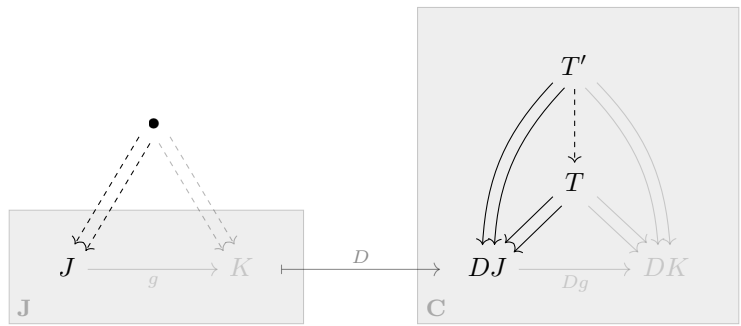
\[
\begin{tikzcd}[sep=small,
	blend group=multiply,
	/tikz/execute at end picture={
		\node [cbox, fit=(J) (K), inner sep=5mm] (CL) {};
		\node [catlabel] at (CL.south west) {$\cat{J}$};
		\node [cbox, fit=(DJ) (DK) (TP), inner sep=5mm] (CR) {};
		\node [catlabel] at (CR.south west) {$\cat{C}$};
	}]
	&&&&&&&&& |[alias=TP]| T' \\ 
	& |[alias=E]| \bullet \\  
	&&&&&&&&& |[alias=T]| T  \\ \\  
	|[alias=J, xshift=-5mm]| J \ar[mgray]{rr}[swap]{g} && |[alias=K, xshift=5mm]| \color{mgray} K
	&&&&&& |[alias=DJ]| DJ \ar[mgray]{rr}[swap]{Dg} && |[alias=DK]| \color{mgray} DK 
	\ar[virtual, from=E, to=J, shift left]
	\ar[virtual, from=E, to=J, shift right]
	\ar[virtual, from=E, to=K, shift left=1, color=cgray]
	\ar[virtual, from=E, to=K, shift right=1, color=cgray]
	\ar[from=T, to=DJ, shift left]
	\ar[from=T, to=DJ, shift right]
	\ar[from=T, to=DK, shift left=1, color=mgray]
	\ar[from=T, to=DK, shift right=1, color=mgray]
	\ar[from=TP, to=DJ, shift left, bend right=20,shorten >=0.7mm]
	\ar[from=TP, to=DJ, shift right, bend right=20]
	\ar[from=TP, to=DK, shift left=1, color=mgray, bend left=20]
	\ar[from=TP, to=DK, shift right=1, color=mgray, bend left=20,shorten >=0.7mm]
	\ar[virtual, from=TP, to=T, shorten <=0.5mm]
	\ar[mapsto,color=dgray, from=K, to=DJ, "D",shorten=2mm]
\end{tikzcd}
\]
Now suppose that the category $\cat{C}$ has powers. By definition of power, a $WJ$-indexed tuple of arrows $T\to DJ$ corresponds naturally to a single arrow $T\to {DJ\,}^{WJ}$:
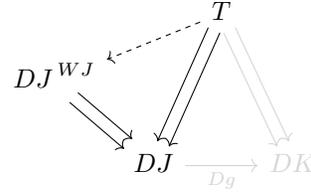
\[
\begin{tikzcd}[row sep=small, column sep=tiny]
	&& |[alias=T]| T \ar[virtual]{dll} \\ 
	{DJ\,}^{WJ} \ar[shift left]{ddr} \ar[shift right]{ddr} \\ \\  
	&|[alias=DJ]| DJ \ar[mgray]{rr}[swap]{Dg} && |[alias=DK]| \color{mgray} DK 
	\ar[from=T, to=DJ, shift left]
	\ar[from=T, to=DJ, shift right]
	\ar[from=T, to=DK, shift left=1, color=mgray]
	\ar[from=T, to=DK, shift right=1, color=mgray]
\end{tikzcd}
\]
We can now do this for all the objects of $\cat{J}$. We want to express $T$ as a particular limit (an \emph{end}, as we will see) of the ${DJ\,}^{WJ}$, instead of as a weighted limit of the $DJ$. 
Now, for each morphism $g:J\to K$ of $\cat{J}$, we could try to induce from the morphism $Dg:DJ\to DK$ a morphism ${DJ\,}^{WJ}\to {DK\,}^{WK}$, forming a new diagram ``parallel'' to the original one:
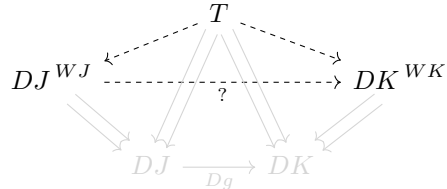
\[
\begin{tikzcd}[row sep=small, column sep=tiny]
	&& |[alias=T]| T \ar[virtual]{dll} \ar[virtual]{drr} \\ 
	{DJ\,}^{WJ} \ar[shift left, color=mgray]{ddr} \ar[shift right, color=mgray]{ddr} \ar[virtual]{rrrr}[swap]{?} &&&& {DK\,}^{WK} \ar[shift left, color=mgray]{ddl} \ar[shift right, color=mgray]{ddl} \\ \\  
	&|[alias=DJ]| \color{mgray} DJ \ar[mgray]{rr}[swap]{Dg} && |[alias=DK]| \color{mgray} DK
	\ar[from=T, to=DJ, shift left, color=mgray]
	\ar[from=T, to=DJ, shift right, color=mgray]
	\ar[from=T, to=DK, shift left=1, color=mgray]
	\ar[from=T, to=DK, shift right=1, color=mgray]
\end{tikzcd}
\]
Things are however not so simple. Indeed, from $Dg$ we could induce a morphism ${DJ\,}^{WJ}\to {DK\,}^{WJ}$, with the same exponent, but not ${DJ\,}^{WJ}\to {DK\,}^{WK}$. 
What we \emph{can} do, however, is induce a pair of morphisms (called a \emph{cospan}) as follows:
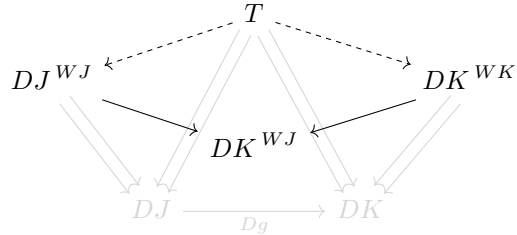
\[
\begin{tikzcd}[row sep=small, column sep=tiny]
	&& |[alias=T]| T \ar[virtual]{dll} \ar[virtual]{drr} \\ 
	|[alias=ONE]| {DJ\,}^{WJ} \ar[shift left, color=mgray]{ddr} \ar[shift right, color=mgray]{ddr} &&&& |[alias=THREE]| {DK\,}^{WK} \ar[shift left, color=mgray]{ddl} \ar[shift right, color=mgray]{ddl} \\ 
	&& |[alias=TWO]| {DK\,}^{WJ} \\  
	&|[alias=DJ]| \color{mgray} DJ \ar[mgray]{rr}[swap]{Dg} && |[alias=DK]| \color{mgray} DK
	\ar[from=T, to=DJ, shift left, color=mgray]
	\ar[from=T, to=DJ, shift right, color=mgray]
	\ar[from=T, to=DK, shift left=1, color=mgray]
	\ar[from=T, to=DK, shift right=1, color=mgray]
	\ar[from=ONE, to=TWO]
	\ar[from=THREE, to=TWO]
\end{tikzcd}
\]
This is due to the fact that the power ${DK\,}^{WJ}$ is functorial in $K$, and \emph{contravariantly functorial in $J$}.
To see how, in detail, let's first fix some terminology.

\begin{definition}\label{def_bifunctor}
	A \newterm{bifunctor} $B:\cat{C}\funtimes\cat{D}\funto\cat{E}$ is a functorial assignment in both variables, 
	\[
	\forall\,C\in\cat{C} \;\;:\quad 
	\begin{tikzcd}[sep=0]
		\cat{D} \ar[functor]{rr}{B(C,-)} && \cat{E} \\
		\color{dgray} D \ar[dgray]{dd}{d} & \color{dgray} \longmapsto & \color{dgray} B(C,D) \ar[dgray]{dd}{B(C,d)} \\
		\phantom{0} & \color{dgray} \longmapsto & \phantom{2} \\
		\color{dgray} D' & \color{dgray} \longmapsto & \color{dgray} B(C,D') 
	\end{tikzcd}
	\qquad\qquad
	\forall\,D\in\cat{D} \;\;:\quad 
	\begin{tikzcd}[sep=0]
		\cat{D} \ar[functor]{rr}{B(-,D)} && \cat{E} \\
		\color{dgray} C \ar[dgray]{dd}{c} & \color{dgray} \longmapsto & \color{dgray} B(C,D) \ar[dgray]{dd}{B(c,D)} \\
		\phantom{0} & \color{dgray} \longmapsto & \phantom{2} \\
		\color{dgray} C' & \color{dgray} \longmapsto & \color{dgray} B(C',D) 
	\end{tikzcd}
	\]
	where moreover the following diagram commutes for all morphisms $c:C\to C'$ of $\cat{C}$ and $d:D\to D'$ of $\cat{D}$:
	\begin{equation}\label{interchange}
	\begin{tikzcd}[sep=large]
		B(C,D) \ar{d}[]{B(C,d)} \ar{r}{B(c,D)} & B(C',D) \ar{d}{B(C',d)} \\
		B(C,D') \ar{r}[]{B(c,D')} & B(C',D')
	\end{tikzcd}
	\end{equation}
	(Equivalently, it is a functor on the product category $\cat{C}\funtimes\cat{D}$.)
\end{definition}

We will also call a \newterm{bidiagram} a bifunctor in the form $\cat{J}^\op\funtimes\cat{J}\funto\cat{C}$. 
A canonical example is the hom functor $\cat{J}(-,-)$. 

\begin{lemma}\label{pow_functorial}
	Let $\cat{C}$ be a category with powers. 
	\begin{enumerate}
		\item For every set $S$, the assignment $X\mapsto X^S$ extends canonically to a functor $(-)^S:\cat{C}\funto\cat{C}$.
		\item For every object $X$ of $\cat{C}$, the assignment $S\mapsto X^S$ extends canonically to a functor $X^{(-)}:\cat{Set}^\op\funto\cat{C}$.
		\item The two assignment make the following diagram commute for all functions $f:S\to T$ and morphisms $g:X\to Y$, 
		\[
		\begin{tikzcd}
			X^T \ar{r}{g^T} \ar{d}{X^f} & Y^T \ar{d}{Y^f} \\
			X^S \ar{r}{g^S} & Y^S
		\end{tikzcd}
		\]
		giving a bifunctor $P:\cat{Set}^\op\funtimes\cat{C}\funto\cat{C}$.
	\end{enumerate}
	
	Dually, let $\cat{C}$ be a category with copowers. The assignment $(S,X)\mapsto S\cdot X$ is a bifunctor $P:\cat{Set}\funtimes\cat{C}\funto\cat{C}$.
\end{lemma}

The result could be seen as following directly from functoriality of the cone presheaves, as we saw in \Cref{contra_weight}.
We give a direct proof here, which will hopefully provide additional clarity.

\begin{proof}[Proof of \Cref{pow_functorial}]
	First, let's start with the power case. 
	\begin{enumerate}
		\item Let $S$ be a set, and let $g:X\to Y$ be a morphism of $\cat{C}$. For every $s\in S$, denote by $c_{X,s}:X^S\to X$ be the $s$-th arrow of the weighted cone of $X^S$, and by $c_{Y,s}:Y^S\to Y$ be the $s$-th arrow of the weighted cone of $Y^S$.
		Consider now the following diagram.
		\[
		\begin{tikzcd}[sep=large]
			X^S \ar[shift left,mgray]{d} \ar[shift right]{d}[swap]{c_{X,s}} \ar[virtual]{r} & Y^S \ar[shift left,mgray]{d} \ar[shift right]{d}[swap]{c_{Y,s}}  \\
			X \ar{r}[swap]{g} & Y
		\end{tikzcd}
		\]
		We have that for every $s\in S$, the composition $g\circ c_{X,s}$ gives an arrow $X^S\to Y$. We therefore have an $S$-tuple of arrows $X^S\to Y$, which by the universal property of $Y^S$ must factor uniquely through the weighted cone of $Y$. That is, there exists a unique morphism $u:X^S\to Y^S$ such that for all $s\in S$, $g\circ c_{X,s}=c_{Y,s}\circ u$.
		We denote this $u$ by $g^S:X^S\to Y^S$, which gives the action on morphisms of the desired functor $\cat{C}\funto\cat{C}$. Functoriality follows from uniqueness of $u$.
		
		\item Let $X$ be an object of $\cat{C}$, and let $f:S\to T$ be a function. Similarly to above, for $s\in S$ denote by $c_{X,s}:X^S\to X$ be the $s$-th arrow of the weighted cone of $X^S$, and for $t\in T$ denote by $c'_{X,t}:X^T\to X$ be the $t$-th arrow of the weighted cone of $X^T$.
		Construct now an $S$-tuple of arrows $(d_s:X^T\to X)_{s\in S}$ as follows (mind that we used both $S$ and $T$): for each $s\in S$, we set $d_s=c_{X,f(s)}:X^T\to X$. In other words, for $t\in T$,
		\begin{itemize}
			\item If $f^{-1}(t)$ has a single element $s\in S$, the map $c'_{X,t}:X^T\to X$ will appear in the tuple $(d_s)_{s\in S}$ exactly once, namely, as the $s$-th entry;
			\item If $f^{-1}(t)$ is a set with several elements, the map $c'_{X,t}:X^T\to X$ will appear in the tuple $(d_s)_{s\in S}$ exactly as many times as the elements of $f^{-1}(t)$;
			\item If $f^{-1}(t)$ is empty, the map $c'_{X,t}:X^T\to X$ will not appear in the tuple $(d_s)_{s\in S}$.
		\end{itemize}
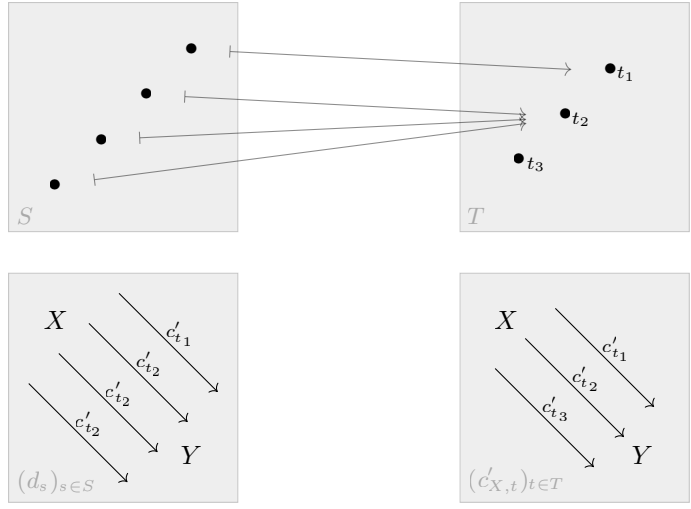
		\[
		\begin{tikzpicture}[blend group=multiply,scale=0.6]
			\node (S1) at (3,3) {$\bullet$};
			\node (S2) at (2,2) {$\bullet$};
			\node (S3) at (1,1) {$\bullet$};
			\node (S4) at (0,0) {$\bullet$};
			\node (SA) at (0,0) {};
			\node (SB) at (3,3) {};
			\node [cbox, fit=(SA) (SB), inner sep=5mm] (S) {};
			\node [catlabel] at (S.south west) {$S$};
			
			\node (T1) at (12.5,2.5) {$\bullet_{t_1}$};
			\node (T2) at (11.5,1.5) {$\bullet_{t_2}$};
			\node (T3) at (10.5,0.5) {$\bullet_{t_3}$};
			\node (TA) at (10,0) {};
			\node (TB) at (13,3) {};
			\node [cbox, fit=(TA) (TB), inner sep=5mm] (T) {};
			\node [catlabel] at (T.south west) {$T$};
			
			\draw [mapsto,dgray,thin,shorten <=3mm,shorten >=3mm] (S1) -- (T1) ;
			\draw [mapsto,dgray,thin,shorten <=3mm,shorten >=3mm] (S2) -- (T2) ;
			\draw [mapsto,dgray,thin,shorten <=3mm,shorten >=3mm] (S3) -- (T2) ;
			\draw [mapsto,dgray,thin,shorten <=3mm,shorten >=3mm] (S4) -- (T2) ;
			
			\node (X1) at (0,-3) {$X$};
			\node (Y1) at (3,-6) {$Y$};
			\begin{scope}[transform canvas={xshift=6mm,yshift=6mm}]
				\draw [->,thin] (X1) -- (Y1) node[midway,above right,inner sep=-0.5mm] {\scriptsize{$c'_{t_1}$}};
			\end{scope}
			\begin{scope}[transform canvas={xshift=2mm,yshift=2mm}]
				\draw [->,thin] (X1) -- (Y1) node[midway,above right,inner sep=-0.5mm] {\scriptsize{$c'_{t_2}$}};
			\end{scope}
			\begin{scope}[transform canvas={xshift=-2mm,yshift=-2mm}]
				\draw [->,thin] (X1) -- (Y1) node[midway,above right,inner sep=-0.5mm] {\scriptsize{$c'_{t_2}$}};
			\end{scope}
			\begin{scope}[transform canvas={xshift=-6mm,yshift=-6mm}]
				\draw [->,thin] (X1) -- (Y1) node[midway,above right,inner sep=-0.5mm] {\scriptsize{$c'_{t_2}$}};
			\end{scope}
			\node (DA) at (0,-3) {};
			\node (DB) at (3,-6) {};
			\node [cbox, fit=(DA) (DB), inner sep=5mm] (D) {};
			\node [catlabel] at (D.south west) {$(d_s)_{s\in S}$};
			
			\node (X2) at (10,-3) {$X$};
			\node (Y2) at (13,-6) {$Y$};
			\begin{scope}[transform canvas={xshift=4mm,yshift=4mm}]
				\draw [->,thin] (X2) -- (Y2) node[midway,above right,inner sep=-0.5mm] {\scriptsize{$c'_{t_1}$}};
			\end{scope}
			\draw [->,thin] (X2) -- (Y2) node[midway,above right,inner sep=-0.5mm] {\scriptsize{$c'_{t_2}$}};
			\begin{scope}[transform canvas={xshift=-4mm,yshift=-4mm}]
				\draw [->,thin] (X2) -- (Y2) node[midway,above right,inner sep=-0.5mm] {\scriptsize{$c'_{t_3}$}};
			\end{scope}
			\node (CA) at (10,-3) {};
			\node (CB) at (13,-6) {};
			\node [cbox, fit=(CA) (CB), inner sep=5mm] (C) {};
			\node [catlabel] at (C.south west) {$(c'_{X,t})_{t\in T}$};
		\end{tikzpicture}
		\]
		Now by the universal property of $X^S$, there exists a unique map $u:X^T\to X^S$ such that for all $s\in S$, $d_s=c_{X,s}\circ u=c'_{X,f(s)}\circ u$:
		\[
		\begin{tikzcd}[row sep=large]
			X^T \ar[shift left,mgray]{dr} \ar[shift right]{dr}[swap,inner sep=0]{c'_{X,f(s)}} \ar[virtual]{rr} && X^S \ar[shift left]{dl}{c_{X,s}} \ar[shift right,mgray]{dl}  \\
			& X
		\end{tikzcd}
		\]
		We denote this map by $X^f:X^T\to X^S$, which gives the action on morphisms of the desired functor $\cat{Set}^\op\funto\cat{C}$. Again, functoriality follows from the uniqueness of $u$.
		
		\item Consider the following diagram, where we consider $S$-many arrows $X^S\to S$ and $Y^S\to Y$ (the universal weighted cones), as well as also $S$-many arrows $X^T\to X$ and $Y^T\to X$ (the ones we constructed in (ii) above.)
		\begin{equation}\label{prism}
		\begin{tikzcd}[column sep=small]
			X^T \ar[shift left,mgray]{ddr} \ar[shift right]{ddr}[swap,inner sep=0]{c'_{X,f(s)}} \ar{drr}[inner sep=0.5mm]{X^f} \ar{rrrr}{g^T} &&&& Y^T \ar{drr}{Y^f} \ar[shift left,mgray]{ddr} \ar[shift right]{ddr}[swap, pos=0.1,inner sep=0]{c'_{Y,f(s)}} \\
			&& X^S \ar[shift left]{dl}[inner sep=0.5mm]{c_{X,s}} \ar[shift right,mgray]{dl} \ar[crossing over]{rrrr}[swap,pos=0.4]{g^S} &&&& Y^S \ar[shift left]{dl}[inner sep=0.5mm]{c_{Y,s}} \ar[shift right,mgray]{dl} \\
			& X \ar{rrrr}[swap]{g} &&&& Y
		\end{tikzcd}
		\end{equation}
		We have to prove that the top parallelogram commutes. 
		Let's first show that all other faces commute for all $s$:
		\begin{itemize}
			\item First of all, the front (lower) parallelogram commutes, since by definition of $g^S$ in (i), $g\circ c_{X,s}=c_{Y,s}\circ g_S$;
			\item The left and right triangles commute by definition of $X^f$ and $Y^f$ in (ii);
			\item Finally, the back parallelogram commutes, since for every $s\in S$, $c'_{Y,t}\circ g^T = g\circ c'_{X,t}$. This follows from the definition of $g^T$ as in (i), where for every $t\in T$, $c'_{Y,t}\circ g^T = g\circ c'_{X,t}$.
		\end{itemize}
		This means that for all $s\in S$,
		\[
		c_{Y,s} \circ g^S\circ X^f \;=\; c_{Y,s}\circ Y^f \circ g^T .
		\]
		In other words, either path in the top of the diagram, post-composed with $c_{Y,s}$, gives the same map to $Y$. Since this is true for all $s$, we have a weighted cone over $Y$, and hence a unique map $X^T\to Y^S$ making the whole diagram commute. This implies that $g^S\circ X^f=\circ Y^f \circ g^T$.		
	\end{enumerate}
	
	It might be also helpful to sketch (ii) for the copower case. Given $X\in\cat{C}$ and a function $f:S\to T$, from the $T$-tuple of arrows $(c_t:X\to T\cdot X)_{t\in T}$ we can create the $S$-tuple of arrows $(c_{f(s)}:X\to T\cdot X)_{s\in S}$.
	By the universal properties of copowers, this induces a unique map $S\cdot X\to T\cdot X$.
\end{proof}

\begin{corollary}\label{bidiag_from_wdiag}
	Suppose $\cat{C}$ has powers.
	A diagram $D:\cat{J}\funto\cat{C}$ weighted by $W:\cat{J}\funto\cat{Set}$ induces a bidiagram $B_{D,W}:\cat{J}^\op\funtimes\cat{J}\funto\cat{C}$ by the following composition:
	\[
	\begin{tikzcd}[row sep=0]
		\cat{J}^\op\funtimes\cat{J} \ar[functor]{rr}{W^\op\funtimes D} && \cat{Set}^\op\funtimes\cat{C} \ar[functor]{r}{P} & \cat{C} \\
		\color{dgray} (J,K) \ar[mapsto,dgray]{rr} && \color{dgray} (WJ,DK) \ar[mapsto,dgray]{r} & \color{dgray} {DK\,}^{WJ} .
	\end{tikzcd}
	\]
	Moreover, this bidiagram is compatible with the original diagram in the following sense: by setting $X=DJ$, $S=WJ$, $Y=DK$, $T=WK$, $g=Dg$ and $f=Wg$ in \eqref{prism}, we have that for all $g:J\to K$ of $\cat{J}$, the respective paths in the following diagram commute,  
	\begin{equation}\label{comp_bidiag}
	\begin{tikzcd}[column sep=tiny, row sep=large]
		{DJ\,}^{WJ} \ar[shift right]{dr}[swap,pos=0.4]{c_J(w)} \ar[shift left,mgray]{dr} \ar{rr}{{Dg\,}^{WJ}} 
		 && {DK\,}^{WJ} \ar[shift right]{dr}[swap,pos=0.4]{c_K(w)} \ar[shift left,mgray]{dr} 
		 && {DK\,}^{WK} \ar{ll}[swap]{{DK\,}^{Wg}} \ar[shift right,mgray]{dl} \ar[shift left]{dl}[pos=0.4]{c'_K(g_*w)}  \\
		& DJ \ar{rr}[swap]{Dg}
		 && DK
	\end{tikzcd}
	\end{equation}
	where $c_J(w)$, etc.\ denote the arrows of the universal cones given by powers.
	
	Dually, if $\cat{C}$ has copowers, a diagram weighted by $W:\cat{J}^\op\funto\cat{Set}$ induces a bidiagram in the following form,
	\[
	\begin{tikzcd}[row sep=0]
		\cat{J}^\op\funtimes\cat{J} \ar[functor]{rr}{W\funtimes D} && \cat{Set}\funtimes\cat{C} \ar[functor]{r}{P} & \cat{C} \\
		\color{dgray} (J,K) \ar[mapsto,dgray]{rr} && \color{dgray} (WJ,DK) \ar[mapsto,dgray]{r} & \color{dgray} WJ\cdot DK .
	\end{tikzcd}
	\]
	compatible with the original diagram dually to above.
\end{corollary}

Let's now come back to our original weighted limit, expressed in terms of powers. By the previous corollary we have a pair of arrows as follows.
\begin{equation}\label{cone_diamond2}
\begin{tikzcd}[row sep=small, column sep=tiny]
	&& |[alias=T]| T \ar[virtual]{dll} \ar[virtual]{drr} \\ 
	|[alias=ONE]| {DJ\,}^{WJ} \ar[shift left, color=mgray]{ddr} \ar[shift right, color=mgray]{ddr} &&&& |[alias=THREE]| {DK\,}^{WK} \ar[shift left, color=mgray]{ddl} \ar[shift right, color=mgray]{ddl} \\ 
	&& |[alias=TWO]| {DK\,}^{WJ} \\  
	&|[alias=DJ]| \color{mgray} DJ \ar[mgray]{rr}[swap]{Dg} && |[alias=DK]| \color{mgray} DK
	\ar[from=T, to=DJ, shift left, color=mgray]
	\ar[from=T, to=DJ, shift right, color=mgray]
	\ar[from=T, to=DK, shift left=1, color=mgray]
	\ar[from=T, to=DK, shift right=1, color=mgray]
	\ar[from=ONE, to=TWO, "{Dg\,}^{WJ}"'{inner sep=0.5mm}]
	\ar[from=THREE, to=TWO, "{DK\,}^{Wg}"{inner sep=0mm}]
\end{tikzcd}
\end{equation}
Moreover, as we will show, the diamond above commutes. 
Let's again fix some terminology. 

\begin{definition}\label{def_end1}
	A \newterm{wedge} over a bidiagram $B:\cat{J}^\op\funtimes\cat{J}\funto\cat{C}$ is an object $T$ together with arrows $c_J:T\to B(J,J)$ for all objects $J\in\cat{J}$, such that for every morphism $g:J\to K$ of $\cat{J}$, the following diagram commutes:
	\[
	\begin{tikzcd}[sep=small]
		& T \ar{dl}[swap]{z_J} \ar{dr}{z_K} \\
		B(J,J) \ar{dr}[swap]{B(J,g)} && B(K,K) \ar{dl}{B(g,K)} \\
		& B(J,K)
	\end{tikzcd}
	\]
	A \newterm{co-wedge} is an object $T$ together with arrows $c_J:B(J,J)\to T$ such that for every $g:J\to K$  the following diagram commutes:
	\[
	\begin{tikzcd}[sep=small]
		& B(K,J) \ar{dl}[swap]{B(g,J)} \ar{dr}{B(K,g)} \\
		B(J,J) \ar{dr}[swap]{z_J} && B(K,K) \ar{dl}{z_K} \\
		& T
	\end{tikzcd}
	\]
\end{definition}

\begin{remark}\label{not_a_cone}
	Notice that a wedge over $B$, in general:
	\begin{itemize}
		\item is not a \emph{cone} over $B$ (seeing $B$ as a diagram indexed by $\cat{J}^\op\funtimes\cat{J}$): it only has arrows to the ``diagonal'' objects $T\to B(J,J)$, and not, for example, $T\to B(J,K)$ for $J\ne K$;
		\item is also not exactly a \emph{weighted} cone, weighted by the empty set at $J\ne K$: a weighted cone would need to have, for all morphisms $g:J\to K$ of $\cat{J}$, also an arrow as the one in the middle, making both triangles commute: 
		
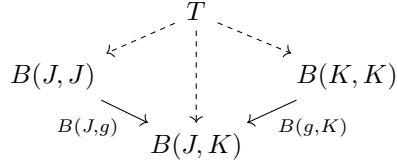
\[
		\begin{tikzcd}[sep=small]
			& T \ar[virtual]{dl} \ar[virtual]{dr} \ar[virtual]{dd} \\
			B(J,J) \ar{dr}[swap]{B(J,g)} && B(K,K) \ar{dl}{B(g,K)} \\
			& B(J,K)
		\end{tikzcd}
		\]
		Indeed, recall that ``a virtual arrow can be composed with a real arrow and give another virtual arrow''.
		(We would have a single arrow for both triangles, since the diamond commutes.)
		However, if we take into account those induced arrows, we \emph{do} obtain a weighted cone canonically, see \Cref{doch_a_cone}.
	\end{itemize}
\end{remark}

\begin{remark}\label{wedgepsh}
	 Wedges, just as cones and cocones, are closed under precomposition with morphisms of $\cat{C}$. Therefore they define a presheaf on $\cat{C}$, analogous to the one of cones. Let's denote it by $\mathrm{Wedge}(B,-):\cat{C}^\op\funto\cat{Set}$.
\end{remark}

We can now make the diamond in \eqref{cone_diamond2} commute:

\begin{lemma}\label{wedge_from_wcone}
	Let $\cat{C}$ be a category with powers. Let $D:\cat{J}\funto\cat{C}$ be a diagram weighted by $\cat{W}:\cat{J}\funto\cat{Set}$, and consider the bidiagram $B_{D,W}:\cat{J}^\op\funtimes\cat{J}\funto\cat{Set}$ constructed in \Cref{bidiag_from_wdiag}.
	Given an object $T$ of $\cat{C}$, there is a bijection between 
	\begin{enumerate}
		\item $W$-weighted cones over $D$ with tip $T$, and
		\item wedges over $B_{D,W}$ with tip $T$.
	\end{enumerate}
	This bijection is moreover natural in $T$, as it commutes with precomposition by arrows of $\cat{C}$.
\end{lemma}
\begin{proof}
	Consider the diagram \eqref{comp_bidiag} involving $D$ and $B_{D,W}$, where the respective paths commute. 
	As we saw in the beginning of this section, a $W$-weighted cone over $D$ consists first of all of, for each $J$, a $WJ$-indexed tuple of arrows $(t_J(w):T\to DJ)_{w\in WJ}$. These are in bijection with single arrows $T\to {DJ\,}^{WJ}$, let's call these arrows $z_J$:
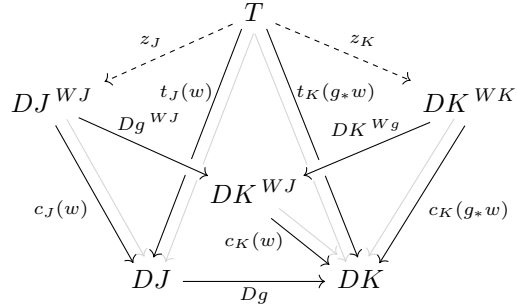
	\[
	\begin{tikzcd}[column sep=tiny]
		&& |[alias=T]| T \ar[virtual]{dll}[swap]{z_J} \ar[virtual]{drr}{z_K} \\ 
		|[alias=ONE]| {DJ\,}^{WJ} \ar[shift left, color=mgray]{ddr} \ar[shift right]{ddr}[swap]{c_J(w)} &&&& |[alias=THREE]| {DK\,}^{WK} \ar[shift left]{ddl}{c_K(g_*w)} \ar[shift right, color=mgray]{ddl} \\ 
		&& |[alias=TWO]| {DK\,}^{WJ} \\ 
		&|[alias=DJ]| DJ \ar{rr}[swap]{Dg} && |[alias=DK]| DK
		\ar[from=T, to=DJ, shift left, color=mgray]
		\ar[from=T, to=DJ, shift right, "t_J(w)"'{pos=0.32,inner sep=0mm}]
		\ar[from=T, to=DK, shift left=1, "t_K(g_*w)"{pos=0.32,inner sep=0mm}]
		\ar[from=T, to=DK, shift right=1, color=mgray]
		\ar[from=ONE, to=TWO, "{Dg\,}^{WJ}"{inner sep=0mm,near start},shorten=-1.5mm,crossing over]
		\ar[from=THREE, to=TWO, "{DK\,}^{Wg}"'{inner sep=-0.2mm,near start},shorten >=-1mm, crossing over]
		\ar[from=TWO, to=DK, shift left, color=mgray]
		\ar[from=TWO, to=DK, shift right, "c_K(w)"'{near start,inner sep=0.2mm}]
	\end{tikzcd}
	\]
	We now have to show that the $(z_J)_{J\in\cat{J}}$ form a wedge over $B_{D,W}$ if and only if the $(t_J(w):T\to DJ)_{w\in WJ, J\in\cat{J}}$ form a weighted cone. That is, chasing the diagram above, we have to show that the top diamond commutes if and only if the respective paths in the back tall triangle commute:
	\begin{itemize}
		\item First of all, suppose that $(t_J(w):T\to DJ)_{w\in WJ, J\in\cat{J}}$ is a weighted cone. That is, for every $g:J\to K$ of $\cat{J}$ and every $w\in WJ$, $t_J(w)=Dg\circ t_K(g_*w)$. Therefore
		\begin{align*}
			c_K(w)\circ {Dg\,}^{WJ} \circ z_J \;&=\; Dg\circ c_J(w)\circ z_J \\
			&=\; Dg\circ t_J(w) \\
			&=\; t_K(g_*w) \\
			&=\; c_K(g_*w)\circ z_K \\
			&=\; c_K(w)\circ{DK\,}^{Wg} \circ z_K .
		\end{align*}
		Since the equality above holds for every $w\in WJ$, we have equal tuples of arrows $T\to DK$, which means equal single arrows $T:{DK\,}^{WK}$. That is,
		\[
		{Dg\,}^{WJ} \circ z_J \;=\; {DK\,}^{Wg} \circ z_K .
		\]
		This holds for all $g$, which means that we have a wedge.
		\item Conversely, suppose that $(z_J)_{J\in\cat{J}}$ is wedge. Then for all $g:J\to K$, ${Dg\,}^{WJ} \circ z_J \;=\; {DK\,}^{Wg} \circ z_K$. Therefore for all $w\in WJ$,
		\begin{align*}
		Dg\circ t_J(w) \;&=\; Dg\circ c_J(w)\circ z_J \\
		&=\; c_K(w)\circ {Dg\,}^{WJ}\circ z_J \\
		&=\; c_K(w)\circ {DK\,}^{Wg}\circ z_K \\
		&=\; c_K(g_*w)\circ z_K \\
		&= t_K(g_*w) ,
		\end{align*}
		which means that we have a weighted cone.
	\end{itemize}
	Naturality, i.e.\ compatibility with precomposition, is immediate.
\end{proof}

Keeping again our main task in mind, we now want to say that if we start with a weighted \emph{limit} cone, the resulting wedge is universal too. Here is the precise term for that.

\begin{definition}
	An \newterm{end} of a bidiagram $B:\cat{J}^\op\funtimes\cat{J}\funto\cat{C}$, if it exists, is a terminal wedge, and is denoted by 
	\[
	\End{J\in\cat{J}} B(J,J) .
	\]
	Explicitly, it means that given any wedge $T'$ over $B$ there exists a unique morphism $u$ making the following diagram commute:
	\[
	\begin{tikzcd}[sep=small]
		& T' \ar[virtual]{dd}[pos=0.6]{u} \ar{dddl} \ar{dddr} \\\\
		& \End{J} B(J,J) \ar{dl} \ar{dr} \\
		B(J,J) \ar{dr}[swap]{B(J,g)} && B(K,K) \ar{dl}{B(g,K)} \\
		& B(J,K)
	\end{tikzcd}
	\]
	Dually, a \newterm{coend} is an initial co-wedge, and is denoted by 
	\[
	\Coend{J\in\cat{J}} B(J,J) .
	\]
\end{definition}

Equivalently, an end is a representing object for the presheaf $\mathrm{Wedge}(B,-)$ of \Cref{wedgepsh}.

Coming back again to our weighted colimit \eqref{cone_diamond2},
we are finally ready to express $T$ as limit of powers: it's their \emph{end}. 
Here is the precise statement.

\begin{theorem}\label{coend_formula}
	Let $\cat{C}$ be a category with (small) powers. Let $D:\cat{J}\funto\cat{C}$ be a functor with weights $W:\cat{J}\funto\cat{Set}$. Then
	\[
	\lim_{J\in\cat{J}}\big\langle WJ, DJ \big\rangle \;\cong\; \End{J\in\cat{J}} {DJ\,}^{WJ} .
	\]
	Dually, let $\cat{C}$ be a category with (small) copowers. Let $D:\cat{J}\funto\cat{C}$ be a functor with weights $W:\cat{J}^\op\funto\cat{Set}$. Then
	\[
	\colim_{J\in\cat{J}}\big\langle WJ, DJ \big\rangle \;\cong\; \Coend{J\in\cat{J}} WJ \cdot DJ .
	\]
\end{theorem}

\begin{proof}
	\Cref{wedge_from_wcone} gives a natural bijection between  $\mathrm{Wedge}(B_{D,W},-)$ and $\mathrm{Cone}^W(D,-)$.
	Therefore the representing objects, if they exist, are isomorphic, and so a $W$-weighted limit of $D$ is equivalently an end of $B_{D,W}$.
\end{proof}

Because of this theorem, it is very common, in categories such as $\cat{Set}$, to write weighted limits and colimits as ends and coends of powers and copowers. 

\begin{example}
	The Yoneda reduction for a functor $D:\cat{J}\funto\cat{C}$ (\Cref{ninja}) can be expressed as follows, provided $\cat{C}$ has powers and copowers:
	\[
	\End{J\in\cat{J}} {DK\,}^{\cat{J}(J,K)} \;\cong\; DJ \;\cong\;\Coend{J\in\cat{J}} \cat{J}(K,J)\cdot DK .
	\]
\end{example}

\begin{remark}	
	Compare the cones of a weighted product (\Cref{wprod}) and of a kernel pair (\Cref{kernel_pair}):
	\[
	\begin{tikzcd}[column sep=small]
		& X^2\times Y \ar[shift left]{dl} \ar[shift right]{dl} \ar{dr} \\
		X && Y
	\end{tikzcd}
	\qquad\qquad
	\begin{tikzcd}[column sep=small]
		& \ker(f) \ar[shift left]{dl} \ar[shift right]{dl} \ar{dr} \\
		X \ar{rr}[swap]{f} && Y
	\end{tikzcd}
	\]
	We see that in the first case, the cone simply consists of a collection of ``projection'' arrows, while in the second case, not any collection of arrows counts as a cone: it needs to make some triangles commute. 
	(The situation is analogous in the case of ordinary limits: a limit is \emph{like} a product, but some triangles must commute in addition.)
	Let's look at this in general: consider the difference between a weighted product and a generic weighted limit, expressed as an end of powers:
	\[
	\prod_{t\in T} {D(t)\,}^{W(t)} \qquad\qquad\mbox{vs.}\qquad\qquad \End{J\in\cat{J}} {DJ\,}^{WJ}
	\]
	If we replace the end sign by a product, we get exactly a weighted product (indexed by the objects of $\cat{J}$). 
	So, a possible interpretation of an end is that \emph{it is like a product, except that some additional things must commute.} These additional constraints come from the morphisms of $\cat{J}$. (In $\cat{Set}$, as usual for limits, we have a \emph{subset} of a product.)

	Similarly, weighted colimits, expressed as coends of copowers, are similar to weighted sums (\Cref{wsum}):
	\[
	\coprod_{t\in T} W(t)\cdot D(t) \qquad\qquad\mbox{vs.}\qquad\qquad  \Coend{J\in\cat{J}} WJ\cdot DJ
	\]
	Again, the difference is that on the right-hand side some triangles, induced by the morphisms of $\cat{J}$, must commute. (In $\cat{Set}$, as usual for colimits, we have a \emph{quotient} of a disjoint union.)
	This idea that a coend is \emph{almost like a sum}  is possibly the reason why it is customary to write it with an integral sign: after all, an integral is also \emph{almost like a sum}.\footnote{There are deeper analogies between coends and integrals, which are beyond the scope of this work. Some of them were explored in \cite{fosco}, some of them in \cite{perronetholen2022kan}, and some of them still remain to be explored.}
\end{remark}

Ends and coends are also interesting beyond the case of \Cref{coend_formula}: they are more than just a convenient presentation of weighted limits. 
Here is an example of an end which is not (on the nose) in the form of \Cref{coend_formula}. For way more theory and examples, see \cite{fosco}.

\begin{example}[Set of natural transformations]\label{nat_set}
	Let $\cat{C}$ be a small category, and consider functors $F,G:\cat{C}\funto\cat{D}$. 
	The set of natural transformations $F\Rightarrow G$ is given by the following end:
	\[
	\End{c\in\cat{C}} \cat{D}(FC,GC) .
	\]
	To see this, consider a one-point set forming a wedge over the bidiagram $\cat{D}(F-,G-)$: 
	First of all we need for all $C\in\cat{C}$ a function $\alpha_C:1\to \cat{D}(FG,GC)$, i.e.\ a morphism $\alpha_C:FC\to GC$ of $\cat{D}$. 
	Moreover, for every morphism $c:C\to C'$, the following diagram has to commute.
	\[
	\begin{tikzcd}[sep=small]
		& 1 \ar{dl}[swap]{\alpha_C} \ar{dr}{\alpha_{C'}} \\
		\cat{D}(FC,GC) \ar{dr}[swap]{Gc\circ-} && \cat{D}(FC',GC') \ar{dl}{-\circ Fc} \\
		& \cat{D}(FC, GC')
	\end{tikzcd}
	\]
	This means precisely that the arrows $\alpha_C$ need to make the following diagram commute:
	\[
	\begin{tikzcd}
		FC\ar{r}{\alpha_C} \ar{d}{Fc} & GC \ar{d}{Gc} \\
		FC' \ar{r}{\alpha_{C'}} & GC'
	\end{tikzcd}
	\]
	This needs to hold for all $c:C\to C'$, that is, they need to form a natural transformation. The end is then the set of all of them. 
\end{example}

To conclude this section, let's give an equivalent definition of ends and coends as particular weighted limits and colimits. 

\begin{theorem}\label{def_end2}
	Let $B:\cat{J}^\op\funtimes\cat{J}\funto\cat{C}$. 
	The end of $B$, if it exists, is equivalently the weighted limit 
	\[
	\End{J\in\cat{J}} B(J,J) \;\cong\; \lim_{J,K} \big\langle\cat{J}(J,K), B(J,K)\big\rangle .
	\]
	The coend of $B$, if it exists, is equivalently the weighted colimit 
	\[
	\Coend{J\in\cat{J}} B(J,J) \;\cong\; \colim_{J,K} \big\langle\cat{J}(K,J), B(J,K)\big\rangle .
	\]
\end{theorem}

The main argument in the proof is given by the following statement, which we hinted at in \Cref{not_a_cone}.

\begin{lemma}\label{doch_a_cone}
	Let $B:\cat{J}^\op\funtimes\cat{J}\funto\cat{C}$ be a bifunctor.
	There is a bijection between 
	\begin{enumerate}
		\item wedges over $B$ (with tip $T$), and 
		\item cones over $B$ (with tip $T$), seeing $B$ as a diagram, weighted by the hom-functor:
		\[
		\begin{tikzcd}[row sep=0]
			\cat{J}^\op\funtimes\cat{J} \ar[functor]{r}{W} & \cat{Set} \\
			\color{dgray} (J,K) \ar[mapsto,dgray]{r} & \color{dgray} \cat{J}(J,K) .
		\end{tikzcd}
		\]
	\end{enumerate}
	This bijection is moreover natural in $T$.
\end{lemma}
\begin{proof}[Proof of \Cref{doch_a_cone}]
	Let $\big(z_J:T\to B(J,J)\big)_{J\in\cat{J}}$ be a wedge. As hinted at in \Cref{not_a_cone}, we can complete it to a weighted cone.
	In order to have a cone with the desired weighting, we need first of all assignments as follows for all $J,K\in\cat{J}$.
	\[
	\begin{tikzcd}
		\cat{J}(J,K) \ar{r}{c_{J,K}} & \cat{C}\big( T, B(J,K) \big) .
	\end{tikzcd}
	\]
	That is, for every $g:J\to K$ of $\cat{J}$ we need an arrow $c_{J,K}(g):T\to B(J,K)$ of $\cat{C}$.
	We then need to show that this assignment is natural in $J$ and $K$.
	Let's now define $c_{J,K}$ from the wedge $(z_J)$ as the unique vertical arrow making the following diagram commute,
	\begin{equation}\label{cone_diamond}
	\begin{tikzcd}[sep=small]
		& T \ar{dl}[swap]{z_J} \ar{dr}{z_K} \ar[virtual]{dd}{c_{J,K}} \\
		B(J,J) \ar{dr}[swap]{B(J,g)} && B(K,K) \ar{dl}{B(g,K)} \\
		& B(J,K)
	\end{tikzcd}
	\end{equation}
	that is, we set 
	\[
	c_{J,K} \;\coloneqq\; B(J,g)\circ z_J \;=\; B(g,K)\circ z_K,
	\]
	where the last equality holds since $(z_J)$ is a wedge. 
	Note that this gives potentially several arrows $T\to B(J,K)$ (not necessarily distinct), one for each $g:J\to K$.
	In particular, for $J=K$, we potentially get several arrows $T\to B(J,J)$, induced by the endomorphisms $J\to J$. (The arrow of the wedge $c_J$ is the one induced by the identity of $J$.)
	To see that our assignment is natural in $J$, let $j:J\to J'$ be a morphism of $\cat{J}$. Then the following diagram commutes,
	\[
	\begin{tikzcd}[sep=small]
		\color{gray} g \ar[mapsto,dgray]{ddddd} \ar[mapsto,dgray]{rrrrr} &&&&& |[overlay,xshift=5mm]| \color{gray} B(g,K)\circ z_{K} \ar[mapsto,dgray]{ddddd} \\
		& \cat{J}(J',K) \ar{rrr}{c_{J',K}} \ar{ddd}[swap]{-\circ j} &&& \cat{C}\big( T, B(J',K) \big) \ar{ddd}{B(j,K)\circ -} \\ \\ \\ 
		& \cat{J}(J,K) \ar{rrr}[swap]{c_{J,K}} &&& \cat{C}\big( T, B(J,K) \big) \\
		\color{dgray} g\circ j \ar[mapsto,dgray]{rrrrr} &&&&& |[overlay,xshift=5mm]| \color{dgray} B(g\circ j,K)\circ z_K = B(j,K)\circ  B(g,K)\circ z_{K} \qquad\qquad\qquad\qquad
	\end{tikzcd}
	\qquad
	\]
	by (contravariant) functoriality of $B$ in the first variable.	Naturality in $K$ is similar.
	Therefore, $(c_{J,K})$ is a weighted cone. 
	
	Conversely, let $(c_{J,K})_{J,K\in\cat{J}}$ be a weighted cone. We can extract a wedge $(z_J)_{J\in\cat{J}}$ by setting 
	\[
	z_J \;\coloneqq\; c_{J,J}(\id_J) .
	\] 
	To see that this is indeed a wedge, notice that for all $g:J\to K$ of $\cat{J}$, notice that
	\[
	B(J,g)\circ z_J \;=\; B(J,g)\circ c_{J,J}(\id_J) \;=\; c_{J,K}(g) \;=\; B(g,K)\circ c_{K,K}(\id_K) \;=\; B(g,K)\circ z_K,
	\]
	where the middle equalities follow from the fact that $c$ is a weighted cone (i.e.\ it makes triangles analogous to the ones in  \eqref{cone_diamond} commute).
	
	This way we have the desired bijection between wedges and weighted cones with tip $T$.
	Naturality in $T$ now means that this bijection respects precomposition of cones and wedges with morphisms $T'\to T$. 
\end{proof}

\begin{proof}[Proof of \Cref{def_end2}]
	By \Cref{doch_a_cone}, the presheaf $\mathrm{Wedge}(B,-)$ is naturally isomorphic to the one of weighted cones over $B$ (weighted by the hom functor).
	Therefore the representing objects must coincide. 
\end{proof}

\subsection{Pointwise Kan extensions}\label{ptwise_kan}

An important special case of weighted limits and colimits are \emph{pointwise Kan extensions}. 

Let's briefly review ordinary Kan extensions. (For the readers who are unfamiliar with them, we recommend \cite[Chapter~6]{riehl2016category}.)

Let $D:\cat{J}\funto\cat{C}$ and $G:\cat{J}\funto\cat{K}$ be functors. 
A \newterm{right Kan extension} of $D$ along $G$ is a
functor $R:\cat{K}\funto\cat{C}$ together with a natural transformation $\rho:R\circ G\Rightarrow D$
\[
\begin{tikzcd}[column sep=large]
	\cat{J} \ar[functor]{dd}[swap]{G} \ar[functor,""{name=UR,below}]{dr}{D} \\
	& \cat{C} \\
	|[alias=DL]| \cat{K} \ar[functor]{ur}[swap]{R}
	\ar[Rightarrow,"\rho"{swap,inner sep=0.2mm, pos=0.6}, from=DL, to=UR, shorten <=4.5mm, shorten >=2mm]
\end{tikzcd}
\]
such that for every (other) functor $F:\cat{K}\funto\cat{C}$ and natural transformation $\alpha:F\circ I\Rightarrow D$ there exists a unique natural transformation $\nu:F\Rightarrow R$ such that the following equality holds.
\[
\begin{tikzcd}[column sep=large]
	\cat{J} \ar[functor]{dd}[swap]{G} \ar[functor,""{name=UR,below}]{dr}{D} \\
	& \cat{C} \\
	|[alias=DL]| \cat{K} \ar[functor,bend right=45,""{name=F}]{ur}[swap]{F}
	\ar[Rightarrow,"\phi"{swap,inner sep=0.5mm}, from=DL, to=UR, shift right=2mm, shorten <=3.5mm, shorten >=2mm]
\end{tikzcd}
\qquad=\qquad
\begin{tikzcd}[column sep=large]
	\cat{J} \ar[functor]{dd}[swap]{G} \ar[functor,""{name=UR,below}]{dr}{D} \\
	& \cat{C} \\
	|[alias=DL]| \cat{K} \ar[functor,""{name=DP}]{ur}[inner sep=0.2mm,pos=0.75]{R} \ar[functor,bend right=45,""{name=F}]{ur}[swap]{F}
	\ar[Rightarrow,"\rho"{inner sep=0.2mm,pos=0.55}, from=DL, to=UR, shorten <=4.5mm, shorten >=2mm]
	\ar[ds,Rightarrow,"\nu"{inner sep=0.5mm, pos=0.2}, from=F, to=DP, shorten <=-1mm, shorten >=1mm]
\end{tikzcd}
\]
Dually, a left Kan extension is a functor $L:\cat{K}\funto\cat{C}$ natural transformation $\lambda:D\Rightarrow R\circ G$ such that for every $F:\cat{K}\funto\cat{C}$ and $\alpha:D\Rightarrow F\circ G$ there exists a unique $\nu:L\Rightarrow F$ such that the following holds.
\[
\begin{tikzcd}[column sep=large]
	\cat{J} \ar[functor]{dd}[swap]{G} \ar[functor,""{name=UR,below}]{dr}{D} \\
	& \cat{C} \\
	|[alias=DL]| \cat{K} \ar[functor,bend right=45,""{name=F}]{ur}[swap]{F}
	\ar[Rightarrow,"\phi"{inner sep=0.5mm, pos=0.4}, from=UR, to=DL, shift left=2mm, shorten <=2.5mm, shorten >=4mm]
\end{tikzcd}
\qquad=\qquad
\begin{tikzcd}[column sep=large]
	\cat{J} \ar[functor]{dd}[swap]{G} \ar[functor,""{name=UR,below}]{dr}{D} \\
	& \cat{C} \\
	|[alias=DL]| \cat{K} \ar[functor,""{name=DP}]{ur}[inner sep=0.2mm,pos=0.75]{L} \ar[functor,bend right=45,""{name=F}]{ur}[swap]{F}
	\ar[Rightarrow,"\lambda"{swap,inner sep=0.5mm, pos=0.45}, from=UR, to=DL, shorten <=2.5mm, shorten >=4mm]
	\ar[ds,Rightarrow,"\nu"{swap,inner sep=0.5mm, pos=0.7}, from=DP, to=F, shorten <=1.2mm, shorten >=-1.2mm]
\end{tikzcd}
\]

\emph{Pointwise} Kan extensions are now particularly well behaved Kan extensions. 

\begin{definition}\label{def_ptwise_kan}
	Let $D:\cat{J}\funto\cat{C}$ and $G:\cat{J}\funto\cat{K}$ be functors.
	
	\begin{itemize}
		\item A right Kan extension $R$ of $D$ along $G$ is \newterm{pointwise} if for all objects $K$ of $\cat{K}$,
		\begin{equation}\label{ptwise_ran}
		RK \;\cong\; \lim_{J\in\cat{J}} \big\langle \cat{K}(K,GJ) , DJ \big\rangle .
		\end{equation}
		\item A left Kan extension $L$ of $D$ along $G$ is \newterm{pointwise} if for all objects $K$ of $\cat{K}$,
		\begin{equation}\label{ptwise_lan}
		LK \;\cong\; \colim_{J\in\cat{J}} \big\langle \cat{K}(GJ,K) , DJ \big\rangle .
		\end{equation}
	\end{itemize}
\end{definition}

Once again, the definition at first might not look particularly suggestive. So let's interpret this in our usual way. The idea is that the weight functor (resp.\ presheaf), i.e.\ the ``virtual arrows in $\cat{J}$'', come from real arrows of $\cat{K}$, in the form $K\to GJ$:
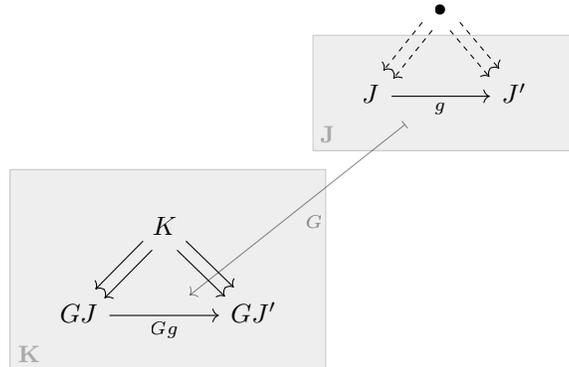
\[
\begin{tikzcd}[column sep=small,%
	blend group=multiply,
	/tikz/execute at end picture={
		\node [cbox, fit=(J) (JP), inner sep=5mm] (CL) {};
		\node [catlabel] at (CL.south west) {$\cat{J}$};
		\node [cbox, fit=(GJ) (GJP) (K), inner sep=5mm] (CR) {};
		\node [catlabel] at (CR.south west) {$\cat{K}$};
	}]
	&&&& & \bullet \ar[virtual,shift right]{dl} \ar[virtual,shift left]{dl} \ar[virtual,shift right]{dr} \ar[virtual,shift left]{dr} \\
	&&&& |[alias=J]| J \ar{rr}[swap]{g} & |[alias=CJ]| & |[alias=JP]| J' \\ \\
	& |[alias=K]| K \ar[shift right]{dl} \ar[shift left]{dl} \ar[shift right]{dr} \ar[shift left]{dr} \\
	|[alias=GJ]| GJ \ar{rr}[swap]{Gg} & |[alias=CK]| & |[alias=GJP]| GJ'
	\ar[mapsto, dgray, "G", from=CJ, to=CK, shorten <=4mm, shorten >=3.5mm]
\end{tikzcd}
\]
That is, the virtual arrows to $J$ are exactly modeled after the real arrows $K\to GJ$ in $\cat{K}$. 
Now the weighted limit of $D$ with these weights is a universal weighted cone with this shape, i.e.\ we are trying to optimally (terminally) fit arrows \emph{like the ones in $\cat{K}$ from $K$ to the image of $\cat{J}$, except that we are in the category $\cat{C}$}:
\begin{equation}\label{kan_in}
\begin{tikzcd}[column sep=small,%
	blend group=multiply,
	/tikz/execute at end picture={
		\node [cbox, fit=(J) (JP), inner sep=5mm] (CL) {};
		\node [catlabel] at (CL.south west) {$\cat{J}$};
		\node [cbox, fit=(GJ) (GJP) (K), inner sep=5mm] (CT) {};
		\node [catlabel] at (CT.south west) {$\cat{K}$};
		\node [cbox, fit=(DJ) (DJP) (T), inner sep=5mm] (CR) {};
		\node [catlabel] at (CR.south west) {$\cat{C}$};
	}]
	&&&& & \bullet \ar[virtual,shift right]{dl} \ar[virtual,shift left]{dl} \ar[virtual,shift right]{dr} \ar[virtual,shift left]{dr} \\
	&&&& |[alias=J]| J \ar{rr}[swap]{g} & |[alias=CJ]| & |[alias=JP]| J' \\ \\
	& |[alias=K]| K \ar[shift right]{dl} \ar[shift left]{dl} \ar[shift right]{dr} \ar[shift left]{dr}
	&&& &&&& & |[alias=T]| RK \ar[shift right]{dl} \ar[shift left]{dl} \ar[shift right]{dr} \ar[shift left]{dr} \\
	|[alias=GJ]| GJ \ar{rr}[swap]{Gg} & |[alias=CK]| & |[alias=GJP]| GJ'
	&& &&&& |[alias=DJ]| DJ \ar{rr}[swap]{Dg} & |[alias=CC]| & |[alias=DJP]| DJ'
	\ar[mapsto, dgray, "G"', from=CJ, to=CK, shorten <=4mm, shorten >=3.5mm]
	\ar[mapsto, dgray, "D", from=CJ, to=CC, shorten <=4mm, shorten >=3.5mm]
	\ar[mapsto, dgray, "R"', from=K, to=T, shorten <=2mm, shorten >=2mm]
\end{tikzcd}
\end{equation}

Let's now check that this definition indeed gives us a Kan extension in the usual sense.

\begin{theorem}\label{pkan_are_kan}
	Any functor $R:\cat{K}\funto\cat{C}$ satisfying \eqref{ptwise_ran}, and where its action on morphisms is induced by the universal property, is necessarily a right Kan extension.
	
	Dually, any functor $L:\cat{K}\funto\cat{C}$ satisfying \eqref{ptwise_lan}, and where its action on morphisms is induced by the universal property, is necessarily a left Kan extension.
\end{theorem}
\begin{proof}[Proof of \Cref{pkan_are_kan}]
	Let $R:\cat{K}\funto\cat{C}$ be satisfying \eqref{ptwise_ran}.
	To show that $R$ is a right Kan extension in the usual sense, we first of all have to give a natural transformation $\rho:R\circ G\Rightarrow D$. 
	In components, we need maps $\rho_J: RGJ\to DJ$.
	Now instantiating \eqref{kan_in} for $K=GJ$, the identity $GJ\to GJ$ gives a distinguished arrow among the (possibly many) ones of the limit cone:
	\[
	\begin{tikzcd}[column sep=small,%
		blend group=multiply,
		/tikz/execute at end picture={
			\node [cbox, fit=(J) (JP), inner sep=5mm] (CL) {};
			\node [catlabel] at (CL.south west) {$\cat{J}$};
			\node [cbox, fit=(GJ) (GJP) (K), inner sep=5mm] (CT) {};
			\node [catlabel] at (CT.south west) {$\cat{K}$};
			\node [cbox, fit=(DJ) (DJP) (T), inner sep=5mm] (CR) {};
			\node [catlabel] at (CR.south west) {$\cat{C}$};
		}]
		&&&& & \bullet \ar[virtual,shift right]{dl} \ar[virtual,cgray,shift left]{dl} \ar[virtual,cgray,shift right]{dr} \ar[virtual,cgray,shift left]{dr} \\
		&&&& |[alias=J]| J \ar{rr}[swap]{g} & |[alias=CJ]| & |[alias=JP]| J' \\ \\
		& |[alias=K]| GJ \ar[shift right]{dl}[swap]{\id} \ar[mgray,shift left]{dl} \ar[mgray,shift right]{dr} \ar[mgray,shift left]{dr}
		&&& &&&& & |[alias=T]| RGJ \ar[shift right]{dl} \ar[mgray,shift left]{dl} \ar[mgray,shift right]{dr} \ar[mgray,shift left]{dr} \\
		|[alias=GJ]| GJ \ar{rr}[swap]{Gg} & |[alias=CK]| & |[alias=GJP]| GJ'
		&& &&&& |[alias=DJ]| DJ \ar{rr}[swap]{Dg} & |[alias=CC]| & |[alias=DJP]| DJ'
		\ar[mapsto, dgray, "G"', from=CJ, to=CK, shorten <=4mm, shorten >=3.5mm]
		\ar[mapsto, dgray, "D", from=CJ, to=CC, shorten <=4mm, shorten >=3.5mm]
	\end{tikzcd}
	\]
	We take that as our component $\rho_J$. (Naturality holds since the analogous diagram commutes in $\cat{J}$.)
	
	Let's now turn to the universal property of the Kan extension. Consider $F$ and $\phi$ as follows. 
	\[
	\begin{tikzcd}[column sep=large]
		\cat{J} \ar[functor]{dd}[swap]{G} \ar[functor,""{name=UR,below}]{dr}{D} \\
		& \cat{C} \\
		|[alias=DL]| \cat{K} \ar[functor,bend right=45,""{name=F}]{ur}[swap]{F}
		\ar[Rightarrow,"\phi"{swap,inner sep=0.5mm}, from=DL, to=UR, shift right=2mm, shorten <=3.5mm, shorten >=2mm]
	\end{tikzcd}
	\]
	We have to show that there is a unique natural transformation $\nu:F\Rightarrow D^+$ such that $\rho\circ(\nu G)=\phi$.
	We can depict $F$ and $\phi$ as follows.
	\[
	\begin{tikzcd}[column sep=small,%
		blend group=multiply,
		/tikz/execute at end picture={
			\node [cbox, fit=(J) (JP), inner sep=5mm] (CL) {};
			\node [catlabel] at (CL.south west) {$\cat{J}$};
			\node [cbox, fit=(GJ) (GJP) (K), inner sep=5mm] (CT) {};
			\node [catlabel] at (CT.south west) {$\cat{K}$};
			\node [cbox, fit=(DJ) (DJP) (T) (FGJ), inner sep=5mm] (CR) {};
			\node [catlabel] at (CR.south west) {$\cat{C}$};
		}]
		&&&& & \bullet \ar[virtual,shift right]{ddl} \ar[virtual,shift left]{ddl} \ar[virtual,shift right]{ddr} \ar[virtual,shift left]{ddr} \\ \\
		&&&& |[alias=J]| J \ar{rr}[swap]{g} & |[alias=CJ]| & |[alias=JP]| J' \\ \\
		& |[alias=K]| K \ar[shift right]{ddl} \ar[shift left]{ddl} \ar[shift right]{ddr} \ar[shift left]{ddr}
		&&& &&&& & |[alias=T]| RK \ar[shift right]{ddl} \ar[shift left]{ddl} \ar[shift right]{ddr} \ar[shift left]{ddr} \\ 
		&&&& &&& |[alias=FK]| FK \ar[shift right]{ddl} \ar[shift left]{ddl} \ar[shift right]{ddr} \ar[shift left]{ddr} \\
		|[alias=GJ]| GJ \ar{rr}[swap]{Gg} & |[alias=CK]| & |[alias=GJP]| GJ'
		&& &&&& |[alias=DJ]| DJ \ar{rr}[swap]{Dg} & |[alias=CC]| & |[alias=DJP]| DJ' \\
		&&&& && |[alias=FGJ]| FGJ \ar{urr}[swap]{\phi_J} \ar{rr}[swap]{FGg} && |[alias=FGJP]| FGJ' \ar{urr}[swap]{\phi_{J'}}
		\ar[mapsto, dgray, "G"', from=CJ, to=CK, shorten <=2mm, shorten >=3.5mm]
		\ar[mapsto, dgray, "D", from=CJ, to=CC, shorten <=4mm, shorten >=3.5mm]
		\ar[mapsto, dgray, "F"', from=GJP, to=FGJ, shorten <=0mm, shorten >=3mm, shift left=7mm]
	\end{tikzcd}
	\]
	Recall now that $RK$ is a terminal weighted cone for all $K\in\cat{K}$.
	If we postcompose the arrows $FK\to FGJ$ with the components $\phi_J:FGJ\to DJ$, we get a weighted cone over $D$ with tip $FK$. Therefore, since $RK$ is a weighted limit, there is a unique map $FK\to RK$ making the respective paths in the following diagram commute:
	\[
	\begin{tikzcd}[column sep=small]
		&& & |[alias=T]| RK \ar[shift right]{ddl} \ar[shift left]{ddl} \ar[shift right]{ddr} \ar[shift left]{ddr} \\ 
		& |[alias=FK]| FK \ar[shift right]{ddl} \ar[shift left]{ddl} \ar[shift right]{ddr} \ar[shift left]{ddr} \\
		&& |[alias=DJ]| DJ \ar{rr}[swap]{Dg} & |[alias=CC]| & |[alias=DJP]| DJ' \\
		 |[alias=FGJ]| FGJ \ar{urr}[swap]{\phi_J} \ar{rr}[swap]{FGg} && |[alias=FGJP]| FGJ' \ar{urr}[swap]{\phi_{J'}}
		\ar[virtual, from=FK, to=T]
	\end{tikzcd}
	\]
	We take this map as the component at $K$ of the natural transformation $\nu:F\Rightarrow R$. 
	Naturality follows from the universal property, and moreover, setting $K=GJ$ in the diagram above and chasing it,
	
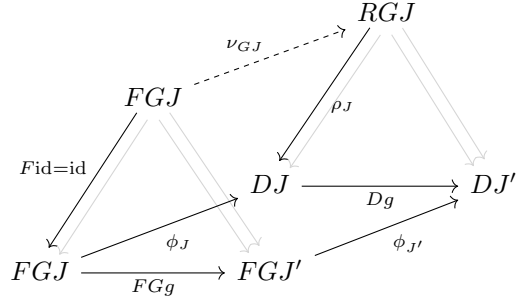
\[
	\begin{tikzcd}[column sep=small]
		&& & |[alias=T]| RGJ \ar[shift right]{ddl}{\rho_J} \ar[shift left,mgray]{ddl} \ar[shift right,mgray]{ddr} \ar[shift left,mgray]{ddr} \\ 
		& |[alias=FK]| FGJ \ar[shift right]{ddl}[swap]{F\id=\id} \ar[shift left,mgray]{ddl} \ar[shift right,mgray]{ddr} \ar[shift left,mgray]{ddr} \\
		&& |[alias=DJ]| DJ \ar{rr}[swap]{Dg} & |[alias=CC]| & |[alias=DJP]| DJ' \\
		 |[alias=FGJ]| FGJ \ar{urr}[swap]{\phi_J} \ar{rr}[swap]{FGg} && |[alias=FGJP]| FGJ' \ar{urr}[swap]{\phi_{J'}}
		\ar[virtual, from=FK, to=T, "\nu_{GJ}"]
	\end{tikzcd}
	\]
	we have that $\rho_J\circ\nu_{GJ}=\phi_J\circ\id=\phi_J$. That holds for all $J$, so that
	\[
	\rho\circ(\nu G) = \phi ,
	\]
	making $(R,\rho)$ a right Kan extension.
\end{proof}

To conclude this section, looking at the diagrams \eqref{wlim_ext}, which we re-propose:
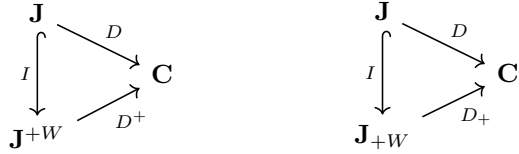
\[
\begin{tikzcd}[row sep=small]
	\cat{J} \ar[functor,hook]{dd}[swap]{I} \ar[functor]{dr}{D} \\
	& \cat{C} \\
	\cat{J}^{+W} \ar[functor]{ur}[swap]{D^+}
\end{tikzcd}
\qquad\qquad\qquad
\begin{tikzcd}[row sep=small]
	\cat{J} \ar[functor,hook]{dd}[swap]{I} \ar[functor]{dr}{D} \\
	& \cat{C} \\
	\cat{J}_{+W} \ar[functor]{ur}[swap]{D_+}
\end{tikzcd}
\]
one may suspect that weighted cones and cocones are actually Kan extensions. They are, and pointwise too.

\begin{theorem}\label{wlim_are_pkan}
	Let $D:\cat{J}\funto\cat{C}$ be a diagram weighted by $W:\cat{J}\funto\cat{Set}$.
	A $W$-weighted limit cone of $D$ is exactly a pointwise right Kan extension of $D$ along the inclusion functor $I:\cat{J}\funto\cat{J}^{+W}$:
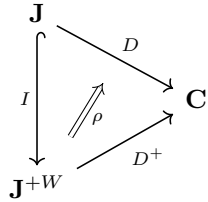
	\[
	\begin{tikzcd}[column sep=large]
		\cat{J} \ar[functor,hook]{dd}[swap]{I} \ar[functor,""{name=UR,below}]{dr}{D} \\
		& \cat{C} \\
		|[alias=DL]| \cat{J}^{+W} \ar[functor]{ur}[swap]{D^+}
		\ar[Rightarrow,"\rho"{swap,inner sep=0.5mm, pos=0.6}, from=DL, to=UR, shorten <=4.5mm, shorten >=2mm]
	\end{tikzcd}
	\]
	Dually, given $W:\cat{J}^\op\funto\cat{Set}$, a $W$-weighted colimit cone of $D$ is exactly a pointwise left Kan extension of $D$ along the inclusion $I$:
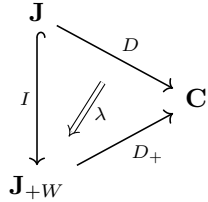
	\[
	\begin{tikzcd}[column sep=large]
		\cat{J} \ar[functor,hook]{dd}[swap]{I} \ar[functor,""{name=UR,below}]{dr}{D} \\
		& \cat{C} \\
		|[alias=DL]| \cat{J}_{+W} \ar[functor]{ur}[swap]{D_+}
		\ar[Rightarrow,"\lambda"{inner sep=0.5mm, pos=0.35}, from=UR, to=DL, shorten <=2mm, shorten >=4.5mm]
	\end{tikzcd}
	\]
\end{theorem}

To prove the theorem we use the following result, which might be interesting on its own:
\begin{lemma}\label{pkan_ff}
	The universal natural transformation ($\rho$ or $\lambda$) of a pointwise (left or right) Kan extension along a fully faithful functor is a natural equivalence.
\end{lemma}
\begin{proof}[Proof of \Cref{pkan_ff}]
	As usual, let's prove the limit (i.e.\ right) case. 
	Let $G:\cat{J}\funto\cat{K}$ be a fully faithful functor. The pointwise right Kan extension of $D:\cat{J}\funto\cat{C}$ along $G$, on an object $K$, is given by
	\[
	R(K) \;\cong\; \lim_{J'\in\cat{J}} \big\langle \cat{K}(K,GJ'), DJ' \big\rangle .
	\]
	Therefore the composition $R\circ G$, on an object $J$ of $\cat{J}$, is given by 
	\begin{align*}
	RGJ \;&\cong\; \lim_{J'\in\cat{J}} \big\langle \cat{K}(GJ,GJ'), DJ' \big\rangle \\
	&\cong\; \lim_{J'\in\cat{J}} \big\langle \cat{J}(J,J'), DJ' \big\rangle \\
	&\cong\; DJ ,
	\end{align*}
	using the fact that $G$ is fully faithful and the Yoneda reduction (\Cref{ninja}). 
	The arrow $\rho_J:RGJ\to DJ$ is now given by this isomorphism. 
\end{proof}

\begin{proof}[Proof of \Cref{wlim_are_pkan}]
	As usual, let's focus on the limit case.
	By \Cref{pkan_ff}, up to natural isomorphism we have that the natural transformation $\rho$ can be taken to be the identity. That is, the functor $D^+$ is actually an extension of $D$, it agrees with $D$ on the objects of $\cat{J}$. 
	On the only other object $E$ we now have, by definition of pointwise Kan extension,
	\begin{align*}
	D^+(E) \;&\cong\; \lim_{J}\big\langle \cat{J}^{+W}(E,DJ), DJ \big\rangle \\
	&=\; \lim_{J}\big\langle WJ, DJ \big\rangle ,
	\end{align*}
 	i.e.\ exactly the $W$-weighted limit of $D$, with 
	the arrows $E\to DJ$ mapped to the arrows of the universal limit cone.
\end{proof}

\subsection{Functor pairing}

We now turn out attention to the prototypical example of a weighted colimit of sets, the \emph{pairing}. 
As sets have copowers, this weighted colimit is usually expressed as a coend, and several coends appearing in the literature are instances of pairings. We will see it for example in Cauchy completions (\Cref{cauchy}), as well as in the composition of profunctors and in the Day convolution of presheaves (see \Cref{further}).

\begin{definition}\label{pairing}
	Let $\cat{C}$ be a small category. 
	Consider a presheaf $P:\cat{C}^\op\funto\cat{Set}$ and a set functor $F:\cat{C}\funto\cat{Set}$. Their \newterm{pairing} (or sometimes \emph{tensor product}, but this name may mean other things too) is the set 
	\begin{equation}\label{eq_pairing}
	\langle P,F\rangle \;\coloneqq\; \colim_{C\in\cat{C}} \big\langle PC, FC \big\rangle \;\cong\; 
	\Coend{C\in\cat{C}} PC\cdot FC  \;\cong\; 
	\Coend{C\in\cat{C}} PC\times FC ,
	\end{equation}
	where we recall that, in $\cat{Set}$, the copower is equivalently the product. 
\end{definition}

Let's now interpret this in terms of virtual arrows.
So far we have been looking at adding extra arrows to a category $\cat{C}$ only in a ``fixed'' direction.

The pairing of a functor and a presheaf is best interpreted using both directions: virtual arrows \emph{to} our category (the functor $F:\cat{C}\funto\cat{Set}$), and \emph{from} our category (the presheaf $P:\cat{C}^\op\funto\cat{Set}$):
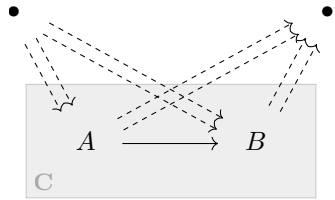
\[
\begin{tikzcd}[sep=small,
	blend group=multiply,
	/tikz/execute at end picture={
		\node [cbox, fit=(A) (B), inner sep=5mm] (CC) {};
		\node [catlabel] at (CC.south west) {$\cat{C}$};
	}]
	 |[alias=E]| \bullet &&&&&& |[alias=F]| \bullet \\ \\ \\ \\
	 &|[alias=A]| A &&&& |[alias=B]| B \\
	 \ar[from=A,to=B,shorten=2mm]
	 \ar[virtual, from=E,to=A,shift left,shorten=2mm]
	 \ar[virtual, from=E,to=A,shift right,shorten=2mm]
	 \ar[virtual, from=E,to=B,shift left,shorten=2mm]
	 \ar[virtual, from=E,to=B,shift right,shorten=2mm]
	 \ar[virtual, from=A,to=F,shift left,shorten=2mm]
	 \ar[virtual, from=A,to=F,shift right,shorten=2mm]
	 \ar[virtual, from=B,to=F,shift left,shorten=2mm]
	 \ar[virtual, from=B,to=F,shift right,shorten=2mm]
\end{tikzcd}
\]
The pairing $\langle P,F\rangle$ can now be seen as the set of ``virtual paths'' $\bullet\dashrightarrow\bullet$ in the diagram above, \emph{with equal paths counted only once}. 

Let's see what we mean. 
The simplest way of forming a path $\bullet\dashrightarrow\bullet$ is to have, at an object $A$ of $\cat{C}$, a virtual arrow $f:\bullet\dashrightarrow A$ (i.e.\ an element $f\in FA$) and a virtual arrow $p:A\dashrightarrow\bullet$ (i.e.\ an element $p\in PA$):
\[
\begin{tikzcd}
	\bullet\ar[virtual]{r}{f} & A \ar[virtual]{r}{p} & \bullet
\end{tikzcd}
\]
So one first way of defining all possible paths $\bullet\dashrightarrow\bullet$ is to take all possible pairs $(f\in FA,p\in PA)$ for all objects $A$:
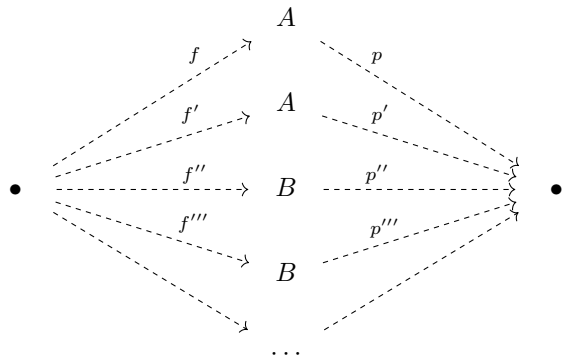
\[
\begin{tikzcd}[column sep=30mm]
	& A \ar[virtual,shorten <=2mm,shorten >=3mm]{ddr}[near start]{p} \\
	& A \ar[virtual,shorten <=2mm,shorten >=3mm]{dr}[near start]{p'} \\
	\bullet \ar[virtual,shorten <=3mm,shorten >=2mm]{uur}[near end]{f} \ar[virtual,shorten <=3mm,shorten >=2mm]{ur}[near end]{f'} \ar[virtual,shorten <=3mm,shorten >=2mm]{r}[pos=0.7]{f''} \ar[virtual,shorten <=3mm,shorten >=2mm]{dr}[pos=0.6]{f'''} \ar[virtual,shorten <=3mm,shorten >=2mm]{ddr} & B \ar[virtual,shorten <=2mm,shorten >=3mm]{r}[pos=0.3]{p''} & \bullet \\
	& B \ar[virtual,shorten <=2mm,shorten >=3mm]{ur}[pos=0.4,inner sep=0mm]{p'''} \\
	& \cdots \ar[virtual,shorten <=2mm, shorten >=3mm]{uur}
\end{tikzcd}
\]
That would be the set\footnote{Recall that we are assuming $\cat{C}$ is small.}
\[
\coprod_{A\in\cat{C}} PA\times FA .
\]

Things are however not so simple.
Given $g:A\to B$ in $\cat{C}$, consider an element $f\in FA$ (i.e.\ a virtual arrow $f:\bullet\dashrightarrow A$) and an element $p\in PB$ (i.e.\ a virtual arrow $p:B\dashrightarrow\bullet$): the form a path $\bullet\dashrightarrow\bullet$ as follows.
\[
\begin{tikzcd}
	\bullet\ar[virtual]{r}{f} & A \ar{r}{g} & B \ar[virtual]{r}{p} & \bullet
\end{tikzcd}
\]
This ``three-arrow path'' was already counted before, since we can express it as a ``two-arrow path'', but we can do so in two ways, namely, as $p\circ (g_*f)$ and as $(g^*p)\circ f$:
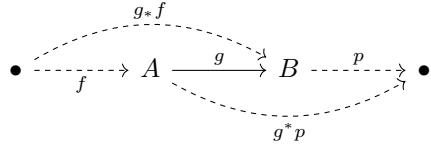
\[
\begin{tikzcd}[sep=large]
	\bullet \ar[virtual, bend left]{rr}{g_*f} \ar[virtual]{r}[swap]{f} & A \ar{r}{g} \ar[virtual, bend right]{rr}[swap]{g^*p} & B \ar[virtual]{r}{p} & \bullet
\end{tikzcd}
\] 
Therefore, the pairs
\[
(g^*p,f) \in PA\times FA \quad\mbox{and}\quad (p,g_*f) \in PB\times FB
\]
should be considered equivalent. 
In other words, we want to take the quotient set 
\begin{equation}\label{quot_sum}
\left(\coprod_{A\in\cat{C}} PA\times FA \right) /_\sim 
\end{equation}
where $\sim$ is the equivalence relation generated by 
\[
(g^*p,f) \in PA\times FA \quad\sim\quad (p,g_*f) \in PB\times FB
\]
for all $f\in FA$ and $p\in PB$ and all morphisms $g:A\to B$ of $\cat{C}$. 
We can see \eqref{quot_sum} as the set of all virtual paths $\bullet\dashrightarrow\bullet$, without double-counting.
We will denote each equivalence class as follows:
\[
[A,g^*p,f] \;=\; [B,p,g_*f] .
\]

\begin{remark}
One may now ask: what about four-arrow paths, and more? These are already accounted for: in the sequence 
\[
\begin{tikzcd}
	\bullet \ar[virtual]{r}{f} & A_0 \ar{r}{g_1} & A_1 \ar{r}{g_2} & \dots \ar{r}{g_n} & A_n \ar[virtual]{r}{p} & \bullet
\end{tikzcd}
\]
one can always replace $g_1,\dots,g_n$ by their composite, which is a single morphism.
\end{remark}

Let's now show that the pairing of $F$ and $P$ is exactly this set. 

\begin{theorem}
	Let $\cat{C}$ be a small category. Given $F:\cat{C}\funto\cat{Set}$ and $P:\cat{C}^\op\funto\cat{Set}$, their pairing $\langle P, F\rangle$ is given up to isomorphism by the set \eqref{quot_sum}.
\end{theorem}

\begin{proof}
	We can for example show that the set \eqref{quot_sum} gives the coend in \eqref{pairing}. So let's first of all show that if forms a co-wedge, i.e.\ that for all $g:A\to B$ of $\cat{C}$, the following diagram commutes,
	\[
	\begin{tikzcd}[column sep=0]
		& PB\times FA \ar{dl}[swap]{g^*\times FA} \ar{dr}{PA\times g_*} \\
		PA\times FA \ar{dr} && PB\times FB \ar{dl} \\
		& \left(\coprod_{A} PA\times FA \right) /_\sim
	\end{tikzcd}
	\]
	where the unlabeled arrows are the inclusions $PA\times FA\hookrightarrow \coprod_A PA\times FA$ followed by the quotient map.
	Starting with $(p\in PB,f\in FA)$,
	
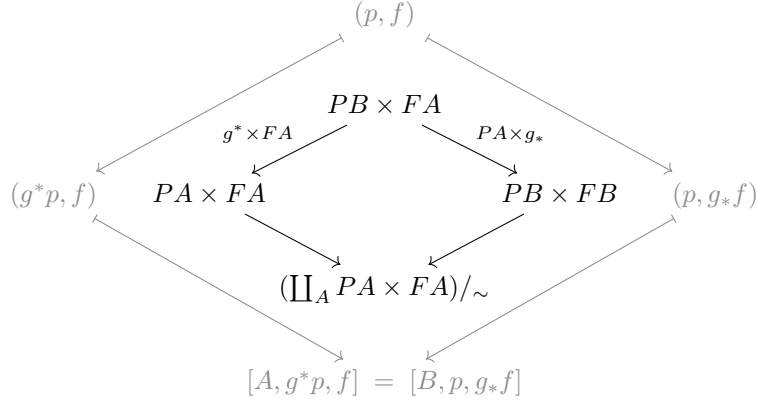
\[
	\begin{tikzcd}
		&& \color{dgray} (p,f) \ar[mapsto,dgray]{ddll} \ar[mapsto,dgray]{ddrr} \\
		&& PB\times FA \ar{dl}[swap]{g^*\times FA} \ar{dr}{PA\times g_*} \\
		\color{dgray} (g^*p,f) \ar[mapsto,dgray]{ddrr} & |[xshift=5mm, overlay]| PA\times FA \ar{dr} && |[xshift=-5mm,overlay]| PB\times FB \ar{dl} & \color{dgray} (p,g_*f) \ar[mapsto,dgray]{ddll} \\
		&& \left(\coprod_{A} PA\times FA \right) /_\sim \\
		&& \color{dgray} {[A,g^*p,f]} \;=\; {[B,p,g_*f]}
	\end{tikzcd}
	\]
	we see that the diagram commutes precisely by how we define the equivalence relation.
	
	To see that this co-wedge is universal, consider a family of functions $(z_A:PA\times FA\to X)_{A\in\cat{C}}$.
	By the universal property of coproducts, such a family is equivalently given by a function on the disjoint union $z:\coprod_A PA\times FA\to X$.
	Suppose now that the $z_A$ form a wedge, i.e.\ that for all $g:A\to B$, the following diagram commutes.
	
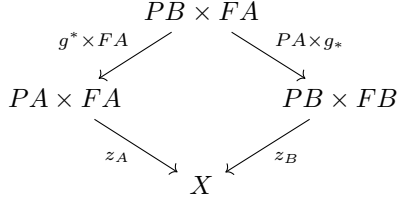
\[
	\begin{tikzcd}[column sep=0]
		& PB\times FA \ar{dl}[swap]{g^*\times FA} \ar{dr}{PA\times g_*} \\
		PA\times FA \ar{dr}[swap]{z_A} && PB\times FB \ar{dl}{z_B} \\
		& X
	\end{tikzcd}
	\]
	This means precisely that for all $f\in FA$ and $p\in B$, $z_A(g^*p,f)=z_B(p,g_*f)$. That is, the function $z:\coprod_A PA\times FA\to X$ is invariant within each equivalence class. Therefore it factors uniquely through the quotient \eqref{quot_sum}.
\end{proof}

Before we leave this section, it can be helpful to what happens if $P$ and $F$ are representable: in that case, all the ``virtual'' arrows are actually real, and the pairing recovers the usual composition of arrows:

\begin{example}
	Let $X$ and $Y$ be objects of a category $\cat{C}$. Consider the functor $\cat{C}(X,-):\cat{C}\funto\cat{Set}$ (of arrows from $X$) and the presheaf $\cat{C}(-,Z)$ (of arrows into $Z$). Their pairing, up to isomorphism, is the set of arrows from $X$ to $Z$:
	\[
	\langle \cat{C}(X,-), \cat{C}(-,Z) \rangle \;=\; \colim_{Y\in\cat{C}}\big\langle \cat{C}(X,Y), \cat{C}(Y,Z) \big\rangle \;\cong\; \cat{C}(X,Z) .
	\]
	Indeed, the isomorphism is exactly an instance of Yoneda reduction (\Cref{ninja}).
\end{example}

\section{Cauchy completion}\label{cauchy}

After representability and weighted limits, the last concept we see in depth using our diagrams is \emph{Cauchy completion} (also known as \emph{idempotent completion}, \emph{Karoubi envelope}, and \emph{absolute completion}, and other names).

\subsection{Cauchy points}

The idea of Cauchy completion can be understood, using our usual point of view, in terms adding to a category an extra object with virtual arrows possibly going in both directions:
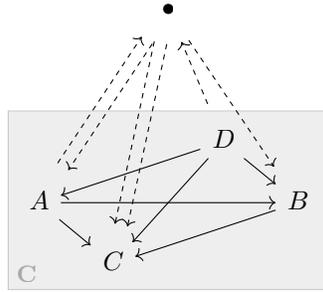
\[
\begin{tikzcd}[row sep=small, column sep=2mm,
	blend group=multiply,
	/tikz/execute at end picture={
		\node [cbox, fit=(A) (B) (C) (D), inner sep=1.2mm] (CC) {};
		\node [catlabel] at (CC.south west) {$\cat{C}$};
	}]
	&& |[alias=E,xshift=2mm]| \bullet \\ \\ \\ \\
	&&& |[alias=D,xshift=-5mm]| D \\
	|[alias=A]| A &&&& |[alias=B]| B \\
	& |[alias=C,xshift=5mm]| C
	\ar[from=A,to=B]
	\ar[from=A,to=C]
	\ar[from=B,to=C]
	\ar[from=D,to=A]
	\ar[from=D,to=B]
	\ar[from=D,to=C]
	\ar[virtual,from=E,to=A,shift left,shorten=2.3mm]
	\ar[virtual,from=A,to=E,shift left,shorten=2.3mm]
	\ar[virtual,from=E,to=C,shift left,shorten=2.3mm]
	\ar[virtual,from=E,to=C,shift right,shorten=2.3mm]
	\ar[virtual,from=E,to=B,shorten=2.3mm]
	\ar[virtual,from=D,to=E,shorten=2.3mm]
\end{tikzcd}
\]
Let's see this in detail. First of all, in order to get virtual arrows in both directions we need both a functor $F$ (giving arrows from $\bullet$ to the category) and a presheaf $P$ (giving arrows from the category to $\bullet$), as we did in \Cref{pairing}. However, compared to \Cref{pairing} we have to imagine the two virtual points identified into a single one.
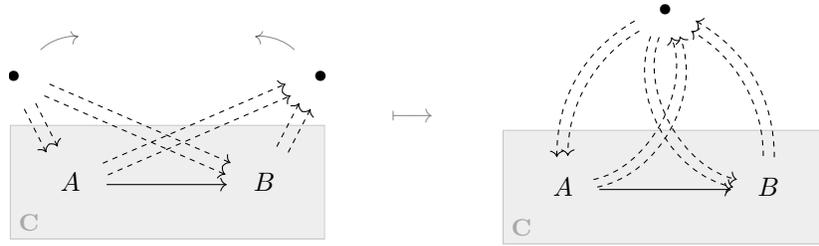
\[
\begin{tikzcd}[row sep=small, column sep=tiny,
	blend group=multiply,
	/tikz/execute at end picture={
		\node [cbox, fit=(A) (B), inner sep=5mm] (CC) {};
		\node [catlabel] at (CC.south west) {$\cat{C}$};
	}]
	&&& |[alias=TOP]| \phantom{bullet}\\
	|[alias=E]| \bullet &&&&&& |[alias=F]| \bullet \\ \\ \\
	&|[alias=A]| A &&&& |[alias=B]| B \\
	\ar[from=A,to=B,shorten=2mm]
	\ar[virtual, from=E,to=A,shift left,shorten=2mm]
	\ar[virtual, from=E,to=A,shift right,shorten=2mm]
	\ar[virtual, from=E,to=B,shift left,shorten=2mm]
	\ar[virtual, from=E,to=B,shift right,shorten=2mm]
	\ar[virtual, from=A,to=F,shift left,shorten=2mm]
	\ar[virtual, from=A,to=F,shift right,shorten=2mm]
	\ar[virtual, from=B,to=F,shift left,shorten=2mm]
	\ar[virtual, from=B,to=F,shift right,shorten=2mm]
	\ar[dgray,from=E,to=TOP,out=45,in=190, shorten <=2mm, shorten >=6mm]
	\ar[dgray,from=F,to=TOP,out=135,in=-10, shorten <=2mm, shorten >=6mm]
\end{tikzcd}
\qquad{\color{dgray}\longmapsto}\qquad
\begin{tikzcd}[sep=small,
	blend group=multiply,
	/tikz/execute at end picture={
		\node [cbox, fit=(A) (B), inner sep=5mm] (CC) {};
		\node [catlabel] at (CC.south west) {$\cat{C}$};
	}]
	 &&&|[alias=E]| \bullet \\ \\ \\ \\ \\ \\
	&|[alias=A]| A &&&& |[alias=B]| B \\
	\ar[from=A,to=B,shorten=2mm]
	\ar[virtual, from=E,to=A,bend right=30,shift left,shorten=2mm]
	\ar[virtual, from=E,to=A,bend right=30,shift right,shorten=1mm]
	\ar[virtual, from=E,to=B,bend right=50,shift left,shorten=2mm]
	\ar[virtual, from=E,to=B,bend right=50,shift right,shorten=1mm]
	\ar[virtual, from=A,to=E,bend right=50,shift left,shorten=2mm]
	\ar[virtual, from=A,to=E,bend right=50,shift right,shorten=1mm]
	\ar[virtual, from=B,to=E,bend right=30,shift left,shorten=2mm]
	\ar[virtual, from=B,to=E,bend right=30,shift right,shorten=1mm]
\end{tikzcd}
\]
Now:
\begin{enumerate}
	\item\label{endo} As in \Cref{pairing}, if we have virtual arrows $\bullet\dashrightarrow X$ and $X\dashrightarrow\bullet$, their composite $\bullet\dashrightarrow\bullet$ is encoded an element of the pairing $\langle P, F\rangle$;
	\item\label{comp} Moreover, thanks to the fact that we have a single extra object, we can now compose virtual arrows the other way as well: if we have virtual $A\dashrightarrow\bullet$ and $\bullet\dashrightarrow B$, we get an arrow $A\to B$ as their ``composite'':
	\begin{equation}\label{comp_diag}
	\begin{tikzcd}
		& \bullet \ar[virtual]{dr}{f} \\
		A \ar[virtual]{ur}{p} \ar{rr}[swap]{p\circ f} && B
	\end{tikzcd}
	\end{equation}
\end{enumerate}

Similarly to what we did for weighted limits, a Cauchy completion is a sort of ``universal'' or ``minimal way'' to add this object and these arrows to $\cat{C}$: in a certain way, we want this extra object to be ``as close as possible to the original category''. More precisely: 
\begin{itemize}
	\item Every endomorphism of the extra object $\bullet$ must necessarily arise from the pairing $\langle P,F\rangle$. In particular, this must be true for the identity of $\bullet$:
	\item For any two objects $A$ and $B$ of the original category $\cat{C}$, the arrows $A\to B$ are exactly those coming from $\cat{C}$. This means in particular that any composite morphisms arising from \eqref{comp_diag} must be already morphisms of $\cat{C}$. 
\end{itemize}
Therefore, besides the functor $F$ and the presheaf $P$, we need the following extra data:
\begin{itemize}
	\item A distinguished element $i\in\langle P,F\rangle$, playing the role of the identity  $\bullet\dashrightarrow\bullet$;
	\item For every two objects $A$ and $B$ of $\cat{C}$, a function 
	\[
	\begin{tikzcd}[row sep=0]
		PA \times FB \ar{r} & \cat{C}(A,B) \\
		\color{dgray} \left(A\dashrightarrow \bullet \;,\; \bullet\dashrightarrow B\right) \ar[mapsto,dgray]{r} & \color{dgray} \left(A\dashrightarrow \bullet\dashrightarrow B\right)
	\end{tikzcd}
	\]
	which forms composites as in \eqref{comp_diag} and assures they are morphisms of $\cat{C}$. We moreover want this function to be natural in $A$ and $B$ (for why, see the proof of \Cref{iscat});
\end{itemize}

Recall now that the elements of $\langle P,F\rangle$ are equivalence classes of objects $[A,p,f]$ with $A\in\cat{C}$, $p\in PA$ and $f\in FA$. 
So, to play the role of the identity, we need a distinguished equivalence class $[X,\pi,\iota]$ (with $X\in\cat{C}$, $\pi\in PX$ and $\iota\in FX$), satisfying some extra conditions so that it behaves like an identity morphism (see \eqref{id_cond} below). 

Here is the precise definition.
\begin{definition}\label{defcauchypt}
	A \newterm{Cauchy point} or \newterm{point of the Cauchy completion} of a category $\cat{C}$ consists of 
	\begin{itemize}
		\item A functor $F:\cat{C}\funto\cat{Set}$;
		\item A presheaf $P:\cat{C}^\op\funto\cat{Set}$;
		\item For all $A$ and $B$ of $\cat{C}$, a mapping $c_{A,B}:FA\times PB\to \cat{C}(A,B)$
		natural in both $A$ and $B$;
		\item A distinguished equivalence class $i=[X,\pi,\iota]$ such that for all $A\in\cat{C}$, $f\in FA$ and $p\in PA$, 
		\begin{equation}\label{id_cond}
			c_{X,A}(\pi,f)_*\iota \;=\; f, \qquad\qquad c_{A,X}(p,\iota)^*\pi \;=\; p .
		\end{equation}
		In diagrams:
		\[
		\begin{tikzcd}[column sep=small]
			\bullet \ar[virtual]{dr}{\iota} && \bullet \ar[virtual]{dr}{f}\\
			& X \ar[virtual]{ur}{\pi} \ar{rr}[swap]{c(\pi,f)} && A
		\end{tikzcd}
		\;=\;
		\begin{tikzcd}[column sep=small]
			\bullet \ar[virtual]{dr}{f}\\
			& A
		\end{tikzcd}
		\qquad\qquad\qquad
		\begin{tikzcd}[column sep=small]
			& \bullet \ar[virtual]{dr}{\iota} && \bullet\\
			A \ar[virtual]{ur}{p} \ar{rr}[swap]{c(p,\iota)} && X \ar[virtual]{ur}{\pi}
		\end{tikzcd}
		\;=\;
		\begin{tikzcd}[column sep=small]
			& \bullet \\
			A \ar[virtual]{ur}{p}
		\end{tikzcd}
		\]
	\end{itemize}
\end{definition}

We can represent the situation as follows,
\begin{equation}\label{virtual_retract}
\begin{tikzcd}[baseline=-1em,
	row sep=small, column sep=2mm,
	blend group=multiply,
	/tikz/execute at end picture={
		\node [cbox, fit=(A) (B) (C) (D), inner sep=1.2mm] (CC) {};
		\node [catlabel] at (CC.south west) {$\cat{C}$};
	}]
	&& |[alias=E,xshift=2mm]| \bullet \\ \\ \\ \\
	&&& |[alias=D,xshift=-5mm]| \color{cgray} A \\
	|[alias=A]| \color{cgray} B &&&& |[alias=B]| \color{cgray} C \\
	& |[alias=C,xshift=5mm]| X
	\ar[from=A,to=B,color=mgray]
	\ar[from=A,to=C,color=mgray]
	\ar[from=C,to=B,color=mgray]
	\ar[from=D,to=A,color=mgray]
	\ar[from=D,to=B,color=mgray]
	\ar[from=D,to=C,color=mgray]
	\ar[virtual,from=E,to=A,shift left,shorten=2.3mm,color=cgray]
	\ar[virtual,from=A,to=E,shift left,shorten=2.3mm,color=cgray]
	\ar[virtual,from=E,to=C,shift left,shorten=2.3mm,"\iota"{pos=0.45}]
	\ar[virtual,from=C,to=E,shift left,shorten=2.3mm,"\pi"{pos=0.45}]
	\ar[virtual,from=B,to=E,shorten=2.3mm,color=cgray]
	\ar[virtual,from=E,to=D,shorten=2.3mm,color=cgray]
\end{tikzcd}
\qquad\qquad\quad``\pi\circ\iota=\id_\bullet"
\end{equation}
where we see that if $[X,\pi,\iota]$ behaves like an identity, it's as if the virtual composition $\pi\circ\iota$ were the identity of the virtual object. 
In other words, we have a ``virtual retract'' of some object $X$. 

Note that conditions \eqref{id_cond} do not depend on the representative of the class. Given $g:X\to X'$, $\iota'=g_*\iota$ and $\pi'$ such that $\pi=g^*\pi'$, we have that 
\[
c_{X',A}(\pi',f)_*\iota' \;=\; c_{X',A}(\pi',f)_*(g_*\iota) \;=\; (c_{X',A}(\pi',f)\circ g)_*\iota \;=\; c_{X,A}(g^*\pi,f)_*\iota \;=\; c_{X,A}(\pi,f)_*\iota ,
\]
using functoriality of $F$ and naturality of $c_{X,A}$ in $X$, and similarly
\[
c_{A,X'}(p,\iota')^*\pi' \;=\; c_{A,X'}(p,g_*\iota)^*\pi' \;=\; (g\circ c_{A,X'}(p,\iota))^*\pi' \;=\; c_{A,X'}(p,\iota)_*(g^*\pi') \;=\; c_{A,X}(p,\iota)^*\pi ,
\]
using naturality of $c_{A,X}$ in $X$ and functoriality of $P$.

It is helpful, as we did for set functors and presheaves, to explicitly construct the ``category with an extra object'' that the data of \Cref{defcauchypt} encode.

\begin{definition}\label{defcplus}
	Let $(F,P,c,i)$ be a Cauchy point of $\cat{C}$.
	The \emph{extension induced by $(F,P,c,i)$} is a category $\cat{C'}$ where:
	\begin{itemize}
		\item The objects are the ones of $\cat{C}$, plus an extra object $E$;
		\item The morphisms between the objects coming from $\cat{C}$ are the same as in $\cat{C}$ (i.e.\ $\cat{C}$ is embedded fully and faithfully into $\cat{C'}$);
		\item For every object $A$ of $\cat{C}$, the morphisms $A\to E$ are the elements of $PA$ (``virtual arrows out of $A$'');
		\item For every object $A$ of $\cat{C}$, the morphisms $E\to A$ are the elements of $FA$ (``virtual arrows into $A$'');
		\item The morphisms $E\to E$ are the elements of the pairing $\langle P,F\rangle$;
		\item The identities of all objects of $\cat{C}$ are the ones of $\cat{C}$, and the identity of $E$ is given by $i$;
		\item The composition between morphisms of $\cat{C}$ is the one in $\cat{C}$, and the one between morphisms of $\cat{C}$ and morphisms to or from $E$ is specified by functoriality of $F$ and $P$, and the one between two morphisms to and from $E$ is specified by $c$.
	\end{itemize}
\end{definition}

Let's see how this works in detail.

\begin{proposition}\label{iscat}
	The category $\cat{C'}$ of \Cref{defcplus} is indeed a category.
\end{proposition}
\begin{proof}
	First of all, it is helpful to write down the composition of morphisms explicitly:
	\begin{itemize}
		\item Between morphisms of $\cat{C}$, the composition is as in $\cat{C}$.
		\item Given $f:E\to A$ (i.e.~$f\in FA$) and $g:A\to B$, the composition $g\circ f:E\to B$ is the element of $FB$ given by $g_*f$, as in \Cref{CplusF}.
		\item Given $g:A\to B$ and $p:B\to E$ (i.e.~$p\in PB$), the composition $p\circ g:E\to B$ is the element of $PA$ given by $g^*p$, as in \Cref{CplusP}.
		\item Given $p:A\to E$ and $f:E\to B$ (i.e.~$p\in PA$ and $f\in FB$), the composition $f\circ p:A\to B$ is the (non-virtual) arrow given by $c_{A,B}(p,f)$.
		\item Given an endomorphism $m=[A,p,f]:E\to E$ and a morphism $f':E\to B$, the composition $f'\circ m:E\to B$ is given by $c(p,f')_*f\in FB$, as depicted in the following diagram; 
		\begin{equation}\label{endo_comp}
			\begin{tikzcd}
				E \ar[virtual, bend left]{rr}{m} \ar[virtual]{r}[swap]{f} & A \ar[virtual]{r}{p} \ar[bend right]{rr}[swap]{c(p,f')} & E \ar[virtual]{r}{f'} & B
			\end{tikzcd}
		\end{equation}
		To see why this composition is well defined on equivalence classes, consider and equivalent triplet $(\tilde{A},\tilde{p},\tilde{f})\sim(A,p,f)$ such that $g_*\tilde{f}=f$ and $\tilde{p}=g^*p$. 
		\[
		\begin{tikzcd}[row sep=tiny]
			& \tilde{A} \ar[virtual]{dr}{\tilde{p}} \ar{dd}{g} \ar[bend left]{drr}{c(\tilde{p},f')} \\
			E \ar[virtual]{ur}{\tilde{f}} \ar[virtual]{dr}{f} && 
			E \ar[virtual]{r}{f'} & B \\
			& A \ar[virtual]{ur}[swap]{p} \ar[bend right]{urr}[swap]{c_{p,f'}}
		\end{tikzcd}
		\]
		Then 
		\[
		c(\tilde{p},f')_*\tilde{f} \;=\; c(g_*p,f')_*\tilde{f} \;=\; (c(p,f')\circ g)_*\tilde{f} \;=\; c(p,f')_*(g_*\tilde{f}) \;=\; c(p,f')_* f ,
		\] 
		where we used naturality of $c$ in its first argument, and functoriality of $F$.
		\item Given a morphism $p':A\to E$ and an endomorphism $m=[B,p,f]:E\to E$ , the composition $m\circ p':A\to E$ is given by $c(p',f)^*p\in PA$, as depicted in the following diagram; 
		\[
		\begin{tikzcd}
			A \ar[bend right]{rr}[swap]{c(p',f)} \ar[virtual]{r}{p'} & E \ar[virtual]{r}{f} \ar[virtual, bend left]{rr}{m} & B \ar[virtual]{r}[swap]{p} & E
		\end{tikzcd}
		\]
		This composition is well defined by an argument similar to the point above (using naturality of $c$ in its second argument).
		\item Finally, given two endomorphisms $m=[A,p,f]$ and $m'=[A',p',f']:E\to E$, their composition is given by $[A,c(p,f'),f]=[A',p',c(p,f')^*p]$, as depicted in the following diagram:
		\[
		\begin{tikzcd}
			E \ar[virtual, bend left]{rr}{m} \ar[virtual]{r}[swap]{f} & A \ar[bend right]{rr}[swap]{c(p,f')} \ar[virtual]{r}{p} & E \ar[virtual, bend left]{rr}{m'} \ar[virtual]{r}{f'} & A' \ar[virtual]{r}[swap]{p'} & E
		\end{tikzcd}
		\]
		Note that the two tuples are equivalent under the usual relation (equivalently, the lower path in the diagram is associative).
		Again, by a similar reasoning as above, this composition is well defined.
	\end{itemize}
	
	To prove associativity of composition, we distinguish a few cases:
	\begin{itemize}
		\item For morphisms of $\cat{C}$, associativity holds since it holds in $\cat{C}$.
		\item For arrows in the following form,
		\[
		\begin{tikzcd}
			E \ar[virtual]{r}{f} & A \ar{r}{g} & B \ar{r}{h} & C
		\end{tikzcd}
		\]
		we have that $(h\circ g)\circ f=(h\circ g)_*f=h_*(g_*f)=h\circ(g\circ f)$ by functoriality of $F$. Similarly we have associativity for arrows in the form 
		\[
		\begin{tikzcd}
			A \ar{r}{g} & B \ar{r}{h} & C \ar[virtual]{r}{p} & E
		\end{tikzcd}
		\]
		by functoriality of $P$.
		\item For arrows in the following form,
		\[
		\begin{tikzcd}
			A \ar[virtual]{r}{p} & E \ar[virtual]{r}{f} & B \ar{r}{g} & C
		\end{tikzcd}
		\]
		we have $g\circ c_{A,B}(p,f)=c_{A,C}(p,g_*f)$ by naturality of $c$ in $B$. We can reason similarly for arrows in the form 
		\[
		\begin{tikzcd}
			A \ar{r}{g} & B \ar[virtual]{r}{p} & E \ar[virtual]{r}{f} & C
		\end{tikzcd}
		\]
		using naturality of $c$ in $B$.
		\item For arrows in the following form,
		\[
		\begin{tikzcd}
			E \ar[virtual]{r}{f} & A \ar{r}{g} & B \ar[virtual]{r}{p} & E
		\end{tikzcd}
		\]
		the relation of \eqref{quot_sum} makes the paths automatically equivalent.
		\item Arrows in the following forms,
		\[
		\begin{tikzcd}
			E \ar[virtual]{r}{f} & A \ar[virtual]{r}{p} & E \ar[virtual]{r}{f'} & A' \\
			A \ar[virtual]{r}{p} & E \ar[virtual]{r}{f} & A' \ar[virtual]{r}{p'} & E'
		\end{tikzcd}
		\]
		are already taken care of by composition of endomorphisms of $E$ and virtual arrows (see \eqref{endo_comp}).
		\item All other cases are obtained by iterating the arguments we have just made.
	\end{itemize}
	
	Finally, let's turn to unitality. Again, we distinguish a few cases.
	\begin{itemize}
		\item All identities of objects of $\cat{C}$ behave like identities with all morphisms of $\cat{C}$.
		\item For every object $A$ of $\cat{C}$ and every morphisms $f:E\to A$, we have 
		\[
		\id_A\circ f \;=\; (\id_A)_*f \;=\; f
		\]
		by functoriality of $F$. Similarly, for every morphism $p:A\to E$ we have $p\circ\id_A=p$ by functoriality of $P$.
		\item To show that $i=[X,\pi,\iota]$ is the identity of $E$, let $f:E\to A$. Then using \eqref{endo_comp} and \eqref{id_cond},
		\[
		f\circ [X,\pi,\iota] \;=\; c(\pi,f)_*\iota \;=\; f.
		\]
		Similarly, given $p:A\to E$,
		\[
		[X,\pi,\iota] \circ p \;=\; c(p,\iota)^*\pi \;=\; p.
		\]
	\end{itemize}
	This makes $\cat{C'}$ a category.
\end{proof}

Let's now look at the situation where, in some sense, ``the extra point is already in $\cat{C}$:

\begin{theorem}\label{thm_cauchy}
	Given a Cauchy point $(F,P,c,i)$ of $\cat{C}$, the following conditions are equivalent.
	\begin{enumerate}
		\item\label{F_R} The functor $F$ is representable, and represented by an object $R$;
		\item\label{P_R} The presheaf $P$ is representable, and represented by an object $R$;
		\item\label{I_eq} The inclusion $I:\cat{C}\funto\cat{C'}$ is an equivalence of categories.
	\end{enumerate}
\end{theorem}

\begin{definition}
	We say that a Cauchy point of $\cat{C}$ is \newterm{already in $\cat{C}$} if any of the equivalent conditions of \Cref{thm_cauchy} are satisfied.
	
	A category is called \newterm{Cauchy complete} if every Cauchy point is already in it. 
\end{definition}

The definition is similar to Cauchy completions of metric spaces: in a certain sense, a category, like a metric space, is Cauchy complete if ``all the points that should be there are indeed there''. This analogy can be made mathematically precise in terms of enriched categories, but that is beyond the scope of this exposition. (See for example \cite{borceux_idempotents}.)

To prove the theorem we will make use of the following lemma.

\begin{lemma}\label{FP_repr}
	Given a Cauchy point $(F,P,c,i)$, the following conditions are equivalent.
	\begin{enumerate}
		\item\label{F_R2} $F$ is representable;
		\item\label{P_R2} $P$ is representable;
		\item\label{FP_split} There exists a representative $(X,\pi,\iota)$ of the equivalence class $i$ such that $c(p,\iota)=\id_X$.
	\end{enumerate}
	\[
	\begin{tikzcd}
		& \bullet \ar[virtual]{dr}{\iota} \\
		X \ar[virtual]{ur}{\pi} \ar{rr}[swap]{\id} && X
	\end{tikzcd}
	\]
	Moreover, if the conditions above are satisfied, $X$ is a representing object.  
\end{lemma}
\begin{proof}
	Let's start with $\ref{FP_split}\Rightarrow\ref{F_R2}$:
	Notice that ``virtual precomposition'' with $\pi$ and $\iota$ gives mappings as follows, natural in $A$:
	\[
	\begin{tikzcd}[row sep=0]
		FA \ar{rr}{c(\pi,-)} && \cat{C}(X,A) \\
		\bullet \ar[virtual]{ddr}[swap]{f} && X \ar[virtual,cgray]{ll}[swap]{\pi} \ar{ddl}{c(\pi,f)} \\
		& \color{dgray} \longmapsto\vphantom{\int} \\
		& A 	
	\end{tikzcd}
	\qquad\qquad
	\begin{tikzcd}[row sep=0]
		 \cat{C}(X,A) \ar{rr}{(-)_*\iota} && FA \\
		X \ar{ddr}[swap]{g} && \bullet \ar[virtual,cgray]{ll}[swap]{\iota} \ar[virtual]{ddl}{g_*\iota} \\
		& \color{dgray} \longmapsto\vphantom{\int} \\
		& A 	
	\end{tikzcd}
	\]
	One of the conditions of \eqref{id_cond} says now exactly that $\pi$ is a ``virtual retract'' of $\iota$, as we said:
	\begin{equation}\label{virtual_retract_eq}
	\begin{tikzcd}[row sep=0]
		FA \ar{rr}{c(\pi,-)} && \cat{C}(X,A) \ar{rr}{(-)_*\iota} && FA \\
		\bullet \ar[virtual]{ddrr}[swap]{f} && X \ar[virtual,cgray]{ll}[swap]{\pi} \ar{dd}{c(\pi,f)} && \bullet \ar[virtual,cgray]{ll}[swap]{\iota} \ar[virtual]{ddll}{c(\pi,f)_*\iota \;=\; f} \\
		&& \vphantom{\int} \\
		&& A 	
	\end{tikzcd}
	\end{equation}
	If we moreover have $c(\pi,\iota)=\id_X$, using naturality of $c$,
	\[
	\begin{tikzcd}[row sep=0]
		\cat{C}(X,A) \ar{rr}{(-)_*\iota} && FA \ar{rr}{c(\pi,-)} && \cat{C}(X,A) \\
		X \ar{ddrr}[swap]{g} && \bullet \ar[virtual,cgray]{ll}[swap]{\iota} \ar[virtual]{dd}{g_*\iota} 
		&& X \ar[virtual,cgray]{ll}[swap]{\pi} \ar{ddll}{c(\pi,g_*\iota)\;=\; c(\pi,\iota)\circ g \;=\; g}
		\\
		&& \color{dgray} \vphantom{\int} \\
		&& A 	
	\end{tikzcd}
	\]
	then $\pi$ and $\iota$ are part of a natural isomorphism $\cat{C}(X,-)\cong F$, making $X$ a representing object for $F$.
	
	$\ref{F_R2}\Rightarrow\ref{FP_split}$:
	Let $R$ be a representing object of $F$, together with a natural isomorphism $\phi:\cat{C}(R,-)\Rightarrow F$. Define now the element $\iota'\in FR$ as follows.
	\[
	\begin{tikzcd}[row sep=0]
		\cat{C}(R,R) \ar{r}{\phi_R}[swap]{\cong} & FR \\
		\color{dgray} \id_R \ar[mapsto,dgray]{r} & \color{dgray} \iota' \coloneqq \phi_R(\id_R)
	\end{tikzcd}
	\]
	By the Yoneda lemma, we have that all components of $\phi$ must be in the following form:
	\[
	\begin{tikzcd}[row sep=0]
		\cat{C}(R,A) \ar{rr}{\phi_A}[swap]{\cong} && FA \\
		R \ar{ddr}[swap]{g} && \bullet \ar[virtual,cgray]{ll}[swap]{\iota'} \ar[virtual]{ddl}{g_*\iota'} \\
		& \color{dgray} \longmapsto\vphantom{\int} \\
		& A 	
	\end{tikzcd}
	\]
	Let now $(X,\pi,\iota)$ be a representative of the class $i$, and set $A=X$ in the diagram above. Since we have a bijection, there exists a unique $g:R\to X$ such that $g_*\iota'=\iota$:
	\[
	\begin{tikzcd}[row sep=0]
		\cat{C}(R,X) \ar{rr}{\phi_X}[swap]{\cong} && FX \\
		R \ar{ddr}[swap]{g} && \bullet \ar[virtual,cgray]{ll}[swap]{\iota'} \ar[virtual]{ddl}{g_*\iota' \;=\;\iota} \\
		& \color{dgray} \longmapsto\vphantom{\int} \\
		& X 	
	\end{tikzcd}
	\]
	Define now $\pi'= g^*\pi\in PR$:
	\[
	\begin{tikzcd}
		R \ar{r}{g} & X \ar[virtual]{r}{\pi} & \bullet
	\end{tikzcd}
	\]
	Consider now the triplet $(R,\pi',\iota')$: first of all, since  $g_*\iota' =\iota$ and $\pi'= g^*\pi$, we have 
	\[
	[R,\pi',\iota'] \;=\; [X,\pi,\iota] \;\in\; \langle P,F \rangle .
	\]
	Moreover, 
	\[
	c(\pi',\iota') \;=\; c(g^*\pi,\iota') \;=\; c(\pi,\iota')\circ g \;=\;\id_R ,
	\]
	where we used that, by naturality the inverse $\phi^{-1}$, the following diagram commutes,
	\[
	\begin{tikzcd}[sep=small]
		\color{dgray} \iota \ar[mapsto,dgray]{ddddd} \ar[mapsto,dgray]{rrrrr} &&&&& |[overlay,xshift=10mm]| \color{dgray} g \ar[mapsto,dgray]{ddddd} \\
		& FX \ar{ddd}[swap]{c(\pi,\iota')_*} \ar{rrr}{\phi^{-1}_X} &&& \cat{C}(R,X) \ar{ddd}{c(\pi,\iota')\circ -} \\ \\ \\
		& FR \ar{rrr}[swap]{\phi^{-1}_R} &&& \cat{C}(R,R) \\
		\color{dgray} c(\pi,\iota')_*\iota \;=\; \iota' \ar[mapsto,dgray]{rrrrr} &&&&& |[overlay,xshift=10mm]| \color{dgray} \id_R \;=\; c(\pi,\iota')\circ g
	\end{tikzcd}
	\qquad\qquad
	\]
	as well as \eqref{id_cond}.
	In summary, the representing object $R$ is part of a triplet $(R,\pi',\iota')$ in the class $i$ such that $c(\pi'\iota')=\id_R$.
	
	The proof of $\ref{P_R2}\Leftrightarrow\ref{FP_split}$ is completely analogous and dual.
\end{proof}

We are now ready to prove the main theorem. 

\begin{proof}[Proof of \Cref{thm_cauchy}]
	First of all, the equivalence $\ref{F_R}\Leftrightarrow\ref{P_R}$ was proven in \Cref{FP_repr}.
	
	Second, notice that by definition of $\cat{C'}$, the inclusion $I:\cat{C}\funto\cat{C'}$ is fully faithful. Therefore it is an equivalence if and only if it is essentially surjective.
	
	$\ref{F_R}\Rightarrow\ref{I_eq}$: Suppose that $F$ is representable. By \Cref{FP_repr}, the representing object $X$ is part of a triple $(X,\pi,\iota)$ representing the class $i$, and such that $c(\pi,\iota)=\id_X$. 
	Considering now $\pi$ and $\iota$ as morphisms $\pi:X\to E$ and $\iota:E\to X$ in the category $\cat{C'}$, we have that $\iota\circ\pi=c(\pi,\iota)=\id_X$ by what we just said, as well as $\pi\circ\iota = \id_E$ by definition of $\id_E$ in $\cat{C'}$.
	Therefore $E\cong X$, making the inclusion $I$ essentially surjective.
	
	$\ref{I_eq}\Rightarrow\ref{F_R}$: Suppose that for some object $R$ of $\cat{C}$ we have an isomorphism $f:E\xrightarrow{\cong} R$ in $\cat{C'}$. Precomposition with $f$ gives a mapping
	\[
	\begin{tikzcd}[row sep=0, column sep=large]
		\cat{C}(R,A) \ar{r}{I(-)\circ f} & \cat{C'}(E,A) \;=\; FA \\
		\color{dgray} g \ar[mapsto,dgray]{r} & \color{dgray} Ig\circ f = g_*f
	\end{tikzcd}
	\]
	natural in $A$. This is moreover a bijection, with the following inverse.
	\[
	\begin{tikzcd}[row sep=0, column sep=large]
		\cat{C}(R,A) = \cat{C'}(IR,IA) & \cat{C'}(E,A) \;=\; FA \ar{l}[swap]{-\circ f^{-1}} \\
		\color{dgray} g = Ig & \color{dgray} Ig\circ f = g_* f \ar[mapsto,dgray]{l}
	\end{tikzcd}
	\]
	Therefore $\cat{C}(R,-)$ is naturally isomorphic to $F$, i.e.\ $R$ represents $F$.	
\end{proof}

Before we leave this section, let's also define the morphisms of Cauchy points, which will look as follows.
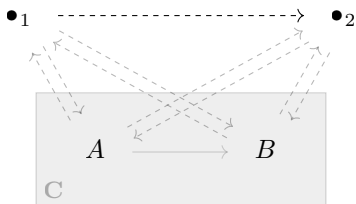
\[
\begin{tikzcd}[sep=small,
	blend group=multiply,
	/tikz/execute at end picture={
		\node [cbox, fit=(A) (B), inner sep=5mm] (CC) {};
		\node [catlabel] at (CC.south west) {$\cat{C}$};
	}]
	|[alias=E]| \bullet_1 &&&&&& |[alias=F]| \bullet_2 \\ \\ \\ \\
	&|[alias=A]| A &&&& |[alias=B]| B \\
	\ar[from=A,to=B,shorten=2mm, color=mgray]
	\ar[virtual, from=E,to=A,shift left,shorten=2mm, color=cgray]
	\ar[virtual, from=A,to=E,shift left,shorten=2mm, color=cgray]
	\ar[virtual, from=E,to=B,shift left,shorten=2mm, color=cgray]
	\ar[virtual, from=B,to=E,shift left,shorten=2mm, color=cgray]
	\ar[virtual, from=A,to=F,shift left,shorten=2mm, color=cgray]
	\ar[virtual, from=F,to=A,shift left,shorten=2mm, color=cgray]
	\ar[virtual, from=B,to=F,shift left,shorten=2mm, color=cgray]
	\ar[virtual, from=F,to=B,shift left,shorten=2mm, color=cgray]
	\ar[virtual, from=E, to=F, shorten=2mm]
\end{tikzcd}
\]
Intuitively, similar to the endomorphisms of a Cauchy point, we want a morphism to be defined only by the existing virtual arrows, without adding any new ones.

\begin{definition}
	Let $C_1=(F_1,P_1,c_1,i_1)$ and $C_2=(F_2,P_2,C_2,i_2)$ be Cauchy points of $\cat{C}$. A \newterm{morphism of Cauchy points} $C_1\to C_2$ is an element of the pairing 
	\[
	\langle P_2, F_1 \rangle \;=\; \Coend{C\in\cat{C}} P_2(C) \times F_1(C) .
	\]
\end{definition}
That is, every morphism $\bullet_1\dashrightarrow\bullet_2$ is necessarily in the following form:
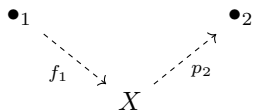
\[
\begin{tikzcd}
	\bullet_1 \ar[virtual]{dr}[swap]{f_1} && \bullet_2 \\
	& X \ar[virtual]{ur}[swap]{p_2}
\end{tikzcd}
\]

\begin{proposition}\label{c_morph_nat}
	A morphisms of Cauchy points $C_1\to C_2$ is equivalently specified by any of the following.
	\begin{itemize}
		\item A natural transformation $P_1\Rightarrow P_2$;
		\item A natural transformation $F_2\Rightarrow F_1$ (mind the direction).
	\end{itemize}
\end{proposition}
\begin{proof}
	Let $[X,p_2\in P_2,f_1\in F_1]\in \langle P_2, F_1 \rangle$ be a morphism of Cauchy points. 
	\[
	\begin{tikzcd}
		\bullet_1 \ar[virtual]{r}{f_1} & X \ar[virtual]{r}{p_2} & \bullet_2
	\end{tikzcd}
	\]
	To construct a natural transformation $\alpha:P_1\Rightarrow P_2$ we need mappings $\alpha_A:P_1(A)\to P_2(A)$ for all $A$ and naturally in $A$. We construct them as the following ``virtual precompositions'':
	\[
	\begin{tikzcd}[row sep=0, column sep=large]
		P_1(A) \ar{r}{c(-,f_1)} & \cat{C}(A,X) \ar{r}{(-)^*p_2} & P_2(A) \\
		\bullet_1 \ar[virtual,cgray]{r}{f_1} & X \ar[virtual,cgray]{r}{p_2} & \bullet_2 \\
		\phantom{\int}\\
		& A \ar[virtual]{uul}{p} \ar{uu}[swap]{c(p,f_1)} \ar[virtual]{uur}[swap]{c(p,f_1)^*p_2}
	\end{tikzcd}
	\]
	
	Conversely, given a natural transformation $\alpha:P_1\Rightarrow P_2$, considering a representative $(X_1,\pi_1,\iota_1)$ of $i_1$ we can take the element $[X_1,\alpha(\pi_1)\in P_2(X_2),\iota_1\in F_1(X_1)]\in \langle P_2,F_1\rangle$. 
	
	To show that these assignments are mutually inverse, start first with $[X,p_2,f_1]\in \langle P_2,F_1\rangle$. 
	Forming the natural transformation and the resulting triple we get
	\[
	[X_1,\alpha(\pi_1),\iota_1] \;=\; [X_1,c(\pi_1,f_1)^*p_2,\iota_1] .
	\]
	Now consider the map $c(\pi_1,f_1):X_1\to X$: we have that 
	\[
	c(\pi_1,f_1)_*\iota_1 \;=\; f_1 ,
	\]
	using \eqref{id_cond}. Therefore, as equivalence classes, $[X_1,c(\pi_1,f_1)^*p_2,\iota_1]=[X,p_2,f_1]$.
	
	Conversely, starting with $\alpha:P_1\Rightarrow P_2$, we have that for all $A$ and $p\in PA$,
	\[
	c(p,\iota_1)^*\alpha_{X_1}(\pi_1) \;=\; \alpha_A(c(p,\iota_1)^*\pi_1) \;=\; \alpha_A(p) ,
	\]
	using naturality of $\alpha$ and \eqref{id_cond}.
	
	The construction in terms of natural transformations $F_2\Rightarrow F_1$ is analogous and dual.
\end{proof}

\begin{corollary}
	Cauchy points and their morphisms form a category.
	If we denote it by $\overline{C}$, we have the following commutative diagram of fully faithful embeddings:
	
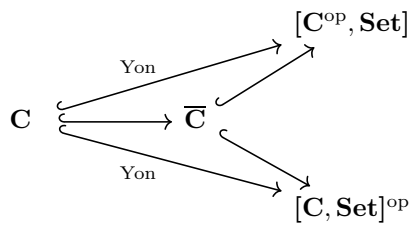
\[
	\begin{tikzcd}
		&&& \sfuncat{\cat{C}^\op} \\
		\cat{C} \ar[functor,hook,shorten <=2mm]{urrr}{\Yon}  \ar[functor,hook,shorten <=2mm]{drrr}[swap]{\Yon} \ar[functor,hook,shorten <=2mm]{rr} && \overline{\cat{C}} \ar[functor,hook]{ur} \ar[functor,hook]{dr} \\
		&&& \sfuncat{\cat{C}}^\op
	\end{tikzcd}
	\]
\end{corollary}

\begin{corollary}
	A category is Cauchy complete if and only if the inclusion $\cat{C}\to\overline{\cat{C}}$ is an equivalence.
\end{corollary}

In the next few sections, we will look at some equivalent ways of looking at Cauchy points and Cauchy completion.

\subsection{In terms of idempotents}

Another way of looking at Cauchy points of unenriched categories is in terms of idempotents and their splittings.

Recall that an \emph{idempotent} in a category $\cat{C}$ is an endomorphism $e:X\to X$ such that $e\circ e=e$. 
A \emph{splitting} of the idempotent $e$ is an object $E$ together with morphisms $\iota:E\to X$ and $\pi:X\to E$ such that $\pi\circ\iota=\id_E$, i.e.\ forming a section-retraction pair, and such that $\iota\circ\pi=e$:
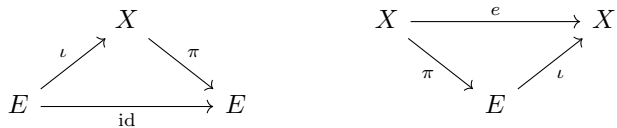
\[
\begin{tikzcd}
	& X \ar{dr}{\pi} \\
	E \ar{ur}{\iota} \ar{rr}[swap]{\id} && E
\end{tikzcd}
\qquad\qquad 
\begin{tikzcd}
	X \ar{rr}{e} \ar{dr}[swap]{\pi} && X \\
	& E \ar{ur}[swap]{\iota}
\end{tikzcd}
\]
Notice that, somewhat conversely, every section-retraction pair gives rise to an idempotent, since
\[
(\iota\circ\pi)\circ(\iota\circ\pi) \;=\; \iota\circ(\pi\circ\iota)\circ\pi \;=\; \iota\circ\id\circ\pi\;=\;\iota\circ\pi .
\]
One of the most common situations in mathematics where one sees this is in linear algebra, where idempotents are projections onto subspaces $E\subseteq X$. 

A splitting of an idempotent can be seen as a particular limit or colimit.

\begin{proposition}\label{e_split_eq}
	Let $e:X\to X$ be an idempotent in a category $\cat{C}$. The following conditions are equivalent.
	\begin{enumerate}
		\item\label{e_split} $e$ has a splitting $(E,\pi,\iota)$;
		\item\label{e_eq} The parallel pair
		\[
		\begin{tikzcd}
			X \ar[shift left]{r}{e} \ar[shift right]{r}[swap]{\id} & X
		\end{tikzcd}
		\]
		has an equalizer $(E,\iota:E\to X)$;
		\item\label{e_coeq} The parallel pair above has a coequalizer $(E,\pi:X\to E)$. 
	\end{enumerate}
	Moreover, the equalizer and coequalizer above, if they exist, are preserved by every functor.
\end{proposition}

\begin{proof}
	For $\ref{e_split}\Rightarrow\ref{e_eq}$, suppose that $(E,\pi,\iota)$ splits $e$. We have to show that $(E,\iota)$ is the equalizer of the pair $(e,\id_X)$. Let $f:A\to X$ be such that $e\circ f=\id\circ f=f$.
	We have to find a unique map $u:A\to E$ making the triangle on the left commute:
	\[
	\begin{tikzcd}
		A \ar[virtual]{dr}[swap]{u} \ar{r}{f} & X \ar[shift left]{r}{e} \ar[shift right]{r}[swap]{\id} & X \\
		& E \ar{u}[swap]{\iota}
	\end{tikzcd}
	\]
	Take now $u=\pi\circ f$. We have that 
	\[
	\iota\circ u \;=\; \iota\circ \pi\circ f \;=\; e\circ f \;=\; f .
	\]
	To see that $u$ is unique, notice that $\iota$ is split monic. 
	
	For $\ref{e_eq}\Rightarrow\ref{e_split}$, suppose that $(E,\iota)$ is the equalizer of the pair $(e,\id_X)$. By idempotency of $e$, we have that $e\circ e=e=\id\circ e$, so that in the following diagram there is a unique $u:X\to E$ making the left triangle commute.
	\[
	\begin{tikzcd}
		X \ar[virtual]{dr}[swap]{u} \ar{r}{e} & X \ar[shift left]{r}{e} \ar[shift right]{r}[swap]{\id} & X \\
		& E \ar{u}[swap]{\iota}
	\end{tikzcd}
	\]
	Set now $\pi=u$. Commutativity of the triangle gives $\iota\circ\pi=e$. To show that $\pi\circ\iota=\id_E$, notice that 
	\[
	\iota\circ \pi\circ\iota \;=\; e\circ\iota \;=\; \iota \;=\; \iota\circ\id_E ,
	\]
	and that $\iota$ is monic (being an equalizer).
	
	The proofs of $\ref{e_split}\Rightarrow\ref{e_coeq}$ and $\ref{e_coeq}\Rightarrow\ref{e_split}$ are analogous and dual.
	
	Finally, let $F:\cat{C}\funto\cat{D}$ be a functor, and suppose that any (hence all) of the conditions above are satisfied. By functoriality, $Fe:FX\to FX$ is an idempotent, split by $F\iota:FE\to FX$ and $F\pi:FX\to FE$. So by the arguments above, $(FE,F\iota)$ is an equalizer and $(FE,f\pi)$ is a coequalizer of $(Fe,\id_{FX})$.
\end{proof}

Let's now connect these concepts to Cauchy completion.

\begin{definition}\label{e_to_cauchy}
	Let $e:X\to X$ be an idempotent of $\cat{C}$. 
	The \newterm{Cauchy point associated to $e$} is constructed as follows.
	\begin{itemize}
		\item We construct, as in \Cref{equalizer}, a functor $\mathrm{Inv}_R:\cat{C}\funto\cat{Set}$ (for ``right-invariant'') by 
		\[
		\mathrm{Inv}_R(A) \;\coloneqq\; \{f:X\to A \;|\; f\circ e = f\} .
		\]
		On morphisms, it just gives postcomposition of morphisms. (It is a subfunctor of the representable functor $\cat{C}(X,-)$.)
		\item Dually, we construct a presheaf $\mathrm{Inv}_L:\cat{C}^\op\funto\cat{Set}$ by
		\[
		\mathrm{Inv}_L(A) \;\coloneqq\; \{p:A\to X \;|\; e\circ p = p\} .
		\]
		On morphisms, it just gives precomposition of morphisms. (It is a subfunctor of the representable presheaf $\cat{C}(-,X)$.)
		\item As mapping $c_{A,B}:FA\times PB\to \cat{C}(A,B)$ we simply take composition of morphisms;
		\item As equivalence class $i$ we take $[X,e,e]$. (Note that by idempotency, $e$ is both in $\mathrm{Inv}_R(X)$ and in $\mathrm{Inv}_L(X)$.)
	\end{itemize}
\end{definition}

To check conditions \eqref{id_cond}, for every $f\in \mathrm{Inv}_R(A)$ and $p\in \mathrm{Inv}_L(A)$,
\begin{equation}\label{id_cond2}
c_{X,A}(e,f)_*e \;=\; (f\circ e)\circ e \;=\; f , \qquad c_{A,X}(p,e)^*e \;=\; e\circ(e\circ p) \;=\; p.
\end{equation}

Conversely, given a Cauchy point $(F,P,c,i)$ with $i=[X,\pi,\iota]$. Recall that, as in \eqref{virtual_retract}, we can see $\pi$ and $\iota$ as forming a ``virtual retract'':
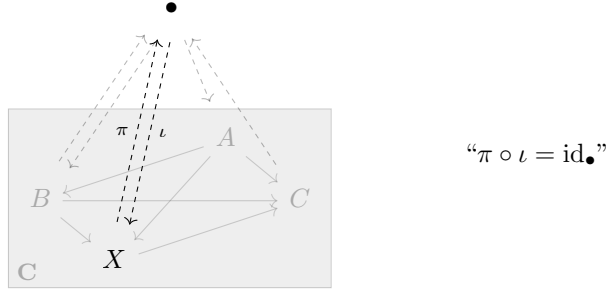
\[
\begin{tikzcd}[baseline=-1em,
	row sep=small, column sep=2mm,
	blend group=multiply,
	/tikz/execute at end picture={
		\node [cbox, fit=(A) (B) (C) (D), inner sep=1.2mm] (CC) {};
		\node [catlabel] at (CC.south west) {$\cat{C}$};
	}]
	&& |[alias=E,xshift=2mm]| \bullet \\ \\ \\ \\
	&&& |[alias=D,xshift=-5mm]| \color{cgray} A \\
	|[alias=A]| \color{cgray} B &&&& |[alias=B]| \color{cgray} C \\
	& |[alias=C,xshift=5mm]| X
	\ar[from=A,to=B,color=mgray]
	\ar[from=A,to=C,color=mgray]
	\ar[from=C,to=B,color=mgray]
	\ar[from=D,to=A,color=mgray]
	\ar[from=D,to=B,color=mgray]
	\ar[from=D,to=C,color=mgray]
	\ar[virtual,from=E,to=A,shift left,shorten=2.3mm,color=cgray]
	\ar[virtual,from=A,to=E,shift left,shorten=2.3mm,color=cgray]
	\ar[virtual,from=E,to=C,shift left,shorten=2.3mm,"\iota"{pos=0.45}]
	\ar[virtual,from=C,to=E,shift left,shorten=2.3mm,"\pi"{pos=0.45}]
	\ar[virtual,from=B,to=E,shorten=2.3mm,color=cgray]
	\ar[virtual,from=E,to=D,shorten=2.3mm,color=cgray]
\end{tikzcd}
\qquad\qquad\quad``\pi\circ\iota=\id_\bullet"
\]
It would therefore be tempting to define an idempotent on $X$ by setting $e=c(\pi,\iota):X\to X$. 
This would indeed be idempotent, since
\begin{equation}\label{would_be_e}
c_{X,X}(\pi,\iota)\circ c_{X,X}(\pi,\iota) \;=\; c_{X,X}(c_{X,X}(\pi,\iota)_*\pi,\iota) \;=\; c_{X,X}(\pi,\iota) ,
\end{equation}
using naturality of $c$ and \eqref{id_cond2}. 
We however need a little care: the resulting construction may not be well defined, since the choice of the object $X$ depends on the chosen representative of $i=[X,\pi,\iota]$. 
Indeed, consider the following situation, where we have a ``cycle'':
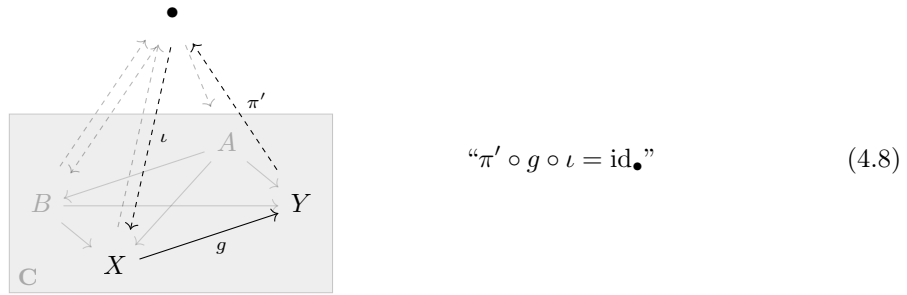
\begin{equation}\label{virtual_retract_equiv}
\begin{tikzcd}[baseline=-1em,
	row sep=small, column sep=2mm,
	blend group=multiply,
	/tikz/execute at end picture={
		\node [cbox, fit=(A) (B) (C) (D), inner sep=1.2mm] (CC) {};
		\node [catlabel] at (CC.south west) {$\cat{C}$};
	}]
	&& |[alias=E,xshift=2mm]| \bullet \\ \\ \\ \\
	&&& |[alias=D,xshift=-5mm]| \color{cgray} A \\
	|[alias=A]| \color{cgray} B &&&& |[alias=B]| Y \\
	& |[alias=C,xshift=5mm]| X
	\ar[from=A,to=B,color=mgray]
	\ar[from=A,to=C,color=mgray]
	\ar[from=C,to=B,"g"']
	\ar[from=D,to=A,color=mgray]
	\ar[from=D,to=B,color=mgray]
	\ar[from=D,to=C,color=mgray]
	\ar[virtual,from=E,to=A,shift left,shorten=2.3mm,color=cgray]
	\ar[virtual,from=A,to=E,shift left,shorten=2.3mm,color=cgray]
	\ar[virtual,from=E,to=C,shift left,shorten=2.3mm,"\iota"{pos=0.45}]
	\ar[virtual,from=C,to=E,shift left,shorten=2.3mm,color=cgray]
	\ar[virtual,from=B,to=E,shorten=2.3mm,"\pi'"{swap,pos=0.45}]
	\ar[virtual,from=E,to=D,shorten=2.3mm,color=cgray]
\end{tikzcd}
\qquad\qquad\quad``\pi'\circ g\circ\iota=\id_\bullet"
\end{equation}
(Setting $\pi=\pi'\circ g$ and $\iota'=g\circ\iota$, we see that $(X,\pi,\iota)\sim(Y,\pi',\iota')$).
We could now take as idempotent either $\iota\circ\pi'\circ g:X\to X$ or $g\circ\iota\circ\pi':Y\to Y$.
So, in order to have a well defined mapping, we have to define certain equivalence classes of idempotents. We will do even more: we will define a \emph{category} of idempotents.

\begin{definition}
	The \newterm{Karoubi envelope} or \newterm{idempotent completion} of a category $\cat{C}$ is a category $\cat{K(C)}$ where
	\begin{itemize}
		\item Objects are pairs $(X,e)$, where $X$ is an object of $\cat{C}$ and $e:X\to X$ is an idempotent;
		\item Morphisms $(X,e)\to (X',e')$ are morphisms $g:X\to X'$ of $\cat{C}$ which are bi-invariant, i.e.\ such that $g\circ e = e'\circ g=g$;
		\item The identity $(X,e)\to (X,e)$ is given by the morphism $e$ (\emph{not} by the identity in $\cat{C}$);
		\item The composition of morphisms is the one of $\cat{C}$.
	\end{itemize}
\end{definition}

The original category $\cat{C}$ is embedded into its Karoubi envelope via the fully faithful functor $X\mapsto (X,\id_X)$.

In the category $\cat{K(C)}$, two idempotents $e:X\to X$ and $e':X'\to X'$ are isomorphic if and only if there are morphisms $g:X\to X'$ and $h:X'\to X$ such that the following diagrams commute.
\begin{equation}\label{iso_K}
\begin{tikzcd}
	X \ar{d}[swap]{e} \ar{r}{g} \ar{dr}{g} & X' \ar{d}{e'} \\
	X \ar{r}[swap]{g} & X' 
\end{tikzcd}
\qquad
\begin{tikzcd}
	X' \ar{d}[swap]{e'} \ar{r}{h} \ar{dr}{h} & X \ar{d}{e} \\
	X' \ar{r}[swap]{h} & X
\end{tikzcd}
\qquad
\begin{tikzcd}
	X \ar{r}{g} \ar{dr}[swap]{e} & X' \ar{d}{h} \\
	& X
\end{tikzcd}
\qquad
\begin{tikzcd}
	X' \ar{r}{h} \ar{dr}[swap]{e'} & X \ar{d}{g} \\
	& X'
\end{tikzcd}
\end{equation}
Note that the maps $g$ and $h$ may not be isomorphisms in the category $\cat{C}$, since in the two triangles on the right the idempotents appear instead of the identities of $\cat{C}$.

\begin{proposition}
	Let $e:X\to X$ and $e':X'\to X'$ be isomorphic idempotents.
	Then $e$ splits if and only if $e'$ does, and if so their splittings coincide.
\end{proposition}
\begin{proof}
	Consider an isomorphism $(g:X\to X',h:X'\to X)$ as defined above, and let $(E,\pi:X\to E,\iota:E\to X)$ be a splitting of $e$. Consider now the maps $\iota'\coloneqq g\circ\iota:E\to X'$ and $\pi'\coloneqq\pi\circ h: X'\to E$. We have that 
	\[
	\iota'\circ\pi' \;=\; g\circ\iota\circ \pi\circ h \;=\; g\circ e\circ h \;=\; g\circ h \;=\; e' 
	\]
	and 
	\[
	\pi'\circ\iota \;=\; \pi\circ h\circ g\circ\iota \;=\; \pi \circ e \circ \iota \;=\; \pi\circ\iota \;=\; \id_E ,
	\]
	so that $(E,\pi',\iota')$ is a splitting of $e'$.
\end{proof}

Let's now establish the correspondence.

\begin{theorem}\label{c_e}
	Let $\cat{C}$ be a category.
	\Cref{e_to_cauchy} establishes an equivalence of categories $\cat{K(C)}\funto\overline{\cat{C}}$. 
\end{theorem}
\begin{proof}
	Let's show that the construction in \Cref{e_to_cauchy} is functorial. Let $g:(X,e)\to (X',e')$ be a morphism of idempotents, and denote the Cauchy points of $(X,e)$ and $(X',e')$ by $\Phi(X,e)=(\mathrm{Inv}_L,\mathrm{Inv}_R,c,i)$ and $\Phi(X',e')=(\mathrm{Inv}'_L,\mathrm{Inv}'_L,c',i')$. 
	By \Cref{c_morph_nat}, we can equivalently specify a morphism of Cauchy points by a natural transformation $\mathrm{Inv}_L\Rightarrow\mathrm{Inv}'_L$. We take the one given by postcomposition with $g$:
	\[
	\begin{tikzcd}[row sep=0]
		\mathrm{Inv}_L(A) \ar{r}{\Phi(g)} & \mathrm{Inv}'_L(A) \\
		\color{dgray} (A \xrightarrow{p} X) \ar[mapsto,dgray]{r} & \color{dgray} (A \xrightarrow{p} X \xrightarrow{g} X')
	\end{tikzcd}
	\]
	Note indeed that $e'\circ g\circ p = g\circ p$. 
	(In terms of elements of the pairing, this morphism is specified by $[X,g,e]=[X',e',g]\in\langle \mathrm{Inv}_L,\mathrm{Inv}_R \rangle$.)
	To show that this assignment preserves identities (recall that the identity of $(X,e)$ is $e$), notice that 
	\[
	\begin{tikzcd}[row sep=0]
		\mathrm{Inv}_L(A) \ar{r}{\Phi(e)} & \mathrm{Inv}'_L(A) \\
		\color{dgray} p \ar[mapsto,dgray]{r} & \color{dgray} e\circ p \;=\; p
	\end{tikzcd}
	\]
	by the fact that $p$ is left-invariant.
	Composition, instead, works as usual.
	This gives a functor $\Phi:\cat{K(C)}\funto\overline{\cat{C}}$. 
	
	To show that $\Phi$ is faithful, let $g,h:(X,e)\to(X',e')$ be morphisms of idempotents. Suppose that for all $A$ and for all left-invariant morphisms $p:A\to X$, $g\circ p=h\circ p$.
	Then setting $A=X$ and $p=e$, we have that $g=g\circ e=h\circ e=h$.
	
	To show that $\Phi$ is full, let $\alpha:\mathrm{Inv}_L\Rightarrow\mathrm{Inv}'_L$ be a natural transformation. Define $g$ by
	\[
	\begin{tikzcd}[row sep=0]
		\mathrm{Inv}_L(X) \ar{r}{\alpha_X} & \mathrm{Inv}'_L(X) \\
		\color{dgray} e \ar[mapsto,dgray]{r} & \color{dgray} g\coloneqq \alpha_X(e) . 
	\end{tikzcd}
	\]
	Notice that, since $g$ is by construction an element of $\mathrm{Inv}'_L(X)$, $e'\circ g=g$. 
	Also, by naturality of $\alpha$, the following diagram commutes.
	\[
	\begin{tikzcd}[sep=small]
		\color{dgray} e \ar[mapsto,dgray]{rrrrr} \ar[mapsto,dgray]{ddddd} &&&&& |[xshift=12mm,overlay]| \color{dgray} g \ar[mapsto,dgray]{ddddd} \\
		& \mathrm{Inv}_L(X) \ar{ddd}[swap]{e^*} \ar{rrr}{\alpha_X} &&& \mathrm{Inv}'_L(X) \ar{ddd}{e^*} \\ \\ \\
		& \mathrm{Inv}_L(X) \ar{rrr}[swap]{\alpha_X} &&& \mathrm{Inv}'_L(X) \\
		\color{dgray} \qquad e\circ e = e \ar[mapsto,dgray]{rrrrr} &&&&& |[xshift=12mm,overlay]| \color{dgray} g=g\circ e \qquad
	\end{tikzcd}
	\qquad\qquad
	\]
	Therefore $g$ is a morphism of idempotents.
	Now again by naturality of $\alpha$, for every left-invariant $p:A\to X$, the following diagram commutes.
	\[
	\begin{tikzcd}[sep=small]
		\color{dgray} e \ar[mapsto,dgray]{rrrrr} \ar[mapsto,dgray]{ddddd} &&&&& |[xshift=12mm,overlay]| \color{dgray} g \ar[mapsto,dgray]{ddddd} \\
		& \mathrm{Inv}_L(X) \ar{ddd}[swap]{p^*} \ar{rrr}{\alpha_X} &&& \mathrm{Inv}'_L(X) \ar{ddd}{p^*} \\ \\ \\
		& \mathrm{Inv}_L(A) \ar{rrr}[swap]{\alpha_A} &&& \mathrm{Inv}'_L(A) \\
		\color{dgray} \qquad e\circ p = p \ar[mapsto,dgray]{rrrrr} &&&&& |[xshift=12mm,overlay]| \color{dgray} \alpha_A(p) = g\circ p \qquad\qquad
	\end{tikzcd}
 	\qquad\qquad
	\]
	In other words, $\alpha$ is in the form $\Phi(g)$ for some $g:(X,e)\to(X',e')$, and so $\Phi$ is full.
	
	To show that $\Phi$ is essentially surjective, let $(F,P,c,i)$ be a Cauchy point, with $i=[X,\pi,\iota]$. Then $e\coloneqq c(\pi,\iota):X\to X$ is idempotent, as shown in \eqref{would_be_e}. We how have to show that $\Phi(X,e)$ is isomorphic to $(F,P,c,i)$.
	It suffices to prove that the set 
	\begin{equation}\label{rinv}
		\mathrm{Inv}_L(A) \;\coloneqq\; \{p:A\to X \;|\; e\circ p=p\}
	\end{equation}
	is isomorphic to $PA$ for all $A$, and naturally in $A$.
	First of all, by the Yoneda lemma, natural transformations $\cat{C}(-,X)\Rightarrow P$ are one-to-one with elements of $PX$. Let's then take the element $\pi\in PX$. For every $A$ the Yoneda lemma gives us a function 
	\[
	\begin{tikzcd}[row sep=0]
		\cat{C}(A,X) \ar{r} & PA \\
		\color{dgray} p \ar[mapsto,dgray]{r} & \color{dgray} p^*\pi
	\end{tikzcd}
	\]
	Let's restrict this function to $\mathrm{Inv}_L(A)$, and prove that we have a bijection.
	To prove injectivity, suppose that $p,p'\in\mathrm{Inv}_L(A)$ satisfy $p^*\pi={p'}^*\pi$.
	Now,
	\[
	p \;=\; e\circ p\;=\; c_{X,X}(\pi,\iota) \circ p \;=\; c(p^*\pi,\iota) 
	\]
	and the same is true for $p'$, so that we have $p=p'$.
	To prove surjectivity, notice that by \eqref{id_cond}, every $\tilde{p}\in PA$ satisfies
	\[
	\tilde{p} \;=\; c(\tilde{p},\iota)^*\pi ,
	\]
	so that setting $p=c(\tilde{p},\iota)$ gives $p^*\pi=\tilde{p}$.
	This shows that $\mathrm{Inv}_L\cong P$, which means that $\Phi$ is essentially surjective, and hence an equivalence of categories.
\end{proof}

Finally, this correspondence preserves splittings of idempotents.

\begin{proposition}\label{repr_split}
	An idempotent $(X,e)$ splits if and only if its corresponding Cauchy point $\Phi(X,e)$ is already in $\cat{C}$ (in the sense of \Cref{thm_cauchy}). In that case, the splitting coincides with the representing object.
\end{proposition}
\begin{proof}
	By \Cref{e_split_eq}, $e$ splits if and only if it admits an equalizer with $\id:X\to X$. As we said in \Cref{equalizer}, such an equalizer, if it exists, is exactly a representing object for $\mathrm{Inv}_L$.
\end{proof}

In other words, we have a commutative diagram of fully faithful functors as follows.
\[
\begin{tikzcd}[row sep=small] 
	& \cat{K(C)} \ar[functor, <->]{dd}{\simeq} \\
	\cat{C} \ar[functor,hook]{ur} \ar[functor,hook]{dr} \\
	& \overline{\cat{C}}
\end{tikzcd}
\]

\begin{corollary}\label{comp_split}
	A category is Cauchy-complete if and only if all its idempotents split.
\end{corollary}

\subsection{In terms of retracts of representables}

Another equivalent characterization of the Cauchy completion for unenriched categories, often used in the literature, is as retracts of representables.
This is another way of making precise the idea of ``virtual retract'' that we saw in \eqref{virtual_retract}:
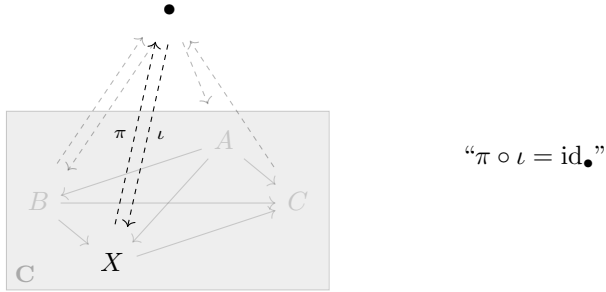
\[
\begin{tikzcd}[baseline=-1em,
	row sep=small, column sep=2mm,
	blend group=multiply,
	/tikz/execute at end picture={
		\node [cbox, fit=(A) (B) (C) (D), inner sep=1.2mm] (CC) {};
		\node [catlabel] at (CC.south west) {$\cat{C}$};
	}]
	&& |[alias=E,xshift=2mm]| \bullet \\ \\ \\ \\
	&&& |[alias=D,xshift=-5mm]| \color{mgray} A \\
	|[alias=A]| \color{mgray} B &&&& |[alias=B]| \color{mgray} C \\
	& |[alias=C,xshift=5mm]| X
	\ar[from=A,to=B,color=mgray]
	\ar[from=A,to=C,color=mgray]
	\ar[from=C,to=B,color=mgray]
	\ar[from=D,to=A,color=mgray]
	\ar[from=D,to=B,color=mgray]
	\ar[from=D,to=C,color=mgray]
	\ar[virtual,from=E,to=A,shift left,shorten=2.3mm,color=cgray]
	\ar[virtual,from=A,to=E,shift left,shorten=2.3mm,color=cgray]
	\ar[virtual,from=E,to=C,shift left,shorten=2.3mm,"\iota"{pos=0.45}]
	\ar[virtual,from=C,to=E,shift left,shorten=2.3mm,"\pi"{pos=0.45}]
	\ar[virtual,from=B,to=E,shorten=2.3mm,color=cgray]
	\ar[virtual,from=E,to=D,shorten=2.3mm,color=cgray]
\end{tikzcd}
\qquad\qquad\quad``\pi\circ\iota=\id_\bullet"
\]

\begin{proposition}\label{cauchy_retract}
	A Cauchy point of a category $\cat{C}$ is equivalently specified by a set functor (or presheaf) which is a retract of a representable one.
\end{proposition}
\begin{proof}
	Let's prove the assert for functors, the presheaf case is analogous and dual.
	
	First of all, given a Cauchy point $(F,P,c,i)$ with $i=[X,\pi,\iota]$, the functor $F$ is a retract of $\cat{C}(X,-)$, as shown in \eqref{virtual_retract_eq}. 
	
	Conversely, suppose that a set functor $F$ is a retract of a representable, meaning that for some object $X$ we have natural transformations $\iota:F\Rightarrow\cat{C}(X,-)$ and $\pi: \cat{C}(X,-)\Rightarrow F$ such that $\pi\circ\iota=\id_F$. 
	Then the other composition, $\iota\circ\pi:\cat{C}(X,-)\to \cat{C}(X,-)$, is idempotent, and since the Yoneda embedding is fully faithful, it corresponds to an idempotent morphism $e:X\to X$. By \Cref{c_e}, we have a Cauchy point $\Phi(X,e)$.
	
	To show that this recovers the original functor, i.e.\ that $\mathrm{Inv}_R\cong F$, Recall that $\mathrm{Inv}_R\subseteq\cat{C}(X,-)$, and that we already have a retraction $\iota:F\Rightarrow\cat{C}(X,-)$, $\pi: \cat{C}(X,-)\Rightarrow F$. We need to show that $\mathrm{Inv}_R\subseteq\cat{C}(X,-)$, just like $F$, splits the idempotent natural transformation $\iota\circ\pi$. Now the latter natural transformation acts as follows,
	\[
	\begin{tikzcd}[row sep=0]
		\cat{C}(X,A) \ar{r}{\pi_A} & FA \ar{r}{\iota_A} & \cat{C}(X,A) \\
		\color{dgray} f \ar[mapsto,dgray]{rr} && f\circ e , 
	\end{tikzcd}
	\]
	and so its equalizer with the identity is exactly the set
	\[
	\mathrm{Inv}_R(A) = \{f:A\to R \;|\; f\circ e = f \} .
	\]
	This is true for all objects $A$, and so $\mathrm{Inv}_R$ and $F$ are naturally isomorphic.
\end{proof}

\subsection{In terms of absolute limits}

Another important characterization of Cauchy completion is that it is a completion of a category under all those limits and/or colimits which are preserved by any functor.

\begin{definition}
	Let $D:\cat{J}\funto\cat{C}$ be a diagram, weighted by $W:\cat{J}\funto\cat{Set}$. A $W$-weighted limit of $D$ is called \newterm{absolute} if it is preserved by every functor $F:\cat{C}\funto\cat{D}$. 
	
	Absolute colimits are defined analogously.
\end{definition}

One can interpret absolute limits as cones whose universal property follows directly from the fact that a certain diagram commutes. This tends to look like a ``retraction'' of some kind (and a ``section'' of some kind for absolute colimits).
Before we make this intuition precise, let's see this at work in an example that should be familiar.

\begin{example}
	A splitting $(E,\pi,\iota)$ of an idempotent $e:X\to X$, as we saw in \Cref{e_split_eq}, is both an absolute equalizer and an absolute coequalizer of the parallel pair $e,\id:X\to X$.
	
	Note that, if a map $\iota:E\to X$ is an equalizer of $e$ and $\id_X$ if and only if postcomposition with $\iota$ makes the natural map
	\begin{equation}\label{should_bij}
	\begin{tikzcd}[row sep=0]
		\cat{C}(A,E) \ar{r}{\iota\circ -} & \mathrm{Inv}_L (A) \\
		\color{dgray} (A\xrightarrow{f} E) \ar[mapsto,dgray]{r} & \color{dgray} (A\xrightarrow{f} E\xrightarrow{\iota} X)
	\end{tikzcd}
	\end{equation}
	a natural \emph{bijection}.
	Indeed, that's exactly what the universal property of the equalizer says, usually depicted in a diagram as follows.
	\[
	\begin{tikzcd}
		A \ar{dr}{f} \ar[virtual]{d} \\
		E \ar{r}[swap]{\iota} & X \ar[shift left]{r}{e} \ar[shift right]{r}[swap]{\id} & X		
	\end{tikzcd}
	\]
	
	Now, as we saw in the proof of \Cref{e_split_eq}, the reason why the map \eqref{should_bij} is a bijection is because of the map $\pi$: 
	\[
		\begin{tikzcd}[row sep=0]
			\cat{C}(A,E) \ar[shift left]{r}{\iota\circ -} & \mathrm{Inv}_L (A) \ar[shift left]{l}{\pi\circ -}  \\
			\color{dgray} (A\xrightarrow{f} E) \ar[mapsto,dgray]{r} & \color{dgray} (A\xrightarrow{f} E\xrightarrow{\iota} X) \\
			\color{dgray} (A\xrightarrow{g} X \xrightarrow{\pi} E ) & \color{dgray} (A\xrightarrow{g} X) \ar[mapsto,dgray]{l}
		\end{tikzcd}
	\]
	Indeed, the two assignments above are mutually inverse thanks to the fact that $\pi\circ\iota=\id_E$, i.e.\ that the following diagram commutes.
	\[
	\begin{tikzcd}
		& X \ar{dr}{\pi} \\
		E \ar{ur}{\iota} \ar{rr}[swap]{\id} && X
	\end{tikzcd}
	\]
	In other words, $\iota$ is an equalizer because it has a particular retraction $\pi$. 
\end{example}

Let's now make this ``retraction'' idea mathematically precise.

\begin{definition}\label{ev_retract}
	Consider a diagram $D:\cat{J}\funto\cat{C}$ weighted by $W:\cat{J}\funto\cat{Set}$ and a weighted cone $c$ with tip $T$.
	A \newterm{eventual retraction} of $c$ is an element of the pairing (reversing the arrows of $\cat{J}$)
	\[
	[J,\iota,\pi] \;\in\; \langle W-, \cat{C}(D-,T) \rangle \;=\; \Coend{J'\in\cat{J}^\op} WJ'\times  \cat{C}(DJ',T) .
	\] 
	such that $\pi\circ c_{J,\iota}=\id_T$.

	Dually, given weights $W:\cat{J}^\op\funto\cat{Set}$ and a weighted co-cone with tip $T$, an \newterm{eventual section} is an element of the pairing 
	\[
	[J,\pi,\iota] \;\in\; \langle W-,\cat{C}(T,D-) \rangle \;=\; \Coend{J'\in\cat{J}} WJ'\times \cat{C}(T,DJ') .
	\]
	such that $c_{J,\pi}\circ\iota=\id_T$.
\end{definition}

Let's interpret this graphically (as usual, let's do it for cones, the co-cone case is dual).
An eventual retraction $[J,\iota,\pi]$ (or at least a representative of that class) consists explicitly of 
\begin{itemize}
	\item An object $J$ of $\cat{J}$;
	\item A weight $\iota\in WJ$, i.e.\ a virtual arrow into $J$;
	\item A (real) arrow $\pi:DJ\to T$, seen as a virtual arrow out of $J$, as in \Cref{ptwise_kan}.
\end{itemize}
From these data one can obtain, by functoriality of $D$ and naturality of the cone, a (real) pair of arrows of $\cat{C}$:
\begin{itemize}
	\item The object $DJ$ of the diagram in $\cat{C}$;
	\item The arrow of the cone $c_{J,\iota}:T\to DJ$ (note that, by \Cref{wcone_is_nat}, the assignment $\iota\mapsto c_{J,\iota}$ can be seen as part of the natural transformation of components $WJ\to \cat{C}(T,DJ)$, given by the cone);
	\item The arrow $\pi:DJ\to T$, this time seen as a real arrow.
\end{itemize}
\[
\begin{tikzcd}[baseline=-0.5em,
	row sep=small, column sep=2mm,
	blend group=multiply,
	/tikz/execute at end picture={
		\node [cbox, fit=(A1) (B1) (C1), inner sep=1.2mm] (CJ) {};
		\node [catlabel] at (CJ.south west) {$\cat{J}$};
		\node [cbox, fit=(A) (B) (C) (E), inner sep=1.2mm] (CC) {};
		\node [catlabel] at (CC.south west) {$\cat{C}$};
	}]
	&& |[overlay]| \Coend{J'} WJ'\times \cat{C}(DJ',T) \ar{rrrrrrrrrrrrrrrrrr}{c_*} && &&&&&&&&&&&&&& && |[overlay]| \Coend{J'} \cat{C}(T,DJ')\times \cat{C}(DJ',T) \\ \\
	|[alias=E1]| \bullet &&&& |[alias=E2]| \bullet &&&&&&&&&&&&&& && |[alias=E]| T \\ \\ \\ \\
	|[alias=A1]| \color{mgray} L &&&& |[alias=B1]| \color{mgray} K &&&&&&&&&&&&&& |[alias=A]| \color{mgray} DL &&&& |[alias=B]| \color{mgray} DK \\
	&& |[alias=C1]| J && &&&&&&&&&&&&&& && |[alias=C]| DJ
	\ar[from=A1,to=B1,color=mgray]
	\ar[from=A1,to=C1,color=mgray]
	\ar[from=C1,to=B1,color=mgray]
	\ar[from=A,to=B,color=mgray]
	\ar[from=A,to=C,color=mgray]
	\ar[from=C,to=B,color=mgray]
	\ar[virtual,from=E1,to=A1,shift left,shorten=0.5mm,color=cgray]
	\ar[virtual,from=E1,to=A1,shift right,shorten=0.5mm,color=cgray]
	\ar[virtual,from=E1,to=C1,shorten=0.5mm,"\iota"{pos=0.5}]
	\ar[virtual,from=E1,to=B1,shorten=0.5mm,color=cgray]
	\ar[from=E,to=A,shift left,shorten=1.3mm,color=mgray]
	\ar[from=E,to=A,shift right,shorten=1.3mm,color=mgray]
	\ar[from=E,to=C,shift right,shorten=1.3mm,"c_{J,\iota}"{swap,pos=0.58,inner sep=0.5mm}]
	\ar[from=C,to=E,shift right,shorten=1.3mm,"\pi"{swap,pos=0.42}]
	\ar[from=E,to=B,shorten=1.3mm,color=mgray]
	\ar[from=B1, to=A, mapsto, color=dgray, "D", shorten=-2mm,shift right=5mm]
	\ar[virtual,from=C1,to=E2,shorten=1.3mm,"\pi"{swap,pos=0.63}]
\end{tikzcd}
\]
We moreover ask that $\pi\circ c_{J,\iota}=\id_T$, which makes $T$ a retract of $DJ$.

Now let's keep in mind that, even if we have a pair of arrows $[DJ,c_{J,\iota}:T\to DJ,\pi:DJ\to T]\in\int^{J'} \cat{C}(T,DJ')\times \cat{C}(DJ',T)$, we are taking a coend over $\cat{J}$, not over $\cat{C}$, and so the resulting equivalence class is identifying triplets $[J,\iota,\pi]$ and $[J',\iota',\pi']$ connected as follows,
\begin{equation}\label{mid_J}
\begin{tikzcd}
	& J \ar{dd}{g} \ar[virtual]{dr}{\pi} \\
	\bullet \ar[virtual]{ur}{\iota} \ar[virtual]{dr}[swap]{\iota'} && \bullet \\
	& J' \ar[virtual]{ur}[swap]{\pi'}
\end{tikzcd}
\qquad{\color{dgray}\longmapsto}\qquad
\begin{tikzcd}
	& DJ \ar{dd}{Dg} \ar{dr}{\pi} \\
	T \ar{ur}{c_{J,\iota}} \ar{dr}[swap]{c_{J',\iota'}} && T \\
	& DJ' \ar{ur}[swap]{\pi'}
\end{tikzcd}
\end{equation}
where \emph{the mediating arrows come from $\cat{J}$}, i.e.\ are arrows of the diagram $D$. We are not quotienting under generic arrows of $\cat{C}$.

Note also that eventual retractions are, in a certain sense, closed under precomposition: if $\pi$ is a retraction of $c_{J,\iota}$, for any morphism $h:L\to J$ of $\cat{C}$, and for all $\iota'$ in the preimage of $\iota$ for the map $h_*:WL\to WJ$, the composition $\pi\circ Dh:DL\to T$ is a retraction of $c_{L,\iota'}$. Indeed, by the fact that we have a cone,
\[
\pi\circ Dh \circ c_{L,\iota'} \;=\; \pi\circ c_{J,h_*\iota'} \;=\; \pi\circ c_{J,\iota} \;=\; \id_T.
\]
Hence the adjective ``eventual'': \emph{at some point} in the diagram we have a retraction, and such a retraction can be pulled back along the arrows of the diagram, ``all the way towards the limit''.

Somewhat dually, by functoriality of weighted colimits, eventual retractions are in a certain sense closed under postcomposition with arrows out of $T$. Namely, for every object $X$ and arrow $f:T\to X$ of $\cat{C}$, we get a function as follows:
\[
\begin{tikzcd}[row sep=0]
	\langle W-, \cat{C}(D-,T) \rangle \ar{r}{f_*} & \langle W-, \cat{C}(D-,X) \rangle \\
	\color{dgray} {[J,\iota,\pi]} \ar[mapsto,dgray]{r} & \color{dgray} {[J,\iota,f\circ\pi]} \\
	\color{dgray} (T \xrightarrow{c_{J,\iota}} DJ \xrightarrow{\pi} T ) \ar[mapsto,dgray]{r} & \color{dgray} ( T \xrightarrow{c_{J,\iota}} DJ \xrightarrow{\pi} T \xrightarrow{f} X )
\end{tikzcd}
\]
The fact that $\pi$ is a retraction makes $[J,\iota,f\circ\pi]$ satisfy the following condition,
\begin{equation}\label{lax_retract}
(f\circ \pi)\circ c_{J,\iota} \;=\; f\circ (\pi\circ c_{J,\iota}) \;=\; f .
\qquad\qquad\qquad
\begin{tikzcd}[row sep=small]
	T \ar{dr}{f} \ar[shift right]{dd}[swap]{c_{J,\iota}} \\
	& X \\
	DJ \ar{ur}[swap]{f\circ\pi} \ar[cgray,shift right]{uu}[swap]{\pi}
\end{tikzcd}
\end{equation}

\begin{definition}
	An eventual retraction $[J,\iota,\pi]$ of $c$ is called a \newterm{universal retraction} if for every object $X$ of $\cat{C}$ and arrow $f:T\to X$, any element 
	\[
	[K\in\cat{J},w\in WK,DK\xrightarrow{g} X] \;\in\; \langle W-, \cat{C}(D-,X) \rangle
	\]
	satisfying
	\[
	g\circ c_{K,w}\;=\; f
	\qquad\qquad\qquad
	\begin{tikzcd}[row sep=small]
		T \ar{dr}{f} \ar{dd}[swap]{c_{K,w}} \\
		& X \\
		DK \ar{ur}[swap]{g}
	\end{tikzcd}
	\]
	is in the form 
	\[
	[J',\iota\in WJ',DJ'\xrightarrow{\pi'} X] \;=\; f_*[J,\iota,\pi] \;=\; [J,\iota,f\circ\pi] .
	\]
	
	A \newterm{universal section} of a co-cone is defined dually.
\end{definition}

This in particular means that any object $DK$ in the diagram is connected to $DJ$ via arrows of the cone, meaning that the category $\cat{J}$ is necessarily connected.

Note moreover that a universal retraction, if it exists, must necessarily be unique. Indeed, if $[J,\iota,\pi]$ and $[J',\iota',\pi']$ are universal retractions of $c$, the commutative triangle
\[
\begin{tikzcd}[row sep=small]
	T \ar{dr}{\id} \ar{dd}[swap]{c_{J',\iota'}} \\
	& T \\
	DJ' \ar{ur}[swap]{\pi'}
\end{tikzcd}
\]
gives, by universality of $[J,\iota,\pi]$, an equality $[J',\iota',\pi']=[J,\iota,\id_T\circ\pi]=[J,\iota,\pi]$.

\begin{example}
	Let $e:X\to X$ be an idempotent with splitting $(E,\pi,\iota)$. We can view this as a diagram of two arrows, with $E$ and $\iota$ as limit cone (with unitary weights):
	\[
	\begin{tikzcd}[column sep=small]
		& E \ar{dl}[swap]{\iota} \ar{dr}{\iota} \\
		X \ar[shift left]{rr}{e} \ar[shift right]{rr}[swap]{\id} && X
	\end{tikzcd}
	\]
	The map $\pi:X\to E$ is now a universal retraction. (Or better, the class $[X,\iota,\pi]$ is, where with a slight abuse we write $X$ and $\iota$ instead of the object and element indexing them.)
	Indeed, let $f:E\to Y$ be any function, and consider a function $g:X\to Y$ (or, a triplet $[X,\iota,X\xrightarrow{g}Y]$)
	making the outer triangle below commute:
	\[
	\begin{tikzcd}[row sep=small]
		E \ar{dr}{f} \ar[shift right]{dd}[swap]{\iota} \\
		& Y \\
		X \ar{ur}[swap]{g} \ar[cgray,shift right]{uu}[swap]{\pi}
	\end{tikzcd}
	\]
	In general we cannot assume that $g=f\circ\pi$ necessarily. But it is so \emph{up to composing it with arrows of the diagram}:
	\[
	g\circ e \;=\; g\circ \iota\circ\pi \;=\; f\circ\pi .
	\]
	Therefore, by the usual equivalence relation, and since $e\circ\iota=\iota$,
	\[
	[X,\iota,g] \;=\; [X,e\circ\iota,g] \;=\; [X,\iota,g\circ e] \;=\; [X,\iota,f\circ\pi] \;=\; f_*[X,\iota,\pi] .
	\]
\end{example}

The general situation is given by the following two statements.

\begin{proposition}\label{retract_to_abs}
	Suppose that a weighted cone admits a universal retraction.
	Then it is necessarily a limit cone, and it is moreover preserved by all functors (i.e.\ it is an absolute limit).
	
	The same can be said about co-cones with a universal section.
\end{proposition}

\begin{proof}
	Let $D:\cat{J}\funto\cat{C}$ be a diagram weighted by $W:\cat{J}\funto\cat{Set}$, let $c$ be a cone with tip $T$, and let $[J,\pi,\iota]$ be a (necessarily unique) universal retraction. 
	To show that the cone is a limit, consider another weighted cone, with tip $T'$ and arrows $(d_{K,w})_{K\in\cat{J},w\in DJ}$. 
	We have to show that there exist a unique arrow $u:T'\to T$ making the respective triangles commute.
	
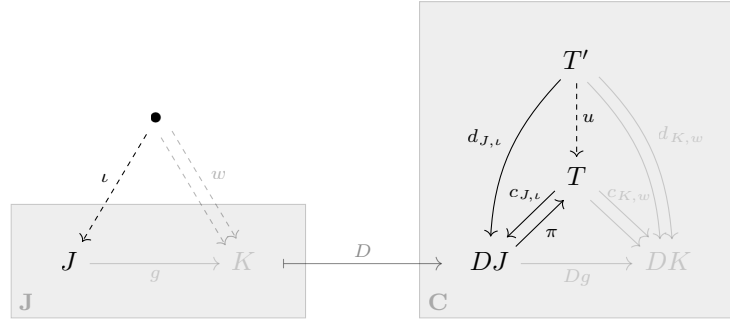
\[
	\begin{tikzcd}[sep=small,
		blend group=multiply,
		/tikz/execute at end picture={
			\node [cbox, fit=(J) (K), inner sep=5mm] (CL) {};
			\node [catlabel] at (CL.south west) {$\cat{J}$};
			\node [cbox, fit=(DJ) (DK) (TP), inner sep=5mm] (CR) {};
			\node [catlabel] at (CR.south west) {$\cat{C}$};
		}]
		&&&&&&&&& |[alias=TP]| T' \\ 
		& |[alias=E]| \bullet \\  
		&&&&&&&&& |[alias=T]| T  \\ \\  
		|[alias=J, xshift=-5mm]| J \ar[mgray]{rr}[swap]{g} && |[alias=K, xshift=5mm]| \color{mgray} K
		&&&&&& |[alias=DJ]| DJ \ar[mgray]{rr}[swap]{Dg} && |[alias=DK]| \color{mgray} DK 
		\ar[virtual, from=E, to=J,"\iota"']
		\ar[virtual, from=E, to=K, shift left=1, color=cgray,"w"]
		\ar[virtual, from=E, to=K, shift right=1, color=cgray]
		\ar[from=DJ, to=T, shift right, "\pi"']
		\ar[from=T, to=DJ, shift right,"c_{J,\iota}"{swap,inner sep=0mm,pos=0.25}]
		\ar[from=T, to=DK, shift left=1, color=mgray,"c_{K,w}"{inner sep=-0.4mm,pos=0.2}]
		\ar[from=T, to=DK, shift right=1, color=mgray]
		\ar[from=TP, to=DJ, bend right=20,shorten >=0.7mm,"d_{J,\iota}"']
		\ar[from=TP, to=DK, shift left=1, color=mgray, bend left=20,"d_{K,w}"]
		\ar[from=TP, to=DK, shift right=1, color=mgray, bend left=20,shorten >=0.7mm]
		\ar[virtual, from=TP, to=T, shorten <=0.5mm, "u"]
		\ar[mapsto,color=dgray, from=K, to=DJ, "D",shorten=2mm]
	\end{tikzcd}
	\]
	Set now $u\coloneqq \pi\circ d_{J,\iota}:T'\to T$. 
	
	To show that $u$ is well defined, consider an arrow $g:J\to J'$, and a triplet $(J',\iota',\pi')\sim(J,\iota,\pi)$ such that $g^*\pi'=\pi$ and $g_*\iota=\iota'$, as in \eqref{mid_J}. 
	Chasing the diagram
	
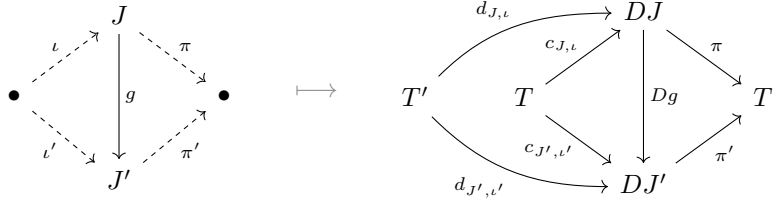
\[
	\begin{tikzcd}
		& J \ar{dd}{g} \ar[virtual]{dr}{\pi} \\
		\bullet \ar[virtual]{ur}{\iota} \ar[virtual]{dr}[swap]{\iota'} && \bullet \\
		& J' \ar[virtual]{ur}[swap]{\pi'}
	\end{tikzcd}
	\qquad{\color{dgray}\longmapsto}\qquad
	\begin{tikzcd}
		& & DJ \ar{dd}{Dg} \ar{dr}{\pi} \\
		T' \ar[out=45,in=180]{urr}{d_{J,\iota}} \ar[out=-45,in=180]{drr}[swap]{d_{J',\iota'}} & T \ar{ur}[inner sep=0.5mm]{c_{J,\iota}} \ar{dr}[swap,inner sep=0.5mm]{c_{J',\iota'}} && T \\
		& & DJ' \ar{ur}[swap]{\pi'}
	\end{tikzcd}
	\]
	we have, since $d$ is a cone, that
	\[
	\pi'\circ d_{J',\iota'} \;=\; \pi'\circ Dg\circ d_{J,\iota} \;=\; \pi\circ d_{J,\iota} .
	\]
	
	To show that $u$ makes the desired triangles commute, let $K$ be an object of $\cat{J}$. We have to show that for all $w\in WK$, $d_{K,w}=c_{K,w}\circ u$. Now by universality of $[J,\iota,\pi]$, the commutative triangle
	
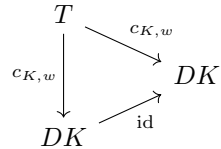
\[
	\begin{tikzcd}[row sep=small]
		T \ar{dr}{c_{K,w}} \ar{dd}[swap]{c_{K,w}} \\
		& DK \\
		DK \ar{ur}[swap]{\id}
	\end{tikzcd}
	\]
	gives us an equality $[K,w,\id_{DK}]=[J,\iota,c_{K,w}\circ\pi]$. 
	We can carry the equality under the composition mapping (which is well defined) as follows.
	\[
	\begin{tikzcd}[row sep=0]
		\displaystyle\Coend{J'} WJ'\times \cat{C}(DJ',DK) \ar{r}{c_*} & \Coend{J'} \cat{C}(T,DJ')\times \cat{C}(DJ',DK) \ar{r}{\circ} & \cat{C}(T,DK) \\
		\color{dgray} {[K,w,\id_{DK}]} \ar[equal,dgray,shorten=1mm]{dd} & \color{dgray} {[T\xrightarrow{c_{K,w}} DK\xrightarrow{\id} DK]} \ar[equal,dgray,shorten=0.3mm]{dd} & \color{dgray}  T\xrightarrow{c_{K,w}} DK  \ar[equal,dgray,shorten=0.8mm]{dd} \\
		\phantom{a} \ar[mapsto,dgray,shorten <=15mm, shorten >=25mm]{r} & \phantom{a} \ar[mapsto,dgray,shorten=20mm]{r} & \phantom{a} \\
		\color{dgray} {[J,\iota,c_{K,w}\circ\pi]} & \color{dgray} {[T\xrightarrow{c_{J,\iota}} DJ \xrightarrow{c_{K,w}\circ\pi} DK]} & \color{dgray} T\xrightarrow{c_{J,\iota}} DJ \xrightarrow{\pi} T \xrightarrow{c_{K,w}} DK 
	\end{tikzcd}
	\]
	Therefore
	\[
	c_{K,w}\circ u \;=\; c_{K,w}\circ \pi\circ d_{J,\iota} \;=\; c_{K,w} ,
	\]
	as desired.
	
	To show that $u$ is unique, notice that if $u':T'\to T$ is such that $d_{J,\iota}=c_{J,\iota}\circ u'$, then
	\[
	u' \;=\; \pi\circ c_{J,\iota}\circ u' \;=\; \pi\circ d_{J,\iota} \;=\; u.
	\]
	
	Finally, to show that the limit is absolute, notice that we can repeat the same procedure applying a functor to all the arrows: given $F:\cat{C}\funto\cat{D}$ and a limit cone in $\cat{D}$ with tip $T'$ and arrows $(d_{K,w})_{K\in\cat{J},w\in DJ}$, the unique map $u:T'\to FT$ is given by $F\pi\circ d_{J,\iota}$. 
\end{proof}

\begin{proposition}\label{abs_to_retract}
	Conversely to \Cref{retract_to_abs}, if a weighted cone is an absolute weighted limit, then it has a universal retraction.
	
	Dually, the same holds for universal weighted colimits and universal sections.
\end{proposition}

\begin{proof}
	As usual, we will focus on the limit case. 
	Let $D:\cat{J}\funto\cat{C}$ be a diagram weighted by $W:\cat{J}\funto\cat{Set}$, and let $c$ be an absolute limit cone with tip $T$. 
	Then by definition this limit is preserved by all functors $\cat{C}\funto\cat{D}$. 
	We now take as functor the Yoneda embedding \footnote{Technically, the category $\sfuncat{\cat{C}}^\op$ may fail to be locally small if $\cat{C}$ is not small. So one can, instead, take the category of all \emph{small} functors $\cat{C}\funto\cat{Set}$, i.e.~only those that can be expressed as a small limit of representable ones. The resulting category is locally small, see for example \cite{smallfunctors}.}
	\[
	\begin{tikzcd}[row sep=0]
		\cat{C} \ar[functor]{r}{\Yon} & \sfuncat{\cat{C}}^\op \\
		\color{dgray} X \ar[mapsto,dgray]{r} & \color{dgray} \cat{C}(X,-)
	\end{tikzcd}
	\]
	Since $T$ is the $W$-weighted limit of $\cat{D}$, and since we are assuming it is absolute, then $\Yon(T)=\cat{C}(T,-)$ is the weighted colimit of \[
	\begin{tikzcd}[row sep=0]
		\cat{J} \ar[functor]{r}{D} & \cat{C} \ar[functor]{r}{\Yon} & \sfuncat{\cat{C}}^\op \\
		\color{dgray} J \ar[mapsto,dgray]{r} & \color{dgray} DJ \ar[mapsto,dgray]{r} & \color{dgray} \cat{C}(DJ,-) .
	\end{tikzcd}
	\] 
	Notice now that limits in $\sfuncat{\cat{C}}^\op$ are colimits in $\sfuncat{\cat{C}}$, and that the latter are computed pointwise, as colimits of sets. We therefore have that for all objects $X$ of $\cat{C}$,
	\begin{equation}\label{coend_X}
	\cat{C}(T,X) \;\cong\; \colim_{J\in\cat{J}^\op} \big\langle WJ, \cat{C}(DJ,X) \big\rangle .
	\end{equation}
	This weighted colimit can be expressed as a pairing,
	\[
	\colim_{J\in\cat{J}^\op} \big\langle WJ, \cat{C}(DJ,X) \big\rangle \;=\; \langle W-, \cat{C}(D-,X) \rangle
	\]
	and for $X=T$ it is exactly the one used in \Cref{ev_retract}:
	\[
	\cat{C}(T,T) \;\cong\;  \langle W-, \cat{C}(D-,T) \rangle \;=\; \Coend{J'\in\cat{J}^\op} WJ' \times \cat{C}(DJ',T)
	\]
	Recall now that, since $\Yon$ preserves our limit, the arrows of the universal co-cone above an are given by precomposition with the arrows $c_{J,w}:T\to DJ$ of the original limit cone in $\cat{C}$:
	\[
	\begin{tikzcd}[row sep=0,column sep=large]
		\cat{C}(DJ,T) \ar{r}{-\circ c_{J,w}} & \cat{C}(T,T) \\
		\color{dgray} (DJ\xrightarrow{f} T) \ar[mapsto,dgray]{r} & \color{dgray} (T\xrightarrow{c_{J,w}} DJ\xrightarrow{f} T)
	\end{tikzcd}
	\]
	We can now decompose this map as follows,
	\[
	\begin{tikzcd}[row sep=0,column sep=large]
		\cat{C}(DJ,T) \ar{r} &  \Coend{J'\in\cat{J}^\op} WJ' \times \cat{C}(DJ',T) \ar[leftrightarrow]{r}{\cong} & \cat{C}(T,T) \\
		\color{dgray} (DJ\xrightarrow{f} T) \ar[mapsto,dgray]{r} & \color{dgray} [J,w,f] \ar[mapsto,dgray]{r} & \color{dgray} (T\xrightarrow{c_{J,w}} DJ\xrightarrow{f} T)
	\end{tikzcd}
	\]
	where the first map is the usual inclusion (of the $w$-th arrow in the tuple) followed by quotienting, and the second map is the isomorphism we considered above.
	Denote now by $[J,\pi,\iota]$ the element of the pairing corresponding to the identity $\id_T$ under the isomorphism above.
	Chasing the diagram,
	\begin{equation}\label{coend_triangle}
	\begin{tikzcd}[sep=tiny] 
		&&& |[xshift=2mm,yshift=3mm,overlay]| \color{dgray} [J,\iota,\pi] \ar[mapsto,dgray]{dddd}\\
		&& \int^{J'\in\cat{J}^\op} WJ' \times \cat{C}(DJ',T) \ar[leftrightarrow]{dd}{\cong} \\
		|[xshift=-10mm,overlay]| \color{dgray} \pi \ar[mapsto,dgray,out=30,in=180]{uurrr} \ar[mapsto,dgray,out=-30,in=180]{ddrrr} & \cat{C}(DJ,T) \ar{ur} \ar{dr}[swap]{-\circ c_{J,\iota}} \\
		&& \cat{C}(T,T) \\
		&&& |[xshift=2mm,yshift=-3mm,overlay]| \color{dgray} \pi\circ c_{J,\iota} =\id_T \qquad\qquad
	\end{tikzcd}
	\end{equation}
	we see that $[J,\iota,\pi]$ is an eventual retraction.
	
	To show that $[J,\iota,\pi]$ is universal, consider an object $K$ of $\cat{J}$, a weight $w\in WJ$, and a commutative triangle of $\cat{C}$ as follows.
	\[
	\begin{tikzcd}[row sep=small]
		T \ar{dr}{f} \ar{dd}[swap]{c_{K,w}} \\
		& X \\
		DK \ar{ur}[swap]{g}
	\end{tikzcd}
	\]
	We can form a diagram analogous to \eqref{coend_triangle} for the weighted colimit \eqref{coend_X}, 
	\[
	\begin{tikzcd}[sep=tiny] 
		&&& |[xshift=2mm,yshift=3mm,overlay]| \color{dgray} [K,w,g] \ar[mapsto,dgray]{dddd}\\
		&& \int^{J'\in\cat{J}^\op} WJ' \times \cat{C}(DJ',X) \ar[leftrightarrow]{dd}{\cong} \\
		|[xshift=-10mm,overlay]| \color{dgray} g \ar[mapsto,dgray,out=30,in=180]{uurrr} \ar[mapsto,dgray,out=-30,in=180]{ddrrr} & \cat{C}(DK,X) \ar{ur} \ar{dr}[swap]{-\circ c_{K,w}} \\
		&& \cat{C}(T,X) \\
		&&& |[xshift=2mm,yshift=-3mm,overlay]| \color{dgray} g\circ c_{K,w} = f 
	\end{tikzcd}
	\]
	where the bottom right equality holds now by hypothesis. Now by naturality in $X$ of the isomorphism \eqref{coend_X}, the following diagram commutes.
	\[
	\begin{tikzcd}[sep=small] 
		\color{dgray} \id_T \ar[mapsto,dgray]{rrrrr} \ar[mapsto,dgray]{ddddd} &&&&& |[xshift=5mm,overlay]| \color{dgray} f \ar[mapsto,dgray]{ddddd} \\
		& \cat{C}(T,T) \ar{rrr}{f\circ-} \ar[leftrightarrow]{ddd}[swap]{\cong} &&& \cat{C}(T,X)  \ar[leftrightarrow]{ddd}{\cong} \\ \\ \\
		& \int^{J'\in\cat{J}^\op} WJ' \times \cat{C}(DJ',T) \ar{rrr}[swap]{f_*} &&& \int^{J'\in\cat{J}^\op} WJ' \times \cat{C}(DJ',X) \\
		\color{dgray} [J,\iota,\pi] \ar[mapsto,dgray]{rrrrr} &&&&& |[xshift=5mm,overlay]| \color{dgray} f_*[J,\iota,\pi]\;=\; [K,w,g] \qquad\qquad\quad
	\end{tikzcd}
	\qquad
	\]
	This is exactly the condition making $[J,\iota,\pi]$ universal.
\end{proof}

Let's now connect the theory of absolute limits and colimits to Cauchy completion.
It may be helpful first to fix the following terminology.
\begin{definition}
	We call a diagram $D:\cat{J}\funto\cat{C}$ weighted by $W:\cat{J}\funto\cat{Set}$ an \newterm{absolute diagram} if for every functor $G:\cat{C}\funto\cat{D}$ such that the $W$-weighted limit of $G\circ D$ exists in $\cat{D}$, such a limit in $\cat{D}$ is absolute.
	
	The analogous concept for contravariantly weighted diagrams is defined dually.
\end{definition}

In particular, if a weighted diagram has a limit, then it is necessarily a weighted limit. 
(Note that being absolute is a property of a diagram \emph{and its weight}.)

\begin{definition}
	A (small) set functor $F:\cat{J}\funto\cat{Set}$ is called an \newterm{absolute weight} if and only if every $F$-weighted diagram is absolute.
	
	A (small) presheaf $P:\cat{J}^\op\funto\cat{Set}$ is called an \newterm{absolute weight} if and only if every $P$-weighted diagram is absolute.
\end{definition}

\begin{theorem}\label{pushforward}
A diagram $D:\cat{J}\funto\cat{C}$ weighted by $W:\cat{J}\funto\cat{Set}$
is absolute if and only if the limit set functor in $\sfuncat{\cat{C}}^\op$ of $\Yon\circ D:\cat{J}\funto\cat{Set}$ is part of a Cauchy point.
\end{theorem}

We will use first of all the following auxiliary statement, which a way of decomposing limits of composite functors as a two-step process.

\begin{lemma}\label{limit_decomp}
	Consider functors $D:\cat{J}\funto\cat{C}$, $G:\cat{C}\funto\cat{D}$ and $W:\cat{J}\funto\cat{Set}$. 
	We have that
	\[
	\lim_{J\in\cat{J}} \big\langle WJ, GDJ \big\rangle \;\cong\; \lim_{X\in\cat{C}} \left\langle \colim_{J\in\cat{J}^\op} \big\langle WJ, \cat{C}(DJ,X) \big\rangle , GX \right\rangle ,
	\]
	meaning that one expression exists if and only if the other one does, and if so, they are isomorphic.
\end{lemma}

Compare for example with sums: given finite sets $X$ and $Y$, and functions $X\xrightarrow{d}Y\xrightarrow{g}\R$,
\[
\sum_{x\in X} g(d(x)) \;=\; \sum_{y\in Y}\sum_{x\in X} \delta_{d(x),y}\,g(y) .
\]

Note moreover that the expression 
\[
\colim_{J\in\cat{J}^\op} \big\langle WJ, \cat{C}(DJ,X) \big\rangle
\]
is the pointwise expression of the weighted \emph{limit} of the functor $\Yon\circ D:\cat{J}\funto\sfuncat{\cat{C}}^\op$, as in \eqref{coend_X}.

\begin{proof}
	Let's first write the weighted colimit as a coend:
	\[
	\colim_{J\in\cat{J}^\op} \big\langle WJ, \cat{C}(DJ,X) \big\rangle \;\cong\; \int^J WJ\times \cat{C}(DJ,X) .
	\]
	Denote this set by $S_X$, with elements $[J,w,f]$, where $J\in\cat{J}$, $w\in WJ$, and $g:DJ\to X$. 
	
	To prove equivalence of the limits, we establish a natural bijection between the respective presheaves of weighted cones. This way, if any of the two presheaves is representable, so is the other one, and the representing objects (i.e.\ the limits) must coincide.
	
	First, let's start with a cone $c$ over $G:\cat{C}\funto\cat{D}$ with weights $S_X$ (for all $X\in\cat{C}$) and tip $T\in\cat{D}$. 
	Its arrows are in the form $c_{[J,w,f]}:T\to GX$.
	To obtain a $W$-weighted cone $d$ over $G\circ D:\cat{J}\funto\cat{D}$ we take $d_{J,w}\coloneqq c_{J,w,\id_{DJ}}$.
	To see it's a cone, notice that for all $g:J\to K$, the following diagram commutes,
	\[\begin{tikzcd}[row sep=small]
		& GDJ \ar{dd}{GDg} \\
		T \ar{ur}{d_{J,w}} \ar{dr}[swap]{d_{K,g_*w}} \\
		& GDK
	\end{tikzcd}
	\qquad=\qquad
	\begin{tikzcd}[row sep=small]
		& GDJ \ar{dd}{GDg} \\
		T \ar{ur}{c_{[J,w,\id_{DJ}]}} \ar{dr}[swap]{c_{[K,g_*w,\id_{DK}]}} \\
		& GDK 
	\end{tikzcd}
	\]
	since, by the fact that $c$ is a cone and by the usual equivalence relation,
	\[
	GDg\circ c_{[J,w,\id_{DJ}]} \;=\; c_{(Dg)_*[J,w,\id_{DJ}]} \;=\; c_{[J,w,Dg]} \;=\; c_{[K,g_*w,\id_{DK}]} .
	\]
	
	Conversely, starting with a $W$-weighted cone $d$ over $G\circ D$, define a weighted cone $c$ over $G$ by $c_{[J,w,f]}:=Gf\circ d_{J,w}$:
	\begin{equation}\label{c_from_d}
	\begin{tikzcd}
		T \ar{r}{d_{J,w}} & GDJ \ar{r}{Gf} & GX
	\end{tikzcd}
	\end{equation}
	To see that it is well defined on equivalence classes, let $g:J\to K$ and suppose $f=f'\circ Dg$ for some $f':DK\to X$. Then the following diagram commutes,
	\[
	\begin{tikzcd}[row sep=small]
		& GDJ \ar{dr}{Gf} \ar{dd}{GDg} \\
		T \ar{ur}{d_{J,w}} \ar{dr}[swap]{d_{K,g_*w}} && GX \\
		& GDK \ar{ur}[swap]{Gf'}
	\end{tikzcd}
	\]
	since the triangle on the left is the cone condition for $d$, and the one on the right is functoriality of $G$. 
	To see that $c$ is a cone, notice that for $h:X\to Y$ of $\cat{C}$,
	\[
	Gh\circ c_{[J,w,f]} \;=\; Gh\circ Gf\circ d_{J,w} \;=\; g(h\circ f) \circ d_{J,w} \;=\; c_{[J,w,h\circ f]} \;=\; c_{h_*[J,w,f]} .
	\]
	
	To see that these assignments are mutually inverse, start first with a $W$-weighted cone $d$ over $G\circ D$. Then setting $f=\id_{DJ}$ in \eqref{c_from_d} we recover exactly $d_{J,w}$.
	Conversely, starting with a weighted cone $c$ over $G$, notice that for all $[J,w,f]\in S_X$,
	\[
	c_{[J,w,f]} \;=\; c_{f_*[J,w,\id_{DJ}]} \;=\; Gf\circ c_{[J,w,\id_{DJ}]} .
	\]
	
	Naturality of this bijection is just invariance under precomposition with arrows of $\cat{D}$.
\end{proof}

The main part of the theorem is contained in the following lemma, which could be interpreted as the fact that Cauchy points have ``virtual universal retractions'' encoded in the class $i=[X,\pi,\iota]$.

\begin{lemma}\label{further_ext} 
	Let $F$ be part of a Cauchy point on $\cat{J}$, let $D:\cat{J}\funto\cat{C}$ be a diagram, and let $d$ be an $F$-weighted cone. 
	We have that $d$ is a weighted limit cone if and only if the functor $D^+:\cat{J}^{+F}\funto\cat{C}$ induced by $d$ (as in \Cref{def_wcone}) extends to a functor $\cat{J'}\funto\cat{C}$. 
	In that case, moreover, the limit is absolute.
\end{lemma}

(The category $\cat{J'}$ is defined in \Cref{defcplus}.)

\begin{proof}[Proof of \Cref{further_ext}]
	Let $(F,P,c,i)$ be a Cauchy point of $\cat{J}$ with $i=[J,\pi,\iota]$, and pick a representative $(J,\pi\in PJ,\iota\in FJ)$. 
	As usual, denote by $E$ the extra point of $\cat{C'}$. This way, we can write $\pi:J\to E$ and $\iota:E\to J$.
	Denote moreover the tip of the cone $D^+(E)$ by $T$.
	
	First, suppose that $d$ extends to a functor $D^{++}:\cat{J'}\funto\cat{C}$. Then the arrow $\pi: J\to E$ of $\cat{C'}$ is mapped functorially to an arrow $D^{++}(\pi):DJ\to T$. Moreover, again by functoriality,
	\[
	D^{++}(\pi)\circ d_{J,\iota}\;=\; D^{++}(\pi)\circ D^{++}(\iota) \;=\; D^{++}(\pi\circ\iota) \;=\; D^{++}(\id_E) \;=\; \id_T .
	\]
	Therefore $[J,\iota,D^{++}(\pi)]$ is an eventual retraction. To show that it is universal, given $K\in\cat{J}$ consider a commutative diagram of $\cat{C}$ as follows.
	\[
	\begin{tikzcd}[row sep=small]
		T \ar{dr}{f} \ar{dd}[swap]{d_{K,w}} \\
		& X \\
		DK \ar{ur}[swap]{g}
	\end{tikzcd}
	\]
	We have to prove that, as equivalence classes, $[K,w,g]=[J,\iota,f\circ D^{++}(\pi)]$. Notice now that we can see $d_{K,w}$ as an arrow in the form $D^{++}(w)$ for some ``virtual'' arrow $w:E\to K$ of $\cat{C}$. Therefore, chasing the following commutative diagram,
	\[
	\begin{tikzcd}[row sep=small]
		& T \ar{dd}{d_{K,w}} \ar{dr}{f} \\
		DJ \ar{ur}{D^{++}(\pi)} \ar{dr}[swap]{D^{++}(w\circ\pi)} && X \\
		& DK \ar{ur}[swap]{g}
	\end{tikzcd}
	\]
	\begin{align*}
	[J,\iota,f\circ D^{++}(\pi)] \;&=\; [J,\iota,g\circ d_{K,w}\circ D^{++}(\pi)] \\
	&=\; [J,\iota,g\circ D^{++}(w\circ\pi)] \\
	&=\; [K,w\circ\pi\circ\iota,g] \\
	&=\; [K,w,g] .
	\end{align*}
	This makes $[J,\iota,\pi]$ a universal retraction, and so $d$ is an absolute limit cone.
	
	Conversely, suppose that $d$ is a limit cone. In order to extend $d$ to a functor $D^{++}:\cat{J'}\funto\cat{C}$ we need to specify its action on those arrows in the form $K\to E$ for $K\in\cat{J}$, i.e.\ on the elements of the sets $PK$. 
	So let $p\in PK$. We need to define an arrow $DK\to T$ of $\cat{C}$. Since $T$ is a limit, such an arrow is uniquely specified by a cone with tip $DK$ (or equivalently, we can specify the map by saying what the composition with the arrows of $d$ should be). 
	For all $J'\in\cat{J}$ and $w\in FJ'$ take the following arrow,
	\[
	\begin{tikzcd}[sep=small]
		K \ar[virtual]{r}{p} & E \ar[virtual]{r}{w} & J'
	\end{tikzcd}
	\qquad{\color{dgray}\longmapsto}\qquad 
	\begin{tikzcd}[sep=huge]
		DK \ar{r}{D(c(p,w))} & DJ'
	\end{tikzcd}
	\]
	recalling that $c(p,w)$ gives the composition of the arrows $p$ and $w$ in $\cat{C'}$. To see that these arrows indeed give a cone, notice that for all $g:J'\to K'$, 
	\[
	Dg\circ D(c(p,w)) \;=\; D(f\circ c(p,w)) \;=\; c(p,g_*w) .
	\]
	Therefore we have a uniquely determined map $DK\to T$. Let's denote this map by $D^{++}(p)$. 
	To show that this assignment extends $D^+$ functorially, notice that 
	\begin{itemize}
		\item It respects precomposition: given $g:J'\to K$, the resulting cone is the precomposition
		\[
		D(c(g^*p,w))\;=\; D(c(p,w)\circ g) \;=\; D(c(p,w)) \circ Dg.
		\]
		Therefore, by functoriality of limits (from the uniqueness in the universal property),
		\[
		D^{++}(g^*p) \;=\; D^{++}(p)\circ Dg.
		\]
		\item It respects postcomposition (given in $\cat{C'}$ by the map $c$): given $f:E\to K'$ (or, $f\in FK'$),
		\[
		D(c(p,f)) \;=\; d_{K',f}\circ D^{++}(g^*p) \;=\; D^+(f)\circ D^{++}(g^*p)
		\]
		exactly by definition of $D^{++}(g^*p)$ via the universal property.
	\end{itemize}
	Therefore $D^{++}$ is a functor extending $D^+$.
\end{proof}

\begin{proof}[Proof of \Cref{pushforward}]
	Consider once again the Yoneda embedding $\Yon:\cat{C}\funto\sfuncat{\cat{C}}^\op$ (or the equivalent in the subcategory of small functors).
	We know the $W$-weighted limit of $\Yon\circ D$ exists, denote if by $F$, and denote by $c_{J,w}:F\Rightarrow\cat{C}(DJ,-)$ the arrows of the limit cone (which are natural transformations). 
	
	Now first suppose that the diagram is absolute. This means that it admits a universal retraction $[J,\iota,\pi]$, where $\pi$ is a natural transformation $\cat{C}(DJ)\Rightarrow F$ such that 
	\[
	F \xRightarrow{c_{J,\iota}} \cat{C}(DJ,-) \xRightarrow{\;\pi\;} F \quad=\quad \id_F .
	\]
	This says exactly that $F$ is a retract of the representable functor $\cat{C}(DJ,-)$, and so, by \Cref{cauchy_retract}, it is part of a Cauchy point.
	
	Conversely, suppose that the limit set functor $F$ is part of a Cauchy point. Let $G:\cat{C}\funto\cat{D}$ be a functor, and suppose that the $W$-weighted limit of $G\circ D:\cat{J}\funto\cat{D}$ exists. By \Cref{limit_decomp},
	\[
	\lim_{J} \big\langle WJ, GDJ \big\rangle \;=\; \lim_{X\in\cat{C}} \Big\langle \colim_{J\in\cat{J}} \big\langle WJ, \cat{C}(DJ,X) \big\rangle, GX \Big\rangle \;\cong\; \lim_{X\in\cat{C}} \big\langle FX, GX \big\rangle ,
	\]
	and by \Cref{further_ext} the latter limit is absolute. Therefore $D$ is an absolute diagram. 
\end{proof}

\begin{corollary}\label{cauchy_abs}
	A set functor $F$ is an absolute weight if and only if it is part of a Cauchy point $(F,P,c,i)$. 
	The same can be said about presheaves.
\end{corollary}

\begin{proof}[Proof of \Cref{cauchy_abs}]
	Recall (\Cref{wlim_repr}) that the $F$-weighted limit of the Yoneda embedding $\Yon:\cat{C}\funto\sfuncat{\cat{C}}^\op$ always exists and is given by $F$ itself. (Once again, we may want to work in the subcategory of small functors instead.)
	
	We can now apply \Cref{pushforward}, setting $D$ to be the identity functor. We get that $F$ is part of a Cauchy point if and only if the identity weighted by $F$ is an absolute diagram, which means exactly that $F$ is an absolute weight.
\end{proof}

\begin{corollary}\label{cauchy_limits}
	The following conditions are equivalent for a category $\cat{C}$:
	\begin{enumerate}
		\item\label{ccomp} $\cat{C}$ is Cauchy-complete;
		\item\label{cabsl} $\cat{C}$ has all absolute limits (meaning, every absolute (weighted) diagram has a limit);
		\item\label{cabsc} $\cat{C}$ has all absolute colimits.
	\end{enumerate}
\end{corollary}

\begin{proof}
	$\ref{ccomp}\Rightarrow\ref{cabsl}$:
	Suppose that $\cat{C}$ is Cauchy-complete. Consider an absolute diagram $D:\cat{J}\funto\cat{C}$ with weight $W:\cat{D}\funto\cat{Set}$.
	Then by \Cref{pushforward}, the set functor 
	\[
	F \;=\; \lim_J \big\langle WJ, \cat{C}(DJ,-) \big\rangle
	\]
	whose pointwise expression is 
	\[
	FX \;=\; \colim_J \big\langle WJ, \cat{C}(DJ,X) \big\rangle
	\]
	is part of a Cauchy point. Since $\cat{C}$ is Cauchy complete, then $F$ is representable. Let $R$ be the representing object. Using \Cref{limit_decomp} and Yoneda reduction (\Cref{ninja}), we see that the limit exists and is equal to $R$:
	\[
	\lim_J\big\langle WJ, DJ \big\rangle \;\cong\; \lim_{X\in\cat{C}} \Big\langle \colim_J \big\langle WJ, \cat{C}(DJ,X) \big\rangle, X \Big\rangle \;\cong\; \lim_{X\in\cat{C}} \big\langle \cat{C}(R,X), X \big\rangle \;\cong R .
	\]
	
	$\ref{cabsl}\Rightarrow\ref{ccomp}$: Suppose that $\cat{C}$ has all absolute limits. Then in particular all idempotents split (\Cref{e_split_eq}), and so by \Cref{comp_split}, it is Cauchy complete.
	
	The proof of $\ref{ccomp}\Leftrightarrow\ref{cabsc}$ is completely analogous and dual.
\end{proof}

So, to conclude this section, let's sum up all the equivalent conditions that we found for Cauchy points.

\begin{theorem}
	A Cauchy point of a category $\cat{C}$ is specified up to isomorphism by any of the following:
	\begin{itemize}
		\item A tuple $(F,P,c,i)$ as in \Cref{defcauchypt} (up to isomorphism);
		\item An idempotent $e:X\to X$ on $X$ (up to isomorphism as in \Cref{iso_K});
		\item A set functor or a presheaf which is a retract of a representable one;
		\item An absolute weight (for limits or for colimits).
	\end{itemize}
\end{theorem}

\begin{theorem}
	For a Cauchy point, the following conditions are equivalent:
	\begin{itemize}
		\item It is already in $\cat{C}$;
		\item Any (hence all) of the corresponding idempotents has a splitting;
		\item The corresponding set functor and/or presheaf is representable;
		\item The (absolute) limit or colimit of the identity diagram weighted by the corresponding functor or presheaf exists.
	\end{itemize}
\end{theorem}

\begin{theorem}
 	A category $\cat{C}$ is Cauchy complete if any of the following equivalent conditions hold:
 	\begin{itemize}
 		\item Every Cauchy point is already in $\cat{C}$;
 		\item Every idempotent has a splitting;
 		\item Every set functor or presheaf which is a retract of a representable one is itself representable;
 		\item All absolute limits and/or colimits exist.
 	\end{itemize}
 \end{theorem}

\section{Further topics}\label{further}

Our idea of adding virtual arrows to diagrams can also be used to study more advanced topics, and in those settings, sometimes, it recovers known constructions.
We will briefly look at two of them, profunctors and the Day convolution. For more details on them, we refer the reader to the references.

\subsection{Profunctors}\label{profunctors}

\begin{definition}
	Let $\cat{C}$ and $\cat{D}$ be categories. A \newterm{profunctor} $\Phi:\cat{C}\profunto\cat{D}$ is a bifunctor 
	\[
	\begin{tikzcd}
		\cat{D}^\op\funtimes\cat{C} \ar[functor]{r}{\overline{\Phi}} & \cat{Set} .
	\end{tikzcd}
	\]
\end{definition}

Alternative names for profunctors appearing in the literature are \newterm{bimodules}, \newterm{distributors}, \newterm{relators}, and others.

As for set functors and presheaves, given objects $C$ of $\cat{C}$ and $D$ of $\cat{D}$, the set elements of the set $\overline{\Phi}(D,C)$ are things that we can \emph{postcompose with arrows of $\cat{C}$, and precompose with arrows of $\cat{D}$, but not the other way around}. 
It is therefore natural to draw them as ``virtual arrows'' between $D$ and $C$ as follows:
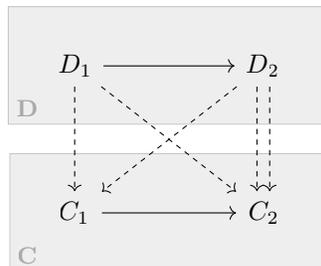
\[
\begin{tikzcd}[row sep=4em, column sep=5em,
	blend group=multiply,
	/tikz/execute at end picture={
		\node [cbox, fit=(C1) (C2), inner sep=5mm] (CC) {};
		\node [catlabel] at (CC.south west) {$\cat{C}$};
		\node [cbox, fit=(D1) (D2), inner sep=5mm] (CD) {};
		\node [catlabel] at (CD.south west) {$\cat{D}$};
	}]
	|[alias=D1]| D_1 \ar{r} & |[alias=D2]| D_2 \\
	|[alias=C1]| C_1 \ar{r} & |[alias=C2]| C_2 
	\ar[virtual, from=D1, to=C1]
	\ar[virtual, from=D1, to=C2]
	\ar[virtual, from=D2, to=C1]
	\ar[virtual, from=D2, to=C2, shift left]
	\ar[virtual, from=D2, to=C2, shift right]
\end{tikzcd}
\]
As we can see, this picture is a generalization of what we did before. Indeed, when either $\cat{C}$ or $\cat{D}$ has a single object and a single arrow, we recover the ``virtual arrows'' of functors and presheaves:
\begin{itemize}
	\item A profunctor $\cat{C}\profunto\cat{1}$ is a functor $\cat{1}^\op\funtimes\cat{C}\cong\cat{C}\funto\cat{Set}$, i.e.\ a set functor on $\cat{C}$. 
	\item A profunctor $\cat{1}\profunto\cat{D}$ is a functor $\cat{C}^\op\funtimes\cat{1}\cong\cat{C}^\op\funto\cat{Set}$, i.e.\ a presheaf on $\cat{C}$.
\end{itemize}

\begin{remark}
	Mind the possible confusion: in a profunctor $\cat{C}\profunto\cat{D}$, i.e.\ \emph{from $\cat{C}$ to $\cat{D}$}, the virtual arrows go \emph{from the objects of $\cat{D}$ to the ones of $\cat{C}$}. This convention, which seems backwards, is motivated by the special case of functors: in a set functor $\cat{C}\funto\cat{Set}$, i.e.\ \emph{from} $\cat{C}$, the virtual arrows go \emph{to} the objects of $\cat{C}$. (And for presheaves they go \emph{from} the objects of $\cat{C}$.)
\end{remark}

As we did for functors, presheaves, and Cauchy points, it is sometimes helpful to ``promote the virtual arrows to real arrows''. The resulting concept is well known: 

\begin{definition}
	The \newterm{collage} of a profunctor $\Phi:\cat{C}\profunto\cat{D}$ is the category where
	\begin{itemize}
		\item The objects are the disjoint union of those of $\cat{C}$ and those of $\cat{D}$, with their identities;
		\item Between two objects of $\cat{C}$, the arrows are those of $\cat{C}$, with their composition;
		\item Between two objects of $\cat{D}$, the arrows are those of $\cat{D}$, with their composition;
		\item Given objects $C$ of $\cat{C}$ and $D$ of $\cat{D}$, the arrows $D\to C$ are the elements of the set $\overline{\Phi}(D,C)$, with precomposition and postcomposition by arrows of $\cat{C}$ and $\cat{D}$ given by functoriality;
		\item There are no arrows in the form $C\to D$.
	\end{itemize}
\end{definition}
Similarly to what happened the categories $\cat{C}^{+F}$, $\cat{C}_{+P}$ and $\cat{C'}$ of \Cref{CplusF}, \Cref{CplusP} and \Cref{defcplus}, the collage of a profunctor is indeed a category thanks to the functoriality of $\Phi$. In particular, associativity of arrows in the form 
\[
\begin{tikzcd}
	C_1 \ar{r}{c} & C_2 \ar[virtual]{r}{\phi} & D_1 \ar{r}{d} & D_2
\end{tikzcd}
\]
is guaranteed by \emph{bi}functoriality, in the sense of \Cref{def_bifunctor} and \eqref{interchange}.
	
The ``virtual arrows'' connecting $\cat{C}$ and $\cat{D}$ are sometimes called \newterm{heteromorphisms}, to distinguish them from the morphisms (or \emph{homomorphisms}) within $\cat{C}$ and within $\cat{D}$. (Compare for example with the words \emph{homogeneous} and \emph{heterogeneous}.)

Consider now two profunctors as follows. 
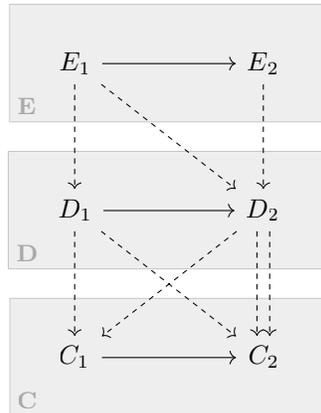
\[
\begin{tikzcd}[row sep=4em, column sep=5em,
	blend group=multiply,
	/tikz/execute at end picture={
		\node [cbox, fit=(C1) (C2), inner sep=5mm] (CC) {};
		\node [catlabel] at (CC.south west) {$\cat{C}$};
		\node [cbox, fit=(D1) (D2), inner sep=5mm] (CD) {};
		\node [catlabel] at (CD.south west) {$\cat{D}$};
		\node [cbox, fit=(E1) (E2), inner sep=5mm] (CE) {};
		\node [catlabel] at (CE.south west) {$\cat{E}$};
	}]
	|[alias=E1]| E_1 \ar{r} & |[alias=E2]| E_2 \\
	|[alias=D1]| D_1 \ar{r} & |[alias=D2]| D_2 \\
	|[alias=C1]| C_1 \ar{r} & |[alias=C2]| C_2 
	\ar[virtual, from=D1, to=C1]
	\ar[virtual, from=D1, to=C2]
	\ar[virtual, from=D2, to=C1]
	\ar[virtual, from=D2, to=C2, shift left]
	\ar[virtual, from=D2, to=C2, shift right]
	\ar[virtual, from=E1, to=D1]
	\ar[virtual, from=E1, to=D2]
	\ar[virtual, from=E2, to=D2]
\end{tikzcd}
\]
We can compose them by \emph{counting the paths by means of pairings}, as in \Cref{pairing}:

\begin{definition}
	Let $\Phi:\cat{C}\profunto\cat{D}$ and $\Psi:\cat{D}\profunto\cat{E}$ be profunctors.
	The \emph{composite profunctor} $\Psi\boldsymbol{\circ}\Phi:\cat{C}\profunto\cat{E}$ is given on object, up to isomorphism, by
	\[
	\overline{\Psi\boldsymbol{\circ}\Phi} (E,C) \;\coloneqq\; \langle  \overline{\Phi}(-,C), \overline{\Psi}(E,-) \rangle \;=\; \Coend{D\in\cat{D}} \overline{\Phi}(D,C)\times \overline{\Psi}(E,D) ,
	\]
	and induced on morphisms by the universal property of weighted colimits.
\end{definition}
In other words, we are taking as virtual arrows $E\to C$ all possible paths via $\cat{D}$, 
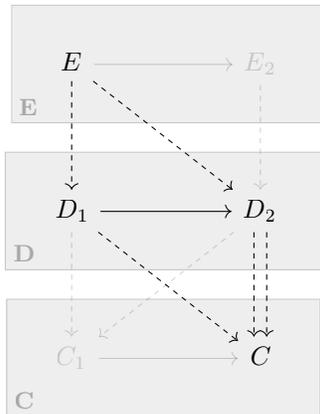
\[
\begin{tikzcd}[row sep=4em, column sep=5em,
	blend group=multiply,
	/tikz/execute at end picture={
		\node [cbox, fit=(C1) (C2), inner sep=5mm] (CC) {};
		\node [catlabel] at (CC.south west) {$\cat{C}$};
		\node [cbox, fit=(D1) (D2), inner sep=5mm] (CD) {};
		\node [catlabel] at (CD.south west) {$\cat{D}$};
		\node [cbox, fit=(E1) (E2), inner sep=5mm] (CE) {};
		\node [catlabel] at (CE.south west) {$\cat{E}$};
	}]
	|[alias=E1]| E \ar[mgray]{r} & |[alias=E2]| \color{mgray} E_2 \\
	|[alias=D1]| D_1 \ar{r} & |[alias=D2]| D_2 \\
	|[alias=C1]| \color{mgray} C_1 \ar[mgray]{r} & |[alias=C2]| C 
	\ar[virtual, from=D1, to=C1,mgray]
	\ar[virtual, from=D1, to=C2]
	\ar[virtual, from=D2, to=C1,mgray]
	\ar[virtual, from=D2, to=C2, shift left]
	\ar[virtual, from=D2, to=C2, shift right]
	\ar[virtual, from=E1, to=D1]
	\ar[virtual, from=E1, to=D2]
	\ar[virtual, from=E2, to=D2,mgray]
\end{tikzcd}
\]
as usual, counting only once those paths that only differ by associativity.

With this notion of composition, profunctors form \emph{almost} a category: it is a \emph{bi}category, two-dimensional and weak. (Indeed, the pairing only specifies the composite profunctor up to natural isomorphism.) It can be considered a higher analogue of the category of relations, between categories instead of sets.
The identities of this bicategory are given by the hom-functors $\cat{C}(-,-):\cat{C}^\op\funtimes\cat{C}\funto\cat{Set}$, which we can see as profunctors. 
Indeed, by Yoneda reduction (\Cref{ninja}),
\[
\Coend{Y} \cat{C}(Y,X) \times \overline{\Phi}(A,Y) \;\cong\; \overline{\Phi}(A,X) .
\]
and a condition on the other side holds analogously.

\begin{remark}
	One can define adjunctions in a bicategory similarly to how one does in $\cat{Cat}$. This way, a Cauchy point $(F,P,c,i)$ on a category $\cat{C}$ (\Cref{defcauchypt}) is exactly a pair of adjoint profunctors as follows.
	
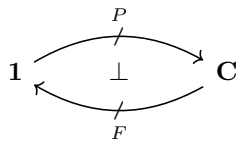
\[
	\begin{tikzcd}[blend mode=multiply]
		\cat{1} \ar[profunctor,bend left]{rr}[inner sep=2mm]{P} & \perp & \cat{C} \ar[profunctor,bend left]{ll}[inner sep=2mm]{F}
	\end{tikzcd}
	\]
	Recall indeed that profunctors $\cat{C}\funto\cat{1}$ are set functors, and profunctors $\cat{1}\funto\cat{C}$ are presheaves.
	Moreover the maps
	\[
	c_{A,B} : PA \times FB \to \cat{C}(A,B)
	\]
	and 
	\[
	i : 1\to \Coend{X} PX\times FX
	\]
	can be seen as the unit $i:\id_\cat{1}\Rightarrow F\boldsymbol{\circ} P$ and counit $c:P\boldsymbol{\circ} F\Rightarrow\id_\cat{C}$ of an adjunction in this profunctor sense. (The conditions \eqref{id_cond} can be seen as the triangle identities.)
\end{remark}

\subsection{String diagrams and Day convolution}

In a monoidal category $(\cat{C},\otimes,I)$ we have, besides the usual notion of composition of morphisms, a notion of \emph{parallel composition}, or \emph{tensor product}. 
Because of that, it is often useful to move from ordinary diagrams to \emph{string diagrams}, where objects are ``wires'' and morphisms are ``boxes''. 
In this document we will orient string diagrams from left to right. 
A morphism $f:X\to Y$ is represented as follows,
\[
\begin{tikzpicture}[scale=0.6]
	\begin{pgfonlayer}{nodelayer}
		\node [style=none] (0) at (-1.5, 0) {};
		\node [style=none] (1) at (1.5, 0) {};
		\node [style=none] (2) at (-2, 0) {$X$};
		\node [style=none] (3) at (2, 0) {$Y$};
		\node [style=medium box] (4) at (0, 0) {$f$};
	\end{pgfonlayer}
	\begin{pgfonlayer}{edgelayer}
		\draw (0.center) to (1.center);
	\end{pgfonlayer}
\end{tikzpicture}
\]
the composition of two morphisms is represented as follows,
\[
\begin{tikzpicture}[scale=0.6]
	\begin{pgfonlayer}{nodelayer}
		\node [style=none] (0) at (-3, 0.05) {};
		\node [style=none] (1) at (3, 0.05) {};
		\node [style=none] (3) at (-3.5, 0.05) {$X$};
		\node [style=none] (4) at (0, 0.55) {$Y$};
		\node [style=medium box] (5) at (-1.5, 0.05) {$f$};
		\node [style=none] (6) at (3.5, 0.05) {$Z$};
		\node [style=medium box] (7) at (1.5, 0.05) {$\vphantom{f}g$};
	\end{pgfonlayer}
	\begin{pgfonlayer}{edgelayer}
		\draw [style=none] (0.center) to (1.center);
	\end{pgfonlayer}
\end{tikzpicture}
\]
and the tensor product as follows.
\[
\begin{tikzpicture}[scale=0.6]
	\begin{pgfonlayer}{nodelayer}
		\node [style=none] (7) at (-2, 0.875) {};
		\node [style=none] (10) at (2, 0.875) {};
		\node [style=none] (11) at (-2.5, 0.875) {$X$};
		\node [style=none] (12) at (2.5, 0.875) {$Y$};
		\node [style=medium box] (13) at (0, 0.875) {$f$};
		\node [style=none] (14) at (-2, -0.875) {};
		\node [style=none] (15) at (2, -0.875) {};
		\node [style=none] (16) at (-2.5, -0.875) {$A$};
		\node [style=none] (17) at (2.5, -0.875) {$B$};
		\node [style=medium box] (18) at (0, -0.875) {$h$};
	\end{pgfonlayer}
	\begin{pgfonlayer}{edgelayer}
		\draw (7.center) to (10.center);
		\draw (14.center) to (15.center);
	\end{pgfonlayer}
\end{tikzpicture}
\]

Now given a functor $F:\cat{C}\funto\cat{Set}$, we can represent the elements $f\in FX$, i.e.\ the ``virtual arrows to $X$'', as dashed morphisms from an ``extra wire'' to $X$:
\[
\begin{tikzpicture}[scale=0.6,baseline=-0.25em]
	\begin{pgfonlayer}{nodelayer}
		\node [style=none] (1) at (2, 0) {};
		\node [style=none] (2) at (2.5, 0) {$X$};
		\node [style=none] (3) at (-2, 0) {};
		\node [style=ds medium box] (4) at (0.25, 0) {$f$};
	\end{pgfonlayer}
	\begin{pgfonlayer}{edgelayer}
		\draw [style=ds] (3.center) to (4);
		\draw (4) to (1.center);
	\end{pgfonlayer}
\end{tikzpicture}
\qquad\mbox{or}\qquad
\begin{tikzpicture}[scale=0.6,baseline=-0.25em]
	\begin{pgfonlayer}{nodelayer}
		\node [style=ds horiz state] (0) at (0, 0) {$f$};
		\node [style=none] (1) at (2, 0) {};
		\node [style=none] (2) at (2.5, 0) {$X$};
		\node [style=none] (3) at (-2, 0) {};
	\end{pgfonlayer}
	\begin{pgfonlayer}{edgelayer}
		\draw (0) to (1.center);
		\draw [style=ds] (3.center) to (0);
	\end{pgfonlayer}
\end{tikzpicture}
\]
We could label the extra wire, in case we have several functors, and we can also not display the extra wire,
\[
\begin{tikzpicture}[scale=0.6,baseline=-0.25em]
	\begin{pgfonlayer}{nodelayer}
		\node [style=none] (1) at (2, 0) {};
		\node [style=none] (2) at (2.5, 0) {$X$};
		\node [style=ds medium box] (4) at (0, 0) {$f$};
	\end{pgfonlayer}
	\begin{pgfonlayer}{edgelayer}
		\draw (4) to (1.center);
	\end{pgfonlayer}
\end{tikzpicture}
\qquad\mbox{or}\qquad
\begin{tikzpicture}[scale=0.6,baseline=-0.25em]
	\begin{pgfonlayer}{nodelayer}
		\node [style=ds horiz state] (0) at (0, 0) {$f$};
		\node [style=none] (1) at (2, 0) {};
		\node [style=none] (2) at (2.5, 0) {$X$};
	\end{pgfonlayer}
	\begin{pgfonlayer}{edgelayer}
		\draw (0) to (1.center);
	\end{pgfonlayer}
\end{tikzpicture}
\]
if one keeps in mind not to confuse these with morphisms from the monoidal unit, which are usually represented similarly (but in solid lines).
As usual, the functorial action of $F$ on morphism is a ``virtual postcomposition'':
\[
\begin{tikzpicture}[scale=0.6]
	\begin{pgfonlayer}{nodelayer}
		\node [style=none] (1) at (4, 0) {};
		\node [style=none] (2) at (4.5, 0) {$X$};
		\node [style=ds horiz state] (4) at (0, 0) {$f$};
		\node [style=medium box] (5) at (2.25, 0) {$g$};
	\end{pgfonlayer}
	\begin{pgfonlayer}{edgelayer}
		\draw (4) to (1.center);
	\end{pgfonlayer}
\end{tikzpicture}
\]

Similarly, and dually, given a presheaf $P:\cat{C}^\op\funto\cat{Set}$ we can represent the elements $p\in PX$, i.e.\ the ``virtual arrows from $X$'' as dashed morphisms from $X$ to an ``extra wire'',
\[
\begin{tikzpicture}[scale=0.6]
	\begin{pgfonlayer}{nodelayer}
		\node [style=none] (0) at (-2.5, 0) {$X$};
		\node [style=none] (1) at (-2, 0) {};
		\node [style=ds horiz effect] (2) at (-0.25, 0) {$\vphantom{f}p$};
		\node [style=none] (3) at (2, 0) {};
	\end{pgfonlayer}
	\begin{pgfonlayer}{edgelayer}
		\draw (1.center) to (2);
		\draw [style=ds] (2) to (3.center);
	\end{pgfonlayer}
\end{tikzpicture}
\]
or once again, omitting the wire:
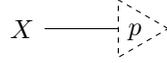
\[
\begin{tikzpicture}[scale=0.6]
	\begin{pgfonlayer}{nodelayer}
		\node [style=none] (0) at (-1.5, 0) {$X$};
		\node [style=none] (1) at (-1, 0) {};
		\node [style=ds horiz effect] (2) at (1, 0) {$\vphantom{f}p$};
	\end{pgfonlayer}
	\begin{pgfonlayer}{edgelayer}
		\draw (1.center) to (2);
	\end{pgfonlayer}
\end{tikzpicture}
\]
The functorial action is ``virtual precomposition'':
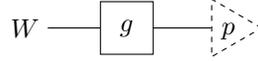
\[
\begin{tikzpicture}[scale=0.6]
	\begin{pgfonlayer}{nodelayer}
		\node [style=none] (0) at (-2.5, 0) {$W$};
		\node [style=none] (1) at (-2, 0) {};
		\node [style=ds horiz effect] (2) at (2, 0) {$\vphantom{f}p$};
		\node [style=medium box] (3) at (-0.25, 0) {$g$};
	\end{pgfonlayer}
	\begin{pgfonlayer}{edgelayer}
		\draw (1.center) to (2);
	\end{pgfonlayer}
\end{tikzpicture}
\]

\paragraph{Cauchy completion.}
Recall from \Cref{cauchy} that Cauchy points can be seen as particular ``extra objects'' equipped with ``virtual maps'' $\iota$ and $\pi$ forming a ``virtual retraction''. We can draw them for example as follows,
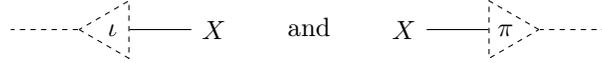
\[
\begin{tikzpicture}[scale=0.6,baseline=-0.25em]
	\begin{pgfonlayer}{nodelayer}
		\node [style=ds horiz state] (1) at (0.25, 0) {$\vphantom{f}\iota$};
		\node [style=none] (6) at (2.5, 0) {$X$};
		\node [style=none] (7) at (-2, 0) {};
		\node [style=none] (8) at (2, 0) {};
	\end{pgfonlayer}
	\begin{pgfonlayer}{edgelayer}
		\draw [style=ds] (1) to (7.center);
		\draw (8.center) to (1);
	\end{pgfonlayer}
\end{tikzpicture}
\qquad\mbox{and}\qquad
\begin{tikzpicture}[scale=0.6,baseline=-0.25em]
	\begin{pgfonlayer}{nodelayer}
		\node [style=none] (0) at (-2, 0) {};
		\node [style=none] (1) at (2, 0) {};
		\node [style=ds horiz effect] (3) at (-0.25, 0) {$\vphantom{f}\pi$};
		\node [style=none] (4) at (-2.5, 0) {$X$};
	\end{pgfonlayer}
	\begin{pgfonlayer}{edgelayer}
		\draw (0.center) to (3);
		\draw [style=ds] (1.center) to (3);
	\end{pgfonlayer}
\end{tikzpicture}
\]
with the condition that
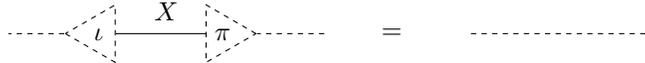
\[
\begin{tikzpicture}[scale=0.6]
	\begin{pgfonlayer}{nodelayer}
		\node [style=ds horiz effect] (0) at (-2.25, 0) {$\vphantom{f}\pi$};
		\node [style=ds horiz state] (1) at (-5, 0) {$\vphantom{f}\iota$};
		\node [style=none] (3) at (0, 0) {};
		\node [style=none] (4) at (0, 0) {};
		\node [style=none] (6) at (-3.5, 0.5) {$X$};
		\node [style=none] (7) at (-7, 0) {};
		\node [style=none] (8) at (1.5, 0) {$=$};
		\node [style=none] (9) at (7.25, 0) {};
		\node [style=none] (10) at (3.25, 0) {};
	\end{pgfonlayer}
	\begin{pgfonlayer}{edgelayer}
		\draw (1) to (0);
		\draw [style=ds] (4.center) to (0);
		\draw [style=ds] (1) to (7.center);
		\draw [style=ds] (10.center) to (9.center);
	\end{pgfonlayer}
\end{tikzpicture}
\]
so that the resulting morphism
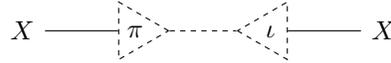
\[
\begin{tikzpicture}[scale=0.6]
	\begin{pgfonlayer}{nodelayer}
		\node [style=ds horiz effect] (0) at (-1.5, 0) {$\vphantom{f}\pi$};
		\node [style=ds horiz state] (1) at (1.5, 0) {$\vphantom{f}\iota$};
		\node [style=none] (2) at (-3.5, 0) {};
		\node [style=none] (3) at (3.5, 0) {};
		\node [style=none] (4) at (3.5, 0) {};
		\node [style=none] (5) at (-4, 0) {$X$};
		\node [style=none] (6) at (4, 0) {$X$};
	\end{pgfonlayer}
	\begin{pgfonlayer}{edgelayer}
		\draw (2.center) to (0);
		\draw (4.center) to (1);
		\draw [style=ds] (0) to (1);
	\end{pgfonlayer}
\end{tikzpicture}
\]
is idempotent.
(Again we can possibly label the dashed wire in case we are considering more than one point.) 

Diagrams of this kind (with slightly different notation) are already widely used in categorical quantum mechanics, where idempotents play a major role, for example, in modeling measurements and decoherence. See for example \cite[Chapters 6 and 7]{dodo} and references therein.

\paragraph{Yoneda embedding and Day convolution.}

Just as we did with non-string diagrams, we can visualize morphisms between presheaves (natural transformations) as virtual morphisms between (labeled) virtual objects:
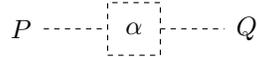
\[
\begin{tikzpicture}[scale=0.6]
	\begin{pgfonlayer}{nodelayer}
		\node [style=none] (0) at (-2, 0) {};
		\node [style=none] (1) at (2, 0) {};
		\node [style=ds medium box] (2) at (0, 0) {$\alpha$};
		\node [style=none] (3) at (-2.5, 0) {$P$};
		\node [style=none] (4) at (2.5, 0) {$Q$};
	\end{pgfonlayer}
	\begin{pgfonlayer}{edgelayer}
		\draw [style=ds] (0.center) to (1.center);
	\end{pgfonlayer}
\end{tikzpicture}
\]
Similarly, morphisms of functors $\alpha:F\Rightarrow G$ are \emph{opposite} morphisms, once again ($\cat{C}$ is embedded in $\sfuncat{\cat{C}}^\op$).
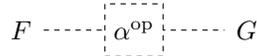
\[
\begin{tikzpicture}[scale=0.6]
	\begin{pgfonlayer}{nodelayer}
		\node [style=none] (0) at (-2, 0) {};
		\node [style=none] (1) at (2, 0) {};
		\node [style=ds medium box] (2) at (0, 0) {$\alpha^\op$};
		\node [style=none] (3) at (-2.5, 0) {$F$};
		\node [style=none] (4) at (2.5, 0) {$G$};
	\end{pgfonlayer}
	\begin{pgfonlayer}{edgelayer}
		\draw [style=ds] (0.center) to (1.center);
	\end{pgfonlayer}
\end{tikzpicture}
\]
This way we can represent via string diagrams the entire categories $\sfuncat{\cat{C}}^\op$ and $\sfuncat{\cat{C}^\op}$, as extensions of the category $\cat{C}$. 

At this point, a question might arise: \emph{Do we have the parallel composition for virtual arrows?}
The answer is affirmative, and the canonical way to do so is via \emph{Day convolution}. (For this to work, we need small functors or presheaves.)

\begin{definition}\label{def_day}
	Let $(\cat{C},\otimes,I)$ be a monoidal category, and let $F,G:\cat{C}\funto\cat{Set}$ be (small) set functors. 
	The \newterm{Day convolution} of $F$ and $G$ is the functor $F\otimes G$ given on objects up to isomorphism by the following pairing:
	\[
	(F\otimes G) (Z) \;\coloneqq\; \langle \cat{C}(-\otimes -, Z), {F-}\times{G-} \rangle \;\cong\; \Coend{X,Y\in\cat{C}} \cat{C}(X\otimes Y, Z)\times FX\times GY
	\]
	On morphisms it is induced by the universal property.
	
	Dually, given (small) presheaves $P$ and $Q$, the Day convolution $P\otimes Q$ is the presheaf defined on objects by the following pairing:
	\[
	(P\otimes Q) (X) \;\coloneqq\; \langle {P-}\times{Q-}, \cat{C}(Z,-\otimes-) \rangle \;\cong\; \Coend{Y,Z\in\cat{C}} PY\times QZ \times \cat{C}(X,Y\otimes Z)
	\]
\end{definition}

Let's now interpret this. As usual, let's focus on set functors. 
The idea is that we want to define a product $F\otimes G$. In string diagrams, it should look like the parallel composition of the following virtual objects:
\[
\begin{tikzpicture}[scale=0.6]
	\begin{pgfonlayer}{nodelayer}
		\node [style=none] (0) at (-1.5, 1) {};
		\node [style=none] (1) at (1.5, 1) {};
		\node [style=none] (2) at (-1.5, -1) {};
		\node [style=none] (3) at (1.5, -1) {};
		\node [style=none] (4) at (2, 1) {$F$};
		\node [style=none] (5) at (2, -1) {$G$};
		\node [style=none] (6) at (-2, 1) {$F$};
		\node [style=none] (7) at (-2, -1) {$G$};
	\end{pgfonlayer}
	\begin{pgfonlayer}{edgelayer}
		\draw [style=ds] (0.center) to (1.center);
		\draw [style=ds] (3.center) to (2.center);
	\end{pgfonlayer}
\end{tikzpicture}
\]
Now $F\otimes G$ is a \emph{set functor}. What are its elements? Given an object $Z$, what are the elements of the set $(F\otimes G)(X)$? 
Or, equivalently (by the Yoneda lemma), what are the natural transformations $\Yon(Z)\Rightarrow F\otimes G$, i.e.\ the morphisms $F\otimes G\to\Yon(Z)$ of $\sfuncat{\cat{C}}^\op$?
\[
\begin{tikzpicture}[scale=0.6]
	\begin{pgfonlayer}{nodelayer}
		\node [style=none] (0) at (-1.5, 1) {};
		\node [style=none] (1) at (0, 1) {};
		\node [style=none] (2) at (-1.5, -1) {};
		\node [style=none] (3) at (0, -1) {};
		\node [style=none] (6) at (-2, 1) {$F$};
		\node [style=none] (7) at (-2, -1) {$G$};
		\node [style=none] (8) at (2, 0) {$Z$};
		\node [style=none] (9) at (0, 0) {};
		\node [style=none] (10) at (1.5, 0) {};
		\node [style=ds medium box,minimum height=16mm] (11) at (0, 0) {$?$};
	\end{pgfonlayer}
	\begin{pgfonlayer}{edgelayer}
		\draw [style=ds] (0.center) to (1.center);
		\draw [style=ds] (3.center) to (2.center);
		\draw (10.center) to (9.center);
	\end{pgfonlayer}
\end{tikzpicture}
\]
The pairing in \Cref{def_day} has, as elements, equivalence classes of tuples $(X,Y,h,f,g)$ with $h:X\otimes Y\to Z$, $f\in FX$ and $g\in GY$,
\[
\begin{tikzpicture}[scale=0.6,baseline=-0.25em]
	\begin{pgfonlayer}{nodelayer}
		\node [style=ds horiz state] (0) at (0.25, 0) {$f$};
		\node [style=none] (1) at (-1.25, 0) {};
		\node [style=none] (2) at (1.5, 0) {};
		\node [style=none] (3) at (2, 0) {$X$};
		\node [style=none] (4) at (-1.75, 0) {$F$};
	\end{pgfonlayer}
	\begin{pgfonlayer}{edgelayer}
		\draw [style=ds] (1.center) to (0);
		\draw (0) to (2.center);
	\end{pgfonlayer}
\end{tikzpicture}
\qquad
\begin{tikzpicture}[scale=0.6,baseline=-0.25em]
	\begin{pgfonlayer}{nodelayer}
		\node [style=ds horiz state] (0) at (0.25, 0) {$\vphantom{f}g$};
		\node [style=none] (1) at (-1.25, 0) {};
		\node [style=none] (2) at (1.5, 0) {};
		\node [style=none] (3) at (2, 0) {$Y$};
		\node [style=none] (4) at (-1.75, 0) {$G$};
	\end{pgfonlayer}
	\begin{pgfonlayer}{edgelayer}
		\draw [style=ds] (1.center) to (0);
		\draw (0) to (2.center);
	\end{pgfonlayer}
\end{tikzpicture}
\qquad
\begin{tikzpicture}[scale=0.6,baseline=-0.25em]
	\begin{pgfonlayer}{nodelayer}
		\node [style=none] (1) at (0, 1) {};
		\node [style=none] (3) at (0, -1) {};
		\node [style=none] (8) at (2, 0) {$Z$};
		\node [style=none] (9) at (0, 0) {};
		\node [style=none] (10) at (1.5, 0) {};
		\node [style=medium box,minimum height=16mm] (11) at (0, 0) {$h$};
		\node [style=none] (12) at (-1.5, 1) {};
		\node [style=none] (13) at (-1.5, -1) {};
		\node [style=none] (14) at (-2, 1) {$X$};
		\node [style=none] (15) at (-2, -1) {$Y$};
	\end{pgfonlayer}
	\begin{pgfonlayer}{edgelayer}
		\draw (10.center) to (9.center);
		\draw (12.center) to (1.center);
		\draw (3.center) to (13.center);
	\end{pgfonlayer}
\end{tikzpicture}
\]
which we can imagine arranged as follows:
\[
\begin{tikzpicture}[scale=0.6]
	\begin{pgfonlayer}{nodelayer}
		\node [style=ds horiz state] (0) at (-1, 1) {$f$};
		\node [style=none] (1) at (-2.5, 1) {};
		\node [style=none] (2) at (1, 1) {};
		\node [style=none] (4) at (-3, 1) {$F$};
		\node [style=medium box,minimum height=16mm] (5) at (1, 0) {$h$};
		\node [style=ds horiz state] (6) at (-1, -1) {$\vphantom{f}g$};
		\node [style=none] (7) at (-2.5, -1) {};
		\node [style=none] (8) at (1, -1) {};
		\node [style=none] (9) at (-3, -1) {$G$};
		\node [style=none] (10) at (2.5, 0) {};
		\node [style=none] (11) at (3, 0) {$Z$};
	\end{pgfonlayer}
	\begin{pgfonlayer}{edgelayer}
		\draw [style=ds] (1.center) to (0);
		\draw (0) to (2.center);
		\draw [style=ds] (7.center) to (6);
		\draw (6) to (8.center);
		\draw (10.center) to (5);
	\end{pgfonlayer}
\end{tikzpicture}
\]
Moreover, the equivalence relation identifies tuples that differ by any morphisms $a:X\to X'$ and $b:Y\to Y'$ as follows:
\[
\left(
\begin{tikzpicture}[scale=0.6,baseline=-0.25em]
	\begin{pgfonlayer}{nodelayer}
		\node [style=ds horiz state] (0) at (0.25, 0) {$f$};
		\node [style=none] (1) at (-1.25, 0) {};
		\node [style=none] (2) at (1.5, 0) {};
		\node [style=none] (3) at (2, 0) {$X$};
		\node [style=none] (4) at (-1.75, 0) {$F$};
	\end{pgfonlayer}
	\begin{pgfonlayer}{edgelayer}
		\draw [style=ds] (1.center) to (0);
		\draw (0) to (2.center);
	\end{pgfonlayer}
\end{tikzpicture}
\qquad\qquad
\begin{tikzpicture}[scale=0.6,baseline=-0.25em]
	\begin{pgfonlayer}{nodelayer}
		\node [style=ds horiz state] (0) at (0.25, 0) {$\vphantom{f}g$};
		\node [style=none] (1) at (-1.25, 0) {};
		\node [style=none] (2) at (1.5, 0) {};
		\node [style=none] (3) at (2, 0) {$Y$};
		\node [style=none] (4) at (-1.75, 0) {$G$};
	\end{pgfonlayer}
	\begin{pgfonlayer}{edgelayer}
		\draw [style=ds] (1.center) to (0);
		\draw (0) to (2.center);
	\end{pgfonlayer}
\end{tikzpicture}
\qquad\qquad
\begin{tikzpicture}[scale=0.6,baseline=-0.25em]
	\begin{pgfonlayer}{nodelayer}
		\node [style=none] (2) at (1, 1) {};
		\node [style=medium box,minimum height=16mm] (4) at (1.25, 0) {$h$};
		\node [style=none] (9) at (2.75, 0) {};
		\node [style=none] (10) at (3.25, 0) {$Z$};
		\node [style=none] (11) at (-2.5, 1) {};
		\node [style=none] (12) at (-2.5, -1) {};
		\node [style=none] (13) at (1, -1) {};
		\node [style=medium box] (14) at (-1, 1) {$a$};
		\node [style=medium box] (15) at (-1, -1) {$b$};
		\node [style=none] (16) at (-3, 1) {$X$};
		\node [style=none] (17) at (-3, -1) {$Y$};
	\end{pgfonlayer}
	\begin{pgfonlayer}{edgelayer}
		\draw (9.center) to (4);
		\draw (12.center) to (13.center);
		\draw (11.center) to (2.center);
	\end{pgfonlayer}
\end{tikzpicture}
\right)
\]
\[
\sim
\]
\[
\left(
\begin{tikzpicture}[scale=0.6,baseline=-0.25em]
	\begin{pgfonlayer}{nodelayer}
		\node [style=ds horiz state] (0) at (-1, 0) {$f$};
		\node [style=none] (1) at (-2.5, 0) {};
		\node [style=none] (2) at (2.5, 0) {};
		\node [style=none] (4) at (-3, 0) {$F$};
		\node [style=medium box] (12) at (1, 0) {$a$};
		\node [style=none] (13) at (3, 0) {$X'$};
	\end{pgfonlayer}
	\begin{pgfonlayer}{edgelayer}
		\draw [style=ds] (1.center) to (0);
		\draw (0) to (2.center);
	\end{pgfonlayer}
\end{tikzpicture}
\qquad
\begin{tikzpicture}[scale=0.6,baseline=-0.25em]
	\begin{pgfonlayer}{nodelayer}
		\node [style=ds horiz state] (0) at (-1, 0) {$\vphantom{f}g$};
		\node [style=none] (1) at (-2.5, 0) {};
		\node [style=none] (2) at (2.5, 0) {};
		\node [style=none] (4) at (-3, 0) {$G$};
		\node [style=medium box] (12) at (1, 0) {$b$};
		\node [style=none] (13) at (3, 0) {$Y'$};
	\end{pgfonlayer}
	\begin{pgfonlayer}{edgelayer}
		\draw [style=ds] (1.center) to (0);
		\draw (0) to (2.center);
	\end{pgfonlayer}
\end{tikzpicture}
\qquad
\begin{tikzpicture}[scale=0.6,baseline=-0.25em]
	\begin{pgfonlayer}{nodelayer}
		\node [style=none] (1) at (0, 1) {};
		\node [style=none] (3) at (0, -1) {};
		\node [style=none] (8) at (2, 0) {$Z$};
		\node [style=none] (9) at (0, 0) {};
		\node [style=none] (10) at (1.5, 0) {};
		\node [style=medium box,minimum height=16mm] (11) at (0, 0) {$h$};
		\node [style=none] (12) at (-1.5, 1) {};
		\node [style=none] (13) at (-1.5, -1) {};
		\node [style=none] (14) at (-2, 1) {$X'$};
		\node [style=none] (15) at (-2, -1) {$Y'$};
	\end{pgfonlayer}
	\begin{pgfonlayer}{edgelayer}
		\draw (10.center) to (9.center);
		\draw (12.center) to (1.center);
		\draw (3.center) to (13.center);
	\end{pgfonlayer}
\end{tikzpicture}
\right)
\]
Indeed, by the Yoneda lemma, we can view $f$ and $g$ as morphisms of $\sfuncat{\cat{C}}^\op_{\rm small}$, and both sides above give rise to the following composition, which is now well defined.
\[
\begin{tikzpicture}[scale=0.6]
	\begin{pgfonlayer}{nodelayer}
		\node [style=none] (2) at (1.75, 1) {};
		\node [style=medium box,minimum height=16mm] (4) at (1.75, 0) {$h$};
		\node [style=none] (9) at (3.25, 0) {};
		\node [style=none] (10) at (3.75, 0) {$Z$};
		\node [style=medium box] (14) at (-0.5, 1) {$a$};
		\node [style=ds horiz state] (18) at (-2.5, 1) {$f$};
		\node [style=none] (19) at (-4, 1) {};
		\node [style=none] (21) at (-4.5, 1) {$F$};
		\node [style=none] (22) at (1.75, -1) {};
		\node [style=medium box] (23) at (-0.5, -1) {$b$};
		\node [style=ds horiz state] (24) at (-2.5, -1) {$\vphantom{f}g$};
		\node [style=none] (25) at (-4, -1) {};
		\node [style=none] (26) at (-4.5, -1) {$G$};
	\end{pgfonlayer}
	\begin{pgfonlayer}{edgelayer}
		\draw (9.center) to (4);
		\draw [style=ds] (19.center) to (18);
		\draw (2.center) to (18);
		\draw [style=ds] (25.center) to (24);
		\draw (22.center) to (24);
	\end{pgfonlayer}
\end{tikzpicture}
\]
In particular, in line with everything we did so far, the only elements of $(F\otimes G)(Z)$ are the ones which one can obtain purely via the elements of $F$, the elements of $G$, and the morphisms of $\cat{C}$.

With these definitions, one can prove that the categories $\sfuncat{\cat{C}}^\op_{\rm small}$ and $\sfuncat{\cat{C}}^\op_{\rm small}$ are monoidal categories, with the unit given by $\Yon(I)$, where $I$ is the unit of $\cat{C}$. 
Moreover, the Yoneda embedding is not just fully faithful, but also strong monoidal:

\begin{proposition}\label{day_strong}
	Let $(\cat{C},\otimes,I)$ be a monoidal category, and assume $\sfuncat{\cat{C}}^\op_{\rm small}$ equipped with the Day convolution tensor product.
	The Yoneda embedding $\cat{C}\ffunto\sfuncat{\cat{C}}^\op_{\rm small}$ is a strong monoidal functor.
	The same thing can be said for presheaves.
\end{proposition}

Because of this, the parallel composition of virtual arrows extends faithfully the one of ordinary arrows. 
Note moreover that, since $\cat{C}$ is embedded, we can also form diagrams, for example, as follows
\[
\begin{tikzpicture}[scale=0.6,baseline=-0.25em]
	\begin{pgfonlayer}{nodelayer}
		\node [style=none] (0) at (-2, 1) {};
		\node [style=none] (1) at (2, 1) {};
		\node [style=none] (2) at (-2.5, 1) {$X$};
		\node [style=none] (3) at (2.5, 1) {$Y$};
		\node [style=medium box] (4) at (0, 1) {$g$};
		\node [style=none] (5) at (-2, -1) {};
		\node [style=none] (6) at (-2.5, -1) {$Y$};
		\node [style=none] (7) at (2.5, -1) {$P$};
		\node [style=ds horiz effect] (8) at (0, -1) {$\vphantom{f}p$};
		\node [style=none] (9) at (2, -1) {};
	\end{pgfonlayer}
	\begin{pgfonlayer}{edgelayer}
		\draw (0.center) to (1.center);
		\draw (5.center) to (8);
		\draw [style=ds] (8) to (9.center);
	\end{pgfonlayer}
\end{tikzpicture}
\qquad\mbox{i.e.}\qquad  \Yon(X)\otimes\Yon(A) \xrightarrow{\Yon(g)\otimes p}\Yon(Y) \otimes P 
\]
(recalling that by the Yoneda lemma, $p$ is an element of $PX$ or, equivalently, a natural transformation $\Yon(X)\Rightarrow P$), or even as follows.
\[
\begin{tikzpicture}[scale=0.6,baseline=-0.25em]
	\begin{pgfonlayer}{nodelayer}
		\node [style=none] (0) at (-1, 1) {};
		\node [style=none] (1) at (3, 1) {};
		\node [style=none] (2) at (-3.5, 0) {$X$};
		\node [style=none] (3) at (3.5, 1) {$Y$};
		\node [style=none] (5) at (-1, -1) {};
		\node [style=none] (7) at (3.5, -1) {$P$};
		\node [style=ds horiz effect] (8) at (1, -1) {$\vphantom{f}p$};
		\node [style=none] (9) at (3, -1) {};
		\node [style=none] (10) at (-1, 0) {};
		\node [style=none] (11) at (-3, 0) {};
		\node [style=medium box,minimum height=16mm] (12) at (-1, 0) {$h$};
	\end{pgfonlayer}
	\begin{pgfonlayer}{edgelayer}
		\draw (0.center) to (1.center);
		\draw (5.center) to (8);
		\draw [style=ds] (8) to (9.center);
		\draw (11.center) to (10.center);
	\end{pgfonlayer}
\end{tikzpicture}
\qquad\mbox{i.e.}\qquad  \Yon(X) \xrightarrow{\Yon(h)} \Yon(Y)\otimes\Yon(A) \xrightarrow{\id\otimes p}\Yon(Y) \otimes P 
\]

\begin{proof}[Sketch of proof]
	As usual, let's focus on the functor case.
	The condition on the unit holds by definition.
	For products, using Yoneda reduction (\Cref{ninja}), 
	\begin{align*}
		\big(\Yon(A)\otimes\Yon(B)\big)(Z) \;&\cong\; \Coend{X,Y\in\cat{C}} \cat{C}(X\otimes Y,Z) \times \cat{C}(A,X)\times \cat{C}(B,Y) \\
		&\cong\; \cat{C}(A\otimes B,Z) \\
		&=\; \Yon(A\otimes B)(Z) . 
	\end{align*}
	The associativity and unitality conditions hold by the universal property.
\end{proof}

\subsection{Additional material}

Here are some references where one can learn more about the material presented here.
\begin{itemize}
	\item For more ideas on these ``virtual arrows'', see \cite{ellerman3}.
	\item For Kan extensions, we recommend Chapter 6 of \cite{riehl2016category}.
	\item For ends and coends, as well as Day convolution, see \cite{fosco}.
	\item For Cauchy completion, idempotent splitting and absolute limits, \cite{absolutecolimits} and \cite{borceux_idempotents}.
	\item For general notions of limits and weighted limits, \cite{pizero}, \cite{borceux1994handbook2}, \cite{kelly-enriched}, and \cite{morphcolimits}.
	\item For profunctors, see again \cite[Chapter~5]{fosco}, and \cite{distributeurs} (in French).
\end{itemize}

\bibliographystyle{alpha}
\bibliography{markov}

\begin{thebibliography}{DLR24}

\bibitem[BD86]{borceux_idempotents}
Francis Borceux and Dominique Dejean.
\newblock Cauchy completion in category theory.
\newblock {\em Cahiers de Topologie et Géométrie Différentielle
  Catégoriques}, 27(2):133--146, 1986.

\bibitem[Bor94]{borceux1994handbook2}
Francis Borceux.
\newblock {\em Handbook of Categorical Algebra II}.
\newblock Cambridge University Press, 1994.

\bibitem[Bé73]{distributeurs}
Jean Bénabou.
\newblock Les distributeurs.
\newblock {\em Université Catholique de Louvain, Institut de Mathématique
  Pure et Appliquée}, 33, 1973.

\bibitem[CK17]{dodo}
Bob Coecke and Alex Kissinger.
\newblock {\em Picturing Quantum Processes}.
\newblock Cambridge University Press, 2017.

\bibitem[DL07]{smallfunctors}
Brian~J. Day and Stephen Lack.
\newblock Limits of small functors.
\newblock {\em Journal of Pure and Applied Algebra}, 210:651--663, 2007.

\bibitem[DLR24]{topoiwithenoughpoints}
Ivan Di~Liberti and Morgan Rogers.
\newblock Topoi with enough points, 2024.
\newblock \href{http://arxiv.org/abs/2403.15338}{arXiv:2403.15338}.

\bibitem[Ell06]{ellerman1}
David Ellerman.
\newblock A theory of adjoint functors -- with some thoughts on their
  philosophical significance.
\newblock In G.~Sica, editor, {\em What is Category Theory?} Milan:
  Polimetrica, 2006.

\bibitem[Ell07]{ellerman2}
David Ellerman.
\newblock Adjoint functors and heteromorphisms, 2007.
\newblock \href{http://arxiv.org/abs/0704.2207}{arXiv:0704.2207}.

\bibitem[Ell15]{ellerman3}
David Ellerman.
\newblock {Mac Lane}, {Bourbaki}, and adjoints: A heteromorphic retrospective,
  2015.

\bibitem[Kel82]{kelly-enriched}
G.M. Kelly.
\newblock {\em Basic concepts of enriched category theory}.
\newblock Cambridge University Press, 1982.
\newblock Reprint available at
  \href{http://www.tac.mta.ca/tac/reprints/articles/10/tr10abs.html}{www.tac.mta.ca/tac/reprints/articles/10/tr10abs.html}.

\bibitem[Lei04a]{leinster-similarity1}
Tom Leinster.
\newblock A general theory of self-similarity i, 2004.
\newblock \href{http://arxiv.org/abs/math/0411344}{arXiv:math/0411344}.

\bibitem[Lei04b]{leinster-similarity2}
Tom Leinster.
\newblock A general theory of self-similarity ii: recognition, 2004.
\newblock \href{http://arxiv.org/abs/math/0411345}{arXiv:math/0411345}.

\bibitem[Lei07]{leinster-similarity3}
Tom Leinster.
\newblock General self-similarity: An overview.
\newblock {\em Real and Complex Singularities}, 2007.
\newblock \href{http://arxiv.org/abs/math/0411343}{arXiv:math/0411343}.

\bibitem[Lei11]{leinster-similarity4}
Tom Leinster.
\newblock A general theory of self-similarity.
\newblock {\em Advances in Mathematics}, 226(4), 2011.
\newblock \href{http://arxiv.org/abs/1010.4474}{arXiv:1010.4474}.

\bibitem[Lei14]{basiccats}
Tom Leinster.
\newblock {\em Basic category theory}.
\newblock Cambridge University Press, 2014.

\bibitem[Lor21]{fosco}
Fosco Loregian.
\newblock {\em Coend calculus}.
\newblock Cambridge University Press, 2021.

\bibitem[Lur09]{htt}
Jacob Lurie.
\newblock {\em Higher Topos Theory}.
\newblock Princeton University Press, 2009.

\bibitem[Par70]{pareigis}
Bodo Pareigis.
\newblock {\em Categories and Functors}.
\newblock Academic Press, 1970.

\bibitem[Par71]{absolutecolimits}
Robert Paré.
\newblock On absolute colimits.
\newblock {\em Journal of Algebra}, 19:80--95, 1971.

\bibitem[Par73]{pizero}
Robert Paré.
\newblock Connected components and colimits.
\newblock {\em Journal of Pure and Applied Algebra}, 3:21--42, 1973.

\bibitem[Par15]{morphcolimits}
Robert Paré.
\newblock Morphisms of colimits: From paths to profunctors.
\newblock {\em Applied Categorical Structures}, 23:1--28, 2015.

\bibitem[Per24]{startingcats}
Paolo Perrone.
\newblock {\em Starting Category Theory}.
\newblock World Scientific, 2024.

\bibitem[PT22]{perronetholen2022kan}
Paolo Perrone and Walter Tholen.
\newblock Kan extensions are partial colimits.
\newblock {\em Applied Categorical Structures}, 30:685--753, 2022.
\newblock \href{https://arxiv.org/abs/2101.04531}{arXiv:2101.04531}.

\bibitem[Rie16]{riehl2016category}
Emily Riehl.
\newblock {\em Category theory in context}.
\newblock Mineola, NY: Dover Publications, 2016.
\newblock
  \href{https://math.jhu.edu/~eriehl/context.pdf}{https://math.jhu.edu/$\sim$eriehl/context.pdf}.

\end{thebibliography}

\end{document}